\documentclass[12pt]{amsbook}

\newcommand{\abknotation}[1]{\left(#1\right)}

\usepackage{ifthen}
\newboolean{casewise}
\setboolean{casewise}{TRUE}

\usepackage{hyperref, amssymb}
\usepackage[all]{xy}
\usepackage{enumerate}
\usepackage{xspace}
\usepackage[table]{xcolor}
\usepackage{multirow}
\usepackage{nicematrix}
\usepackage{tikz}
\usepackage[margin]{fixme} 

\usepackage{arydshln}
\newcommand{\setof}[2]{\left\{ #1 \;\middle|\; #2 \right\}}
\newcommand{\n}{n}
\newcommand{\lt}{\operatorname{rt}}
\newcommand{\rt}{\operatorname{rt}}

\newcommand{\cok}{\operatorname{coker}}

\newcommand{\open}{\cellcolor{black}{\textcolor{white}{?}}}
\newcommand{\ee}{{\mathbf e}}
\newcommand{\BB}{\mathsf B}
\newcommand{\DD}{\mathsf D}
\newcommand{\XX}[1]{\mathsf X_{#1}}
\newcommand{\XXp}[1]{\mathsf X^+_{#1}}
\newenvironment{discu}{{\sc Discussion:}}{}
\usepackage{multirow}
\renewcommand{\phi}{\varphi}
\renewcommand{\theta}{\vartheta}

\newtheorem{theor}{Theorem}[section]
\newtheorem{propo}[theor]{Proposition}
\newtheorem{lemma}[theor]{Lemma}
\newtheorem{corol}[theor]{Corollary}

\theoremstyle{defin}
\newtheorem{defin}[theor]{Definition}

\newtheorem{conje}[theor]{Conjecture}

\theoremstyle{remark}
\newtheorem{remar}[theor]{Remark}
\newtheorem{examp}[theor]{Example}

\newtheorem{openq}[theor]{Open question}

\numberwithin{equation}{chapter}
\numberwithin{theor}{section}

\renewcommand{\thetable}{\arabic{chapter}.\arabic{table}}

\newenvironment{proofoffirstpart}[1]{{\sc Proof (first part).}}{\hfill$\square$}
\newenvironment{proofofsecondpart}[1]{{\sc Proof of #1 (second part).}}{\hfill$\square$}

\newcommand{\Ff}{\mathcal{F}}

\newcommand{\Kk}{\mathbb{K}}

\newcommand{\Oo}{\mathcal{O}}

\newcommand{\Dd}{\mathcal{D}}
\newcommand{\kk}{\underline{k}}
\newcommand{\mylll}{\underline{L}}
\newcommand{\myll}{\underline{\ell}}
\newcommand{\kkk}{\underline{K}}
\newcommand{\nn}{\underline{n}}
\newcommand{\mm}{\underline{m}}

\newcommand{\K}{\mathbb{K}}
\newcommand{\CC}{\mathbb{C}}
\newcommand{\NN}{\mathbb{N}}
\newcommand{\QQ}{\mathbb{Q}}

\newcommand{\TT}{\mathbb{T}}
\newcommand{\ZZ}{\mathbb{Z}}
\newcommand{\spaz}{\operatorname{sp}^\times}
\newcommand{\BF}{\mathsf{BF}}

\newcommand{\Asf}{\mathsf{A}}
\newcommand{\Bsf}{\mathsf{B}}

\newcommand{\xyz}[1]{\mbox{$\mathsf{#1}$}}
\newcommand{\xyzrel}[1]{\mbox{$\overline{\underline{\mathsf{ #1}}}$}}
\newcommand{\xyzLPArel}[1]{\mbox{$\overline{\overline{\underline{\underline{\mathsf{ #1}}}}}$}}
\newcommand{\III}{\mbox{\texttt{\textup{(I)}}}}
\newcommand{\IIIm}{\mbox{\texttt{\textup{(I-)}}}}
\newcommand{\IIIp}{\mbox{\texttt{\textup{(I+)}}}}
\newcommand{\TTT}{\mbox{\texttt{\textup{(T)}}}}
\newcommand{\OOO}{\mbox{\texttt{\textup{(O)}}}}
\newcommand{\CCC}{\mbox{\texttt{\textup{(C)}}}}
\newcommand{\CCCp}{\mbox{\texttt{\textup{(C+)}}}}
\newcommand{\RRR}{\mbox{\texttt{\textup{(R)}}}}
\newcommand{\RRRp}{\mbox{\texttt{\textup{(R+)}}}}
\newcommand{\PPP}{\mbox{\texttt{\textup{(P)}}}}
\newcommand{\PPPp}{\mbox{\texttt{\textup{(P+)}}}}
\newcommand{\KKKp}{\mbox{\texttt{\textup{(K+)}}}}
\newcommand{\KKKm}{\mbox{\texttt{\textup{(K-)}}}}
\newcommand{\KKK}{\mbox{\texttt{\textup{(K)}}}}
\newcommand{\SSS}{\mbox{\texttt{\textup{(S)}}}}
\newcommand{\SSSs}{\mbox{\texttt{\textup{(S*)}}}}
\newcommand{\SSSsn}{\mbox{\texttt{\textup{(S=n)}}}}
\newcommand{\SSSsnm}{\mbox{\texttt{\textup{(S<n)}}}}
\newcommand{\AAArr}{\mbox{\texttt{\textup{(Arr)}}}}
\newcommand{\AAArrp}{\mbox{\texttt{\textup{(Arr+)}}}}

\newcommand{\singgen}{\mathcal G_{\circ\star}}

\newcommand{\DT}{\mathcal{DT}}
\newcommand{\DQ}{\mathcal{DQ}}
\newcommand{\DQp}{\mathcal{DQ}_1}

\renewcommand{\emptyset}{\varnothing}
\newcommand{\id}{\operatorname{id}}

\newcommand{\idmatrix}{\mathsf{I}}
\newcommand{\zmatrix}{\mathsf{0}}

 \newcommand{\standard}{\mathsf{std}}
 \newcommand{\standardp}{\mathsf{std}^+}

\newcommand{\GF}{{\mathsf F}}
\newcommand{\GC}{{\mathsf E}}
\newcommand{\GO}{{\mathsf G}}
\newcommand{\mybox}[1]{\fbox{\xyz{#1}}}

\newcommand{\GLrel}{\mathsf{GL}}
\newcommand{\GLprel}{\mathsf{GL}^+}
\newcommand{\SLrel}{\mathsf{SL}}

\newcommand{\FKX}{\operatorname{FK}_X}
\newcommand{\FKXp}{\operatorname{FK}_X^+}
\newcommand{\FKXpu}{\operatorname{FK}_X^{+,1}}

\definecolor{lgray}{rgb}{0.9,0.9,0.9}
\newcommand{\nsp}{\cellcolor{lgray}{ }}

\newcommand{\step}[1]{\mbox{}\\\noindent{\textsc{Step} #1:}}
\newcommand{\substep}[2]{\mbox{}\\\noindent{\textsc{Step} #1(#2):}}
\newcommand{\mystep}[1]{{\textsc{Step}} #1}
\newcommand{\mysubstep}[2]{{\textsc{Step}} #1(#2)}
\newcommand{\isxyz}[3]{\xymatrix{{#1}\ar@{<~>}[rr]^-{\mathsf{#3}}&&{#2}}}
\newcommand{\ismove}[3]{\xymatrix{{#1}\ar@{<~>}[rr]^-{\langle \scriptsize #3\rangle}&&{#2}}}

\newcommand{\WDG}{weighted, decorated graph}
\newcommand{\DG}{decorated graph}

\newcommand{\DGs}{decorated graphs}

\newcommand{\SL}{{\sf SL}}
\newcommand{\SLp}{\mbox{${\sf SL}_+$}}
\newcommand{\GL}{{\sf GL}}
\newcommand{\GLp}{\mbox{${\sf GL}_+$}}

\newcommand{\RRc}{\mbox{``$\bullet\bullet$''}}
\newcommand{\RSc}{\mbox{``$\bullet\circ$''}}
\newcommand{\SRc}{\mbox{``$\circ\bullet$''}}
\newcommand{\SSc}{\mbox{``$\circ\circ$''}}
\newcommand{\RRi}{{\bullet\bullet}}
\newcommand{\RSi}{{\bullet\circ}}
\newcommand{\SRi}{{\circ\bullet}}
\newcommand{\SSi}{{\circ\circ}}

\newcommand{\nsX}{\Sigma\ }
\newcommand{\osX}{\Pi^+}
\newcommand{\tsX}{\Pi\ }
\newcommand{\oXe}[2]{\left[#1\left\|{#2}\right.\right]}
\newcommand{\regX}{{\bullet}}
\newcommand{\singX}{{\circ}}
\newcommand{\ww}{\mathbf{w}}
\newcommand{\yy}{\mathbf{y}}
\newcommand{\zz}{\mathbf{z}}
\newcommand{\xx}{\mathbf{x}}
\newcommand{\vv}{\mathbf{v}}
\newcommand{\uu}{\mathbf{1}}
\newcommand{\uuu}{\underline{\mathbf{1}}}
\newcommand{\www}{\underline{\mathbf w}}
\newcommand{\eee}{\underline{\mathbf e}}
\newcommand{\oo}{{\mathbf 0}}
\newcommand{\ooo}{\underline{\mathbf 0}}
\newcommand{\nos}{{n_\singX}}
\newcommand{\nor}{{n_\regX}}
\newcommand{\pK}[1]{\pi_{\mathcal R}(#1)}
\newcommand{\alal}{\pmb{\alpha}}
\newcommand{\bebebe}{\underline{\pmb{\beta}}}
\newcommand{\bebe}{{\pmb{\beta}}}
\newcommand{\adjRR}{\mathsf{A}^{\bullet\bullet}}
\newcommand{\adjRS}{\mathsf{A}^{\bullet\circ}}
\newcommand{\adjSR}{\mathsf{A}^{\circ\bullet }}
\newcommand{\adjSS}{\mathsf{A}^{\circ\circ}}

\newcommand{\piK}[1]{\pi_{\mathcal K}(#1)}
\newcommand{\nDT}{\Omega}

\makeindex

\title[Refined moves for structure-preserving isomorphism]{Refined moves for structure-preserving isomorphism of graph $C^*$-algebras}

	\author{S{\o}ren Eilers}
        \address{Department of Mathematical Sciences \\
        University of Copenhagen\\
        Universitetsparken~5 \\
        DK-2100 Copenhagen, Denmark}
        \email{eilers@math.ku.dk }

	\author{Efren Ruiz}
        \address{Department of Mathematics\\University of Hawaii,
Hilo\\200 W. Kawili St.\\
Hilo, Hawaii\\
96720-4091 USA}
        \email{ruize@hawaii.edu}
\date{\today}

\begin{document}

\begin{abstract}
We formalize eight different notions of isomorphism among (unital) graph $C^*$-algebras, and initiate the study of which of these notions may be described geometrically as generated by moves. We propose a list of eight types of moves that we conjecture has the property that the collection of moves respecting each of these notions of isomorphism indeed generate that notion, in the sense that two graphs are equivalent in that sense if and only if one may transform one into another using only the invariant kinds of moves.

We completely resolve  invariance properties of all moves on our list except one, and collect generation results supporting our conjectures. In particular, we prove the conjecture \emph{in toto} for the cases of
all amplified and all 
 acyclic graphs, and for all monocyclic graphs defining $C^*$-algebras with at most one non-trivial gauge invariant ideal. We also present new and very precise results for large classes of simple $C^*$-algebras.
\end{abstract}
\maketitle

\newcommand{\companion}[1]{\cite[#1]{seaseer:gciugc}}

\tableofcontents

\chapter{Preliminaries}


\section{Introduction}

The \emph{geometric classification} of unital graph $C^*$-algebras obtained by the authors with Restorff and S\o{}rensen  gives a description of the equivalence relation induced on all directed graphs with finitely many vertices (but possibly countably infinitely many edges) by stable isomorphism of their associated graph $C^*$-algebras, as the coarsest equivalence relation containing a number of moves on the graphs. Thus, in a way resembling the role played by Reidemeister moves for homotopy of knots, two graphs define the same $C^*$-algebra up to stable isomorphism if and only if there is a finite number of such moves which leads  from one graph to another. Those moves are all local in nature -- affecting only the graph in a small neighborhood of the vertex to which it is applied -- and the basic moves have their origin in symbolic dynamics, viz.\ \emph{in-splitting} and \emph{out-splitting} as defined and studied by Williams (\cite{rfw:csft}) as a way of characterizing conjugacy among shifts of finite type.

The emphasis on moves in \cite{segrerapws:ccuggs} originally served as a vehicle of providing $K$-theoretical classification of this class of non-simple 	$C^*$-algebras, but we have found that such a characterization carries substantial weight in itself, by explaining many known phenomena and providing a convenient  tool for establishing new insights.

In parallel with these efforts, it has lately been discovered in a movement pioneered by Kengo Matsumoto (cf.~\cite{km:sdsc}) that the subclass of graph algebras associated in this way to finite essential graphs -- the so-called Cuntz-Krieger algebras -- have profound rigidity properties when one considers them not as $C^*$-algebras alone, but as $C^*$-algebras equipped with natural extra structure. More precisely, the symbolic dynamical systems defined as shifts of finite type associated to such graphs are remembered by the $C^*$-algebras at varying level of precision depending upon how much structure is considered. Key objects are the diagonal and the gauge action  either associated to the Cuntz-Krieger algebra $\Oo_A$ itself (denoted $\Dd_A$ and $\gamma^A$, respectively), or associated to the stabilization	$\Oo_A\otimes\K$  ($\Dd_A\otimes c_0 $ and $\gamma^A\otimes \id$, respectively). Table \ref{varres} 
lists a collection of such results presently known, showing that standard dynamical notions such as conjugacy and flow equivalence are indeed rigidly remembered by the operator algebras, and providing motivation for the study of new concepts of sameness of such dynamical systems.
\begin{table}[t]
\begin{center}
\begin{tabular}{r|r|r|r}
\emph{Notion}&\emph{Data}&\emph{Which SFTs?}&\emph{Ref.}\\
\hline\hline
Flow equivalence&$(\Oo_A\otimes\K,\Dd_A\otimes c_0)$&Irreducible&\cite{kmhm:coetmscka}\\\cline{3-4}
&&All&\cite{ceor:coe}\\\hline
Conjugacy&$(\Oo_A\otimes\K,\Dd_A\otimes c_0,\gamma^A\otimes\id)$&All&\cite{tmcjr:dgigc}\\\hline
Shift equivalence&$(\Oo_A\otimes\K,\gamma^A\otimes\id)$&Primitive&\cite{obak:tsacsa}\\\cline{3-4}
&&Irreducible&\cite{segs}\\\hline
Continuous orbit&$(\Oo_A,\Dd_A)$&Irreducible&\cite{km:oetmscka}\\\cline{3-4}
 equivalence&&No isolated points&\cite{nbtmcmfw:gaoe}\\\cline{3-4}
&&All&\cite{seaseer:dcdpgc}\\
&&&\cite{tmcmlw:oegigg}\\\hline
Eventual conjugacy&$(\Oo_A,\Dd_A,\gamma^A)$&Irreducible&\cite{km:ucoemsgacka}\\\cline{3-4}
&&All&\cite{tmcjr:dgigc}
\end{tabular}
\end{center}
\caption{Assorted rigidity results}\label{varres}
\end{table}

To systematically address questions of this nature, we propose here a nomenclature to express isomorphisms of eight different type by means of 3-bit words that we always denote \xyz{xyz} with $\xyz{x},\xyz{y},\xyz{z}\in\{\xyz{0},\xyz{1}\}$. The first bit determines whether the isomorphism is exact ($\xyz{x}=\xyz{1}$) or just stable ($\xyz{x}=\xyz{0}$), the second whether it must commute with the gauge action (appropriately extended to the stabilization when $\xyz{x}=\xyz{0}$), and the third whether it must send diagonals (again suitably extended when $\xyz{x}=\xyz{0}$) to diagonals.

The eight cases are best considered as the corners of a cube
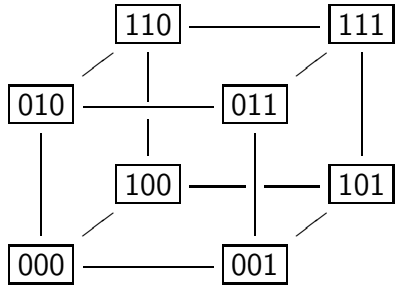
\[
\xymatrix@=0.3cm{
&&\mybox{110}\ar@{-}[rr]&&\mybox{111}\\
&\mybox{010}\ar@{-}[rr]\ar@{-}[ur]&&\mybox{011}\ar@{-}[ur]&&\\
&&\mybox{100}\ar@{-}'[r][rr]\ar@{-}'[u][uu]&&\mybox{101}\ar@{-}[uu]&\\
&\mybox{000}\ar@{-}[rr]\ar@{-}[uu]\ar@{-}[ur]&&\mybox{001}\ar@{-}[ur]\ar@{-}[uu]&&
}
\]
where the data considered by the indicated types of isomorphisms is
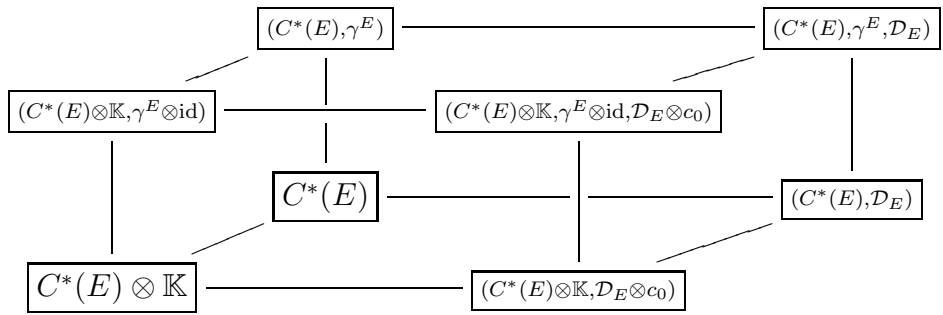
\[
\xymatrix@=0.3cm{
&&\fbox{$\scriptstyle (C^*(E),\gamma^E)$}\ar@{-}[rr]&&\fbox{$\scriptstyle (C^*(E),\gamma^E,\Dd_E)$}\\
&\fbox{$\scriptstyle (C^*(E)\otimes \K,\gamma^E\otimes\id)$}\ar@{-}[rr]\ar@{-}[ur]&&\fbox{$\scriptstyle (C^*(E)\otimes\K,\gamma^E\otimes\id,\Dd_E\otimes c_0)$}\ar@{-}[ur]&&\\
&&\fbox{$C^*(E)$}\ar@{-}'[r][rr]\ar@{-}'[u][uu]&&\fbox{$\scriptstyle (C^*(E),\Dd_E)$}\ar@{-}[uu]&\\
&\fbox{$C^*(E)\otimes \K$}\ar@{-}[rr]\ar@{-}[uu]\ar@{-}[ur]&&\fbox{$\scriptstyle (C^*(E)\otimes \K,\Dd_E\otimes c_0)$}\ar@{-}[ur]\ar@{-}[uu]&&
}
\]
It is essential to note from the outset that since our requirement is that there is a single isomorphism preserving all structure, it is potentially stronger to satisfy a relation with two bits set than two relations each with one bit set, and so on. 

Among the results in Table \ref{varres}, the result of Carlsen and Rout (\cite{tmcjr:dgigc}) that conjugacy of two-sided shifts of finite type is exactly reflected by the stabilized Cuntz-Krieger algebra equipped with the stabilized diagonal and the gauge action trivially extended to the stabilization is the foremost source of motivation for the present work, as it
contains a corollary of direct relevance to geometric classification. Indeed, Williams established already in 1973 (\cite{rfw:csft}) that  when two shift of finite types are conjugate -- i.e. when the two essential finite graphs are \xyz{011}-equivalent -- one may transform one to the other by a finite number of in- or out-splittings, or their inverses. Since such splitting moves are already on the list of moves considered in \cite{segrerapws:ccuggs}, and may be seen to be \xyz{011}-invariant, this begs the question of whether it is the case that the various kinds of sameness defined at the level of operator algebras to reflect the needs of rigidity results are generated as the smallest equivalence relations containing those moves that preserve the relevant notions.

It turns out that any result of this nature fails for the collection of moves previously considered, so the question must be recast as: Does there exist a list of moves generating stable isomorphism among unital graph $C^*$-algebras with the property that those moves that preserve further structure also generate the relevant refined equivalence relations? Obtaining such a list of moves would be a key step towards a unified approach to understanding and comparing the various concepts of sameness, where open questions persist even restricted to the types of graphs considered in  the purely dynamical cases. We are hopeful that it may even cast light on the Williams conjecture in symbolic dynamics, and the corresponding question of decidability of strong shift equivalence. 

We present here a  collection of moves
\begin{equation}\label{movelist}
\OOO,\IIIp,\IIIm,\RRRp,\SSS,\CCCp,\PPPp,\KKKp
\end{equation}
 that we strongly believe has this generating property.

We will systematically prove that the invariance properties of each type of move are as indicated in Table \ref{sevenprops}, and establish -- at varying level of  generality -- a portfolio of theorems of generation which to us serves as evidence for the conjecture that indeed the family of moves we present here have the desired properties. In one of the seven new cases we are considering  we are able to prove so in full generality, and the proof of this fact was given in \cite{seaseer:gciugc} jointly with Arklint, which is to be considered a companion paper to the present one. The remaining six we can only establish under varying added assumptions on the graphs, and in the last part of this work we will present a large collection of partial results that support our conjectures, and outline the outstanding technical issues.

We have focused our efforts on two important subclasses, in both of which we are able to establish the conjectures fully. The  class of \emph{finite graphs defining gauge simple graph $C^*$-algebras}
contains the graphs defining simple Cuntz-Krieger algebras, which has been the object of intense scrutiny over several decades. This has established a clear connection to symbolic dynamics which allows us to prove our conjectures by a careful reduction to that subcase. We also complete the analysis of the class of \emph{acyclic graphs}, which are exactly those that define unital AF algebras, and for some new cases of \emph{type I} graph $C^*$-algebras. Such results follow easily from our earlier work, or by ad hoc methods, whenever the gauge action is not considered, but requires new methods when it is.


\begin{table}
\begin{center}
\begin{tabular}{|c|c|c|c|r|r|r|}\hline&\xyz{x}&\xyz{y}&\xyz{z}&Definition&Invariance&Non-invariance\\\hline\hline
\OOO&\xyz{1}&\xyz{1}&\xyz{1}&\ref{OOO}&\ref{thm:outsplitting}&$-$\\\hline
\IIIp&\xyz{1}&\xyz{1}&\xyz{1}&\ref{IIIp}&\ref{thm:insplittingplus}&$-$\\\hline
\IIIm&\xyz{0}&\xyz{1}&\xyz{1}&\ref{IIIm}&\ref{thm:insplitting2}&\ref{OIIniIOO}\\\hline
\RRRp&\xyz{1}&\xyz{0}&\xyz{1}&\ref{RRRp}&\ref{thm:reduction-plus}&\ref{IOIniOIO}\\\hline
\SSS&\xyz{0}&\xyz{0}&\xyz{1}&\ref{SSS}&\ref{thm:source}&\ref{OIIniIOO}, \ref{IOIniOIO}\\\hline
\CCCp&\xyz{1}&\xyz{0}&\xyz{0}&\ref{CCCp}&\ref{thm:iso-cuntzsplice+}&\ref{psexplus}\\\hline
\PPPp&\xyz{1}&\xyz{0}&\xyz{0}&\ref{PPPp}&\ref{thm:iso-pulelehua+}&\ref{psexplus}\\\hline
\KKKp&\xyz{1}&\xyz{1}&\xyz{0}&\ref{KKKp}&(open)&\ref{kimroushplus}\\\hline\end{tabular}\end{center}
\caption{Invariance properties of the eight moves}\label{sevenprops}
\end{table}

The first seven moves in \eqref{movelist} are honestly geometric, and in all cases but two rather small variations of moves that are already in the literature, changed slightly by an addition of sources to ensure that the moves leaving the graph $C^*$-algebra rather than its stabilization invariant. The one true innovation in this collection of moves is a complete rethinking of the concept of \emph{in-splitting}. The lack of symmetry between out- and insplitting in the setting of graph $C^*$-algebras was understood already in Bates and Pask's trailblazing paper \cite{tbdp:fega}, where they showed that whereas Williams' out-splitting can be transplanted to operator algebras  in a way which does not change the graph $C^*$-algebra itself (or indeed, as we shall see, any other relevant structure), that is impossible for in-splitting which is only an invariant of the stabilization.
But whereas this did not carry any further weight in  \cite{tbdp:fega} or \cite{segrerapws:ccuggs}, as soon as one attempts to keep track of the $C^*$-algebras ``on the nose'', or of information stored in the gauge action, it becomes clear that the definition of in-splitting in \cite{tbdp:fega}  is insufficient.

The solution we present here is to replace the traditional in-splitting move by two new types of moves, called \IIIm\ and \IIIp. Here \IIIm\ is a rather innocuous variation of Williams' in-split allowing empty sets in the partition defining the in-splitting, thus introducing sources in a way that we shall see is \xyz{011}-invariant.  The more restrictive \IIIp\ move specializes this to allow the passage between two graphs that result from two different in-splittings of the same graph \emph{having the same number of sets in their partition}. This, as we shall prove, preserves all relevant structure of the graph $C^*$-algebra and is hence \xyz{111}-invariant.

As an example, consider the graph
\begin{equation}\label{original}
\xymatrix{\bullet\ar@/^/@<-1.5mm>[r]\ar@/^/@<-0.5mm>[r]\ar@/^/@<0.5mm>[r]\ar@/^/@<1.5mm>[r]&\bullet\ar@/^1em/[l]}
\end{equation}
and in-split the edges having the rightmost vertex as their ranges into two sets evenly and unevenly to obtain
\begin{equation}\label{bc}
\xymatrix{&\bullet\ar[dr]\ar[dl]\ar@<-1mm>[dl]\ar@<1mm>[dl]\\\bullet\ar@/^2em/[ur]&&\bullet\ar@/_2em/[ul]}\qquad
\xymatrix{&\bullet\ar@<-0.5mm>[dr]\ar@<0.5mm>[dr]\ar@<-0.5mm>[dl]\ar@<0.5mm>[dl]\\\bullet\ar@/^2em/[ur]&&\bullet\ar@/_2em/[ul]}
\end{equation}
where as usual, the incoming edges are distributed according to the partition and the outgoing edges are duplicated. We will say one of  these two graphs arises from the other by an \IIIp\ move, and will see that their associated $C^*$-algebras are the same in our strongest  (\xyz{111}) sense. They are also \IIIp- and \xyz{111}-equivalent to 
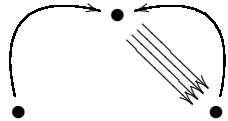
\begin{equation}\label{bcalt}
\xymatrix{&\bullet\ar@<0.5mm>[dr]\ar@<-0.5mm>[dr]\ar@<1.5mm>[dr]\ar@<-1.5mm>[dr]\\\bullet\ar@/^2em/[ur]&&\bullet\ar@/_2em/[ul]}
\end{equation}
which is obtained by partitioning the four incoming vertices into two sets, one of which is empty. These three graphs are \textbf{not} \xyz{111}-equivalent to that in \eqref{original}.

We were led to \IIIp\ by Example 3.6 of \cite{kabtmc:ckaocsftg} which exactly compares the two graphs in \eqref{bc}, but found that the idea of considering only pairs of partitions with the same number of nonempty set was too restrictive to be generating in any useful sense. Allowing empty sets seems to provide exactly the right level of generality. A good way of thinking of the \IIIp\ move is that whenever two or more vertices have exactly the same future, one may \emph{redistribute their pasts among them freely}. It is natural to consider the \IIIp\ move as an operation on the class of one-sided shift spaces associated to finite graphs without sinks, and  work by Kevin Aguyar Brix (\cite{kab:bssebiec}) shows that the \IIIp\ move in conjunction with Williams' original out-split generates a notion of sameness of such shift spaces in the same way that Williams' original moves generate conjugacy of two-sided shift spaces.

It turns out, however, that these seven moves are not sufficient to generate the two types of isomorphism that we denote \xyz{110} and \xyz{010}, and we do not have any obvious candidates for local and honestly geometric moves to fill this gap.
In order to address all eight cases in a uniform manner, we hence introduce here a ``Krieger move'' \KKKp\ which is of a dramatically different nature than the others, being arithmetic and global rather than geometric and local in nature with reference to the dimension triples introduced in \cite{wk:dftmc} which we augment to \emph{dimension quadruples} necessary at our level of generality. As opposed to situation for the collection of geometric moves, it follows easily that this move generates \xyz{110}-equivalence, whereas establishing invariance is a prominent open problem known from graded algebras as the \emph{Hazrat conjecture}. We show the weaker result that the move is  \xyz{100}-invariant, and analyze the structure of the dimension quadruples carefully. This leads us to confirm Hazrat's conjecture in some new cases, and we hope that our approach may help shed light on this extremely important outstanding problem in general, but we must leave several key questions in conjectural form.

Apart from the Hazrat conjecture, our undertaking  is motivated by three other, probably more distant, goals which we will conclude by listing.
\begin{itemize}
\item
The question of \emph{decidability} of each of the notions of isomorphisms amongst graphs with finitely many vertices is extremely interesting and contains the outstanding key issue of whether or not Willliams' strong shift equivalence is decidable. Outside of the \xyz{x00} case which has been known to be decidable since  \cite{segrerapws:ccuggs}, we  have no general evidence,  but  we can present assorted results in Theorems  \ref{classifyamplified}, \ref{singxOz},\ref{oneidealstable}, \ref{oneidealexact}, Proposition \ref{monoxOz} and Corollary \ref{acyclicxIz} that may allow for a recasting of the problem in a constructive way.
\item  The \emph{Abrams-Tomforde conjectures} can be generalized to ask the question of whether or not the eight types of isomorphisms we are studying here coincide with algebraic isomorphisms of the Leavitt path algebras associated to the same pair of graphs, and indeed such questions were answered affirmatively in all \xyz{xy1} cases in sweeping work by Carlsen and Rout (\cite{tmcjr:dpgisa}). Combining these results with ours leads to several new results for some \xyz{xy0} settings, as we will detail in Appendix \ref{LPA}, and may conceivably be of use in answering this subtle question.
\item  It is natural to use our framework to extend manifestly dynamical notions on the graphs defining symbolic dynamical systems to general graphs and study them as such, and with generalized tools. In fact in the case of \xyz{101}-equivalence this is already well  developed in existing work through groupoid theory and topological full groups, and we see no reason for the journey to stop there.
\end{itemize}
\subsection{Acknowledgments}

We would like to thank Sara E.~Arklint, Becky Armstrong, Kevin Aguyar Brix, Toke Meier Carlsen, James Gabe, Roozbeh Hazrat, Kengo Matsumoto, David Pask, Gunnar Restorff, Aidan Sims, and Lia Va\v s for sharing insights, references and ideas with us at crucial points in this project.

This research is part of the EU Staff Exchange project 101086394 ``Operator Algebras That
One Can See''. Also, the first named author was supported by the Independent Research Fund Denmark 
grants no.\ 7014-00145B and 1026-00371B.  He also acknowledges support from the Rise network H2020-MSCA-RISE-2015-691246-QUANTUM DYNAMICS in the early phases of this work. The second named author was supported by
a Simons Foundation Collaboration Grant, \# 567380. 

\section{Fundamentals}

\subsection{Notation and conventions}

We use the definition of graph $C^*$-algebras in \cite{njfmlir:cig} in which sinks and infinite emitters are singular vertices, and always consider $C^*(E)$ as universal $C^*$-algebras generated by Cuntz-Krieger\index{Cuntz-Krieger families}\index{Cuntz-Krieger relations}\index{CK@(CK1)--(CK3)} families $\{s_e,p_v\}$ with $e$ ranging over edges and $v$ ranging over vertices in $E$.  We expect the reader to be familiar with fundamental graph $C^*$-algebra theory, but for easy reference we list the Cuntz-Krieger relations below:
\begin{enumerate}[(CK1)]
\item $s_e^* s_e = p_{r(e)}$ for every $e\in E^1$
\item $s_e s_e^* \leq p_{s(e)}$ for every $e\in E^1$
\item $ p_v = \sum_{s(e)=v} s_e s_e^*$ for every $v\in E^0_{\text{reg}}$.
\end{enumerate}
In drawings we always use $\circ$ for singular vertices and $\bullet$ for regular vertices, and we say that a graph is regular, resp.~singular, if all of its vertices are regular, resp.~singular. \index{regular!vertex} \index{regular!graph}  \index{singular!graph}\index{singular!vertex} 

{\bf Unless stated otherwise, graphs $E$, $F$ will always be considered as having finitely many vertices and finitely or countably infinitely many edges.}

We will need to appeal to the full range of standard results on the structure of graph $C^*$-algebras, and establish some terminology and notation that will be used throughout the paper.

\subsection{\xyz{xyz}-isomorphism}\label{xIztools}

As usual, the \emph{diagonal}\index{diagonal} of any graph $C^*$-algebra $C^*(E)$ is defined as
\[
\Dd_E=\overline{\operatorname{span}}\{s_\mu s_\mu^*\mid \mu \text{ a finite path in $E$}\}
\]
and the \emph{gauge action} defines  $\gamma^E_z\in \operatorname{Aut}(C^*(E))$ for each $z\in \TT$ by
\[
\gamma^E_z(s_e)=zs_e\qquad \gamma^E_z(p_v)=p_v.
\]
We will suppress the superscript $E$ when $E$ is clear from the context.

We are ready to formally state our basic notion of \xyz{xyz}-equivalence.\index{xyz@$\xyz{xyz}$-equivalence}\index{xyz@$\xyzrel{xyz}$}

\begin{defin}\label{xyzdef}
With $\xyz{x}, \xyz{y}, \xyz{z} \in \{ \xyz{0}, \xyz{1}\}$ we 
say that $E$ and $F$ are \xyz{xyz}-equivalent when there exists a $*$-isomorphism $\phi \colon C^*(E) \otimes \K \to C^*(F) \otimes \K$ which additionally satisfies 
\begin{itemize}
\item $\phi ( 1_{ C^*(E) } \otimes e_{11} ) = 1_{ C^*(F) } \otimes e_{11}$ when $\xyz{x} = \xyz{1}$

\item $\phi \circ ( \gamma_z^E \otimes \id_{ \K } ) = ( \gamma_z^F \otimes \id_{\K} ) \circ  \phi$ when $\xyz{y} = \xyz{1}$

\item $\phi ( \mathcal{D}_E \otimes c_0 ) = \mathcal{D}_F \otimes c_0$ when $\xyz{z} = \xyz{1}$.
\end{itemize}
\end{defin}

Here, $c_0$ is the canonical diagonal of $\K$ (the compacts on a separable Hilbert space).\index{c0@$c_0$}\index{K@$\K$}

We use the notation $\xyzrel{xyz}$ to refer to the equivalence relation among graphs with finitely many vertices defined this way, considered as a subset of $\mathcal G\times \mathcal G$, with $\mathcal G$ denoting the set of all graphs with finitely many vertices. We use $\langle \cdot \rangle$ to denote the smallest equivalence relation generated by a collection of relations, allowing use to state succinctly by 
\begin{equation}\label{xyzbracket}
\xyzrel{xyz}=\langle X_1,\dots, X_n \rangle
\end{equation}
that some collection of moves generates some notion of sameness. We will often need to pass to some subclass $\widetilde{\mathcal G}\subseteq \mathcal G$, and whenever 
\[
\xyzrel{xyz}\cap (\widetilde{\mathcal G}\times \widetilde{\mathcal G})=\langle X_1,\dots, X_n \rangle\cap (\widetilde{\mathcal G}\times \widetilde{\mathcal G})
\]
we will consistently use the phrasing that \eqref{xyzbracket} \emph{holds in the class $\widetilde{\mathcal G}$}. We will also state with $\subseteq$ and $\supseteq$ in a similar way that one relation is finer or coarser than another, in general or in specialized situations. We always tacitly identify isomorphic graphs, and denote the relation of graph isomorphism by $\langle\rangle$.

\begin{remar}
Note that all the \xyz{1yz}-relations may be formulated directly, as properties for $\phi:C^*(E)\to C^*(F)$.
Note also that there is more than one reasonable way to extend $\gamma^E$ to $C^*(E)\otimes \K$. Our choice is motivated by \cite{tmcjr:dgigc} and differs from the standard choice in Leavitt path algebras (cf.\ \cite{rh:ggclpa}). 
\end{remar}

\subsection{Admissible pairs}\label{adpairs}

  Let $E$ be a graph and let $H \subseteq E^0$.   The set $H$ is said to be \emph{hereditary}\index{hereditary vertex sets} provided that for all $e \in E^1$, $s(e) \in H$ implies $r(e) \in H$ and the set $H$ is  said to be \emph{saturated}\index{saturated vertex sets} provided that for all regular vertices $v$, $r(s^{-1}(v)) \subseteq H$ implies $v \in H$.  For a hereditary and saturated subset $H$ of $E^0$, set 
$$
B_H := \{ v \in E^0 : | s^{-1}(v)|=\infty \text{ and } 0 < | s^{-1}(v) \cap r^{-1}(E^0 \setminus H)| < \infty\}.
$$
An \emph{admissible pair}\index{admissible} is an ordered pair $(H,S)$, where $H$ is a hereditary and saturated subset of $E^0$ and $S$ is a subset of $B_H$.  The set of all admissible pairs becomes a lattice when given the following ordering 
$$
(H, S) \leq (H',S') \quad \iff \quad H \subseteq H' \text{ and } S \subseteq H' \cup S'.
$$
For an admissible pair $(H,S)$, $I_{(H,S)}$, will denote the ideal of $C^*(E)$ generated by 
$$
\{ p_v: v \in H\} \cup \left\{ p_w^H := p_w - \sum_{ \substack{ s(e)=w \\ r(e) \notin H }} s_e s_e^* : w \in S\right\}.
$$
By \cite[Theorem~3.6 and Corollary~3.10]{bhrs:iccig}, the map $(H,S) \mapsto I_{(H,S)}$ is a lattice isomorphism from the lattice of admissible pairs of $E$ to the lattice of gauge-invariant ideals of $C^*(E)$.

\subsection{Classification}\label{classummary}

For extensive use in the paper, we will summarize and streamline notation from our earlier work 
 \cite{segrerapws:ccuggs,segrerapws:gcgcfg}. 
 
We denote adjacency matrices by  $\mathsf{A}_E$, allowing entries ``$\infty$'' at infinite emitters. When studying graphs by their matrices, we often impose the following three modest assumptions:
\begin{enumerate}[(i)]
\item There is at most one regular source;
\item Any regular vertex which is not a source supports a path back to itself;
\item Any infinite emitter emits with zero or infinite multiplicity to any vertex.\fxnote{Does $\GL$ still imply $\xyzrel{000}$ if we replace (iii) by ``no breaking vertices''?}
 \end{enumerate}
 These conditions allow for a convenient reorganization of the data in $\mathsf{A}_E$. First, we subdivide $\mathsf{A}_E$ into\index{AE@$\mathsf{A}^\bullet_E$}\index{AE@$\mathsf{A}^\circ_E$}
\[
\left[ \begin{array}{c|ccc} 0&a_{sv_1}&\cdots &a_{sv_n}\\\hline
 0&\\
 \vdots&&\mathsf{A}^\bullet_E&\\
 0&\\\hline
 0&\\
 \vdots&&\mathsf{A}^\circ_E&\\
 0&
 \end{array}\right]
  \]
  where the vertices have been reordered with the regular source $s$ first (if it exists), then other regular vertices, and then singular vertices. We define, for use with $K$-theory,\index{BE@$\mathsf{B}^\bullet_E$}\
  \[
  \mathsf{B}^\bullet_E=(\mathsf{A}^\bullet_E-\idmatrix)^T.
  \]
  (often dropping the ``$\bullet$'' from the notation when the graph is regular), and 
  collect the multiplicities of edges emitting from $s$ in a column vector
  \[
  \mathsf{D}_E=\begin{bmatrix}a_{sv_1}+1&\cdots &a_{sv_n}+1\end{bmatrix}^T.
  \]
When there is no source, we set
  \[
  \mathsf{D}_E=\begin{bmatrix}1&\cdots &1\end{bmatrix}^T.
  \]

We may (by (ii)) partition the remaining elements of $E^0$ into sets of \emph{generalized components} \index{generalized components} which are either maximal strongly connected subsets (with a non-trivial path between any two vertices -- including from any vertex back to itself) or singletons with singular vertices. We note further that (by (iii)) any hereditary set $H$ is automatically saturated, and that always $B_H=\emptyset$. Consequently 
the admissible pairs of the previous section are exactly given by the hereditary  subsets, and such sets are unions of generalized components.

This makes it very easy to read off the structure of the lattice of gauge invariant ideal as the hereditary subsets of the set $\mathcal P$ of generalized components, ordered so that $c_1 \geq c_2$ when there is a path from one (hence all) vertex in $c_1$ to one (hence all) in $c_2$.
Reordering the vertices as needed, we may arrange the adjacency matrices $\mathsf{A}_E$ and  $\mathsf{A}^\bullet_E$ in lower triangular block matrices where all blocks corresponding to index pairs $(i,j)\in\mathcal P\times \mathcal P$ for which $i\not\leq j$ vanish.
Using the same ordering as for   $\mathsf{A}_E$ and  $\mathsf{A}^\bullet_E$, the matrices $\mathsf{B}_E$ and  $\mathsf{B}^\bullet_E$ are then upper block triangular.

Note that whenever $(E,F)\in\xyzrel{000}$, the associated  sets $\mathcal P_E$ and $\mathcal P_F$ are isomorphic, and in this case we often tacitly identify to one set $\mathcal P$. Note that very often, there are several choices of identification -- we will assume one has been fixed in the notation, but of course any of the finitely many solutions are relevant.
In the other direction, we say that a pair of graphs $(E,F)$ is \emph{matched}\index{matched graphs} when both $E$ and $F$  satisfy (i)--(ii) above, they have the same $\mathcal P$,  and when the number of regular as well as singular vertices agree within all of the corresponding generalized components agree. This means that all of the matrices $\mathsf{A},\mathsf{A}^\bullet,\mathsf{B},\mathsf{B}^{\bullet}$ can be arranged to agree in size blockwise.

We say that two matched graphs are 
 \emph{\GL-equivalent}\index{GL-equivalence@\GL-equivalence} when they both satisfy (i)--(iii) above, their $\mathcal P$ sets agree, and there exist block triangular matrices $U,V$ with vanishing blocks whenever $i\not\leq j$, and such that all diagonal blocks are always invertible, and equal to $[1]$ when they are $1\times 1$, satisfying
\begin{equation}\label{GLdef}
U\mathsf{B}^\bullet_E=\mathsf{B}^\bullet_FV
\end{equation}
We write $(E,F)\in\GL$ in this case, but note that $\GL$-equivalence is only a reflexive relation when restricted to the set of graphs satisfying (i)--(iii).

We say that a graph satisfying (i)--(iii) is in \emph{canonical form}\index{canonical form} when further
\begin{enumerate}[(i)]\addtocounter{enumi}{3}
\item Whenever there is a path from $v$ to $w$, there is an edge from $v$ to $w$;
\item Whenever there is a path from $v$ to $w$, and $v$ is an infinite emitter, there are infinitely many edges from $v$ to $w$;
\item 
If there are two different paths from $v$ back to itself (neither visiting $v$ along the way), then 
\begin{enumerate}
\item  $v$ supports  two loops;
\item There is at least 3 regular vertices in the generalized component containing $v$; 
 \item The Smith form of the restriction of $\BB^\bullet$ to the generalized component containing $v$ has at least two ones in its diagonal;
\end{enumerate} 
\item 
Every regular vertex supports a loop.
\end{enumerate}
Finally, we say that a pair of graphs are in \emph{standard form}\index{standard form} if the pair is matched, and the graphs are both in canonical form.

The reader may have noticed that (vii) renders the mentioning of sources in (i) void, and that (iv) implies (iii). The weaker condition for a matched pair is sometimes convenient, and we allow sources in preparation for a generalization which we present in Section \ref{GLandfriends}.

This discussion sets the stage for statements (i) and (iii) of the main result of \cite{segrerapws:ccuggs}, which we give in the form:

\begin{theor}[{\cite[Theorems 3.1, 14.2]{segrerapws:ccuggs}}]\label{ERRSmain}
For any two graphs $E$ and $F$ the statements
\begin{enumerate}[(i)]
\item $(E,F)\in \xyzrel{000}$
\item $\FKXp(C^*(E))\simeq \FKXp(C^*(F))$
\end{enumerate}
are equivalent, and when $(E,F)$ is matched, they are implied by
\begin{enumerate}[(i)]\addtocounter{enumi}{2}
\item $(E,F)\in \GLrel$.
\end{enumerate}
When $(E,F)$ are in standard form, all statements are equivalent.
\end{theor}

We will elaborate in Section \ref{GLandfriends} below  on how it is always possible to arrange for standard form in the appropriate sense, fully and algorithmically relating  \xyz{000}-equivalence to \GL-equivalence.

The invariant $\FKXp(-)$ used in (ii) is the \emph{ordered, (gauge) filtered $K$-theory} obtained by computing the six term exact sequence for any extension of gauge-invariant subquotients of the $C^*$-algebras as we will now detail. We use the language of $C^*$-algebras \emph{over a topological space}\index{Cs-algebra over X@$C^*$-algebra over $X$} $X$, which associates an ideal $A[U]$ to any open subset $U\subset X$ and a subquotient $A[V\backslash U]$ to any pair $U\subset V\subset X$ of open sets, defined uniquely as\index{FKx@$\FKX(-)$}\index{FKxp@$\FKXp(-)$}
\[
A[U\backslash V]:=A[V]/A[U],
\]
so that
\begin{equation}\label{UVdefs}
\xymatrix
{0\ar[r]&A[U]\ar[r]^-{\iota _U^V}&A[V]\ar[r]^-{\kappa_U^V}&A[U\backslash V]\ar[r]&0}
\end{equation}
is exact. We define $\FKX(A)$ as the collection of all the data in the six term exact sequence in $K$-theory applied to all the $C^*$-extensions in \eqref{UVdefs}, identified whenever the defining subsets of $X$ agree. Consequently, two $C^*$-algebras over $X$ have isomorphic $\FKX(-)$ when the $K$-groups of the distinguished ideals and subquotients are isomorphic with the isomorphisms commuting with the maps
\[
(\iota_U^V)_*,(\kappa_U^V)_*,\partial_U^V,
\]
the latter denoting the index maps. We define $\FKXp(-)$ using the canonical order on $K_0(-)$, requiring the $K_0$-part of the isomorphisms to preserve the order.,
When $C^*(E)$ is a graph $C^*$-algebra with finitely many gauge-invariant ideals we always take the set  $X = \mathrm{Prim}_\gamma (C^*(E))$, denoting the set of all proper ideals that are prime within the set of proper gauge invariant ideals.  By \cite[Lemma~3.7]{segrerapws:ics}, $C^*(E)$ is a $C^*$-algebra with distinguished ideals which are exactly the gauge-invariant ideals of $C^*(E)$.  Let $U$ be an open subset of $X$. 
It is straightforward to define $X$ as an Alexandroff space given by $\mathcal P$ above, and the lattice of open sets is exactly the the same lattice that we presented using $\mathcal P$.
Just as in that discussion, (ii) of Theorem \ref{ERRSmain} is only true when these spaces are isomorphic, and when they are, an explicit identification is chosen.

We also prove in \cite{segrerapws:ccuggs}:

\begin{theor}\label{ERRSunital}
For any two graphs $E$ and $F$ the statements
\begin{enumerate}[(i)]
\item $(E,F)\in \xyzrel{100}$
\item $\FKXpu(C^*(E))\simeq \FKXpu(C^*(F)) $
\end{enumerate}
are equivalent, and when $(E,F)$ are matched, they are implied by
\begin{enumerate}[(i)]\addtocounter{enumi}{2}
\item $(E,F)\in \GLprel$.
\end{enumerate}
\end{theor}

The invariant $\FKXpu(A)$\index{FKxpu@$\FKXpu(-)$} in (ii) is obtained from $\FKXp(A)$  when $A$ is unital
by keeping track of  the class $[1_A]$  in all groups of the form $K_0(A[X\backslash V])$, and requiring that they are preserved by any isomorphism. 
Finally, a $\GL$-equivalent pair $(E,F)$ lies in $\GLp$ if $(U,V)$ of \eqref{GLdef} can be chosen so that 
\[
U\mathsf{D}_E-\mathsf{D}_F\in \operatorname{im} \mathsf{B}_F^\bullet.
\]
This notation has been chosen so that this condition exactly says that $U$ induces a map on $K$-theory which sends the unit of $K_0(C^*(E))$ to the unit of  $K_0(C^*(F))$.\index{GLp-equivalence@$\GLp$-equivalence}

It is in fact true that all statements are equivalent for pairs of graphs in standard form, but obtaining this form is not generally possible. Following \cite{seaseer:gciugc}, we will provide a better result in Section \ref{GLandfriends}.

\part{Invariance of moves}\label{invpart}

\chapter{Refined geometric moves}


In this section we define seven types of geometric moves and establish their invariance properties.

\section{Outsplitting}

Our first move is a standard outsplitting move, which is exactly as in \cite{tbdp:fega}. We include the definition for completeness, and check that the moves preserve all relevant structure via the canonical $*$-isomorphism studied in that paper.

\begin{defin}[Move \OOO: Outsplit at a non-sink]\label{OOO}\index{O@\OOO\ move}
Let $E = (E^0 , E^1 , r, s)$ be a graph, and let $w \in E^0$ be a vertex that is not a sink. 
Partition $s^{-1} (w)$ as a disjoint union of a finite number of nonempty sets
$$s^{-1}(w) = \mathcal{E}_1\sqcup \mathcal{E}_2\sqcup \cdots \sqcup\mathcal{E}_n$$
with the property that at most one of the $\mathcal{E}_i$ is infinite. 
Let $E_O$ denote the graph $(E_O^0, E_O^1, r_O , s_O )$ defined by 
\begin{align*}
E_O^0&:= \setof{v^1}{v \in E^0\text{ and }v \neq w}\cup\{w^1, \ldots, w^n\} \\
E_O^1&:= \setof{e^1}{e \in E^1\text{ and }r(e) \neq w}\cup \setof{e^1, \ldots , e^n}{e \in E^1\text{ and }r(e) = w} \\
r_{O} (e^i ) &:= 
\begin{cases}
r(e)^1 & \text{if }e \in E^1\text{ and }r(e) \neq w\\
w^i & \text{if }e \in E^1\text{ and }r(e) = w
\end{cases} \\
s_{O} (e^i ) &:= 
\begin{cases}
s(e)^1 & \text{if }e \in E^1\text{ and }s(e) \neq w \\
s(e)^j & \text{if }e \in E^1\text{ and }s(e) = w\text{ with }e \in \mathcal{E}_j.
\end{cases}
\end{align*}
We  say $E_O$ is formed by
performing move \OOO\ to $E$.
\end{defin}

\begin{theor}[$\OOO\subseteq \xyzrel{111}$]\label{thm:outsplitting}
Let $E = ( E^0, E^1, r , s )$ be a graph and let $w \in E^0$ be a vertex that is not a sink.  Partition $s^{-1} (w)$ as a finite disjoint union of  subsets
\[
s^{-1} (w) = \mathcal{E}_1 \sqcup \mathcal{E}_2 \sqcup \cdots \sqcup \mathcal{E}_n
\]
with the property that at most one of the $\mathcal{E}_i$ is infinite.  Define $\psi \colon C^*(E) \to C^*(E_{O})$ by $\psi ( p_w ) = \sum_{i=1}^np_{ w^i}$ and $\psi(p_v)=p_v$ for all $v \in E^0\backslash\{w\}$ and 
\[
\psi ( s_{e} ) =
\begin{cases}
s_{e^1} &\text{if $r (e) \neq w$} \\
\sum_{i=1}^ns_{e^i} &\text{if $r (e) = w$.}
\end{cases}
\]
 Then $\psi$ is a $*$-isomorphism such that $\gamma_z^{E_{O}} \circ \psi = \psi \circ \gamma_z^E$ and $\psi ( \mathcal{D}_E ) =  \mathcal{D}_{E_{O}}$.
\end{theor}

\begin{proof}
Throughout the proof, if $w$ is an infinite emitter, then $\mathcal{E}_1$ will be the element in the partition that contains infinitely many edges.  By this choice, if $w$ is an infinite emitter in $E$, then $w^1$ is an infinite emitter in $E_O$ and each $w^i$, $2 \leq i \leq n$ is regular.  Bates and Pask proved in \cite[Theorem~3.2]{tbdp:fega} that $\psi$ is an $*$-isomorphism, and   it is clear that $\gamma_z^{E_{O}} \circ \psi = \psi \circ \gamma_z^E$.  It remains to show that $\psi ( \mathcal{D}_E ) =  \mathcal{D}_{E_{O}}$.

To show $\psi ( \mathcal{D}_E ) \subseteq  \mathcal{D}_{E_{O}}$, we first show that for all $e \in E^1$, $\psi ( s_e) \mathcal{D}_{E_O} \psi ( s_e)^* \subseteq \mathcal{D}_{E_O}$.  Let $\mu$ be a path in $E_O$ and let $e \in E^1$.  When $r (e) \neq w$, we get that $\psi ( s_e ) s_{\mu} s_\mu^* \psi ( s_e )^*= s_{ e^1} s_{\mu} s_{\mu}^* s_{e^1}^* \in \mathcal{D}_{E_O}$.  When $r(e) = w$, we get that $\psi ( s_e ) s_{\mu} s_\mu^* \psi ( s_e )^* = \sum_{ i, j = 1}^n s_{e^i } s_{\mu} s_{\mu}^* s_{e^j }^*$.  Note that if $s_{e^i} s_\mu s_\mu^* s_{e^j}^* \neq 0$, then 
\[
w^i = r_{O}(e^i) = s_{O}(\mu)=r_{O}(e^j) = w^j.
\] 
Thus, if $s_{e^i} s_\mu s_\mu^* s_{e^j} \neq 0$, then $i=j$.  Consequently,
\[
\psi ( s_e ) s_{\mu} s_\mu^* \psi ( s_e )^* = \sum_{ i, j = 1}^n s_{e^i } s_{\mu} s_{\mu}^* s_{e^j }^* =   \sum_{ i= 1}^n s_{e^i } s_{\mu} s_{\mu}^* s_{e^i }^* \in \mathcal{D}_{E_O}.
\]
Since $\mathcal{D}_{E_O}$ is the closed linear span of elements of the form $s_\mu s_\mu^*$, $\mu$ a path in $E_O$, we have that $\psi ( s_e)$ normalize $\mathcal{D}_{E_O}$ for all $e \in E^1$.

Using the fact that $\psi ( s_e) \mathcal{D}_{E_O} \psi ( s_e)^* \subseteq \mathcal{D}_{E_O}$ for all $e \in E^1$, we have that $\psi ( s_\mu s_\mu^* ) \in \mathcal{D}_{E_O}$ provided that $\psi ( s_{\mu'} s_{\mu'}^* ) \in \mathcal{D}_{E_O}$ where $\mu = e \mu'$.  One can then show that $\psi ( s_\mu s_\mu^* ) \in \mathcal{D}_{E_O}$ for all paths $\mu$ in $E$ by induction on the length of the path and by noting that $\psi ( s_\mu s_\mu^*) \in \mathcal{D}_{E_O}$ for all paths in $E$ of length zero.

The proof that $\psi^{-1} ( \mathcal{D}_{E_O} ) \subseteq \mathcal{D}_E$ goes in a similar way, but is more complicated.  To show that for all $f \in E_O^1$, $\psi^{-1} (s_f) \mathcal{D}_E \psi^{-1} (s_f) \subseteq \mathcal{D}_E$, we will need an explicit description of $\psi^{-1}$ on the generators of $C^*(E_O)$.  For $v \in E^0 \setminus \{w\}$ and $e \in E^1 \setminus r^{-1}(w)$, then $\psi^{-1} ( p_{v^1} )= p_{v}$ and $\psi^{-1} ( s_{e^1} )= s_e$.  Suppose $w^j$ is regular.  Since 
\[
s_{O}^{-1} ( w^j ) = \left( \bigsqcup_{ \substack{ e \in \mathcal{E}_j \\ r(e) = w} } \{ e^i : 1 \leq i \leq n \}  \right) \sqcup \left( \bigsqcup_{ \substack{ e \in \mathcal{E}_j \\ r(e) \neq w} } \{ e^1 \} \right),
\]
the Cuntz-Krieger relations on $C^*(E_O)$ imply
\begin{align*}
p_{w^j} &= \sum_{ \substack{ e \in \mathcal{E}_j  \\ r(e)=w } } \sum_{ i = 1}^n s_{e^i} s_{e^{i}}^* + \sum_{ \substack{ e \in \mathcal{E}_j \\ r(e)\neq w } } s_{e^1} s_{e^{1}}^* \\
	&=  \sum_{ \substack{ e \in \mathcal{E}_j  \\ r(e)=w } } \sum_{ i ,k= 1}^n s_{e^i} s_{e^k}^*+ \sum_{ \substack{ e \in \mathcal{E}_j \\ r(e)\neq w } } s_{e^1} s_{e^{1}}^* \\
	&= \sum_{ \substack{ e \in \mathcal{E}_j  \\ r(e)=w } } \psi(s_e s_e^*) +  \sum_{ \substack{ e \in \mathcal{E}_j \\ r(e)\neq w } } \psi(s_e s_e^*) \\
	&= \psi\left(  \sum_{ e \in \mathcal{E}_j } s_e s_e^* \right).
\end{align*}
Consequently, $\psi^{-1} ( p_{w^j} ) =\sum_{e \in \mathcal{E}_j}  s_e s_e^*$.  Suppose $w^j$ is an infinite emitter.  Then by the choice of indexing of the partition, $j = 1$ and $w^2, \ldots, w^n$ are regular.  Therefore, 
\[
\psi^{-1} ( p_{w^1} ) = \psi^{-1} \left( \psi ( p_w ) - \sum_{i =2}^n p_{ w^i } \right) = p_w - \sum_{ i = 2}^n \sum_{e \in \mathcal{E}_i}  s_e s_e^* = p_w - \sum_{ e \in s_{E}^{-1}(w) \setminus \mathcal{E}_1} s_{e}s_e^*.
\]
For $e_0 \in E^1$ with $r( e_0) = w$, 
\begin{align*}
s_{e_0^j} &= s_{e_0^j} p_{w^j}  \\
&= \left( \sum_{ i =1}^n e_0^i \right) p_{w^j} \\
&=
\begin{cases}
\psi( s_{e_0 }) \psi \left( \sum_{ e \in \mathcal{E}_j	 }  s_e s_e^* 	\right) &w^j=\bullet\\
\psi( s_{e_0} )  \psi\left( p_w- \sum_{ e \in s^{-1}(w) \setminus \mathcal{E}_1 }  s_{e} s_{e}^*\right) &w^j=\circ 
\end{cases} \\
&=\begin{cases}
 \psi \left( \sum_{ e \in \mathcal{E}_j	 }  s_{e_0} s_e s_e^* 	\right) &w^j=\bullet\\
 \psi\left( s_{e_0} - \sum_{ e \in s^{-1}(w) \setminus \mathcal{E}_1 } s_{e_0} s_{e} s_{e}^*\right) 
\psi( s_{e_0} )  \psi\left( p_w- \sum_{ e \in s^{-1}(w) \setminus \mathcal{E}_1 }  s_{e} s_{e}^*\right) &w^j=\circ 
\end{cases}
\end{align*}
where we use $w^j=\bullet$ to indicate that $w^j$ is regular, and  $w^j=\circ$ to indicate that it is not. Therefore,
\begin{align*}
\psi^{-1} ( s_{e_0^j} ) &= \psi^{-1}  (s_{ e_0^j} p_{w^j} ) \\
				 &= 
				 \begin{cases}
				 \displaystyle \sum_{e \in \mathcal{E}_j} s_{e_0}  s_e s_e^* &w^j=\bullet\\ 
				 s_{e_0} - \sum_{ e \in s^{-1}(w) \setminus \mathcal{E}_1} s_{e_0} s_{e}s_e^* &w^j=\circ 
				 \end{cases}
\end{align*}

We are now ready to show that for all $f \in E_O^1$, $\psi^{-1} (s_f)$ normalize $\mathcal{D}_E$.  Let $f \in E_O^1$ and let $\mu$ be a path in $E$.  Without loss of generality, we may assume that $\psi^{-1} ( s_f ) s_\mu s_\mu^* \psi^{-1} ( s_f )^*$ is a nonzero element.  Suppose $r_{O} ( f ) \neq w^i$ for any $i$.  Then $f = e^1$ for some $e \in E^1$ such that $r(e) \neq w$.  Consequently,  $\psi^{-1} ( s_f ) s_\mu s_\mu^* \psi^{-1} ( s_f )^* = s_{ e }  s_\mu s_\mu^* s_e^* \in \mathcal{D}_E$ by \cite[Lemma~4.1]{nbtmcmfw:gaoe}.  Suppose $r_{O} ( f ) = w^j$ for some $j$.  Then $f = e_0^j$ for some $e_0 \in E^1$ with $r ( e_0 ) = w$.  Suppose $w^j$ is regular.  Then
\[
\psi^{-1} ( s_{ e_0^j } ) s_\mu s_\mu^* \psi^{-1} ( s_{ e_0^j } )^* = \sum_{e, e' \in \mathcal{E}_j} s_{e_0}  s_e s_e^* s_\mu s_\mu^* s_{e'} s_{e'}^* s_{e_0}^*.
\]
Since $\psi^{-1} ( s_{ e_0^j } ) s_\mu s_\mu^* \psi^{-1} ( s_{ e_0^j } )^*  \neq 0$, $\mu = e_1 \mu'$ for some $e_1 \in \mathcal{E}_j$ or $\mu \in E^0$.  Consequently, $e_0 \mu = e_0 e_1 \mu'$ is a path in $E$ (when $|\mu| \geq 1$) and 
\[
\psi^{-1} ( s_{ e_0^j } ) s_\mu s_\mu^* \psi^{-1} ( s_{ e_0^j } )^* = \sum_{e, e' \in \mathcal{E}_j} s_{e_0}  s_e s_e^* s_\mu s_\mu^* s_{e'} s_{e'}^* s_{e_0}^* =
 \begin{cases}
s_{e_0 \mu } s_{ e_0 \mu }^* &\text{if $|\mu| \geq 1$} \\
\displaystyle \sum_{  e\in \mathcal{E}_j \cap r_{E}^{-1} ( s_E(\mu)) } s_{e_0 e} s_{ e_0 e }^* &\text{if $|\mu| =0$.} 
\end{cases}
\]
In both cases, $\psi^{-1} ( s_{ e_0^j } ) s_\mu s_\mu^* \psi^{-1} ( s_{ e_0^j } )^* \in \mathcal{D}_E$.  Suppose $w^j$ is an infinite emitter.  By the choice of indexing of the partition, $j =1$.  Since $\psi^{-1} ( s_{ e_0^j } ) s_\mu s_\mu^* \psi^{-1} ( s_{e_0^j} )^* \neq 0$ and
\begin{align*}
&\psi^{-1} ( s_{ e_0^j } ) s_\mu s_\mu^* \psi^{-1} ( s_{e_0^j} )^* \\
&= \left( s_{e_0} - \sum_{ e \in s^{-1}(w) \setminus \mathcal{E}_1} s_{e_0} s_{e}s_e^* \right) s_\mu s_\mu^*  \left( s_{e_0}^* - \sum_{ e \in s^{-1}(w) \setminus \mathcal{E}_1} s_{e} s_{e}^*s_{e_0}^* \right) \\
&= s_{e_0} s_\mu s_\mu^* s_{e_0}^* -  \sum_{ e \in s^{-1}(w) \setminus \mathcal{E}_1} s_{e_0} s_{e}s_e^* s_\mu s_\mu^* s_{e_0}^* \\
&\qquad  - s_{e_0}s_\mu s_\mu^* \sum_{ e \in s^{-1}(w) \setminus \mathcal{E}_1} s_{e} s_{e}^*s_{e_0}^*  + \sum_{ e, e'  \in s^{-1}(w) \setminus \mathcal{E}_1 } s_{e_0} s_{e} s_{e}^* s_\mu s_\mu^* s_{e'} s_{e'}^* s_{e_0}^*,
\end{align*}
we have that $s(\mu ) = w$.  Set $\mu = e_1 \mu'$ where $e_1 \in s^{-1} (w)$ if $|\mu| \geq 1$.  A computation shows that 
\[
\psi^{-1} ( s_{ e_0^j } ) s_\mu s_\mu^* \psi^{-1} ( s_{e_0^j} )^* =
\begin{cases}
\displaystyle s_{e_0} s_{e_0}^* - \sum_{ e \in s^{-1} (w) \setminus \mathcal{E}_1 }s_{ e_0}s_{e } s_{e }^*s_{e_0}^* &\text{if } |\mu| = 0 \\
s_{e_0}  s_{\mu} s_{ \mu}^* s_{e_0}^*  &\text{if } | \mu | \geq 1 \text{ and } e_1 \notin \mathcal{E}_1 \\
0 &\text{if } | \mu | \geq 1 \text{ and } e_1 \in \mathcal{E}_1.
\end{cases}
\]
Consequently, $\psi^{-1} ( s_{ e_0^j } ) s_\mu s_\mu^* \psi^{-1} ( s_{e_0^j} )^* \in \mathcal{D}_E$.  Since $\mathcal{D}_E$ is the closed linear span of elements of the form $s_\mu s_\mu^*$, $\mu$ a path in $E$, $\psi^{-1}( s_f ) \mathcal{D}_E \psi^{-1}( s_f ) \subseteq \mathcal{D}_E$ for all $f \in E_O^1$.

Since $\psi^{-1} ( s_\mu s_\mu^*) \in \mathcal{D}_E$ for all $\mu$ of length zero, we can use the fact that for all $f \in E_O^1$, $\psi^{-1}( s_f ) \mathcal{D}_E \psi^{-1}( s_f ) \subseteq \mathcal{D}_E$ and induction on the length of the paths in $E_{O}$ to prove that $\psi^{-1} ( s_\mu s_\mu^*) \in \mathcal{D}_E$ for all paths in $E_O$.  Consequently, $\psi^{-1} ( \mathcal{D}_{E_O} ) \subseteq \mathcal{D}_E$ which implies that $\mathcal{D}_{E_O} \subseteq \psi ( \mathcal{D}_E )$.  

We now can conclude that $\psi ( \mathcal{D}_E ) = \mathcal{D}_{E_O}$.
\end{proof}

\section{Insplitting}
We now move on to insplitting. The insplitting move introduced in \cite{tbdp:fega} can be seen to respect both diagonal and gauge action, but will usually not provide a $*$-isomorphism. We will replace it with two different moves called \IIIm\ and  \IIIp\ so that \IIIm\ is a generalization of the move in \cite{tbdp:fega}, and \IIIp\ is a specialization of \IIIm. We shall see that \IIIp\ respects all structure under study -- i.e. $\IIIp\subseteq\xyzrel{111}$ -- and that $\IIIm\subseteq\xyzrel{011}$. The added flexibility in $\IIIm$  is to allow some of the sets in the insplitting partition to be empty, and we need this in order to introduce sources in a gauge-invariant way.

\begin{defin}[Move \IIIm:  Insplitting] \label{IIIm}\index{Im@\IIIm}
Let $E = ( E^0, E^1, r , s )$ be a graph and let $w \in E^0$ be a regular vertex.  Partition $r^{-1} (w)$ as a finite disjoint union of (possibly empty) subsets,
\[
r^{-1} (w) = \mathcal{E}_1 \sqcup \mathcal{E}_2 \sqcup \cdots \sqcup \mathcal{E}_n.
\]

Let $E_I = ( E_I^0 , E_I^1, r_{I} , s_{I} )$ be the graph defined by 
\begin{align*}
E^0_I &= \{ v^1 \ : \ v \in E^0 \setminus \{ w \} \} \cup \{ w^1, w^2, \ldots, w^n \} \\
E^1_I &= \{ e^1 \ : \ e \in E^1, s (e) \neq w \} \cup \{ e^1 , e^2 , \ldots, e^n \ : \ e \in E , s (e)  = w \} \\
s_{I}(e^i ) &= \begin{cases} s(e)^1 &\text{if $e \in E^1, s (e) \neq w$} \\ w^i &\text{if $e \in E^1, s (e) = w$} \end{cases} \\
r_{I}(e^i ) &= \begin{cases} r(e)^1 &\text{if $e \in E^1, r (e) \neq w$} \\ w^j &\text{if $e \in E^1, r (e) = w, e \in \mathcal{E}_j$} \end{cases}
\end{align*}
We  say $E_I$ is formed by
performing move \IIIm\ to $E$.
\end{defin}

\begin{defin}[Move \IIIp:  Unital insplitting] \label{IIIp}\index{Ip@\IIIp}
The graphs $F$ and $G$ are said to be \emph{move $\mbox{\texttt{\textup{(I+)}}}$ equivalent} if there exists a graph $E=(E^0,E^1,r,s)$ and a regular vertex $w\in E^0$ such that $F$ is the result of an \IIIm\ move applied to  $E$ via a partition of $r^{-1}(w)$ using $n$ sets and $G$ is the result of an \IIIm\ move applied to  $E$ via a partition of $r^{-1}(w)$ using $n$ sets.
\end{defin}

\begin{theor}[$\IIIp\subseteq \xyzrel{111}$]\label{thm:insplittingplus}
Let $E = ( E^0, E^1, r , s )$ be a graph and let $w \in E^0$ be a regular vertex.  Partition $r^{-1} (w)$ twice as a finite disjoint union of (possibly empty) subsets,
\begin{eqnarray*}
r^{-1} (w) &=& \mathcal{F}_1 \sqcup \mathcal{F}_2 \sqcup \cdots \sqcup \mathcal{F}_n\\
&=& \mathcal{G}_1 \sqcup \mathcal{G}_2 \sqcup \cdots \sqcup \mathcal{G}_n  
\end{eqnarray*}
Let $F$ be the insplitting graph of $E$ with respect to the first partition and let $G$ be the insplitting graph of $E$ with respect to the second partition.  Then there exists a $*$-isomorphism $\phi \colon C^*(F) \to C^*(G)$ such that $\phi ( \mathcal{D}_F ) = \mathcal{D}_G$ and $\gamma_z^G \circ \phi = \phi \circ \gamma_z^F$. 
\end{theor}

\begin{proof}
We first establish some notation that we will use throughout the proof.  Define $m \colon E^0 \to \NN$ by $m(v) = 1$ if $v \neq w$ and $m(w) = n$.  Set $\mathcal{F}_i^w = \mathcal{F}_i$ and $\mathcal{G}_i^w = \mathcal{G}_i$ for $1 \leq i \leq m(w)$.  For $v \in E^0$ with $v \neq w$, set $\mathcal{F}^v_1 = \mathcal{G}_1^v = r^{-1} (v)$.  Then $F$ is the graph with vertex set $\setof{ v^i }{ v \in E^0 , 1 \leq i \leq m(v) }$ and edge set $\setof{ e^i }{ e \in E^1 , 1 \leq i \leq m ( s(e) ) }$ such that 
\[
s_F (e^i) = s(e)^i \quad \text{and} \quad r_F (e^i ) = r(e)^j \quad \text{where } e \in \mathcal{F}_j^{r(e)}.
\]
We can similarly describe $G$ as the graph with vertex set $\setof{ v^i }{ v \in E^0 , 1 \leq i \leq m(v) }$ and edge set $\setof{ e^i }{ e \in E^1 , 1 \leq i \leq m ( s(e) ) }$ such that 
\[
s_G (e^i) = s(e)^i \quad \text{and} \quad r_G(e^i ) = r(e)^j \quad \text{where } e \in \mathcal{G}_j^{r(e)}.
\]
For each $e \in r^{-1}(w)$, let $j(e)$ and $k(e)$ be the unique elements of $\{1, 2, \ldots, n \}$ such that $e \in \mathcal{F}^w_{j(e)} \cap \mathcal{G}_{k(e)}^w$.  If we also set $j(e)= k(e)=1$ for $e \notin r^{-1}(w)$, then for all $e \in E$, $r_{F}(e^c) = r(e)^{j(e)}$ and $r_{G}(e^c)= r(e)^{k(e)}$.

We claim that there exists a $*$-isomorphism $\phi \colon C^*(F) \to C^*(G)$ such that
\begin{align*}
\phi ( p_{v^{i}} ) &= p_{v^i} \quad \text{for all $v \in E^0$ and for all $1 \leq i \leq m(v)$}\\
\phi ( s_{e^{i}}) &= 
\begin{cases}
s_{e^i} &\text{if } r(e) \neq w \\
\displaystyle\sum_{ g \in s^{-1} (w) } s_{e^i} s_{g^{k(e)}} s_{ g^{j(e)}}^* &\text{if } r(e)=w
\end{cases}
\end{align*}
for all $e \in E^1$ and for all $1 \leq i \leq m ( s (e) )$.  Set $P_{v^i} = p_{v^i} \in C^*(G)$ for $v \in E^0$, $1 \leq i \leq m(v)$ and define $S_{e^i} \in C^*(G)$ by
\[
S_{e^i} = 
\begin{cases}
s_{e^i} &\text{if } r(e) \neq w \\
\displaystyle \sum_{ g \in s^{-1} (w) } s_{e^i} s_{g^{k(e)}} s_{ g^{j(e)}}^* &\text{if } r(e)= w
\end{cases}
\]
for all $e \in E^1$ and $1 \leq i \leq m ( s(e) )$.  We claim that $\{ P_{v^i} , S_{e^i} \}$ is a Cuntz-Krieger $F$-family in $C^*(G)$.  

Note that 
\begin{align*}
S_{e^c}^* S_{f^d} &= 
\begin{cases}
s_{e^c}^* s_{f^d} &\text{if } e, f \notin r^{-1} (w)  \\
\displaystyle \sum_{ g \in s^{-1} (w) }  s_{g^{j(e)}} s_{g^{k(e)} }^* s_{e^c}^*s_{f^d} &\text{if } r(e) = w \text{ and } r(f)\neq w  \\
\displaystyle \sum_{ g \in s^{-1} (w) } s_{e^c}^*  s_{f^d} s_{ g^{k(f)} } s_{g^{j(f)}}^* &\text{if } r(e) \neq w \text{ and } r(f)=w \\ 
\displaystyle \sum_{ g, h \in s_{E}^{-1} (w) } s_{g^{j(e)}} s_{g^{k(e)} }^* s_{e^c}^* s_{f^d} s_{ h^{k(f)} } s_{h^{j(f)}}^* &\text{if  } r(e)= r(f)=w 
\end{cases} \\
&=\begin{cases}
P_{r(e)^1} &\text{if } e^c =f^d  \text{ and } e, f \notin r^{-1} (w) \\
\displaystyle \sum_{ g \in s^{-1} (w) }  s_{g^{j(e)}} s_{g^{j(e)}}^* &\text{if  } e^c = f^d \text{ and } r(e)=r(f)=w \\
0&\text{otherwise}
\end{cases}\\
&=\begin{cases}
p_{r(e)^1} &\text{if } e^c =f^d  \text{ and } e, f \notin r^{-1} (w) \\
p_{w^{j(e)}}  &\text{if  } e^c = f^d \text{ and } r(e)=r(f)=w \\
0&\text{otherwise}
\end{cases}\\
&=\delta_{e^c,f^d}P_{r(e)^{j(e)}} = \delta_{e^c,f^d}P_{r_F(e^c)}
\end{align*}
for all $e, f\in E^1$, $1 \leq c \leq m( s(e) )$, and $1 \leq d \leq m (s(f))$, and further
\begin{align*}
S_{e^c} S_{e^c}^*&=
\begin{cases}
s_{e^c} s_{e^c}^* &\text{if } r(e) \neq w \\
\displaystyle \sum_{f, g \in  s^{-1} ( w ) } s_{e^c} s_{f^{k(e)}} s_{f^{j(e)}}^* s_{g^{j(e)}} s_{g^{k(e)}}^* s_{e^c}^* &\text{if } r(e)=w
\end{cases} \\
&= 
\begin{cases}
s_{e^c} s_{e^c}^* &\text{if } r(e) \neq w \\
\displaystyle \sum_{g \in  s^{-1} ( w ) } s_{e^c} s_{g^{k(e)}} s_{g^{k(e)}}^* s_{e^c}^* &\text{if } r(e)=w
\end{cases} \\
&= 
\begin{cases}
s_{e^c} s_{e^c}^* &\text{if } r(e) \neq w \\
s_{e^c} p_{ w^{k(e)} } s_{e^c}^* &\text{if } r(e)=w
\end{cases} = \begin{cases}
s_{e^c} s_{e^c}^* &\text{if } r(e) \neq w \\
s_{e^c} p_{ r_G(e^c)  } s_{e^c}^* &\text{if } r(e)=w
\end{cases} \\
&= s_{ e^c } s_{e^c}^* \leq p_{ s_G ( e^c) } = p_{ s(e)^c } = P_{ s_F( e^c) }
\end{align*}
for all $e \in E^1$ and $1 \leq c \leq m ( s(e) )$.  Let $v^c$ be a regular vertex in $F$.  Then $v$ is a regular vertex in $E$ and  $s^{-1}_{F}(v^c)=\{e^c \colon e\in s^{-1}(v)\} = s_G^{-1} ( v^c)$.  Hence, $v^c$ is a regular vertex in $G$ and 
\[
P_{v^c} = p_{v^c} =  \sum_{ e^c \in s_{G}^{-1} (v^c) } s_{e^c} s_{e^c}^*= \sum_{ e \in s^{-1} (v) } s_{e^c} s_{e^c}^*= \sum_{ e^c \in s_{F}^{-1} (v^c) } S_{e^c} S_{e^c}^*.
\]
This proves that $\{ P_v, S_e\}$ is a Cuntz-Krieger $F$-family in $C^*(G)$.  By the universal property of $C^*(F)$, there exists a $*$-homomorphism $\phi \colon C^*(F) \to C^*(G)$ such that  
\begin{align*}
\phi ( p_{v^{i}} ) &= p_{v^i} \quad \text{for all $v \in E^0$ and for all $1 \leq i \leq m(v)$}\\
\phi ( s_{e^{i}}) &= 
\begin{cases}
s_{e^i} &\text{if } r(e) \neq w \\
\displaystyle \sum_{ g \in s^{-1} (w) } s_{e^i} s_{g^{k(e)}} s_{ g^{j(e)}}^* &\text{if } r(e)=w
\end{cases}
\end{align*}
for all $e \in E^1$ and for all $1 \leq i \leq m ( s (e) )$.  One easily checks that  $\gamma_z^{G} \circ \phi = \phi \circ \gamma_z^F$ on $\{ p_{v^i}, s_{e^i} \}$, hence $\gamma_z^{G} \circ \phi = \phi \circ \gamma_z^F$.  Since $\psi( p_{v^i} ) = p_{v^i} \neq 0$ for all $v^i \in F^0$, by the gauge-invariant uniqueness theorem (\cite{bhrs:iccig}), $\phi$ is injective.  

By a symmetric argument, there exists an injective gauge-invariant $*$-homomorphism $\psi \colon C^*(G) \to C^*(F)$ such that 
\begin{align*}
\psi ( p_{v^{i}} ) &= p_{v^i} \quad \text{for all $v \in E^0$ and for all $1 \leq i \leq m(v)$} \\
\psi ( s_{e^{i}}) &= 
\begin{cases}
s_{e^i} &\text{if } r(e) \neq w \\
\displaystyle\sum_{ g \in s^{-1} (w) } s_{e^i} s_{g^{j(e)}} s_{ g^{k(e)}}^* &\text{if  } r(e)=w
\end{cases}
\end{align*}
for all $e \in E^1$ and for all $1 \leq i \leq m ( s (e) )$.  We claim that $\phi \circ \psi = \id_{C^*(G)}$, and hence $\phi$ is surjective.  Thus implying that $\phi$ is a $*$-isomorphism with inverse $\psi$.  Note that $( \phi \circ \psi )( p_{v^i } ) = p_{v^i}$ and $( \phi \circ \psi )( s_{e^i } ) = s_{e^i}$ for all $v \in E^0$ and $e \in E^1$ with $r(e) \neq w$.  Suppose $r(e)=w$.    Then 
\[
( \phi \circ \psi )( s_{ e^\ell } )= \sum_{ g \in s^{-1} (w) }  \phi ( s_{e^\ell} ) \phi ( s_{ g^{j(e)} } ) \phi ( s_{ g^{k(e)}}^* ).
\]
Note that 
\begin{align*}
\sum_{ \substack{ s(g)=w \\  r(g)\neq w}  }  \phi ( s_{e^\ell} ) \phi ( s_{ g^{j(e)} } ) \phi ( s_{ g^{k(e)}}^* )  &=\sum_{ \substack{ s(g)=w \\  r(g)\neq w}  } \left( \sum_{ h \in s^{-1} (w) } s_{e^\ell } s_{ h^{k(e)} } s_{h^{j(e)}}^* \right) s_{ g^{j(e)} } s_{ g^{k(e)} }^* \\
&= \sum_{ \substack{ s(g)=w \\  r(g)\neq w}  }  s_{e^\ell} s_{g^{k(e)}} s_{g^{k(e)}}^*
\end{align*}
and 
\begin{align*}
&\sum_{ \substack{ s(g)=w \\  r(g)= w}  }  \phi ( s_{e^\ell} ) \phi ( s_{ g^{j(e)} } ) \phi ( s_{ g^{k(e)}}^* ) \\
&= \sum_{ \substack{ s(g)=w  \\ r(g)= w }} \left( \sum_{ h \in s^{-1} (w) } s_{e^\ell } s_{ h^{k(e)} } s_{h^{j(e)}}^* \right)   \left( \sum_{ h \in s^{-1} (w) } s_{g^{j(e)} } s_{ h^{k(g)} } s_{h^{j(g)}}^* \right)  \left( \sum_{ h \in s^{-1} (w) } s_{h^{j(g)}}  s_{ h^{k(g)} }^* s_{g^{k(e)} }^*  \right)  \\
&=\sum_{ \substack{ s(g)=w \\  r(g)= w}  }s_{e^\ell } s_{ g^{k(e)} }   \left( \sum_{ h \in s^{-1} (w) }  s_{ h^{k(g)} } s_{ h^{k(g)} }^* s_{g^{k(e)} }^*  \right) \\
&= \sum_{ \substack{ s(g)=w \\  r(g)= w}  } s_{e^\ell } s_{ g^{k(e)} } p_{w^{k(g)} } s_{g^{k(e)}}^*  \\
&= \sum_{ \substack{ s(g)=w \\  r(g)= w}  } s_{e^\ell } s_{ g^{k(e)} } p_{r_G(g^{k(e)}) } s_{g^{k(e)}}^*  \\
&= \sum_{ \substack{ s(g)=w \\  r(g)= w}  } s_{e^\ell } s_{ g^{k(e)} } s_{g^{k(e)}}^*.
\end{align*}
Therefore, 
\begin{align*}
( \phi \circ \psi )( s_{ e^\ell } )&= \sum_{ g \in s^{-1} (w) }  \phi ( s_{e^\ell} ) \phi ( s_{ g^{j(e)} } ) \phi ( s_{ g^{k(e)}}^* )  \\
		&= \sum_{ \substack{ s(g)=w \\  r(g)\neq w}  }   \phi ( s_{e^\ell} ) \phi ( s_{ g^{j(e)} } ) \phi ( s_{ g^{k(e)}}^* ) + \sum_{ \substack{ s(g)=w \\  r(g)= w}  }   \phi ( s_{e^\ell} ) \phi ( s_{ g^{j(e)} } ) \phi ( s_{ g^{k(e)}}^* ) \\
		&=\sum_{ \substack{ s(g)=w \\  r(g)\neq w}  } s_{ e^\ell } s_{g^{k(e)}} s_{g^{k(e)}}^* + \sum_{ \substack{ s(g)=w \\  r(g)= w}  } s_{e^\ell } s_{ g^{k(e)} } s_{g^{k(e)}}^*   \\
&= \sum_{ g \in s^{-1} (w) } s_{ e^\ell } s_{g^{k(e)}} s_{g^{k(e)}}^* \\
&= s_{ e^\ell } p_{w^{k(e)}} = s_{e^{\ell}} p_{r_G(e^\ell)} = s_{e^{\ell}}.  
\end{align*}
We now have $\phi \circ \psi = \id_{C^*(G)}$ since $\phi \circ \psi$ and $\id_{C^*(G)}$ are equal on the generators of $C^*(G)$, hence proving the claim.

We are left with showing that $\phi ( \mathcal{D}_F ) = \mathcal{D}_G$.  Note that it is enough to show $\phi ( \mathcal{D}_F ) \subseteq \mathcal{D}_G$ since a symmetric argument implies that $\psi ( \mathcal{D}_G ) \subseteq \mathcal{D}_F$ and hence $\phi ( \mathcal{D}_F ) =  \mathcal{D}_G$.   Let $e_1^{j_1} e_{2}^{j_2} \cdots e_r^{j_r}$ be a path in $F$.  Then $e_i \in \mathcal{F}_{j_{i+1}}^{r(e_i)}$, $r( e_i ) = s( e_{i+1} )$ for all $1 \leq i \leq r-1$.  Set  $1 \leq k_{i+1} \leq m ( r( e_i ) )$ such that $e_i \in \mathcal{G}_{ k_{i+1} }^{r( e_i ) }$ for $1 \leq i \leq r - 1$.  Then $e_1^{ j_1 } e_2^{ k_2 } \cdots e_{r}^{k_r}$ is a path in $G$. 

We claim that 
\[
\phi ( s_{ e_1^{j_1} e_{2}^{j_2} \cdots e_s^{j_s}} ) = 
\begin{cases}
s_{ e_1^{ j_1 } e_2^{ k_2 } \cdots e_{s}^{k_s}} &\text{if } r( e_s ) \neq w \\
s_{ e_1^{ j_1 } e_2^{ k_2 } \cdots e_{s}^{k_s} } \displaystyle \sum_{g \in s^{-1} (w) } s_{ g^{ k_{s+1} } } s_{ g^{ j_{s+1} }}^* &\text{if } r(e_s)=w.
\end{cases}
\]
for all $1 \leq s \leq r-1$.  Note that $j(e_s) = j_{s+1}$ and $k(e_s)=k_{s+1}$ since $e_s \in \mathcal{F}^{r(e_s)}_{j_{i+1}} \cap \mathcal{G}^{r(e_s)}_{k_{i+1}}$.  Therefore, the claim is true for $s = 1$ by the definition of $\phi$.  Suppose the claim is true for $1 \leq s \leq r - 2$.  Suppose $r( e_{ s +1 } ) \neq w$.  Then 
\begin{align*}
\phi (  s_{ e_1^{j_1} e_{2}^{j_2} \cdots e_{s+1}^{j_{s+1}} } )  &= 
\begin{cases}
s_{ e_1^{ j_1 } e_2^{ k_2 } \cdots e_{s}^{k_s} e_{s+1}^{ j_{ s+1} } } &\text{if } r( e_s ) \neq w  \\ 
s_{ e_1^{ j_1 } e_2^{ k_2 } \cdots e_{s}^{k_{s}} }  \displaystyle \sum_{g \in s^{-1} (w) }  s_{ g^{ k_{s+1} } } s_{ g^{ j_{s+1} }}^*  s_{ e_{s+1}^{ j_{s+1} }}  &\text{if } r(e_s)=w \\
\end{cases} \\
&= s_{ e_1^{ j_1 } e_2^{ k_2 } \cdots e_{s}^{k_{s}} } s_{e_{s+1}^{k_{s+1}}}\\
&= s_{ e_1^{ j_1 } e_2^{ k_2 } \cdots e_{s+1}^{k_{s+1}}}.
\end{align*}  
Suppose $r(e_{s+1})=w$.  Then
\begin{align*}
&\phi (  s_{ e_1^{j_1} e_{2}^{j_2} \cdots e_{s+1}^{j_{s+1}} } )  \\
&= 
\begin{cases}
\displaystyle s_{ e_1^{ j_1 } e_2^{ k_2 } \cdots e_{s}^{k_s}} \sum_{ g \in s^{-1} (w) } s_{e_{s+1}^{j_{s+1}}}  s_{g^ { k(e_{s+1}) } } s_{ g^{ j(e_{s+1}) } }^* &\text{if } r( e_s ) \neq w  \\ 
\displaystyle s_{ e_1^{ j_1 } e_2^{ k_2 } \cdots e_{s}^{k_s} }  \left( \sum_{g \in s^{-1} (w) } s_{ g^{ k_{s+1} } } s_{ g^{ j_{s+1} }}^*\right)\left(  \sum_{ g \in s^{-1} (w) } s_{e_{s+1}^{j_{s+1}}}  s_{g^ {k(e_{s+1})} } s_{ g^{ s(e_{s+1}) } }^* \right)  &\text{if } r(e_s)=w
\end{cases}  \\
&= 
\begin{cases}
\displaystyle s_{ e_1^{ j_1 } e_2^{ k_2 } \cdots e_{s}^{k_s}} s_{e_{s+1}^{j_{s+1}}}  \sum_{ g \in s^{-1} (w) }  s_{g^ { k_{s+2} } } s_{ g^{ j_{s+2} } }^* &\text{if } r( e_s ) \neq w  \\ 
\displaystyle s_{ e_1^{ j_1 } e_2^{ k_2 } \cdots e_{s}^{k_s} } s_{e_{s+1}^{k_{s+1}}}    \sum_{ g \in s^{-1} (w) }  s_{g^ { k_{s+2} } } s_{ g^{ j_{s+2} } }^*  &\text{if  } r(e_s)=w.
\end{cases}  \\
&=\begin{cases}
\displaystyle s_{ e_1^{ j_1 } e_2^{ k_2 } \cdots e_{s}^{k_s} e_{s+1}^{j_{s+1}}}  \sum_{ g \in s^{-1} (w) }  s_{g^ { k_{s+2} } } s_{ g^{ j_{s+2} } }^* &\text{if } r( e_s ) \neq w  \\ 
\displaystyle s_{ e_1^{ j_1 } e_2^{ k_2 } \cdots e_{s}^{k_s} e_{s+1}^{k_{s+1}}}    \sum_{ g \in s^{-1} (w) }  s_{g^ { k_{s+2} } } s_{ g^{ j_{s+2} } }^*  &\text{if  } r(e_s)=w.
\end{cases}  
\end{align*}
Note that if $r( e_s ) \neq w$, then $s( e_{s+1} ) \neq w$ and $j_{s+1} = k_{s+1} = 1$.  Hence, 
\begin{align*}
\phi (  s_{ e_1^{j_1} e_{2}^{j_2} \cdots e_{s+1}^{j_{s+1}} } )   &= s_{ e_1^{ j_1 } e_2^{ k_2 } \cdots e_{s}^{k_s} e_{s+1}^{j_{s+1}}}    \sum_{ g \in s^{-1} (w) }  s_{g^ { k_{s+2} } } s_{ g^{ j_{s+2} } }^* \\
&= s_{ e_1^{ j_1 } e_2^{ k_2 } \cdots e_{s}^{k_s} e_{s+1}^{k_{s+1}}}   \sum_{ g \in s^{-1} (w) }  s_{g^ { k_{s+2} } } s_{ g^{ j_{s+2} } }^*,
\end{align*}
and thus proving the claim.  

The above claim implies that 
\begin{align*}
&\phi ( s_{ e_1^{j_1} e_{2}^{j_2} \cdots e_r^{j_r}}  s_{ e_1^{j_1} e_{2}^{j_2} \cdots e_r^{j_r}}^* ) \\
&= 
\begin{cases}
s_{ e_1^{ j_1 } e_2^{ k_2 } \cdots e_{r}^{k_r}} s_{ e_1^{ j_1 } e_2^{ k_2 } \cdots e_{r}^{k_r}}^*  &\text{if } r( e_r ) \neq w \\
s_{ e_1^{ j_1 } e_2^{ k_2 } \cdots e_{s}^{k_r} } \left( \sum_{g, h \in s^{-1} (w) } s_{ g^{ k_{r+1} } } s_{ g^{ j_{r+1} }}^* s_{ h^{ j_{r+1} }}s_{ h^{ k_{r+1} } }^* \right) s_{ e_1^{ j_1 } e_2^{ k_2 } \cdots e_{s}^{k_r} }^* 
&\text{if  } r(e_r)=w\\
\end{cases} \\
&= 
\begin{cases}
s_{ e_1^{ j_1 } e_2^{ k_2 } \cdots e_{r}^{k_r}} s_{ e_1^{ j_1 } e_2^{ k_2 } \cdots e_{r}^{k_r}}^*  &\text{if } r( e_r ) \neq w \\
s_{ e_1^{ j_1 } e_2^{ k_2 } \cdots e_{s}^{k_r} } \left( \sum_{g \in s^{-1} (w) } s_{ g^{ k_{r+1} } } s_{ g^{ k_{r+1} } }^* \right) s_{ e_1^{ j_1 } e_2^{ k_2 } \cdots e_{s}^{k_r} }^* 
&\text{if  } r(e_r)=w
\end{cases} \\
&= s_{ e_1^{ j_1 } e_2^{ k_2 } \cdots e_{r}^{k_r}}  p_{w^{k_{r+1}}} s_{ e_1^{ j_1 } e_2^{ k_2 } \cdots e_{r}^{k_r}}^* = \ s_{ e_1^{ j_1 } e_2^{ k_2 } \cdots e_{r}^{k_r}}  p_{r_G( e_r^{k_r} ) } s_{ e_1^{ j_1 } e_2^{ k_2 } \cdots e_{r}^{k_r}}^* \\
&= s_{ e_1^{ j_1 } e_2^{ k_2 } \cdots e_{r}^{k_r}} s_{ e_1^{ j_1 } e_2^{ k_2 } \cdots e_{r}^{k_r}}^*  \in \mathcal{D}_G.
\end{align*}
And since $\phi ( p_{v^i} ) = p_{v^i} \in \mathcal{D}_G$, we have $\phi ( \mathcal{D}_F ) \subseteq \mathcal{D}_G$.  This completes the proof. 
\end{proof}

\begin{theor}[$\IIIm\subseteq \xyzrel{011}$]\label{thm:insplitting2}
Let $E = ( E^0, E^1, r , s )$ be a graph and let $w \in E^0$ be a regular vertex.  Partition $r^{-1} (w)$ as a finite disjoint union of (possibly empty) subsets,
\[
r^{-1} (w) = \mathcal{E}_1 \sqcup \mathcal{E}_2 \sqcup \cdots \sqcup \mathcal{E}_n.
\]
There exists a $*$-isomorphism $\theta \colon C^*(E) \otimes \K \to C^*(E_I) \otimes \K$ such that $\theta ( \mathcal{D}_{E} \otimes c_0 ) = \mathcal{D}_{E_I} \otimes c_0 $ and $(\gamma_z^{E_I} \otimes \id_{ \K } ) \circ \theta = \theta \circ (\gamma_z^{E} \otimes \id_{ \K })$.
\end{theor}

\begin{proof}
Consider the partition $r^{-1} (w) =  \mathcal{F}_1 \sqcup \mathcal{F}_2 \sqcup \cdots \sqcup \mathcal{F}_n$ with $\mathcal{F}_1 = r^{-1} (w)$ and $\mathcal{F}_i = \emptyset$ for $1 \leq i \leq n$.  Let $F$ be the graph obtained by insplitting the graph $E$ with respect to this partition.  By Theorem~\ref{thm:insplittingplus}, there exists a $*$-isomorphism $\phi \colon C^*( F ) \to C^*( E_I )$ such that $\gamma_z^{E_I} \circ \phi = \phi \circ \gamma_z^F$ and $\phi ( \mathcal{D}_F ) = \mathcal{D}_{E_I}$.  Therefore $\phi \otimes \id_{\K}  \colon C^*( F ) \otimes \K \to C^*( E_I ) \otimes \K $ is a  $*$-isomorphism such that $( \gamma_z^{E_I} \otimes \id_{\K} )\circ ( \phi \otimes \id_\K) = ( \phi \otimes \id_\K) \circ ( \gamma_z^F \otimes \id_\K)$ and $( \phi \otimes \id_\K)  ( \mathcal{D}_F \otimes c_0 ) = \mathcal{D}_{E_I} \otimes c_0$.

To finish the proof, we will show that there exists a $*$-isomorphism $\Psi \colon C^*(E) \otimes \K \to C^*( F ) \otimes \K$ such that $( \gamma_z^{F} \otimes \id_{\K} )\circ \Psi = \Psi \circ ( \gamma_z^E \otimes \id_\K)$ and $\Psi  ( \mathcal{D}_E \otimes c_0 ) = \mathcal{D}_{F} \otimes c_0$.  Note that $E$ is isomorphic to the subgraph of $F$ given by $( \{ v^1 : v \in E^0\} , s_F^{-1} ( \{ v^1 : v \in E^0\} ) , s_F, r_F )$.  Note also that $w_2, w_3, \ldots, w_n$ are regular sources in $F$.  Set $p = \sum_{ v \in E^0 } p_{v^1}$.  Then $\psi \colon C^*( E ) \to p C^*( F ) p$ defined by $\psi ( p_{v} ) = p_{v^1}$ and $\psi ( s_{e} ) = s_{ e^1}$ for all $v \in E^0$ and $e \in E^1$ gives a $*$-isomorphism such that $\gamma_z^F \circ \psi = \psi \circ \gamma_z^E$ and $\psi ( \mathcal{D}_E ) = p \mathcal{D}_F$. 

We claim that $p$ is full in the fixed point-algebra $\mathcal{F}_F$.  Let $g \in s^{-1} ( w )$ and let $1 \leq k \leq n$.  Then $v_{g^k} = s_{g^1} s_{g^k}^* \in \mathcal{F}_F$ such that 
\[
v_{g^k}^* v_{g^k} = s_{g^k } s_{g^k }^* \quad \text{and} \quad v_{g^k} v_{g^k}^* = s_{g^1} s_{g_1}^* \leq p,
\]  
which implies that $s_{g^k } s_{g^k }^*$ is in the ideal of $\mathcal{F}_F$ generated by the projection $p$.  Therefore, $p_{w^k} = \sum_{ g\in s^{-1} ( w ) } s_{g^k } s_{g^k }$ is in the ideal of $\mathcal{F}_F$ generated by $p$ for all $k$.  Thus, $p$ is full in $\mathcal{F}_F$, hence proving the claim.  

By $(9) \implies (8)$ in \cite[Corollary~11.3]{tmcerasmt:rgcds}, there exists a $*$-isomorphism $\Psi \colon C^*(E) \otimes \K \to C^*( F ) \otimes \K$ such that $( \gamma_z^{F} \otimes \id_{\K} )\circ \Psi = \Psi \circ ( \gamma_z^E \otimes \id_\K)$ and $\Psi  ( \mathcal{D}_E \otimes c_0 ) = \mathcal{D}_{F} \otimes c_0$.  Then $\theta = ( \phi \otimes \id_{\K} ) \circ \Psi$ is the desired $*$-isomorphism.
\end{proof}

\begin{remar}\label{IIIphowtothink}
As touched upon in the introduction, we find it convenient to think of a $\IIIp$ move as the result of redistributing the past of vertices having the same future, as the pair of lower points in \eqref{bc} and \eqref{bcalt}. This even makes sense -- with care -- in the presence of loops; for instance we have
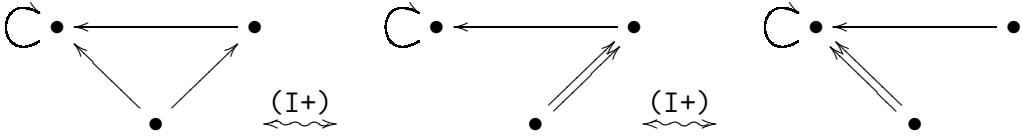
\[
\xymatrix{\bullet\ar@(dl,ul)[]&&\bullet\ar[ll]&&\bullet\ar@(dl,ul)[]&&\bullet\ar[ll]&&\bullet\ar@(dl,ul)[]&&\bullet\ar[ll]\\
&\bullet\ar[ur]\ar[ul]&\ar@{<~>}[r]^{\IIIp}&&&\bullet\ar@<-0.5mm>[ur]\ar@<0.5mm>[ur]&\ar@{<~>}[r]^{\IIIp}&&&\bullet\ar@<-0.5mm>[ul]\ar@<0.5mm>[ul]}
\]
since all graphs are in-splittings of $\quad\quad\xymatrix{\bullet\ar@(dl,ul)[]&\bullet\ar@<-0.5mm>[l]\ar@<0.5mm>[l]}$ with two sets in the partition.
\end{remar}

\section{Time-changing moves}

We now move on to moves related to symbol expansion from symbolic dynamics, characterized by allowing slower transitions between two vertices in the graph. Such moves are inherently not gauge-invariant, and the versions hitherto used have also failed to provide $*$-isomorphisms rather than stable isomorphisms. Inspired by the work in \cite{seaer:ccka} we introduced in \cite{segrerapws:ccuggs} a move which allows to remove any regular vertex not supporting a loop in a way preserving \xyzrel{101}. It is a variation of a generalization of the \RRR\ move considered earlier which seems to have the right generality. We will still need the move \SSS\ of removing a source anywhere considered already in \cite{apws:gcsga} and note that it preserves \xyzrel{001}.

\begin{defin}[Move $\RRRp$:  Unital Reduction]\label{RRRp}\index{Rp@\RRRp}
Let $E=(E^0,E^1,r,s)$ be a graph and let $w$ be a regular vertex which does not support a loop.  Let $E_{R}$ be the graph defined by 
\begin{align*}
E_{R}^0 &= ( E^0  \setminus \{ w \} ) \sqcup \{ \widetilde{w} \} \\
E_{R}^1 &= \left( E^1 \setminus (  r^{-1} (w) \cup s^{-1} (w)  ) \right) \sqcup \{ [ ef ] \ : \ e \in r^{-1}(w) , f \in s_{E}^{-1}(w) \} \sqcup \{ \widetilde{f} \ : \ f\in s^{-1}(w) \}
\end{align*}
where the source and range maps of $E_{R}$ extend those of $E$, and satisfy $s_{{R}} ( [ef] ) = s(e)$, $s_{{R}}( \widetilde{f} ) = \widetilde{w}$, $r_{{R}} ( [ef] ) = r_{E}(f)$ and $r_{{R}} ( \widetilde{f} ) = r(f)$.
\end{defin}

\begin{theor}[$\RRRp\subseteq \xyzrel{101}$]\label{thm:reduction-plus}
Let $E=(E^0,E^1,r,s)$ be a graph and let $w$ be a regular vertex  which does not support a loop, and let $E_{R}$ be the graph in Definition~\ref{RRRp}.  Then there exists a $*$-isomorphism $\psi \colon C^*(E_{R}) \to C^*(E)$ such that $\psi ( \mathcal{D}_{{E_R}}) = \mathcal{D}_E$.
\end{theor}

\begin{proof}
Set $P_v = p_v$ for 	$v\in E^0\backslash\{w\}$, $P_{\widetilde{w}} = p_w$, $S_e = s_e$ for $e\in E^1\backslash(r^{-1}_E(w)\cup s^{-1}_E(w)$, $S_{[ef]} = s_e s_f$ for $e\in r^{-1}_E(w)$, $f\in s^{-1}_E(w)$, and $S_{ \widetilde{f} } = s_f$ for $f\in s^{-1}_E(w)$.  A computation shows that $\{ P_v , P_{\widetilde{w}}, S_e, S_{[ef]}, S_{\widetilde{f}} \}$ is a Cuntz-Krieger $E_{R}$-family in $C^*(E)$.  Therefore, there exists a $*$-homomorphism $\psi \colon C^*(E_{R}) \to C^*(E)$ such that $\psi ( p_v ) = P_v$, $\psi ( p_{\widetilde{w}} ) = P_{\widetilde{w}}$, $\psi ( s_e ) = S_e$, $\psi ( s_{ [ef] } ) = S_{[ef]}$, and $\psi ( s_{\widetilde{f}} ) = S_{ \widetilde{f}}$.  Note that $\psi$ sends each vertex projection to a nonzero projection and if $e_1 e_2 \ldots e_n$ is a vertex-simple cycle in $E_{R}$ with no exits (\emph{vertex-simple} means $s_{{R}} ( e_i) \neq s_{{R}} ( e_j )$ for $i \neq j$) then $\psi ( s_{e_1 e_2\cdots e_n})$ is a unitary in a corner of $C^*(E)$ with full spectrum.  Hence, by \cite[Theorem~1.2]{ws:gckut}, $\psi$ is injective.

We now show that $\psi$ is surjective.  Note that the only generators of $C^*(E)$ that are not obviously in the image of $\psi$ are $s_e$ with $r(e) = w$.  Let $e \in r^{-1}(w)$.  Then 
\[
s_e = s_e \sum_{ f \in s^{-1} (w) } s_f s_f^* = \sum_{ f \in s^{-1} (w) } S_{[ef]} S_{\widetilde{f}}^* = \psi \left( \sum_{ f \in s^{-1} (w) } s_{[ef]} s_{\widetilde{f}}^* \right).
\]
Therefore, the image of $\psi$ contains every generator of $C^*(E)$ which implies that $\psi$ is surjective.

We are left with showing $\psi ( \mathcal{D}_{{E_R}}) = \mathcal{D}_E$.  For a path $\mu = e_1 e_2 \ldots e_n $ in $E_{R}$, set $\nu_\mu = \nu_1 \nu_2 \cdots \nu_n$ where 
\[
\nu_i = 
\begin{cases}
e_i &\text{if $e_i \in E^1 \setminus (  r^{-1} (w) \cup s^{-1} (w)  ) $} \\
ef &\text{if $e_i = [ef]$}\\
f &\text{if $e_i = \widetilde{f}$} \\
v &\text{if $\mu = v$} \\
w &\text{if $\mu = \widetilde{w}$.}
\end{cases}
\] 
Since $\widetilde{f}$ is a source in $E_{R}$ for all $f \in s^{-1} (w)$, $e_i = \widetilde{f}$ for some $f \in s^{-1}(w)$ implies $i = 1$.  With this observation, it is now clear that $\nu_\mu$ is a path in $E$.  Note that if $\mu$ is a path in $E_{R}$, then $\psi ( s_{\mu} ) = s_{ \nu_\mu}$.  Therefore, for all paths $\mu$ in $E_{R}$, $\psi ( s_\mu s_\mu^*) = s_{\nu_\mu} s_{\nu_\mu}^*$ is an element in $\mathcal{D}_E$.  Since $\mathcal{D}_{{R}}$ is generated by elements of the form $s_{\mu}s_\mu^*$, $\psi ( \mathcal{D}_{{E_R}} ) \subseteq \mathcal{D}_E$.  Let $\nu$ be a path in $E$.  If $\nu = v \in E^0$, then by definition of $\psi$, $p_\nu$ is an element of $\psi ( \mathcal{D}_{{E_R}} )$.  Let $\nu = f_1 f_2 \cdots f_n$ be a path of positive length in $E$.  By grouping the edges $f_i f_{i+1}$ where $f_i \in  r^{-1}(w)$ and $f_{i+1} \in s^{-1}(w)$, we may write $\nu = \nu_1 \nu_2 \cdots \nu_m$ where $\nu_i$ is an edge or a path of length two of the form $ef$ with $e \in r^{-1}(w)$ and $f \in s^{-1} (w)$ satisfying $| \nu_i | = 1$ and $s(\nu_i) = w$ implies $i = 1$.  Set $\mu = \mu_1 \cdots \mu_n$ with 
\[
\mu_1 = 
\begin{cases}
\mu_1 &\text{if $|\nu_1| = 1$ and $s( \nu_1) \neq w$} \\
\widetilde{\nu_1} &\text{if $| \nu_1| = 1$ and  $s( \nu_1) = w$} \\
[ \mu_1 ] &\text{if $|\mu_1| = 2$} 
\end{cases}
\]
and 
\[
\mu_i = 
\begin{cases}
\nu_i &\text{if $| \nu_i | = 1$} \\
[ \nu_i ] &\text{if $|\nu_i| = 2$}
\end{cases} \mathcal{D}_{{E_R}} 
\]
for $i \geq 2$.  By construction, $\mu$ is a path in $E_{R}$ and $\psi ( s_{\mu} ) = s_{\nu}$.  Thus, 
\[
s_\nu s_{\nu}^* = \psi ( s_\mu s_\mu^*).
\]
Since $\mathcal{D}_E$ is generated by elements of the form $s_\nu s_\nu^*$, $\mathcal{D}_E \subseteq \psi ( \mathcal{D}_{{E_R}} )$.  We conclude that $\psi ( \mathcal{D}_{{E_R}} ) = \mathcal{D}_E$.
\end{proof}

Using \OOO\ and \RRRp\ moves only, it is now possible to obtain matching forms for two graphs that are sufficiently alike to have a chance of giving isomorphic $C^*$-algebras.  We will note  much more precise results in  Section \ref{GLandfriends}, but for immediate use, we prove:

\begin{propo}\label{getmatched}
There is a procedure which for two graphs $E,F$ either provide a certificate that $(E,F)\not\in \xyzrel{000}$, or produces a matched pair $(E',F')$ so that
\[
(E,E'),(F,F')\in \langle\OOO,\RRRp\rangle
\]
\end{propo}
\begin{proof}
It is straightforward to obtain (i) and (iii) of Section \ref{classummary} by moves of type \OOO, and using \RRRp\ moves until all regular vertices support a loop leads to (ii). We may now compute $\mathcal P$ for $E$ and $F$, and if they are not isomorphic, we have $(E,F)\not\in\xyzrel{000}$. If they are, we divide the generalized components into three types which we denote Type 0, Type 1 and Type 2 defined by assigning Type 0 and 1 to the cases where there are 0 and 1 edges, respectively, in the generalized components, and Type 2 to the remainder. We now investigate how to identify the generalized components in an order-preserving way in a way preserving type. If this is not possible, as will be witnessed by the ordered $K_0$-groups and by $K_1$, we again have $(E,F)\not\in\xyzrel{000}$. If it is, we chose any identification and note that the sizes are matched automatically in the Type 0 and 1 cases. In the Type 2 case, we note that there must be at least one vertex inside the generalized component which emits more than one edge, and that an \OOO\ move may be performed on that to increase the number of vertices in that generalized component, without changing anything else. Consequently, if two generalized components chosen for identification are not of the same size, the smaller one may be expanded until they are.
\end{proof}



\begin{defin}[Move \SSS: Remove a regular source]\label{SSS}\index{S@\SSS}
Let $E = (E^0 , E^1 , r, s)$ be a graph, and let $w\in E^0$ be a source that is also a regular vertex. 
Let $E_S$ denote the graph $(E_S^0 , E_S^1 , r_S , s_S )$ defined by
$$E_S^0 := E^0 \setminus \{w\}\quad
E_S^1 := E^1 \setminus s^{-1} (w)\quad
r_S := r|_{E_S^1}\quad
s_S := s|_{E_S^1}.$$
We  say $E_S$ is formed by performing move \SSS\ to $E$.
\end{defin}

\begin{theor}[$\SSS\subseteq \xyzrel{001}$]\label{thm:source}
Let $E$ be a graph and let $w$ be a regular source. Let $E_S$ be the graph in Definition~\ref{SSS}.  Then there exists a $*$-isomorphism $\psi \colon C^*(E_S)\otimes\K \to C^*(E)\otimes \K$ such that $\psi ( \mathcal{D}_{E_S}\otimes c_0) = \mathcal{D}_E\otimes c_0$. 
\end{theor}

\begin{proof}
Define $\phi \colon C^*(E_S) \to C^*(E)$ by $\phi ( p_v ) = p_v$ and $\phi ( s_e ) = s_e$ for all $v \in E_S^0$ and $e \in E_S^1$.  Let $q = \sum_{ v \in E_S } p_v$.  A computation shows that $\phi ( C^*(E_S) ) = q C^*(E) q$ and $\phi ( \mathcal{D}_{E_S} ) = q \mathcal{D}_E$.  Since $w$ is a regular source, $q$ is a full projection in $C^*(E)$.  The theorem now follows from \cite[Theorem~3.2, Theorem~4.2, and Corollary~4.5]{tmceras:esigdsiga}.
\end{proof}

\section{Advanced moves}

The Cuntz splice and Pulelehua moves are complicated moves designed (by Cuntz in \cite{jc:cpco} and the authors with Restorff and S\o{}rensen in \cite{segrerapws:ccuggs}) to change the graph essentially without changing the $K$-theory of the graph $C^*$-algebra. Proving that these moves do not change the graph $C^*$-algebra either is a key step for establishing classification (\cite{mr:ccka}, \cite{segrerapws:ics}, \cite{segrerapws:ccuggs}) and the stable isomorphisms thus obtained are not concrete and cannot be expected to preserve any additional structure.

It is, however, not hard to obtain versions which preserve $*$-isomorphism as corollaries to classification, and indeed this was first considered for the \CCC\ move in \cite{km:sncswcoe}. We choose a different solution, adding a number of sources to the standard construction which balances the situation enough to preserve the $K$-theoretical class of the unit. We may then prove that these moves are \xyzrel{100} by appealing to the classification results outlined in Section \ref{classummary}.

For a graph $F$ and $v\in F^0$, $\ee_v$ will denote the element in $\ZZ^{F^0}$ that is $1$ in the $v^{\text{th}}$ coordinate and zero elsewhere. By $[\ee_v]$ or $[\ee_v]_F$ we denote the image in $\cok \mathsf{B}_F^\bullet$.

\begin{defin}[Move \CCCp: Cuntz splicing] \label{CCCp}\index{Cp@\CCCp}
Let $E = (E^0 , E^1 , r , s )$ be a graph and let $u \in E^0$ be a regular vertex that supports at least two distinct return paths.
Let $E_{C}$ denote the graph $(E_{C}^0 , E_{C}^1 , r_{C} , s_{C})$ defined by 
\begin{align*}
E_{C}^0 &:= E^0\sqcup\{u_1 , u_2, u_3 \} \\
E_{C}^1 &:= E^1\sqcup\{e_1 , e_2 , f_1 , f_2 , h_1 , h_2 , g\},
\end{align*}
where $r_{C}$ and $s_{C}$ extend $r$ and $s$, respectively, and satisfy
$$s_{C} (e_1 ) = u,\quad s_{C} (e_2 ) = u_1 ,\quad s_{C} (f_i ) = u_1 ,\quad s_{C} (h_i ) = u_2 , \quad s_{C} ( g) = u_3$$
and
$$r_{C} (e_1 ) = u_1 ,\quad r_{C} (e_2 ) = u,\quad r_{C} (f_i ) = u_i ,\quad r_{C} (h_i ) = u_i , \quad r_{C} ( g) = u.$$
We say $E_{C}$ is formed by performing move \CCCp\ to $E$. 
\end{defin}

\begin{theor}[$\CCCp\subseteq \xyzrel{100}$]\label{thm:iso-cuntzsplice+}
Let $E$ be a graph and let $u$ be a regular vertex that supports at least two distinct return paths.  Then $C^*(E) \cong C^*( E_{C} )$.
\end{theor}

\begin{proof}
By \cite[Lemma~3.17]{segrerapws:gcgcfg}, there exists a graph $E'$ with finitely many vertices such that $C^*(E) \cong C^*(E')$, each infinite emitter in $E'$ emits infinitely many edges to any vertex it emits an edge to, and every transition state has exactly one outgoing edge.  To obtain $E'$, one out-splits each infinite emitter and each transition state.  Since these graph moves do not involve $u$ (regular vertex that supports at least two distinct return paths), we can outsplit the same vertices in $E_{C}$ and get a graph $F'$ such that $C^*( E_{ C+} ) \cong C^*(F')$ and $F' \cong E'_{ C+ }$.  Thus, without loss of generality, we may assume that each infinite emitter in $E$ emits infinitely many edges to any vertex it emits an edge to, and every transition state has exactly one outgoing edge.  Consequently, $\mathsf{B}_E = \mathsf{A}_E - \mathsf{I}$ is an element of the class of matrices $\mathfrak{M}_{\mathcal{P}}^\circ ( \mathbf{m} \times \mathbf{n} , \ZZ )$ defined in \cite[Definition~4.15]{segrerapws:gcgcfg}, and $\mathsf{B}_{E_{C}} =  \mathsf{A}_E - \mathsf{I}$ is an element of $\mathfrak{M}_{\mathcal{P}}^\circ ( ( \mathbf{m} + 3 \mathbf{e}_j) \times (\mathbf{n} + 3 \mathbf{e}_j ) , \ZZ )$ where $u$ belongs to the block $j \in \mathcal{P}$. Here $\mathbf{e}_j$ is the vector with $1$ in the $j^{\text{th}}$ component.\fxnote{Can be shortened using above, mainly \ref{getmatched}}

Note that the diagonal block of $\mathsf{B}_{E_{C}}^\bullet\{j\}$ can be described as the matrix
\[
\left[\begin{array}{c|c}
 \mathsf{B}_E^\bullet\{j\}
 &\begin{smallmatrix}\vdots&\vdots&\vdots\\0&0&0\\1&0&0\\&&\end{smallmatrix}\\\hline
 \begin{smallmatrix}&&\\\cdots&0&1\\\cdots&0&0\\\cdots&0&1\end{smallmatrix}&
 \begin{smallmatrix}&&\\0&1&0\\1&0&0\\0&0&-1\end{smallmatrix} 
\end{array}\right].
\]
Let $U \in \mathrm{GL}_{\mathcal{P}} ( ( \mathbf{m} + 3 \mathbf{e}_j) \times (\mathbf{n} + 3 \mathbf{e}_j ) , \ZZ )$ and $V \in \mathrm{SL}_{\mathcal{P}} ( \mathbf{m} + 3 \mathbf{e}_j) \times (\mathbf{n} + 3 \mathbf{e}_j ) , \ZZ )$ be the identity matrix everywhere except the $j$'th diagonal block where they are given by
\[
\left[\begin{array}{c|c}
\idmatrix
 &\begin{smallmatrix}\vdots&\vdots&\vdots\\0&0&0\\0&-1&0\\&&\end{smallmatrix}\\\hline
 \begin{smallmatrix}&&\\\cdots&0\\\cdots&0\\\cdots&0\end{smallmatrix}&
 \begin{smallmatrix}&&\\0&-1&0\\-1&0&0\\0&0&1\end{smallmatrix} 
\end{array}\right]
\]
and
\[
\left[\begin{array}{c|c}
\idmatrix
 &\begin{smallmatrix}\vdots&\vdots&\vdots\\0&0&0\\&&\end{smallmatrix}\\\hline
 \begin{smallmatrix}&&\\\cdots&0&0\\\cdots&0&-1\\\cdots&0&-1\end{smallmatrix}&
 \begin{smallmatrix}&&\\1&0&0\\0&1&0\\0&0&1\end{smallmatrix} 
\end{array}\right]
\]
respectively.
Then $U \mathsf{B}_{E_{C}}^\bullet V = - \iota_{3\mathbf{e}_j } ( - \mathsf{B}_E^\bullet )$ where $\iota_{\mathbf r}$ is the embedding of a block matrix to a larger block matrix for any multiindex $\mathbf r$ (see \cite[Definition~4.1]{segrerapws:gcgcfg}).

As\fxnote{Update using notation from \ref{classummary}} explained in \cite[Section~4.4]{segrerapws:gcgcfg}, $(U,V)$ induces an isomorphism $\mathrm{FK}_{\mathcal{R}}(U,V)$ on the reduced filtered $K$-theory from $\mathrm{FK}_{ \mathcal{R} } ( \mathcal{P} , C^*( E_{ C+ } ) )$ to $\mathrm{FK}_{ \mathcal{R} } ( \mathcal{P} , C^*( E ) )$ where $V^T$ induces the isomorphism on the $K_0$-groups.  Note that $[p_{ u_1 } ] = 0$ and $[ p_{u_3} ] = [ p_u ]$ in $K_0 ( C^*(E_{C}) )$.  Note that 
\[
V^T(\ee_v) = \ee_v \quad\text{and}\quad V^T( \ee_{u_2} ) = -\ee_{ u }+\ee_{u_2}
\]
for all $v \in E^0$.  By identifying $K_0(C^*(E_{C}))$ with $\cok \mathsf{B}_{E_{C}}^\bullet$ via the identification $[p_v] \mapsto [\ee_v]$ and by identifying $K_0(C^*(E))$ with $\cok(-\iota _{3{\mathsf e}_j}(-\mathsf{B}_{E}^\bullet))$ in a similar way, we get $V^T([\ee_v]_{E_{C}})=[\ee_v]_E$ and $V^T([\ee_{u_2}]_{E_{C}})=-[\ee_u]_E$ since $\ee_{u_2}\in \operatorname{im}((-\iota _{3{\mathsf e}_j}(-\mathsf{B}_{E}^\bullet))$. Therefore,
\begin{align*}
V^T\left( \sum_{v \in E_{C}^0} [\ee_v] _{E_{C}}\right) &= V^T \left( \sum_{v \in E^0 } [\ee_v] _{E_{C}} + [\ee_{ u_2 }] _{E_{C}} + [\ee_{ u _3 }] _{E_{C}}  \right) \\
					&= \sum_{ v \in E^0} [\ee_v]_E -[ \ee_{u}]_E + V^T ([ \ee_{u}]_{E_{C}} )\\
					&= \sum_{ v \in E^0 } [\ee_v]_E. 
\end{align*}
Hence, $\mathrm{FK}_{\mathcal{R}}(U,V)$ sends $[ 1_{C^*(E_{C})} ]$ to $[ 1_{C^*(E) } ]$.  Moreover, since $V^T ( \ee_v ) = \ee_v$ for all $v \in E^0$ and the simple subquotient corresponding to the $j$th component is a purely infinite simple $C^*$-algebra, $\mathrm{FK}_{\mathcal{R}}(U,V)$ is an isomorphism on the ordered reduced filtered $K$-theory from $\mathrm{FK}_{ \mathcal{R} }^+ ( \mathcal{P} , C^*( E_{ C+ } ) )$ to $\mathrm{FK}_{ \mathcal{R} }^+ ( \mathcal{P} , C^*( E ) )$.  By \cite[Theorem~3.5]{segrerapws:ccuggs}, $C^*(E_{C}) \cong C^*(E)$.
\end{proof}


We denote the graph where we have applied \CCCp\ to each vertex $w$ in $S$ by \[E_{S,-}=(E_{S,-}^0,E_{S,-}^1,r_{S,-},s_{S,-}).\] For each $w\in S$, $v_1^w$, $v_2^w$, $v_3^w$ will denote the additional vertices from the \CCCp\ move at $w$ enumerated as in Definition \ref{CCCp}.

\begin{defin}[Move $\mbox{\texttt{\textup{(P+)}}}$: Eclosing a cyclic component] \label{PPPp}
Let $E=(E^0,E^1,r,s)$ be a graph and let $u$ be a regular vertex that supports a loop and no other return path, the loop based at $u$ has an exit, and if $w \in E^0 \setminus \{u\}$ and $s^{-1} (u) \cap r^{-1}(w) \neq \emptyset$, then $w$ is a regular vertex that supports at least two distinct return paths.  We construct $E_{u, P}$ as follows.  

Set $S = \{ w \in E^0 \setminus \{u\} : s^{-1}(u) \cap r^{-1}(w) \neq \emptyset \}$ (since the loop based at $u$ has an exit, $S$ is nonempty, and clearly $u \notin S$).  Set $E_{u, P}^0 = E_{S,-}^0$ and
\[
E_{u, P}^1 = E_{S,-}^1 \sqcup \{ \overline{e}_w, \tilde{e}_w  : \ w \in S, e \in s^{-1} (u) \cap r^{-1}(w) \}
\]
with $s_{{u, P}} \vert_{E_{S,-}^1} = s_{{S,-}}$, $r_{{u, P}} \vert_{E_{S,-}^1} = r_{{S,-}}$, $s_{{u, P}} ( \overline{e}_w ) = s_{{u, P}} ( \tilde{e}_w ) = u$, $r_{{u, P}} ( \overline{e}_w ) = r_{{u, P}} ( \tilde{e}_w ) = v_2^w$. 

We say that $E_{u,P}$ is formed by performing move $\mbox{\texttt{\textup{(P+)}}}$ to $E$.   
\end{defin}

\begin{theor}[$\PPPp\subseteq \xyzrel{100}$]\label{thm:iso-pulelehua+}\index{Pp@\PPPp}
Let $E=(E^0,E^1,r,s)$ be a graph  and let $u$ be a regular vertex that supports a loop and no other return path, the loop based at $u$ has an exit, and if $w \in E^0 \setminus \{u\}$ and $s^{-1} (u) \cap r^{-1}(w) \neq \emptyset$, then $w$ is a regular vertex that supports at least two distinct return paths.  Then $C^*(E_{u, P}) \cong C^*(E)$.
\end{theor}

\begin{proof}
We again use a similar argument as in the proof of Theorem~\ref{thm:iso-cuntzsplice+}, without loss of generality, we may assume that each infinite emitter in $E$ emits infinitely many edges to any vertex it emits an edge to, and every transition state has exactly one outgoing edge.  Thus, $\mathsf{B}_E = \mathsf{A}_E - \mathsf{I}$ is an element of $\mathfrak{M}_{\mathcal{P}}^\circ ( \mathbf{m} \times \mathbf{n} , \ZZ )$.  Let $j$ be the component of $u$ and let $j_1, j_2, \ldots, j_n$ be the components in $\mathcal{P}$ such that for each $i$, there exists $w$ in the component $j_i$ such that $\mathsf{B}_E ( u, w ) \neq 0$.  

Let $S = \{ w \in E^0 \setminus \{u\} : s^{-1}(u) \cap r^{-1}(w) \neq \emptyset \}$, let 
\[
\mathcal{E}_i = \{ w \in E^0 \ : \ \text{$w$ is in the $j_i$ component and $\mathsf{B}_E ( u, w ) \neq 0$} \},
\]
and let $k_i = |\mathcal{E}_i |$.  Let $V$ be the operation matrix which implements the matrix operations
\begin{itemize}
\item Add column $v_3^w$ to column $w$ for all $w \in S$
\item Subtract column $v_2^w$ from column $w$ for all $w \in S$
\end{itemize}
 by multiplication from the right.
Let $\mathbf{r} = \sum_{ i = 1 }^n 3 k_i\mathbf{e}_{j_i}$.  Then $V \in \mathrm{SL}_{\mathcal{P}} ( ( \mathbf{m} + \mathbf{r}) \times (\mathbf{n} + \mathbf{r} ) , \ZZ )$ since the column operations are done in each component.  Let $U$ be the operation matrix which implements the matrix operations
\begin{itemize}
\item Subtract row $v_3^w$ from row $w$ for all $w \in S$
\item Subtract $\sum_{w\in S}2|s^{-1}_E(u)\cap  r^{-1}_E(w)|$ times row $v^w_1$ from row $u$ for all $w \in S$
\item Multiply row $u$ by $-1$
\item Multiply rows $v^w_1$ and $v^w_2$ by $-1$ for all $w \in S$
\item Interchange rows $v_1^w$ and  $v_2^w$ for all $w \in S$
\end{itemize}
by multiplication from the left. Since the row operations are done in the components $j_i$, we have that $U \in \mathrm{GL}_{\mathcal{P}} (( \mathbf{m} + \mathbf{r}) \times (\mathbf{n} +  \mathbf{r} ) , \ZZ )$.  

A\fxnote{Update using notation from \ref{classummary}} computation shows that $U \mathsf{B}_{E_{u, P}}^\bullet V = - \iota_{\mathbf{r}} ( - \mathsf{B}_E^\bullet )$.  As explained in \cite[Section~4.4]{segrerapws:gcgcfg}, $(U,V)$ induces an isomorphism $\mathrm{FK}_{\mathcal{R}}(U,V)$ on the reduced filtered $K$-theory from $\mathrm{FK}_{ \mathcal{R} } ( \mathcal{P} , C^*( E_{ u, P } ) )$ to $\mathrm{FK}_{ \mathcal{R} } ( \mathcal{P} , C^*( E ) )$ where $V^T$ induces the isomorphism on the $K_0$-groups. 
Note that $V^T\ee_v=\ee_v$ for $v\in E^0$ and $V^T\ee_{v_2^w}=-\ee_w+\ee_{v_2^w}$ for all $w\in S$. Since $[p_{v^w_1}]=0$, $[p_{v^w_2}]=[p_w]$ in $K_0(C^*(E_{u,P}))$ and $e_{v^w_2}$ is in the image of $-\iota_{\mathbf r}(-\mathsf{B}_E^\bullet)$, identifying the $K_0$-groups with the cokernels of the associated matrices we get 
\begin{align*}
V^T\left( \sum_{v \in E_{u , P}^0} [\ee_v] \right) &= V^T \left( \sum_{v \in E } [\ee_v] \right) + V^T\left( \sum_{w \in S } [\ee_{v^w_2}] \right) + V^T \left( \sum_{w \in S } [\ee_{v^w_3}] \right) \\
					& = \sum_{ v \in E^0 \setminus S } [\ee_v] + \sum_{w \in S} [\ee_{w}] - \sum_{w \in S }  [\ee_{w}] + V^T \left(\sum_{ w \in S } [\ee_w] \right) \\
					&=  \sum_{ v \in E^0 \setminus S} [\ee_v] + \sum_{w \in S } [\ee_w] \\
					&= \sum_{v \in E^0} [\ee_v]. 
\end{align*}
Hence, $\mathrm{FK}_{\mathcal{R}}(U,V)$ sends $[ 1_{C^*(E_{u, P})} ]$ to $[ 1_{C^*(E) } ]$.  Moreover, since $V^T ( \ee_v ) = \ee_v$ for all $v \in E^0$ and the simple subquotient corresponding to the $j_i$th component is a purely infinite simple $C^*$-algebra, $\mathrm{FK}_{\mathcal{R}}(U,V)$ is an isomorphism on the ordered reduced filtered $K$-theory from $\mathrm{FK}_{ \mathcal{R} }^+ ( \mathcal{P} , C^*( E_{ u, P } ) )$ to $\mathrm{FK}_{ \mathcal{R} }^+ ( \mathcal{P} , C^*( E ) )$.  By \cite[Theorem~3.5]{segrerapws:ccuggs}, $C^*(E_{u,P}) \cong C^*(E)$.
\end{proof}

\chapter{The Krieger moves}

As we shall see below, the two types of isomorphism \xyz{010} and \xyz{110} cannot be generated by the moves we have introduced thus far. To remedy this situation, we introduce here two \emph{Krieger moves} \KKKp\ and  \KKKm\ which are of a dramatically different nature than the previously defined ones.

\section{The dimension quadruple}

It follows from Definition \ref{xyzdef} that any  \xyz{11z}-equivalence among $E$ and $F$ implies that $\Ff_E\simeq \Ff_F$ and that any 
 \xyz{01z}-equivalence implies that $\Ff_E\otimes \K\simeq \Ff_F\otimes \K$, with  $\Ff_E$ the fixed point algebra
 \[
 \Ff_E=\{a\in C^*(E)\mid \forall z\in \TT:\gamma^E_z(a)=a\}.
 \]
This $C^*$-algebra is extremely well understood, since it can be described using the \emph{skew product} graph\index{skew product} 
 $E\times_1\ZZ$ defined as  the graph with vertices $E^0 \times \ZZ$ and edges $E^1 \times \ZZ$ with range and source maps given by 
\[
s( e, n) = ( s(e), n-1) \quad \text{and} \quad r( e, n) = ( r(e), n ).
\]
Note that $E\times_1 \ZZ$ with infinitely many vertices is not directly amenable to analysis by moves. It is a classical result by Crisp that  $\Ff_E$ is a corner  of the graph $C^*$-algebra associated to the skew product:

\begin{theor}[\cite{tc:cga}]\label{tyrone}
Let $E$ be a graph with finitely many vertices. Then
\[
\Ff_E\simeq p_0^E C^*(E^1 \times_1 \ZZ)p_0^E
\]
where\index{pE0@$p_0^E$}
\[
p_0^E=\sum_{v\in E^0}p_{(v,0)}.
\]
\end{theor}

Note that since $E\times _1\ZZ$ is acyclic by construction, $C^*(E\times_1 \ZZ)$ is AF by  \cite[Corollary~2.13]{ddmt:cag}. Therefore $\Ff_E$ is always AF.  We now proceed to collect results which describe $\Ff_E$ dynamically. These results are definitely known to experts, but we have not been able to locate them in the literature at the needed level of precision for our purposes.

It is easy to see that when $\{p_{(v,n)},s_{(e,n)}\}$ is a Cuntz-Krieger $E\times_1\ZZ$-family, then so is $\{p_{(v,n+1)},s_{(e,n+1)}\}$, and consequently a canonical (right) translation map $\lt\in\operatorname{Aut}(C^*(E\times_1\ZZ))$ is defined.\index{rt@$\lt$}
The key refinement is to retain information carried by $\lt$, which is essentially a dual action $\widehat{\gamma^E}$.

For each $a \in C^*(E)$ and $f \in C(\TT)$, $f \otimes a$ will denote the continuous function $z \mapsto f(z) a\in C(\TT,C^*(E))$.  We will also consider $f \otimes a$ as an element of $C^*(E) \rtimes_{\gamma^E} \TT$ in the canonical way inside the regular representation. For each $n \in \ZZ$, $f_n$ will denote the function in $C(\TT)$ given by $f_n(z) = z^n$.  By \cite[Lemma~3.1]{irws:ckigm}, there exists an $*$-isomorphism $\phi_E \colon C^*(E \times_1 \ZZ) \to C^*(E) \rtimes_{\gamma^E} \TT$ such that 
$$
\phi_E( p_{(v,n)} ) = f_n \otimes p_v \quad \text{and} \quad \phi_E( s_{(e,n)} ) = f_n \otimes s_{e}
$$
for all $(v,n) \in (E \times_1\ZZ)^0$ and for all $(e,n) \in (E \times_1 \ZZ)^1$ such that $(\widehat{\gamma^E})_1 \circ \phi = \phi \circ \lt$, where $\widehat{\gamma^E}$ is the dual $\ZZ$-action on $C^*(E) \rtimes_{\gamma^E} \mathbb{T}$ and $1$ denotes the canonical generator of the Pontrjagin dual $\widehat{\TT} \cong \ZZ$.

\begin{theor}\label{raecor}
If $E$ and $F$ are \xyz{110}-equivalent, then there exists a $*$-isomorphism $\psi \colon C^*(E\times_1\ZZ) \to C^*(F\times_1\ZZ)$ such that $\psi \circ \lt = \lt \circ \psi$ and $\psi ( p^E_0 )= p_0^F$.
\end{theor}


\begin{proof}
Let $\phi$ be a \xyz{110}-isomorphism from $C^*(E)$ to $C^*(F)$.  By \cite[Corollary~2.4.8]{dpw:cpc}, there exist homomorphisms $ \phi \rtimes \id \colon C^*(E) \rtimes_{\gamma^E} \TT \to C^*(F) \rtimes_{ \gamma^F } \TT$ and $\phi^{-1} \rtimes \id \colon C^*(F) \rtimes_{\gamma^F} \TT \to C^*(E) \rtimes_{\gamma^E } \TT$ such that $\phi \rtimes \id ( f )(z) = \phi ( f(z) )$ and $\phi^{-1} \rtimes \id ( g)(z) = \phi^{-1} ( g(z) )$ for all $f \in C( \TT, C^*(E) )$, $g \in C(\TT, C^*(F))$, and $z \in \TT$.  Consequently, $(\phi^{-1} \rtimes \id) \circ ( \phi \rtimes \id )(f) (z) = f(z)$ and $(\phi \rtimes \id ) \circ ( \phi^{-1} \rtimes \id )( g)(z) = g(z)$ for all $f \in C( \TT, C^*(E) )$, $g \in C(\TT, C^*(F))$, and $z \in \TT$.  Since $C(\TT, C^*(E))$ is dense in $C^*(E) \rtimes_{\gamma^E} \TT$ and $C^*(\TT, C^*(F))$ is dense in $C^*(F) \rtimes_{\gamma^F} \TT$, we have that $(\phi \rtimes \id ) \circ ( \phi^{-1} \rtimes \id ) = \id$ and $( \phi^{-1} \rtimes \id ) \circ ( \phi \rtimes \id ) = \id$.  Hence, $\phi \rtimes \id$ is a $*$-isomorphism.

Note that 
\begin{align*}
( ( \phi \rtimes \id ) \circ (\widehat{\gamma^E})_\tau (f) )(z) &= \phi ( (\widehat{\gamma^E})_\tau ( f )(z)) \\
							&= \phi ( \tau (z)  f(z) ) = \tau (z) \phi ( f(z) ) \\
							&=  \tau (z) (\phi \rtimes \id)(f)(z) \\
							&= (\widehat{\gamma^F})_\tau \circ (\phi \rtimes \id )( f)(z),
\end{align*}
for all $f \in C( \TT , C^*(E))$, for all $z\in \TT$, and for all $\tau \in \widehat{\TT}$.  And since $C( \TT, C^*(E))$ is dense in $C^*(E) \rtimes_{\gamma^E} \TT$, we have that $\phi \rtimes \id$ is an equivariant $*$-isomorphism from $( C^*(E) \rtimes \TT, \widehat{\gamma^E} )$ to $(C^*(F) \rtimes \TT, \widehat{\gamma^F})$.

Let $\phi_E \colon C^*(E \times_1 \ZZ ) \to C^*(E)\rtimes_{\gamma^E} \TT$ and $\phi_F \colon C^*(F \times_1\ZZ) \to C^*(F) \rtimes_{\gamma^F} \TT$ be the $*$-isomorphisms given in \cite[Lemma~3.1]{irws:ckigm} for $E$ and $F$ respectively.  Note that 
\begin{align*}
(\phi \rtimes \id) \left( \sum_{v \in E^0 } f_0 \otimes p_v \right)(\zeta) &= \sum_{ v \in E^0 }  ( \phi \rtimes \id )( f_0 \otimes p_v )( \zeta ) \\
									&= \sum_{ v \in E^0 } \phi ( p_v ) \\\
									&= \phi ( 1_{C^*(E) } ) \\\
									&= 1_{C^*(F)} \\
									&= \sum_{w \in F^0 } p_w \\
									&= \sum_{ w \in F^0 } ( f_0 \otimes p_w )(\zeta)
\end{align*}
for all $\zeta \in \TT$.  Thus, $\phi \rtimes \id \left( \sum_{v \in E^0 } f_0 \otimes p_v \right) = \sum_{ w \in F^0 } f_0 \otimes p_w $.  Set $\psi = \phi_F^{-1} \circ (\phi \rtimes \id ) \circ \phi_E$.  Then $\psi$ is a $*$-isomorphism from $C^*(E\times_1\ZZ)$ to $C^*(F\times_1\ZZ)$, 
\[
\psi \left( p^E_0 \right) =  \phi_F^{-1} \circ ( \phi \rtimes \id ) \circ \left( \sum_{v \in E^0 } f_0 \otimes p_v  \right)  = \phi_F^{-1} \left( \sum_{w \in F^0 }  f_0 \otimes p_w \right) = p^F_0,
\]
and 
\begin{align*}
\lt \circ \psi &= \lt \circ \phi_F^{-1} \circ (\phi \rtimes \id ) \circ \phi_E  \\
			&= \phi_F^{-1} \circ ( \widehat{\gamma^F})_1 \circ (\phi \rtimes \id ) \circ \phi_E   \\
			&= \phi_F^{-1} \circ ( \phi \rtimes \id ) \circ (\widehat{\gamma^E})_1 \circ \phi_E \\
			&= \phi_F^{-1} \circ ( \phi \rtimes \id ) \circ  \phi_E \circ \lt \\
			&= \psi \circ \lt.\qedhere
\end{align*}
\end{proof}

Since the $C^*$-algebras we consider here are AF, there is no loss of information in passing to $K$-theory. We note explicitly:

\begin{corol}\label{xIOinvaK}
Suppose $E$ and $F$ are \xyz{x10}-equivalent.  Then there exists an order isomorphism $h \colon K_0(C^*(E\times_1\ZZ) )\to K_0(C^*(F\times_1\ZZ))$ such that $h \circ \lt_* = \lt_* \circ h$ and such that for some  $n\in \NN$,
\[
h ( [p^E_0 ])\leq n [p^F_0 ] \qquad [p^F_0 ]\leq nh( [p^E_0 ]).
\]
 When $\xyz{x}=\xyz{1}$, i.e. when  $E$ and $F$ are \xyz{110}-equivalent, we may further assume that
\[
h ( [p^E_0 ])= [p^F_0 ].
\]
\end{corol}
\begin{proof}
The claims for $\xyz{x}=\xyz{1}$ follow directly from Theorem \ref{raecor}. When $\xyz{x}=\xyz{0}$ we employ \cite[Lemma 2.75]{dpw:cpc} which shows that
\[
(C^*(E)\rtimes_{\gamma^E} \TT)\otimes\K\simeq (C^*(E)\otimes\K)\rtimes_{\gamma^E \otimes\id}\TT.
\]
In fact, it follows from the proof given there that the implementing isomorphism $\psi$ makes
\[
\xymatrix{C^*(E\times _1\ZZ)\otimes\K\ar^-{\phi\otimes\id}[r]\ar[d]_-{\lt\otimes\id}&(C^*(E)\rtimes_{\gamma^E} \TT)\otimes\K\ar[d]^-{\widehat{\gamma^E}\otimes\id}&\ar[l]_-{\psi} (C^*(E)\otimes\K)\rtimes_{\gamma^E \otimes\id}\TT\ar[d]_-{\widehat{\gamma^E \otimes\id}}\\
C^*(E\times _1\ZZ)\otimes\K\ar^-{\phi\otimes\id}[r]&(C^*(E)\rtimes_{\gamma^E} \TT)\otimes\K&\ar[l]_-{\psi} (C^*(E)\otimes\K)\rtimes_{\gamma^E\otimes\id}\TT.}
\]
commute, so since the isomorphism implementing the \xyz{x10}-equivalence induces an isomorphism intertwining $\widehat{\gamma^E\otimes\id}$ and $\widehat{\gamma^F\otimes\id}$ between $(C^*(E)\otimes\K)\rtimes_{\gamma^E\otimes\id}\TT$ and $(C^*(F)\otimes\K)\rtimes_{\gamma^F\otimes\id}\TT$ on the right, we obtain an an isomorphism intertwining $\lt\otimes\id$ on the left, sending $p^E_0$ inside the ideal generated by $p_0^F$. This easily translates to the given $K$-theoretical conditions.
\end{proof}

 As we will detail in Section \ref{twosided} below, the invariant consisting of $K_0(C^*(E\times_1\ZZ) )$ and $\lt_*$, with the former considered as an ordered group, was introduced by Krieger in \cite{wk:dftmc} as an invariant for Williams' shift equivalence in the foundational case when $E$ is finite and has no sinks or sources. We follow Krieger in calling this data the \emph{dimension triple} for all graphs considered here, but to retain all the information in  Corollary \ref{xIOinvaK} we additionally define dimension \emph{quadruples} as follows.
 
 When $(G,G_+)$ is a partially ordered group, and $x\in G_+$ we say that the \emph{order ideal}\index{order ideal} $I(x)$ generated by $x$ is defined as
\[
I(x)=\{y\in G\mid \exists n\in\NN: 0\leq y\leq nx\}.
\]

\begin{defin}\index{dimension quadruple}\index{DQ@$\DQ(-)$}\index{Kp@$\KKKp$}\index{DQp@$\DQp(-)$}\index{Km@$\KKKm$}\label{KKKp}
When $E$ is a graph, the \emph{dimension triple} is defined as
\[
\DT(E)=(K_0(C^*(E\times_1\ZZ)),K_0(C^*(E\times_1\ZZ))_+,\lt_*)
\]\index{DT@$\DT(-)$}
Two kinds of \emph{dimension quadruples} of $E$ are defined as
\[
\DQp(E)=(K_0(C^*(E\times_1\ZZ)),K_0(C^*(E\times_1\ZZ))_+,\lt_*,[p_0^E])=(\DT(E),[p_0^E])
\]
and
\[
\DQ(E)=(K_0(C^*(E\times_1\ZZ)),K_0(C^*(E\times_1\ZZ))_+,\lt_*,I([p_0^E]))=(\DT(E),I([p_0^E])).
\]
\end{defin}

We use the obvious notions of isomorphism for dimension triples and quadruples; group isomorphisms between $K_0(C^*(E\times_1\ZZ))$ and $K_0(C^*(F\times_1\ZZ))$ that intertwine the $\lt_*$ maps and preserve the remaining data.

We are now ready to introduce the two Krieger moves \KKKp\ and \KKKm\ as follows:

\begin{defin}
We say that $E$ is \KKKp-equivalent to $F$ when
\[
\DQp(E)\simeq \DQp(F)
\]
and that $E$ is \KKKm-equivalent to $F$ when
\[
\DQ(E)\simeq \DQ(F).
\]
\end{defin}

Using this notation, Corollary \ref{xIOinvaK} becomes:

\begin{propo}\label{KKpremembers}
 $\xyzrel{010}\subseteq\langle\KKKm\rangle$,  $\xyzrel{110}\subseteq\langle\KKKp\rangle$.
\end{propo}



\section{Partial invariance}\label{hazratdisc}

Because of the arithmetic nature
of the \KKKp\ and \KKKm\ moves,
establishing invariance for them takes the
form of classification, showing that
$\DQ(-)$ is a complete invariant for \xyz{010}-isomorphism
and $\DQp(-)$ for \xyz{110}. But this in turn is
an open question which has been
studied in key cases
since the late 1990's and was formalized
by Roozbeh Hazrat in \cite{rh:ggclpa}.

Consequently, the invariance problem for \KKKp\
is exactly the Hazrat conjecture for
graded isomorphism, whereas the invariance
problem for \KKKm\ is a variation of
the Hazrat conjecture for graded Morita
equivalence.
We strongly believe that the conjectures
hold for graph $C^*$-algebras, and our work contains new
evidence for it, but for the time being, we must leave the invariance of the \KKKp\ and \KKKm\ moves
in a conjectural state.

We can, however, fully establish partial
invariance which shows, as least, that when the invariants are the same, the graphs are equivalent in the weaker \xyz{100} or \xyz{000} form, respectively.
We will do so by following an approach developed in \cite{parhhl:gkfkcga}, in which 
\[
\DT(E)\simeq \DT(F)\Longrightarrow (E,F)\in \xyzrel{000}
\]
was
shown for all  regular graphs. We shall see below in easy examples that we need the more refined invariant $\DQ(-)$ to have a chance of $\langle \KKKm\rangle=\xyzrel{010}$ to hold for non-regular graphs, but to obtain \xyz{000}-equivalence the $\DT(-)$ part is sufficient in general, and we work just like in  \cite{parhhl:gkfkcga} by interpolating between the dimension triples and stable isomorphism with the filtered, ordered $K$-theory $\mathrm{FK}_X^+ (C^*(E))$  which is a complete invariant for  \xyz{000}-equivalence by the main  classification result of \cite{segrerapws:ccuggs}, Theorem \ref{ERRSmain}. To allow for exact isomorphism (\xyz{100}-equivalence) to be inferred from $\DQp(-)$ we simply keep track of the class of the unit in $K$-theory, since \cite{segrerapws:ccuggs} provides the relevant classification result as well, as summarized in Theorem \ref{ERRSunital}. 

To obtain these results, we find it advantageous to interpolate further with
an invariant introduced by Bentmann and Meyer \cite{rbrm:mgmcek},
which naturally connects to the
data given in $\DT(-)$ and $\DQp(-)$,
and was established to be a complete invariant for stable isomorphism of a
large class of purely infinite $C^*$-algebras
with finitely many ideals. This overlaps
significantly with the classification
results in  \cite{segrerapws:ccuggs}, and extends
substantially beyond it, but to classify
the $C^*$-algebras $C^*(E)$ that are not
purely infinite, and to obtain exact
rather than stable classification, the
Bentmann-Meyer invariant must be augmented
in a natural way. There are no great surprises
in doing so, but since this may be of interest for future work, say for the open problem of whether $\mathrm{FK}^+_X(-)$ is also a complete invariant for stable isomorphism of general graph $C^*$-algebras with finitely many ideals, we will do so in the ensuing section deviating temporarily from our standing assumption that all graphs have finitely many vertices. 

We prove this by establishing the downwards arrows in
\[
\xymatrix{
(E,F)\in\xyzrel{010}\ar@{=>}[r]&\DQ(E)\simeq\DQ(F)\ar@{=>}[d]\\
&\DT(E)\simeq\DT(F)\ar@{=>}[d]\\
&\mathrm{XK}\delta^+(C^*(E))\simeq \mathrm{XK}\delta^+(C^*(F))\ar@{=>}[d]\\
&\mathrm{FK}^+_X(C^*(E))\simeq\mathrm{FK}^+_X (C^*(F))\ar@{<=>}[r]&(E,F)\in\xyzrel{000}
}
\]
and in
\[
\xymatrix{
(E,F)\in\xyzrel{110}\ar@{=>}[r]&\DQp(E)\simeq\DQp(F)\ar@{=>}[d]\\
&\mathrm{XK}\delta^{+,1}(C^*(E))\simeq \mathrm{XK}\delta^{+,1}(C^*(F))\ar@{=>}[d]\\
&\mathrm{FK}^{+,1}_X (C^*(E))\simeq\mathrm{FK}^{+,1}_X(C^*(F))\ar@{<=>}[r]&(E,F)\in\xyzrel{100}
}
\]
As summarized in Section \ref{classummary}, invoking the $\mathrm{XK}\delta$ and $\mathrm{FK}$ invariants also requires control of the finite topological spaces $\mathrm{Prim}_\gamma ( C^*(E))$ and $\mathrm{Prim}_\gamma ( C^*(F))$ which we will show are homeomorphic when $\DT(E)\simeq \DT(F)$.

The key results for this are  Theorem~\ref{thm:XKdelta-filteredkthy} and Theorem~\ref{thm:graded-kthy-XKdelta}.  Using \cite{segrerapws:ccuggs} classification, we get the desired conclusion:


\begin{propo}\label{Kpartialinva}
 $\langle\KKKp\rangle \subseteq \xyzrel{100}$ and  $\langle\KKKm\rangle\subseteq \xyzrel{000}$.
\end{propo}

\section{Comparing invariants}

In this section, we depart from our standing assumption that all graphs have finitely many vertices.  
We start by introducing the invariant $\mathrm{XK}\delta(C^*(E))$ defined by  Bentmann and Meyer in \cite{rbrm:mgmcek}. 

The invariant $\mathrm{XK}_*(A)$\index{XKs@$\mathrm{XK}_*(-)$} was introduced by Bentmann in  \cite[Definition~3.1]{rb:csipiga} (also see \cite[Definition~4.7]{rb:kxrrzic}) and it is defined as follows.  Let $A$ be a $C^*$-algebra over $X$ (cf. Section \ref{classummary}), and recall from there  that for all open sets $U$ and $V$ with $U \subseteq V$, we have an ideal embedding $\iota_U^V \colon A[U] \to A[V]$.  Then $\mathrm{XK}_*(A)$ denotes the invariant consisting of the collection of $\ZZ_2$-graded abelian groups $K_*(A[U_x])$ for each $x \in X$ (where $U_x$ is the smallest open neighborhood containing $x$) together with the collection of graded group homomorphisms $( \iota_{U_x}^{U_y} )_*$ induced on the $K_*$-groups whenever $U_x \subseteq U_y$.  

As when defining $\FKX(-)$, when  $C^*(E)$ is a graph $C^*$-algebra with finitely many gauge-invariant ideals we always work with  $X = \mathrm{Prim}_\gamma (C^*(E))$, and get that the distinguished ideals of  $C^*(E)$  are exactly the gauge-invariant ideals of $C^*(E)$.  Let $U$ be an open subset of $X$.  Since $C^*(E)[U]$ is a gauge-invariant ideal of $C^*(E)$, by  the lattice isomorphism from Section \ref{adpairs}, there exists a unique admissible pair, $(H_U,S_U)$, such that $C^*(E)[U] = I_{ (H_U, S_U)}$.  Note that $\gamma^E$ restricts to an action of the circle on $C^*(E)[U]$ as $C^*(E)[U]$ is a gauge-invariant ideal.  Consequently, we form the crossed product, $I_{(H_U, S_U)} \rtimes_{\gamma^E} \mathbb{T}$, which we canonically identify as an ideal of $C^*(E) \rtimes_{\gamma^E} \mathbb{T}$.  By setting
\[
(C^*(E)\rtimes_{\gamma^E} \TT)[U]  := I_{(H_U, S_U)} \rtimes_{\gamma^E} \mathbb{T}
\]
for all open subsets $U$ of $X$, $C^*(E) \rtimes_{\gamma^E} \mathbb{T}$ becomes a $C^*$-algebra over $X$.   By \cite[Corollary~5.14]{rbrm:mgmcek}, we get a dual Pimsner-Voiculescu exact sequence 
\begin{equation}\label{eq-obstruct-class}
\scalebox{.75}{
\xymatrix{ 0  \ar[r] & \mathrm{XK}_1 ( C^*(E) )  \ar[r]  & \mathrm{XK}_0 ( C^*(E) \rtimes_{\gamma^E} \mathbb{T} )  \ar[rr]^-{\mathrm{id} - [\widehat{\gamma^E}]^{-1} } &  &  \mathrm{XK}_0 ( C^*(E) \rtimes_{\gamma^E} \mathbb{T} )  \ar[r] & \mathrm{XK}_0 ( C^*(E) )  \ar[r] & 0.}
}
\end{equation}
The class of the exact sequence \eqref{eq-obstruct-class} in $\mathrm{Ext}^2( \mathrm{XK}_*(C^*(E)), \mathrm{XK}_{*+1}(C^*(E))$ is called the \emph{obstruction class}\index{obstruction class} and is denoted by $\delta(C^*(E))$.  The invariant $(\mathrm{XK}_*(C^*(E)), \delta(C^*(E)))$ is denoted by $\mathrm{XK}\delta(C^*(E))$.\index{XKdelta@$\mathrm{XK}\delta(-)$}  An isomorphism from $\mathrm{XK}\delta(C^*(E))$ to $\mathrm{XK}\delta(C^*(F))$ is an isomorphism $\alpha \colon \mathrm{XK}_*(C^*(E)) \to \mathrm{XK}_*(C^*(F))$ such that 
$$
[\delta(C^*(F))] \alpha = \alpha [\delta(C^*(E))]
$$
in $\mathrm{Ext}^2( \mathrm{XK}_*(C^*(E)), \mathrm{XK}_{*+1}(C^*(F)))$ or equivalently, $\alpha \colon \mathrm{XK}_*(C^*(E)) \to \mathrm{XK}_*(C^*(F))$ is an isomorphism such that there are homomorphisms $\eta$ and $\theta$ from $\mathrm{XK}_0 ( C^*(E) \rtimes_{\gamma^E} \mathbb{T} )$ to $\mathrm{XK}_0 ( C^*(F) \rtimes_{\gamma^E} \mathbb{T} )$ making the diagram 
$$
\scalebox{.85}{
\xymatrix{ 0  \ar[r] & \mathrm{XK}_1 ( C^*(E) )  \ar[r]  \ar[d]_{\alpha_1}& \mathrm{XK}_0 ( C^*(E) \rtimes_{\gamma^E} \mathbb{T} )  \ar[rr]^-{\mathrm{id} - [\widehat{\gamma^E}]^{-1} }  \ar[d]_{\eta} & &  \mathrm{XK}_0 ( C^*(E) \rtimes_{\gamma^E} \mathbb{T} )  \ar[r] \ar[d]_{\theta} & \mathrm{XK}_0 ( C^*(E) )  \ar[r] \ar[d]_{\alpha_0}& 0 \\
0  \ar[r] & \mathrm{XK}_1 ( C^*(F) )  \ar[r]  & \mathrm{XK}_0 ( C^*(F) \rtimes_{\gamma^F} \mathbb{T} )  \ar[rr]^-{\mathrm{id} - [\widehat{\gamma^F}]^{-1} } &  &  \mathrm{XK}_0 ( C^*(F) \rtimes_{\gamma^F} \mathbb{T} )  \ar[r] & \mathrm{XK}_0 ( C^*(F) )  \ar[r] & 0}
}
$$
commutative.  By \cite[Theorem~2.18]{rbrm:mgmcek} and the discussion in \cite[Section~5.3]{rbrm:mgmcek}, if $\alpha \colon \mathrm{XK}\delta(C^*(E)) \to \mathrm{XK}\delta(C^*(F))$ is an isomorphism, then there exists an invertible element of $\mathrm{KK}_X ( C^*(E) , C^*(F))$ that induces the isomorphism $\alpha \colon \mathrm{XK}_*(C^*(E)) \to \mathrm{XK}_*(C^*(F))$.

Since we are interested in $C^*$-algebras that may have subquotients that are finite $C^*$-algebras, we will also need the positive cone in $K$-theory.  For a $C^*$-algebra $A$ over $X$, we write $\mathrm{XK}^+(A)$ for the invariant $\mathrm{XK}(A)$ together with $K_0(A[U_x])_+$ for all $x \in X$.  If $B$ is another $C^*$-algebra over $X$, then an isomorphism from $\mathrm{XK}^+(A)$ to $\mathrm{XK}^+(B)$ will be an isomorphism $\alpha \colon \mathrm{XK}(A) \to \mathrm{XK}(B)$ such that 
\[
\alpha ( K_0(A[U_x])_+ ) = K_0(B[U_x])_+
\]
for all $x \in X$.  We write $\mathrm{XK}\delta^+(C^*(E))$ for the invariant $(\mathrm{XK}^+(C^*(E)), \delta(C^*(E)))$, and isomorphisms from $\mathrm{XK}\delta^+(C^*(E))$ to $\mathrm{XK}\delta^+(C^*(F))$ are defined in the obvious way.\index{XKdeltap@$\mathrm{XK}\delta^+(-)$}

Suppose $C^*(E)$ is unital.  We recall from \cite{searbtk:rfkccka} that we may still use $\mathrm{XK}\delta^+( C^*(E))$ to determine the isomorphism class of $C^*(E)$, even though the unit may fail to be present in any $C^*(E)[U_x]$.  By \cite[Lemma~8.3]{searbtk:rfkccka}, an isomorphism $\alpha$ from $\mathrm{XK}\delta^+(C^*(E))$ to $\mathrm{XK}\delta^+(C^*(F))$ induces a unique isomorphism $\alpha_0$ from $K_0( C^*(E))$ to $K_0(C^*(F))$.  The isomorphism $\alpha_0$ is the induced isomorphism from the following commutative diagram 
\[
\scalebox{.8}{
\xymatrix{ \displaystyle{\bigoplus_{\substack{ x, x' \in X \\ y \in \inf\{x, x'\} } } K_0( C^*(E)[U_y])} \ar[d]^-{\alpha} \ar[rr]^-{\abknotation{( \iota_{U_y}^{U_x} - \iota_{U_y}^{U_{x'}})_*}}  & & \displaystyle{\bigoplus_{ x \in X } K_0( C^*(E)[U_x])}  \ar[rr]^-{ (\iota_{U_x}^X)_*} \ar[d]^{\alpha} & &K_0(C^*(E)) \ar@{-->}[d]^{\alpha_0} \ar[r] & 0 \\
\displaystyle{\bigoplus_{\substack{ x, x' \in X \\ y \in \inf\{x, x'\} } } K_0( C^*(F)[U_y]) } \ar[rr]^-{ \abknotation{( \iota_{U_y}^{U_x} - \iota_{U_y}^{U_{x'}})_*}}  & & \displaystyle{\bigoplus_{ x \in X } K_0( C^*(F)[U_x])}  \ar[rr]^-{ (\iota_{U_x}^X)_*} & &K_0(C^*(F)) \ar[r] & 0 
}
}
\]
with exact rows ($\inf M$ denotes the set of infima of $M$).  Note that here and below, we follow \cite{searbtk:rfkccka} in using matrix notation for maps between direct sums of $K$-groups.
If the unique isomorphism $\alpha_0$ satisfies $\alpha_0 ( [1_{C^*(E)} ]) = [1_{C^*(F)}]$, then we say that $\alpha$ is an isomorphism from $\mathrm{XK}\delta^{+,1}( C^*(E))$ to $\mathrm{XK}\delta^{+,1} (C^*(F))$.\index{XKdeltapu@$\mathrm{XK}\delta^{+,1}(-)$}

We now connect the invariants $\mathrm{XK}\delta( C^*(E))$, $\mathrm{XK}\delta^+( C^*(E))$, $\mathrm{XK}\delta^{+,1}( C^*(E))$ with the (ordered) filtered $K$-theory $\mathrm{FK}_X(C^*(E))$, $\mathrm{FK}_X^+(C^*(E))$, and $\mathrm{FK}_X^{+,1} (C^*(E))$ defined in \cite[Section~3.3]{segrerapws:gcgcfg}.  One key relationship between $\mathrm{XK}(C^*(E))$ and $\mathrm{FK}_X (C^*(E))$ is that the $K$-groups and maps involved in the filtered $K$-theory includes $K_*( C^*(E)[U_x])$ for all $x \in X$ and the maps $(\iota_{U_x}^{U_y})_*$.  Consequently, if $\beta$ is an isomorphism from $\mathrm{FK}_X(C^*(E))$ to $\mathrm{FK}_X(C^*(F))$, then $\beta$ restricts to an isomorphism from $\mathrm{XK}_*(C^*(E))$ to $\mathrm{XK}_*(C^*(F))$.  The next lemma allows us to recover $K_0(C^*(E)[U])_+$ from $K_0(C^*(E)[U_x])_+$.  We note that \cite[Lemma~8.3]{searbtk:rfkccka} already implies that that the map $\abknotation{(\iota_k)_*}$  in Lemma~\ref{lem-order-lifting} is a surjective homomorphism but it is not clear from the proof given there that $\eta$ sends $\bigoplus_{ k = 1}^n K_0(I_k)_+$ onto $K_0(I)_+$.   Hence, we give an alternative proof to include the positive cone.

\begin{lemma}\label{lem-order-lifting}
Let $E$ be a graph ($|E^0|=\infty$ allowed) and let $I_1, I_2, \ldots, I_n$ be gauge-invariant ideals of $C^*(E)$.  Set $I = \overline{ \sum_{k=1}^n I_k}$ and let $\iota_k \colon I_k \to I$ be the inclusion of ideals.  Then the homomorphism $\abknotation{(\iota_k)_*}  \colon \bigoplus_{ k = 1}^n K_0(I_k) \to K_0( I)$ is a surjection which sends $\bigoplus_{ k = 1}^n K_0(I_k)_+$ onto $K_0( I)_+$.
\end{lemma}

\begin{proof}
First note that surjectivity of $\abknotation{(\iota_k)_*}$ will follow once we prove that $\abknotation{(\iota_k)_*}$ sends $\bigoplus_{ k = 1}^n K_0(I_k)_+$ onto $K_0( I)_+$.

We claim that it is enough to prove the statement for a row-finite graph with no sinks.  Let $F$ be the Drinen-Tomforde desingularization. By \cite[Theorem~2.11]{ddmt:cag}, there exists a $*$-isomorphism $\varphi \colon C^*(E)\otimes \Kk \to C^*(F) \otimes \Kk$.  Since the mapping $J \mapsto J \otimes \Kk$ is a lattice isomorphism from the lattice of ideals of $A$ to the lattice of ideals of $A \otimes \Kk$, there are ideals $J_1, J_2, \ldots,J_n$ of $C^*(F)$ such that $\varphi( I_k \otimes \Kk) = J_k \otimes \Kk$.  Since each $I_k$ is a gauge-invariant ideal of $C^*(E)$, by \cite[Lemma~3.3]{segrerapws:gcgcfg}, each $J_k$ is a gauge-invariant ideal of $C^*(F)$.  Let $J = \overline{ \sum_{ k = 1}^n J_k }$.  Since $I = \overline{ \sum_{k=1}^n I_k}$, we have $J \otimes \Kk = \overline{\sum_{ k = 1}^n J_k \otimes \Kk}$.   Note that the diagram
\[
\xymatrix{
\displaystyle{\bigoplus_{ k = 1}^n K_0(I_k) } \ar[rr]^-{ (\iota_k)_*  }  \ar[d]_{ \abknotation{(\id \otimes e_{11})_*} } & & K_0(I)  \ar[d]^{(\id\otimes e_{11})_*}\\
\displaystyle{\bigoplus_{ k = 1}^n K_0(I_k \otimes \Kk ) }\ar[rr]^-{ \abknotation{(\iota_k \otimes \id)_*}  } \ar[d]_{ \varphi_*}  & & K_0(I \otimes \Kk) \ar[d]^{\varphi_*} \\
\displaystyle{ \bigoplus_{ k = 1}^n K_0(J_k \otimes \Kk ) } \ar[rr]^-{ \abknotation{(\iota_k \otimes \id)_*}  }  & & K_0(J \otimes \Kk) \\
\displaystyle{ \bigoplus_{ k = 1}^n K_0(J_k) }\ar[rr]^-{ (\iota_k )_* } \ar[u]^{ \abknotation{(\id \otimes e_{11})_*} }& & K_0(J)   \ar[u]_{(\id\otimes e_{11})_*}\\
}
\]
is commutative, where $\id \otimes e_{11}$ is the embedding of $A$ into the $(1,1)$ corner  of $A \otimes \Kk$.  Since the vertical maps are order isomorphisms, if the statement of the lemma is true for $F$, $J_1, J_2, \ldots, J_n$, then the statement is also true for $E$, $I_1, I_2, \ldots, I_n$, proving the claim.

Assume that $E$ is a row-finite graph with no sinks.  Since each of the $I_k$ is a gauge-invariant ideal, there exists a hereditary and saturated subset $H_k$ of $E^0$ such that $I_{H_k} = I_{ (H_k, \emptyset)} = I_k$.  Since $I = \overline{ \sum_{ k = 1}^n I_k}$ is also a gauge-invariant ideal, there exists a hereditary and saturated subset $H$ of $E^0$ such that $I_H = I$.  Let $\mathcal{H}$ be the smallest hereditary and saturated subset of $E^0$ containing $\bigcup_{ k = 1}^n H_k$.  Note that $\mathcal{H}$ is the join of $H_1 \vee H_2 \vee \ldots \vee H_n$ in the lattice of hereditary and saturated subset of $E^0$.  Thus, $I_\mathcal{H} = \overline{ \sum_{ k = 1}^n I_{H_k} }$ due to the lattice isomorphism given in Section \ref{adpairs}.  Consequently, \[I_{\mathcal{H}}  = \overline{ \sum_{ k = 1}^n I_{H_k} } = \overline{ \sum_{ k = 1}^n I_{k} }= I = I_H\] which implies $\mathcal{H} = H$.  

We now show that every element in $K_0(I)_+$ lifts to an element in $\bigoplus_{k=1}^n K_0(I_k)_+$ via the homomorphism $\abknotation{(\iota_k)_*} $.  We first show that we may reduce this lifting problem to lifting the elements $[p_v]$ for all $v \in H$.  Let $x \in K_0(I)_+ = K_0(I_H)_+$.  Since $E$ is a row-finite graph with no sinks, by \cite[Theorem~4.1]{tbdpirws:crg}, the graph $C^*$-algebra, $C^*(E_H)$, embeds into a full corner of $I_H$, where $E_H = ( H, s^{-1}(H), s, r)$ via $p_{v} \mapsto p_v, s_e \mapsto s_e$.  By \cite[Theorem~7.1]{pamamep:nkga}, $x = \sum_{ i =1}^n m_i [ p_{v_i} ]$, where $m_i \geq 0$ and $v_i \in H$.  Thus, to show that $x$ lifts to an element in $\bigoplus_{k=1}^n K_0(I_k)_+$, it is enough to show that for all $v \in H$, $[p_v]$ lifts to an element in $\bigoplus_{k=1}^n K_0(I_k)_+$.  

Let $v \in H$.  Since $\bigcup_{k=1}^n  H_k$ is a hereditary subset of $E^0$, by \cite[Remark~3.1]{bhrs:iccig}, $H = \mathcal{H}$ is constructed as the union of the sequence $\sum_m \left(\bigcup_{k = 1}^n H_k\right)$ of subsets of $E^0$ defined inductively by $\sum _0 \left( \bigcup_{k = 1}^n H_k\right) :=  \bigcup_{k = 1}^m H_k$ and $\sum _m\left( \bigcup_{k = 1}^n H_k\right)$ is the union of $\sum_{m-1} \left( \bigcup_{k = 1}^n H_k  \right)$ with the set of all regular vertices $w$ for which all edges emitting from $w$ have range in $\sum_{m-1} \left( \bigcup_{k = 1}^n H_k  \right)$.  We will inductively show that for all $m \in \ZZ_{\geq 0}$, if $v \in \sum _m\left( \bigcup_{k = 1}^n H_k\right)$, then there exists $y \in \bigoplus_{k=1}^n K_0(I_k)_+$ such that $\abknotation{(\iota_k)_*}$ sends $y$ to $[p_v]$.  Suppose $v \in \bigcup_{ k = 1}^n H_k$.  Then $v \in H_k$ for some $k$ which implies $y=( 0, \ldots, 0, [p_v], 0, \ldots , 0)$ is an element of $\bigoplus_{k=1}^n K_0(I_k)_+$ such that $\abknotation{(\iota_k)_*}$ sends $y$ to $[p_v]$.  Assume that for all $v \in H$, if $v \in \sum _m\left( \bigcup_{k = 1}^n H_k\right)$, then there exists $y \in \bigoplus_{k=1}^n K_0(I_k)_+$ such that $\abknotation{(\iota_k)_*}$ sends $y$ to $[p_v]$.   Assume $v \in \sum _{m+1}\left( \bigcup_{k = 1}^n H_k\right)$.  If $v \in \sum _{m}\left( \bigcup_{k = 1}^n H_k\right)$, then by the inductive hypothesis there exists $y \in \bigoplus_{k=1}^n K_0(I_k)_+$ such that $\abknotation{(\iota_k)_*}$ sends $y$ to $[p_v]$.  Assume $v \notin  \sum _{m}\left( \bigcup_{k = 1}^n H_k\right)$.  Then $v$ is a regular vertex such that $r(s^{-1}(v)) \subseteq  \sum _{m}\left( \bigcup_{k = 1}^n H_k \right)$.  Note that 
\[
[p_v] = \left[ \sum_{ s(e)=v } s_e s_e^* \right] = \sum_{ s(e)=v } [ s_e s_e^* ] =  \sum_{ s(e)=v } [ p_{r(e)} ].
\]
Since $r(e) \in \sum _{m}\left( \bigcup_{k = 1}^n H_k \right)$ for all $e$ with $s(e)=v$, by the inductive hypothesis, for all $e$ with $s(e)=v$, there exists $y_e \in \bigoplus_{ k = 1}^n K_0(I_k)_+$ such that $\abknotation{(\iota_k)_*}$ sends $y_e$ to $[p_{r(e)}]$.  Consequently, $\sum_{s(e) = v } y_e$ is an element of  $\bigoplus_{ k = 1}^n K_0(I_k)_+$ such that $\abknotation{(\iota_k)_*}$ sends $\sum_{s(e) =v } y_e$ to $[p_v]$.   Thus, concluding the proof of the claim for all $m \in \ZZ_{\geq 0}$, if $v \in \sum _m\left( \bigcup_{k = 1}^n H_k\right)$, there exists $y$ in $\bigoplus_{ k = 1}^n K_0(I_k)_+$  such that $\abknotation{(\iota_k)_*}$ sends $y$ to $[p_{v}]$ .  Since $H$ is equal to the union $\bigcup_{m =0}^\infty \sum_m \left( \bigcup_{k = 1}^n H_k\right)$, we have for all $v \in H$, there exists $y$ in $\bigoplus_{ k = 1}^n K_0(I_k)_+$ such that $\abknotation{(\iota_k)_*}$ sends $y$ to $[p_v]$.  This concludes the proof that $\abknotation{(\iota_k)_*}$ sends $\bigoplus_{ k = 1}^n K_0(I_k)_+$ onto $K_0(I)_+$.
\end{proof}

\begin{theor}\label{thm:XKdelta-filteredkthy}
Let $E$ and $F$ be graphs ($|E^0|+|F^0|=\infty$ allowed) such that $X =  \mathrm{Prim}_\gamma ( C^*(E)) \cong \mathrm{Prim}_\gamma ( C^*(F))$ and $X$ is finite.
\begin{enumerate}[(1)]
\item If $\alpha \colon \mathrm{XK}\delta ( C^*(E) ) \to \mathrm{XK}\delta (C^*(F))$ is an isomorphism, then there exists an isomorphism $\overline{\alpha} \colon \mathrm{FK}_X ( C^*(E)) \to \mathrm{FK}_X ( C^*(F))$ inducing $\alpha$.

\item If  $\alpha \colon \mathrm{XK}\delta^+ ( C^*(E) ) \to \mathrm{XK}\delta^+ (C^*(F))$ is an isomorphism, then there exists an isomorphism $\overline{\alpha} \colon \mathrm{FK}^+_X(C^*(E)) \to \mathrm{FK}^+_X(C^*(F))$ inducing $\alpha$.
\end{enumerate}
When  $|E^0|+|F^0|<\infty$, the finiteness condition on $X$ is automatic, and we further have 
\begin{enumerate}[(1)]\addtocounter{enumi}{2}
\item   If  $\alpha \colon \mathrm{XK}\delta^{+,1} ( C^*(E) ) \to \mathrm{XK}\delta^{+,1} (C^*(F))$ is an isomorphism, then there exists an isomorphism $\overline{\alpha} \colon \mathrm{FK}^{+,1}_X (C^*(E)) \to \mathrm{FK}^{+,1}_X (C^*(F))$ inducing $\alpha$.
\end{enumerate}
\end{theor}

\begin{proof}
For (1), assume $\alpha \colon \mathrm{XK}\delta ( C^*(E) ) \to \mathrm{XK}\delta (C^*(F))$ is an isomorphism.  By \cite[Theorem~2.18]{rbrm:mgmcek}, there exists an invertible element $\widetilde{\alpha} \in \mathrm{KK}_X(C^*(E), C^*(F))$ that induces $\alpha$.  As $\widetilde{\alpha}$ is an invertible element of $\mathrm{KK}_X(C^*(E), C^*(F))$, $\widetilde{\alpha}$ also induces an isomorphism $\overline{\alpha} \colon \mathrm{FK}_X(C^*(E)) \to \mathrm{FK}_X(C^*(F))$.  Since $\alpha$ and $\overline{\alpha}$ are induced by $\widetilde{\alpha}$, we have $\alpha = \overline{\alpha} \vert_{\mathrm{XK}(C^*(E))}$.   

We now prove (2).  Assume $\alpha \colon \mathrm{XK}\delta^+ ( C^*(E) ) \to \mathrm{XK}\delta^+ ( C^*(F) )$.  By Part~(1), there exists an isomorphism $\overline{\alpha} \colon \mathrm{FK}_X(C^*(E)) \to \mathrm{FK}_X(C^*(F))$ that induces $\alpha$.  To prove that $\overline{\alpha}$ is an isomorphism from $\mathrm{FK}^+_X(C^*(E))$ to $\mathrm{FK}^+_X(C^*(F))$ we must show that for all open subsets $U$ and $V$ of $X$ with $U \subseteq V$, $\overline{\alpha}$ is an order isomorphism from $K_0( C^*(E)[V\setminus U])$ to $K_0( C^*(F)[V\setminus U])$.  Since $\overline{\alpha} \vert_{\mathrm{XK}(C^*(E))} = \alpha$, we have that $\overline{\alpha}$ is an order isomorphism from $K_0( C^*(E)[U_x])$ to $K_0( C^*(F)[U_x])$ for all $x \in X$.  Using this fact together with Lemma~\ref{lem-order-lifting}, we get that $\overline{\alpha}$ is an order isomorphism from $K_0(C^*(E)[V])$ to $K_0(C^*(F)[V])$ for all open subsets $V$ of $X$.  Let $U$ and $V$ be open subsets of $X$ such that $\emptyset \subsetneq U \subseteq V$.  Using the description of a gauge-invariant ideal given in \cite[Theorem~5.1]{ermt:iga} and the description of the quotient of a graph $C^*$-algebra by a gauge-invariant ideal given in \cite[Corollary~3.5]{bhrs:iccig} as graph $C^*$-algebras, and using the fact that the positive cone of a graph $C^*$-algebra is generated by vertex projections and gap projections (see \cite[Corollary~3.5]{dhmlmmer:nklpa}), we get that the quotient map 
\[
\iota_{V}^{V\setminus U} \colon C^*(E)[V] \to C^*(E)[V \setminus U]
\]
induces a surjective homomorphism $(\iota_{V}^{V\setminus U})_* \colon K_0( C^*(E)[V]) \to K_0( C^*(E)[V \setminus U])$ such that $( \iota_{V}^{V\setminus U})_*$ sends $K_0(  C^*(E)[V]  )_+$ onto $K_0( C^*(E)[V\setminus U])_+$.  Similarly, for $C^*(F)$.  And since $\overline{\alpha}$ is an order isomorphism from $K_0(C^*(E)[V])$ to $K_0(C^*(F)[V])$ and since the diagram
\[
\xymatrix{ 
K_0( C^*(E)[V]) \ar[rr]^-{( \iota_{V}^{V\setminus U})_* } \ar[d]_{\overline{\alpha}} &  & K_0( C^*(E)[V\setminus V]) \ar[d]^{\overline{\alpha}} \\
K_0( C^*(F)[V]) \ar[rr]^-{( \iota_{V}^{V\setminus U} )_*} &  & K_0( C^*(F)[V \setminus U])
}
\]
is commutative, we get that $\overline{\alpha}\vert_{ K_0( C^*(E)[V \setminus U])}$ is an order isomorphism.  This concludes the proof that $\overline{\alpha}$ is an isomorphism from $\mathrm{FK}_X^+(C^*(E))$ to $\mathrm{FK}_X^+(C^*(F))$.

Lastly, we prove (3).  Suppose $\alpha \colon \mathrm{XK}\delta^{+,1} ( C^*(E) ) \to \mathrm{XK}\delta^{+,1} ( C^*(F) )$ is an isomorphism.  By Claim~(2), there exists an isomorphism $\overline{\alpha} \colon \mathrm{FK}_X^+(C^*(E)) \to \mathrm{FK}_X^+(C^*(F))$ inducing $\alpha$.  So, all we need to prove is $\overline{\alpha}( [1_{C^*(E)}]) = [1_{C^*(F)}]$.  Since $\overline{\alpha} \vert_{ \mathrm{XK}(C^*(E))} = \alpha$, the unique isomorphism $\alpha_0 \colon K_0(C^*(E)) \to K_0(C^*(F))$ induced by $\alpha$ is equal to $\overline{\alpha}\vert_{ K_0(C^*(E))}$.  Since $\alpha_0( [1_{C^*(E)}]) = [1_{C^*(F)}]$, we have $\overline{\alpha}( [1_{C^*(E)}]) = [1_{C^*(F)}]$.  Hence, $\overline{\alpha}$ is an isomorphism from $\mathrm{FK}_X^{+,1}(C^*(E))$ to $\mathrm{FK}_X^{+,1}(C^*(F))$.
\end{proof}

The rest of the section is devoted to connecting $\DT(C^*(E))$ with $\mathrm{XK}\delta^+(C^*(E))$.  We start by relating order ideals of the skew product to gauge invariant  ideals of the original.

\fxnote{We should align with acyclic section later}

\begin{lemma}\label{lem-adm-pair-skew-prod}
Let $E$ be a graph ($|E^0|=\infty$ allowed).  If $(H, S)$ is an admissible pair for $E$, then $( H \times \ZZ , S \times \ZZ )$ is an admissible pair for $E \times_1 \ZZ$. 
\end{lemma}

\begin{proof}
Let $(H, S)$ be an admissible pair for $E$.  Suppose $(e, n) \in E^{1} \times \ZZ$ with $( s_E(e) , n - 1 ) =  s_{E \times_1 \ZZ} ( e, n) \in H \times \ZZ$.  Then $s_E(e) \in H$ which implies that $r_E (e) \in H$.  Therefore, $r_{E \times_1 \ZZ }( e, n) = ( r_E (e), n ) \in H \times \ZZ$.  Suppose $(v, n ) \in E^0 \times \ZZ$ is a regular vertex such that $r_{E\times_1\ZZ } ( s_{ E \times_1\ZZ}^{-1} ( v, n ) ) \subseteq H \times \ZZ$.  Since $s_{ E \times_1\ZZ}^{-1} ( v, n ) = s_E^{-1} (v) \times \{ n+1\}$, $v$ is a regular vertex in $E^0$.  Moreover, $r_E( s_E^{-1} ( v) ) \times \{ n +1 \} = r_{E\times_1\ZZ } ( s_{ E \times_1\ZZ }^{-1} ( v, n ) )$.  And since $ r_{E\times_1\ZZ } ( s_{ E \times_1\ZZ }^{-1} ( v, n ) ) \subseteq H \times \ZZ$, $r_E( s_E^{-1} (v)) \subseteq H$.  Thus, $v \in H$ as $H$ is saturated.  Consequently $(v, n) \in H \times \ZZ$.  Thus, $H \times \ZZ$ is a hereditary and saturated subset of $E^0 \times \ZZ = (E \times_1 \ZZ)^0$.  

Let $( v, n  ) \in S \times \ZZ$.  Since $s_{ E \times_1 \ZZ}^{-1} ( v, n) = s_E^{-1} ( v ) \times \{ n +1 \}$, $| s_{E\times_1\ZZ}^{-1} ( v, n) | = | s_E^{-1} (v) | = \infty$.  Since 
\begin{align*}
&s_{E\times_1\ZZ}^{-1} ( v, n ) \cap r_{E \times_1 \ZZ}^{-1} ( (E^0 \times \ZZ) \setminus ( H \times \ZZ ) ) \\
	&\quad = ( s_E^{-1} ( v) \times \{ n+1\} ) \cap ( r_{E}^{-1} ( E^0 \setminus H) \times \ZZ ) \\
	&\quad = (s_{E}^{-1} ( v ) \cap r_{E}^{-1} ( E^0 \setminus H ) ) \times \{ n +1\},
\end{align*}
$|s_{E\times_1\ZZ}^{-1} ( v, n ) \cap r_{E \times_1 \ZZ}^{-1} ( (E^0 \times \ZZ) \setminus ( H \times \ZZ ) )  | = | (s_{E}^{-1} ( v ) \cap r_{E}^{-1} ( E^0 \setminus H ) ) |$ which implies that $s_{E\times_1\ZZ}^{-1} ( v, n ) \cap r_{E \times_1 \ZZ}^{-1} ( (E^0 \times \ZZ) \setminus ( H \times \ZZ ) )$ is a nonempty finite subset of $E^0 \times \ZZ$.  Hence, $(v, n) \in B_{ H \times \ZZ }$.  We have just shown that $S \times \ZZ \subseteq B_{ H \times \ZZ }$.  Consequently, $( H \times \ZZ , S \times \ZZ )$ is an admissible pair for $E\times_1\ZZ$.
\end{proof}

Let $A$ be an AF-algebra and let $I$ be an ideal of $A$.  Since an AF-algebra has cancellation of projections, the natural map from $K_0(I)$ to $K_0(A)$ is an injection.  Moreover, with this identification, $K_0(I)$ is an order ideal of $K_0(A)$.  In fact, by Elliott (\cite[Proposition IV.5.1]{krd:ce}) the map $I \mapsto K_0(I)$ is a lattice isomorphism from the lattice of ideals of $A$ to the lattice of order ideals of $K_0(A)$.  We will use this fact throughout the proof of the next result.

\begin{theor}\label{thm:admiss-ordideal}
Let $E$ be a graph  ($|E^0|=\infty$ allowed). Then
\begin{align*}
\lambda \colon \left\{ \text{Admissible pairs in $E$ } \right\} &\to  \left\{  \begin{minipage}{.45\textwidth} \begin{center} Order ideals in $K_0 ( C^*( E \times_1 \ZZ ) )$  that are $\lt_*$-invariant.\end{center}\end{minipage}\right\} \\
\\
	(H, S) &\mapsto K_0 ( I_{ (H \times \ZZ , S \times \ZZ ) } )
\end{align*}
is a lattice isomorphism.  Hence, there is a lattice isomorphism, which we again denote by $\lambda$, from the lattice of gauge-invariant ideals of $C^*(E)$ to  
the order ideals in $K_0 ( C^*( E \times_1 \ZZ ) )$  that are $\lt_*$-invariant. 
\end{theor}

\begin{proof}
Since $I_{ ( H \times \ZZ , S \times \ZZ )  }$ is the ideal generated by $\{ p_{ ( v, n) } :  v \in H , n \in \ZZ \} \cup \{ p_{ ( v, n ) }^{ H \times \ZZ } : v \in S , n \in \ZZ \}$, where 
\[
p_{ (v, n) }^{ H \times \ZZ } = p_{(v, n)} - \sum_{ \substack{  s_{E\times_1 \ZZ}( e, \ell ) =( v, n) \\  r_{E \times_1 \ZZ}( e, \ell ) \notin H \times \ZZ } } s_{ ( e, \ell ) } s_{ ( e, \ell ) }^* =  p_{(v, n)}  - \sum_{ \substack{ s_E^{-1}(e)=    v  \\ r_E(e) \notin H } } s_{ (e, n+1)} s_{(e, n+1)}^*
\]
and since $\lt ( p_{ ( v, n) } ) = p_{ (v, n+1)}$ and $\lt ( p_{ (v, n)}^{ H \times \ZZ } ) = p_{ (v, n+1) }^{ H \times \ZZ }$, $I_{ ( H \times \ZZ , S \times \ZZ ) }$ is $\lt$-invariant.  Consequently, $K_0 ( I_{ ( H \times \ZZ , S \times \ZZ ) } )$ is $\lt_*$-invariant.

Suppose $K_0 ( I_{ (H \times \ZZ , S \times \ZZ )} ) = K_0 ( I_{ (H' \times \ZZ , S' \times \ZZ ) } )$.  Since $E \times_1 \ZZ$ is a graph with no cycles, \cite[Corollary~2.13]{ddmt:cag}, $C^*(E \times_1 \ZZ )$ is an AF-algebra.  Thus, $I_{( H \times \ZZ , S \times \ZZ ) } = I_{ ( H' \times \ZZ , S' \times \ZZ ) }$ since $K_0 ( I_{ (H \times \ZZ , S \times \ZZ )} ) = K_0 ( I_{ (H' \times \ZZ , S' \times \ZZ ) } )$.  Consequently, $H \times \ZZ = H' \times \ZZ$ and $S \times \ZZ = S' \times \ZZ$ which implies that $H = H'$ and $S = S'$.  Thus, $\lambda$ is injective.

Next, we show that $\lambda$ is surjective.  Let $T$ be an order ideal in $K_0 ( C^*( E \times_1 \ZZ ))$ that is $\lt_*$-invariant.  Since $C^*(E\times_1\ZZ)$ is an AF-algebra, there exists an ideal of $C^*(E \times_1\ZZ)$ such that $K_0(I)=T$.  Since $E \times_1\ZZ$ is a graph with no cycles, by \cite[Corollary~3.8]{bhrs:iccig}, every ideal of $C^*(E\times_1\ZZ)$ is gauge-invariant.  Hence, cf.~Section \ref{adpairs}, $I= I_{ ( H' , S' ) }$ for some admissible pair $(H', S')$ for $E \times_1 \ZZ$.  Let $H = \{ v \in E^0 :  \text{ $(v, n) \in H'$ for some $n \in \ZZ$} \}$ and $S = \{ v \in E^0 : \text{$(v, n) \in S'$ for some $n \in \ZZ$} \}$.  We claim that $H \times \ZZ = H'$, $S \times \ZZ = S'$, and $( H, S)$ is an admissible pair for $E$.

We first show that $H \times \ZZ = H'$.  It is clear from the definition of $H$ that $H' \subseteq H \times \ZZ$.  Let $(v, m) \in H \times \ZZ$.  Then there exists $n \in \ZZ$ such that $(v, n) \in H'$.  Choose $k \in \ZZ$ such that $m = n + k$.  Then 
\[
[ p_{ (v, m) } ] = \lt_*^k ( [ p_{(v, n) } ] ) \in T = K_0 ( I_{ (H',S')} ) 
\]
since $T$ is $\lt_*$-invariant and $p_{(v, n) } \in I_{ (H', S') }$.  Since $C^*( E \times_1 \ZZ )$ is an AF-algebra, we have that $p_{(v, m) } \in I_{ (H', S') }$.  Let $H_{(v, m)}$ be the smallest hereditary and saturated subset of $(E\times_1 \ZZ)^0$ that contains $(v, m)$.  Then $I_{ ( H_{ (v,m)} , \emptyset ) } \subseteq I_{ (H', S') }$ which implies that $H_{ (v, m)} \subseteq H'$.  Consequently, $(v,m) \in H'$.  We have just shown that $H \times \ZZ = H'$.

We now show that $S \times \ZZ = S'$.  It is clear from the definition of $S$ that $S ' \subseteq S \times \ZZ$.  Let $(v, m ) \in S \times \ZZ$.  Then there exists $n \in \ZZ$ such that $(v, n ) \in S'$.  Let $k \in \ZZ$ such that $m = n+k$.  Then 
\[
[ p_{(v, m)}^{H'} ] \ = \lt_*^k ( [p_{(v, n)}^{H'} ] ) \in T = K_0 ( I_{(H', S') })
\] 
since $T$ is $\lt_*$-invariant and $p_{(v,n)}^{H'} \in I_{(H',S')}$.  Since $C^*(E \times_1 \ZZ )$ is an AF-algebra, $p_{(v,m)}^{H'} \in I_{ (H',S')}$.  Since $|s_{ E \times_1 \ZZ}^{-1} ( v, m)| = |s_{E}^{-1} (v) \times \{ m + 1 \}|= |s_{E}^{-1} (v) \times \{ n + 1 \}| = |s_{ E \times_1 \ZZ}^{-1} ( v, n)|$ and 
\begin{align*}
\lefteqn{| s_{E\times_1\ZZ}^{-1} ( (v,m) ) \cap r_{E\times_1\ZZ}^{-1} ( (E^0 \times\ZZ) \setminus H') |} \\
			&\quad =  | ( s_{E}^{-1} (v) \times \{ m + 1 \} ) \cap   r_{E\times_1\ZZ}^{-1} ( (E^0 \times\ZZ) \setminus (H \times \ZZ ) ) | \\
				&\quad =  | ( s_{E}^{-1} (v) \times \{ m + 1 \} ) \cap  ( r_{E}^{-1} ( E^0 \setminus H ) \times \ZZ )  | \\
				&\quad = | ( s_E^{-1}(v) \cap  r_{E}^{-1} ( E^0 \setminus H ) ) \times \{ m+1 \} | \\
				&\quad = | ( s_E^{-1}(v) \cap  r_{E}^{-1} ( E^0 \setminus H ) ) \times \{ n+1 \} | \\
				&\quad =  | ( s_{E}^{-1} (v) \times \{ n + 1 \} ) \cap  ( r_{E}^{-1} ( E^0 \setminus H ) \times \ZZ ) | \\
				&\quad =  | ( s_{E}^{-1} (v) \times \{ n + 1 \} ) \cap   r_{E\times_1\ZZ}^{-1} ( (E^0 \times\ZZ) \setminus (H \times \ZZ ) ) | \\
				&\quad = | s_{E\times_1\ZZ}^{-1} ( (v,n) ) \cap r_{E\times_1\ZZ}^{-1} ( (E^0 \times\ZZ) \setminus H') |, 
\end{align*}
and since $(v,n) \in S'$, $| s_{ E \times_1 \ZZ}^{-1} ( v, m)  | = \infty$ and $0 < | s_{E\times_1\ZZ}^{-1} ( (v,m) ) \cap r_{E\times_1\ZZ}^{-1} ( (E^0 \times\ZZ) \setminus H') | <\infty$.  So, $(v, m) \in B_{H'}$ and $(v,m) \notin H'$.  Therefore, $( H ' , \{ (v,m) \} )$ is an admissible pair for $E \times_1 \ZZ$ and $I_{( H ' , \{ (v,m) \} )} \subseteq  I_{ (H',S')}$ since $p_{(v,m)}^{H'} \in I_{ (H',S')}$.  Hence, $( H ' , \{ (v,m) \} ) \leq (H',S')$ which implies that $(v,m) \in H' \cup S'$.  We can conclude that $(v,m) \in S'$ since $(v,m) \notin H'$.  This proves the claim that $S \times \ZZ = S'$.

Next we show that $(H, S)$ is an admissible pair for $E$.  Let $e \in E^1$ such that $s_E (e ) \in H$.  Then $s_{E \times_1 \ZZ } ( ( e, 1) )  = ( s_E ( e) , 0 ) \in H \times \ZZ = H'$.  Since $H'$ is a hereditary subset of $(E \times_1 \ZZ)^0$, $( r_E (e), 1) = r_{ E\times_1 \ZZ } ( e, 1) \in H' = H \times \ZZ$ which implies that $r_E (e) \in H$.  Consequently, $H$ is a hereditary subset of $E^0$.  Let $v \in E^0$ be a regular vertex such that $r_E ( s^{-1} ( v ) ) \subseteq H$.  Then $(v,0)$ is a regular vertex in $E \times \ZZ$ since $s_{E}^{-1} ( v) \times \{ 1\} = s_{ E \times_1 \ZZ }^{-1} ( (v, 0) )$.  Note that 
\[
r_{ E \times_1 \ZZ} ( s_{ E \times_1\ZZ}^{-1} ( v, 0) ) = r_{E \times_1\ZZ} ( s_E^{-1} ( v) \times \{1\} ) = r_{E}( s_E^{-1} ( v) ) \times \{ 1 \} \subseteq H \times \ZZ = H'.
\]
Saturation of $H'$ implies that $(v,0 ) \in H' = H \times \ZZ$ and hence, $v \in H$.  Consequently, $H$ is a hereditary and saturated subset of $E^0$.  Let $v \in S$.  Then $(v, 0 ) \in S'$.  By the computation in the previous paragraph, we have that $s_{ E \times_1 \ZZ}^{-1} ( v, 0) = s_{E}^{-1} (v) \times \{ 1 \}$ and $| s_{E\times_1\ZZ}^{-1} ( (v,0) ) \cap r_{E\times_1\ZZ}^{-1} ( (E^0 \times\ZZ) \setminus H') |  = | ( s_E^{-1}(v) \cap  r_{E}^{-1} ( E^0 \setminus H ) ) \times \{ 1 \} |$.  Therefore, $| s_E^{-1}(v) | = \infty$ and $0 < |s_E^{-1}(v) \cap  r_{E}^{-1} ( E^0 \setminus H )| <\infty$ which implies that $v \in B_H$.  Consequently, $S \subseteq B_H$ and therefore, $(H, S)$ is an admissible pair for $E$.  We have now proved the claim that $\lambda$ is surjective.  

To show $\lambda$ is a lattice isomorphism, we prove that $\lambda$ and its inverse are order preserving.   Let $(H, S )$ and $(H' ,S')$ be admissible pairs with $(H, S ) \leq ( H' , S')$.  Then $H \subseteq H'$ and $S \subseteq H' \cup S'$.  Therefore, $H \times \ZZ \subseteq H' \times \ZZ$ and $S \times \ZZ \subseteq ( ( H ' \cup S' )  \times \ZZ ) = ( H' \times \ZZ ) \cup ( S' \times \ZZ )$ which implies that $I_{ ( H \times \ZZ , S \times \ZZ ) } \subseteq I_{ ( H' \times \ZZ , S' \times \ZZ ) }$.  Consequently, $K_ 0 ( I_{ ( H \times \ZZ , S \times \ZZ ) } )  \subseteq K_0 (  I_{ ( H' \times \ZZ , S' \times \ZZ ) } )$.  Suppose $T$ and $T'$ are order ideals in $K_0 ( C^*( E \times_1 \ZZ ) )$ that are $\lt_*$-invariant and $T \subseteq T'$.  Let $(H, S)$ and $(H',S')$ be admissible pairs in $E$ such that $T = K_ 0 ( I_{ ( H \times \ZZ , S \times \ZZ ) } )$ and $T' = K_0 (  I_{ ( H' \times \ZZ , S' \times \ZZ ) } )$.  Since $T \subseteq T'$, $I_{ ( H \times \ZZ , S \times \ZZ ) } \subseteq I_{ ( H' \times \ZZ , S' \times \ZZ ) }$ (since $C^*(E\times_1\ZZ)$ is an AF-algebra).  Thus, $ ( H \times \ZZ , S \times \ZZ ) \leq   ( H' \times \ZZ , S' \times \ZZ )$.  Consequently, $H \times \ZZ \subseteq H' \times \ZZ$ and $S \times \ZZ \subseteq ( H ' \times \ZZ ) \cup ( S' \times \ZZ )$ which implies that $H \subseteq H'$ and $S \subseteq H' \cup S'$ since $( H '\cup S' ) \times \ZZ = ( H ' \times \ZZ ) \cup ( S' \times \ZZ )$.  We now can conclude that $\lambda$ is a lattice isomorphism.

The last part of the theorem is clear since $(H, S ) \mapsto I_{(H,S)}$ is a lattice isomorphism between the lattice of admissible pairs and gauge-invariant ideals in $C^*(E)$ as noted in Section \ref{adpairs}.
\end{proof}

Let $C^*(E)$ be a graph $C^*$-algebra with finitely many gauge-invariant ideals and let $X = \mathrm{Prim}_\gamma (C^*(E))$.  Using Lemma~\ref{lem-adm-pair-skew-prod}, $C^*(E \times_1\ZZ)$ becomes a $C^*$-algebra over $X$ as follows.  Let $U$ be an open subset of $X$.  Since $C^*(E)[U]$ is a gauge-invariant ideal of $C^*(E)$, cf.~Section \ref{adpairs}, there exists a unique admissible pair, $(H_U,S_U)$, such that $C^*(E)[U] = I_{ (H_, H_U)}$.  By Lemma~\ref{lem-adm-pair-skew-prod}, $(H \times \ZZ, S \times \ZZ)$ is an admissible pair in $E \times_1\ZZ$ which implies $I_{ (H \times \ZZ, S \times \ZZ ) }$ is an ideal of $C^*(E \times_1\ZZ)$.  Now, $C^*(E \times_1\ZZ)$ becomes a $C^*$-algebra over $X$ by setting
\[
C^*(E \times_1\ZZ)[U] := I_{ (H_U \times \ZZ, S_U \times \ZZ ) }
\]
for all open subsets $U$ of $X$.  Using this $X$-structure for $C^*(E\times_1\ZZ)$, we now show how to use the skew product graph $E \times_1 \ZZ$ to represent the obstruction class $\delta(C^*(E))$.

\begin{theor}\label{thm:dual-skew}
Let $E$ be a graph ($|E^0|=\infty$ allowed), and let $X=\mathrm{Prim}_\gamma(C^*(E))$.  The isomorphism $\phi \colon C^*( E \times_1 \ZZ ) \to C^*(E) \rtimes_{\gamma^E} \TT$ given by $\phi ( p_{ (v,n) } ) = f_n \otimes p_v$ and $\phi ( s_{( e, n) } ) = f_n \otimes s_e$ is an $X$-equivariant isomorphism with the property that $\hat{\gamma} \circ \phi = \phi \circ \lt$, where $\hat{\gamma} = (\widehat{\gamma^E})_1$ is the canonical generator for the dual $\ZZ$-action on $C^*(E) \rtimes_{\gamma^E} \mathbb{T}$.    

Consequently, $\phi$ induces a commutative diagram
\[
\scalebox{.75}{
\xymatrix{
0 \ar[r] & \mathrm{XK}_1 ( C^*(E) )  \ar[r] \ar@{=}[d] & \mathrm{XK}_0 ( C^*( E \times_1 \ZZ ) ) \ar[d]_{ [ \phi ] } \ar[rr]^-{\mathrm{id} - \lt^{-1}_* } & &  \mathrm{XK}_0 ( C^*( E \times_1 \ZZ ) )  \ar[d]_{ [ \phi ] } \ar[r] & \mathrm{XK}_0 ( C^*(E) ) \ar@{=}[d] \ar[r] & 0 \\
0  \ar[r] & \mathrm{XK}_1 ( C^*(E) )  \ar[r]  & \mathrm{XK}_0 ( C^*(E) \rtimes_{\gamma^E} \mathbb{T} )  \ar[rr]^-{\mathrm{id} - [\widehat{\gamma^E}]^{-1} } &  &  \mathrm{XK}_0 ( C^*(E) \rtimes_{\gamma^E} \mathbb{T} )  \ar[r] & \mathrm{XK}_0 ( C^*(E) )  \ar[r] & 0}
}
\]
where the bottom exact sequence is the dual Pimsner-Voiculescu exact sequence. 
\end{theor}

\begin{proof}
Let $U$ be an open subset of $X$.  Let $(H_U, S_U)$ be the unique admissible pair such that $C^*(E)[U] = I_{(H_U, S_U)}$.  Then $C^*(E\times_1\ZZ) [ U ] = I_{ (H_U \times \ZZ, S_U \times \ZZ)}$ and $(C^*(E) \rtimes_{\gamma^E} \mathbb{T})[U] =  I_{ (H_U, S_U) } \rtimes_{\gamma^E} \mathbb{T}$.  To show that the isomorphism $\phi$ is an $X$-equivariant isomorphism we must prove that $\phi (  I_{ (H_U \times \ZZ, S_U \times \ZZ)  } ) =  I_{ (H_U, S_U) } \rtimes_{\gamma^E} \mathbb{T}$.  First note that $I_{(H_U \times \ZZ,  S_U \times \ZZ)}$ is generated by 
$$
\{ p_{(v,m)} : (v,m) \in H_U \times \ZZ \} \cup \left\{ p_{(w, n)} - \sum_{ \substack{  s_{E\times_1 \ZZ}( e, \ell ) =( w, n) \\  r_{E \times_1 \ZZ}( e, \ell ) \notin H \times \ZZ } } s_{ ( e, \ell ) } s_{ ( e, \ell ) }^* : (w, n) \in S_U \times \ZZ \right\}
$$   
and $I_{ (H_U, S_U) } \rtimes_{\gamma^E} \mathbb{T}$ is generated by $\{ f_n \otimes a : a \in I_{ (H_U , S_U)} , n \in \ZZ \}$.  Let $(v, m) \in H_U \times \ZZ$ and $(w, n) \in S_U \times \ZZ$.  Then 
\[
\phi ( p_{(v,m)} ) = f_m \otimes p_v \in I_{ (H_U, S_U)} \rtimes_{\gamma^E} \mathbb{T}
\]
since $v \in H_U$.   Since for all $e \in E^1$ and for all $\zeta \in \TT$,
\begin{align*}
(f_{n+1} \otimes s_{e})(f_{n+1} \otimes s_{e})^* (\zeta) &= \int_\TT ( f_{n+1} \otimes s_e)(z) \gamma_z^E ( ( f_{n+1} \otimes s_e)^*( z^{-1} \zeta ) ) dz \\
			&= \int_\TT z^{n+1} s_e \gamma_z^E ( \gamma_{z^{-1}\zeta}^E ( (f_{n+1} \otimes s_e )( z\zeta^{-1})^* ) dz \\
				&= \int_\TT z^{n+1} s_e \gamma_\zeta^E (  z^{-(n+1)} \zeta^{n+1} s_e^* ) dz \\
				&= \int_\TT  \zeta^{n} s_e s_e^* dz \\
				&= ( f_{n} \otimes s_{e} s_e^* )(\zeta),
\end{align*}
we have
\begin{align*}
\phi( p_{(w,n)}^{H_U \times \ZZ} ) &= \phi( p_{(w, n)}) - \sum_{ \substack{  s_{E\times_1 \ZZ}( e, \ell ) =( w, n) \\  r_{E \times_1 \ZZ}( e, \ell ) \notin H \times \ZZ } } \phi(s_{ ( e, \ell ) } ) \phi(s_{ ( e, \ell ) }^*) \\
&=  \phi( p_{(w, n)}) - \sum_{ \substack{ s_E^{-1}(e)=    w  \\ r_E(e) \notin H } }\phi(s_{ ( e, n+1 ) } ) \phi(s_{ ( e, n+1 ) }^*) \\
&= f_n \otimes p_w  - \sum_{ \substack{ s_E^{-1}(e)=    w \\ r_E(e) \notin H } } (f_{n+1} \otimes s_e)(f_{n+1} \otimes s_e)^* \\
&= f_n \otimes p_w - \sum_{ \substack{ s_E^{-1}(e)=    w \\ r_E(e) \notin H } } (f_{n} \otimes s_es_e^*) \\
&= f_n \otimes \left( p_w - \sum_{ \substack{ s_E^{-1}(e)=    w  \\ r_E(e) \notin H } } s_e s_e^* \right).
\end{align*}
Consequently, $\phi( p_{(w,n)}^{H_U \times \ZZ} )$ is an element of $I_{ (H_U, S_U)} \rtimes_{\gamma^E} \mathbb{T}$ since $w \in S_U$. Thus, implies that $\phi ( I_{ (H_U ,S_U) } \subseteq I_{ (H_U, S_U)} \rtimes_{\gamma^E} \mathbb{T} =(C^*(E) \rtimes_{\gamma^E} \mathbb{T})[U]$.  Since $I_{(H_U, S_U)}$ is generated by $\{ p_v : v \in H_U \} \cup \{ p_w^{H_U} : w \in S_U \} $, $I_{ (H_U, S_U) } \rtimes_{\gamma^E} \mathbb{T}$ is generated by $\{ f_n \otimes p_v, f_n \otimes p_w^{H_U} : v \in H_U, w \in S_U \}$.  Hence, by the above computation, we see that 
$$
I_{ (H_U, S_U) } \rtimes_{\gamma^E} \mathbb{T} \subseteq \phi ( I_{(H_U \times \ZZ,  S_U \times \ZZ)} ).
$$
Hence, $\phi (  I_{ (H_U \times \ZZ, S_U \times \ZZ)  } ) =  I_{ (H_U, S_U) } \rtimes_{\gamma^E} \mathbb{T}$.  Therefore, $\varphi$ is an $X$-equivariant isomorphism.  The fact that $\hat{\gamma} \circ \phi = \phi \circ \lt$ follows from \cite[Lemma~3.1]{irws:ckigm}.  The last statement of the theorem follows from the fact that $\varphi$ is an $X$-equivariant isomorphism such that $\hat{\gamma} \circ \phi = \phi \circ \lt$.  
\end{proof}

Theorem~\ref{thm:dual-skew} implies that the obstruction class $\delta(C^*(E))$ given by \eqref{eq-obstruct-class} may be replaced by the class of the exact sequence 
\[
\scalebox{.8}{
\xymatrix{
0 \ar[r] & \mathrm{XK}_1 ( C^*(E) )  \ar[r] & \mathrm{XK}_0 ( C^*( E \times_1 \ZZ ) )  \ar[rr]^-{\mathrm{id} - \lt^{-1}_* } &  &  \mathrm{XK}_0 ( C^*( E \times_1 \ZZ ) )   \ar[r] & \mathrm{XK}_0 ( C^*(E) )\ar[r] & 0}
}
\]
in $\mathrm{Ext}^2( \mathrm{XK}_*(C^*(E)), \mathrm{XK}_{*+1}(C^*(E)))$.  This observation together with the discussion in \cite[Section~5.3]{rbrm:mgmcek}, we get the following result.

\begin{theor}\label{thm:skew-obstruct}
Let $E$ and $F$ be graphs ($|E^0|+|F^0|=\infty$ allowed) such that $X =  \mathrm{Prim}_\gamma ( C^*(E)) \cong \mathrm{Prim}_\gamma ( C^*(F))$ and $X$ is finite.  An isomorphism from $\mathrm{XK}\delta ( C^*(E) )$ to $\mathrm{XK}\delta (C^*(F))$ lifts to an invertible element in $\mathrm{KK}_X ( C^*(E), C^*(F))$.  Moreover the obstruction classes, $\delta(C^*(E))$ and $\delta(C^*(F))$, are represented by the exact sequences
\[
\scalebox{.8}{
\xymatrix{
0 \ar[r] & \mathrm{XK}_1 ( C^*(E) )  \ar[r]  & \mathrm{XK}_0 ( C^*( E \times_1 \ZZ ) ) \ar[rr]^-{\mathrm{id} - \lt_*^{-1} } &  & \mathrm{XK}_0 ( C^*( E \times_1 \ZZ ) ) \ar[r] & \mathrm{XK}_0 ( C^*(E) ) \ar[r] & 0 
}
}
\]
and
\[
\scalebox{.8}{
\xymatrix{
0 \ar[r] & \mathrm{XK}_1 ( C^*(F) )  \ar[r]  & \mathrm{XK}_0 ( C^*( F \times_1 \ZZ ) ) \ar[rr]^-{\mathrm{id} - \lt_*^{-1} } &   & \mathrm{XK}_0 ( C^*( F \times_1 \ZZ ) ) \ar[r] & \mathrm{XK}_0 ( C^*(F) ) \ar[r] & 0 
}
}
\]
\end{theor}

To connect $\DT(C^*(E))$ and $\mathrm{XK}\delta^+(C^*(E))$, we will also need to show that the surjective homomorphism from $\mathrm{XK}_0( C^*(E \times_1 \ZZ))$ to $\mathrm{XK}_0(C^*(E))$ sends the positive cone onto the positive cone.  This may be well-known but we were not able to find a reference to this result, thus we provide its proof here.

\begin{lemma}\label{lem:PVsq-positive}
Let $E$  be a graph ($|E^0|=\infty$ allowed) and set $X =  \mathrm{Prim}_\gamma ( C^*(E))$ with $X$ finite.  Then for all gauge-invariant ideals $I$ of $C^*(E)$, the homomorphism $\alpha_I : K_0 ( I \rtimes_{\gamma^E} \mathbb{T} ) \to K_0 (I)$ in the dual Pimsner-Voiculescu exact sequence takes $K_0 ( I \rtimes_{\gamma^E} \mathbb{T} )_+$ onto $K_0 (I)_+$.  Consequently, the induced homomorphism from $K_0 ( C^*( E \times_1 \ZZ ) [ U ] )$ to $K_0 ( C^*(E)[U])$ takes $K_0 ( C^*( E \times_1 \ZZ ) [ U ] )_+$ onto $K_0 (C^*(E)[U])_+$ for all open subsets $U$ of $X$ and sends $[ p_{ (v, n) } ]$ to $[p_v]$.
\end{lemma}

\begin{proof}
Let $I$ be a gauge-invariant ideal in $C^*(E)$.  Recall that $\alpha_I$ is the homomorphism induced by the $*$-homomorphism
\[
\xymatrix{ I \rtimes_{\gamma^E} \mathbb{T} \ar[r]^-\iota & ( I \rtimes_{\gamma^E} \mathbb{T}) \rtimes_{ \widehat{\gamma^E} } \ZZ \ar[r]^-{\Psi} & I \otimes \mathbb{K}( L^2 ( \mathbb{T} ) ) }.
\]
The isomorphism $\Psi$ is the isomorphism of the Takai Duality Theorem described in \cite[Section~7.1]{dpw:cpc}.  It will be important to know $\Psi$ so we describe the isomorphism here.  Let $\rho \colon \TT \to U( L^2(\TT))$ be the right-regular representation which induces an action of $\TT$ on $\Kk(L^2(\TT))$ via $\rho(\zeta)T\rho(\zeta^{-1})$ for all $\zeta \in \TT$.  Then there is an isomorphism $\Phi \colon (I \rtimes_{\gamma^E} \mathbb{T} ) \rtimes_{ \widehat{\gamma^E}} \mathbb{Z} \to C ( \mathbb{T} , I ) \rtimes_{ \rho \otimes \mathrm{id} } \mathbb{T}$ mapping $F \in C_c ( \ZZ \times \mathbb{T} , I )$ to the element in $C ( \mathbb{T} , C ( \mathbb{T}, I))$ given by 
\[
\Phi (F)(s,r) = \sum_{ m \in \ZZ } (\gamma_{r}^E)^{-1} ( F( m, s ) ) \overline{ s^{-1} r }^m
\]
with the identification that the dual group of $\mathbb{T}$ is $\ZZ$.  Next, \cite[Lemma~7.5]{dpw:cpc} gives an isomorphism $\phi \colon C( \mathbb{T} ) \rtimes_{\rho} \mathbb{T}  \to \mathbb{K} ( L^2 ( \mathbb{T} ))$ mapping $f \in C ( \mathbb{T} \times \mathbb{T} )$ to the element in $\mathbb{K} ( L^2( \mathbb{T} ))$ given by 
\[
\phi ( f)(h)(z) = \int_\mathbb{T} f( w, z) h( w^{-1} z) dw.
\]
Then $\Psi = (\phi \otimes \mathrm{id}) \circ \Phi$.

Since $K_0 (I)_+$ is generated by gauge-invariant projections in $I$, it is enough to show that for every gauge-invariant projection $p$ in $I$, there exists a projection $q$ in $C^*(E) \rtimes_{\gamma^E} \mathbb{T}$ such that $[ \Phi (q) ] = [p]$ in $K_0(I)$.  Let $p$ be a gauge-invariant projection in $I$.  Note that $q_n= f_n \otimes p \in C ( \mathbb{T} , I ) \subseteq I \rtimes_{\gamma^E} \mathbb{T}$ is a projection since
\begin{align*}
( f_n \otimes p )^2 (\zeta ) &= \int_\mathbb{T} ( f_n \otimes p )(z) \gamma_z^E ( ( f_n \otimes p)(z^{-1} \zeta) )dz = \int_\mathbb{T} z^n p \gamma_z^E ( z^{-n} \zeta^n p ) dz \\
				&= \int_\mathbb{T}  \zeta^n p dz = \zeta^n p = ( f_n \otimes p )(\zeta) \quad  \text{and} \\
(f_n \otimes p)^*(\zeta) &= \gamma_\zeta^E ( (f_n \otimes p)(\zeta^{-1})^* ) = \gamma_\zeta^E ( ( \zeta^{-n} p)^*) =\gamma_\zeta^E ( \zeta^n p)= \zeta^n p = (f_n \otimes p)(\zeta)
\end{align*}
for all $\zeta \in \mathbb{T}$.  Define $F_n \in C_c ( \ZZ \times \mathbb{T} )$ by $F_n ( m , z ) = \delta_{0, m} z^n$, where $\delta_{0,m}$ is the Kronecker delta function.  Note that 
\[
(F_n \otimes p)(m, z ) = \delta_{0,m} z^n p = \iota(q_n)(m,z).
\]
Hence, $F_n \otimes p = \iota (q_n)$ in $(I \rtimes_{\gamma^E} \mathbb{T} ) \rtimes_{ \widehat{\gamma^E}} \mathbb{Z}$.  Therefore, $\Phi ( \iota ( q_n ) )$ is the element in $C ( \mathbb{T} , C ( \mathbb{T} , I)) \subseteq C ( \mathbb{T} , I ) \rtimes_{ \rho \otimes \mathrm{id} } \mathbb{T}$ given by 
\begin{align*}
\Phi ( \iota ( q_n ) ) (s,r) &=  \Phi (F_n \otimes p )(s,r) = \sum_{ m \in \ZZ } (\gamma_{r}^E)^{-1} ( (F_n\otimes p )( m, s ) )\overline{ s^{-1} r }^m \\
				&= (\gamma_r^E)^{-1} ( s^n p  ) =  s^n p = (G_n \otimes p)(s,r )
\end{align*}
where $G_n$ is the element of $C ( \mathbb{T} \times \mathbb{T} )$ that sends $(s,r)$ to $s^n$.  Note that 
\[
\phi (G_n )(f_m)(r) = \int_\mathbb{T} s^n  s^{-m} r^m ds = \delta_{n, m} f_n (r)
\]  
which implies that $\phi( G_n) = e_{n,n}$, the projection onto the subspace generated by $f_n$.  Hence, 
\[
\Psi ( \iota ( q_{n} ) ) = p \otimes e_{n,n}
\]
which implies that $[ \Psi ( \iota (q_n) ) ] = [ p ] \in K_0 (I)$.

For the last part of the theorem, observe that the induced homomorphism from $K_0 ( C^*( E \times_1 \ZZ ) [ U ] )$ to $K_0 ( C^*(E)[U])$ is given by $\alpha_{ C^*(E)[U] } \circ [\phi]$.   Therefore, the homomorphism sends $K_0 ( C^*( E \times_1 \ZZ ) [ U ] )_+$ onto $K_0 ( C^*(E)[U])_+$ and 
\[
\alpha_{ C^*(E)[U] } \circ [\phi] ( p_{(v,n) } ) = \alpha_{ C^*(E)[U] }  ( [ f_n \otimes p_v ] ) = [p_v].\qedhere
\]
\end{proof}

We are now ready to prove our main result in this section.

\begin{theor}\label{thm:graded-kthy-XKdelta}
Let $E$ and $F$ be graphs ($|E^0|+|F^0|=\infty$ allowed).
\begin{enumerate}[(1)]
\item If 
$\DT(E)\simeq \DT(F)$, 
then $X =  \mathrm{Prim}_\gamma ( C^*(E)) \cong \mathrm{Prim}_\gamma ( C^*(F))$. If further $X$ is finite, we have that $\mathrm{XK}\delta^+ ( C^*(E)) \cong \mathrm{XK} \delta^+ ( C^*(F))$.
\end{enumerate}
When  $|E^0|+|F^0|<\infty$, the finiteness condition on $X$ is automatic, and  we further have 
\begin{enumerate}[(1)]\addtocounter{enumi}{1}
\item  If $\DQp(E)\simeq \DQp(F)$,  then $\mathrm{XK}\delta^{+,1} ( C^*(E)) \cong \mathrm{XK} \delta^{+,1} (C^*(F))$.
\end{enumerate}
\end{theor}

\begin{proof}
Suppose $\alpha$ is an isomorphism from $\DT(E)$ to $\DT(F)$.  Let $\lambda_E$ and $\lambda_F$ be the lattice isomorphism provided by Theorem~\ref{thm:admiss-ordideal} for $E$ and $F$ respectively.  Then $\beta = \lambda_F^{-1} \circ \alpha \circ \lambda_E$ is a lattice isomorphism between the gauge-invariant ideals of $C^*(E)$ and the gauge-invariant ideals of $C^*(F)$, where the lattice isomorphism that sends $I_{(H, S)}$ to $\lambda_F^{-1} ( \alpha(K_0( I_{ ( H \times \ZZ, S \times \ZZ) } ) ))$.  Consequently, the lattice isomorphism $\beta$ induces a homeomorphism $\overline{\beta}$ from $\mathrm{Prim}_{\gamma}( C^*(E))$ to $\mathrm{Prim}_{\gamma}( C^*(F))$, i.e., $\beta ( C^*(E)[U] ) = C^*(F)[ \overline{\beta}(U)]$.  Consequently, $C^*(F)$ becomes a $C^*$-algebra over $X = \mathrm{Prim}_\gamma ( C^*(E))$ by setting 
\[
C^*(F)[U] := C^*(F)[\overline{\beta}(U)] = \lambda_F^{-1} ( \alpha(K_0( I_{ ( H \times \ZZ, S \times \ZZ) } ) ))
\]   

Let $U$ be an open subset $X$, let $(H_U, S_U)$ be the unique admissible pair for $E$ such that $I_{ (H_U,S_U)} = C^*(E)[U]$.  Recall that $C^*(E\times_1\ZZ) [ U ] = I_{ (H_U \times \ZZ , S_U \times \ZZ) }$.  Let $(\overline{H}_U, \overline{S}_U)$ be the unique admissible pair for $F$ such that $I_{(\overline{H}_U, \overline{S}_U)} = C^*(F)[U]$, so $C^*(F \times_1\ZZ)[U] = I_{ ( \overline{H}_U \times \ZZ, \overline{S}_U \times \ZZ) }$.  Since $C^*(F)[U] = \lambda_F^{-1} ( \alpha(K_0( I_{ ( H \times \ZZ, S \times \ZZ) } ) ))$, 
\[
K_0( I_{(\overline{H}_U \times \ZZ , \overline{S}_U \times \ZZ) } ) =\lambda_F (I_{(\overline{H}_U, \overline{S}_U)} ) = \alpha(K_0( I_{ ( H \times \ZZ, S \times \ZZ) } ))
\]
 Consequently, $\alpha$ is an order isomorphism from $K_0(C^*(E\times_1\ZZ) [ U ] ) = K_0( I_{ (H_U \times \ZZ , S_U \times \ZZ) })$ to $K_0( C^*(F \times_1\ZZ)[U] )= K_0(I_{ ( \overline{H}_U \times \ZZ, \overline{S}_U \times \ZZ) })$.  And since $\lt_* \circ \alpha = \alpha \circ \lt_*$, for each open subset $U$ of $X$, $\alpha$ induces an order isomorphism from $K_0( C^*(E \times_1\ZZ)[U])$ to $K_0( C^*(F \times_1 \ZZ)[U])$ (which we denote by $\alpha_U$) such that the diagram
\[
\xymatrix{
 K_0 ( C^*( E \times_1 \ZZ )[U] ) \ar[r]^-{\mathrm{id} - \lt_*^{-1} } \ar[d]_{\alpha_U} &  K_0 ( C^*( E \times_1 \ZZ )[U] )\ar[d]^{\alpha_U} \\
K_0 ( C^*( F \times_1 \ZZ )[U] ) \ar[r]^-{\mathrm{id} - \lt_*^{-1} } &   K_0 ( C^*( F \times_1 \ZZ )[U] ) 
}
\]
is commutative.  Consequently, for each open subset $U$ of $X$, there exists a homomorphism $\theta_U \colon K_*(C^*(E)[U]) \to K_*(C^*(F))$ such that the diagram 
\[
\scalebox{.8}{
\xymatrix{
0 \ar[r] & K_1 ( C^*(E)[U] )  \ar[r]  \ar[d]^{\theta_U} & K_0 ( C^*( E \times_1 \ZZ )[U] ) \ar[rr]^-{\mathrm{id} - \lt_*^{-1} } \ar[d]^{\alpha_U} & &   K_0 ( C^*( E \times_1 \ZZ )[U] ) \ar[r] \ar[d]^{\alpha} & K_0 ( C^*(E)[U] ) \ar[r] \ar[d]^{\theta_U}  & 0 \\
0 \ar[r] &K_1 ( C^*(F)[U] )  \ar[r]  & K_0 ( C^*( F \times_1 \ZZ )[U] ) \ar[rr]^-{\mathrm{id} - \lt_*^{-1} } &   & K_0 ( C^*( F \times_1 \ZZ )[U] ) \ar[r] & K_0 ( C^*(F)[U] ) \ar[r] & 0 
}
}
\]
is commutative as the rows are exact sequences by \cite[Corollary~5.14]{rbrm:mgmcek} and Theorem~\ref{thm:dual-skew}.  A diagram chase shows that $\theta_U$ is an isomorphism.  By Lemma~\ref{lem:PVsq-positive}, $\theta_U$ is an order isomorphism.  A computation shows that $(\theta_U)$ intertwines the homomorphisms $K_*( \iota_U^V)$ for all open subsets $U$ and $V$ of $X$ such that $U \subseteq V$.  Then $\beta := ( \theta_{U_x})_{x \in X}$ is an isomorphism from $\mathrm{XK}^+(C^*(E))$ to $\mathrm{XK}^+(C^*(F))$
such that the diagram 
\[
\scalebox{.8}{
\xymatrix{
0 \ar[r] & \mathrm{XK}_1 ( C^*(E) )  \ar[r]  \ar[d]^{\beta} & \mathrm{XK}_0 ( C^*( E \times_1 \ZZ ) ) \ar[rr]^-{\mathrm{id} - \lt_*^{-1} } \ar[d]^{\alpha} &  &  \mathrm{XK}_0 ( C^*( E \times_1 \ZZ ) ) \ar[r] \ar[d]^{\alpha} & \mathrm{XK}_0 ( C^*(E) ) \ar[r] \ar[d]^{\beta}  & 0 \\
0 \ar[r] & \mathrm{XK}_1 ( C^*(F) )  \ar[r]  & \mathrm{XK}_0 ( C^*( F \times_1 \ZZ ) ) \ar[rr]^-{\mathrm{id} - \lt_*^{-1} } & &   \mathrm{XK}_0 ( C^*( F \times_1 \ZZ ) ) \ar[r] & \mathrm{XK}_0 ( C^*(F) ) \ar[r] & 0 
}
}
\]
is commutative.  By Theorem~\ref{thm:skew-obstruct}, we get that $\beta$ is an isomorphism from $\mathrm{XK}\delta^+(C^*(E))$ to $\mathrm{XK}\delta^+(C^*(F))$.  This concludes the proof of (1). 

We now prove (2).  Assume $\alpha$ is an isomorphism from $\DQp(E)$ to $\DQp(F)$.  Let $\{ \theta_U : U \text{ open subsets of } X \}$ and $\beta$  be the isomorphisms constructed in the proof of (1).  By construction, the diagram
\[
\scalebox{.75}{
\xymatrix{
\displaystyle{\bigoplus_{x \in X }} K_0 ( C^*(E \times_1\ZZ)[U_x]) \ar@{->>}[dddd]_{\abknotation{(\iota_{U_x}^X)_*}} \ar@{->>}[rrr] \ar[rd]^-{\oplus \alpha_{U_x} }& & & \displaystyle{\bigoplus_{x \in X }} K_0 ( C^*(E)[U_x])    \ar@{->>}[dddd]^-{\abknotation{(\iota_{U_x}^X)_*}} \ar[ld]_-{\oplus \theta_{U_x} }  \\
 &  \displaystyle{\bigoplus_{x \in X }} K_0 ( C^*(F \times_1\ZZ)[U_x])  \ar[r]  \ar@{->>}[dd]_-{\abknotation{(\iota_{U_x}^X)_*}}& \displaystyle{\bigoplus_{x \in X }} K_0 ( C^*(F)[U_x]) \ar@{->>}[dd]^-{\abknotation{(\iota_{U_x}^X)_*}}  \\
 \\
& K_0 ( C^*(F \times_1 \ZZ))    \ar@{->>}[r]    &   K_0(C^*(F)) \\
K_0 ( C^*(E \times_1 \ZZ))    \ar@{->>}[rrr]  \ar[ru]_-{\alpha}  &  & &  K_0(C^*(E))  \ar[lu]^-{\theta_X}
}
}
\]
is commutative.  Commutativity of the right square and the fact that the isomorphism $\beta \colon \mathrm{XK}^+(C^*(E)) \to \mathrm{XK}^+(C^*(F))$ is defined as $\beta=(\theta_{U_x})_{x \in X}$ implies the unique isomorphism from $K_0(C^*(E))$ to $K_0(C^*(F))$ that is induced by $\beta$ is $\theta_X$.  Let $\pi_E$ be the homomorphism from $K_0(C^*(E \times_1\ZZ))$ to $K_0(C^*(E))$ and let $\pi_F$ be the homomorphism from $K_0(C^*(F \times_1\ZZ))$ to $K_0(C^*(F))$ in the above diagram.  By Lemma~\ref{lem:PVsq-positive},  
\[
\pi_E ([p_0^E]) = \sum_{ v \in E^0 }  \pi_E ( [p_{(v,0)} ] ) = \sum_{v \in E^0 } [ p_v ] = [1_{C^*(E)} ].
\]
A similar computation shows that $\pi_F ( [p_0^F])= [1_{C^*(F)} ]$.  Commutativity of the bottom square implies 
\begin{align*}
\theta_X([1_{C^*(E)} ]) &= \theta_X \circ \pi_E ([p_0^E])  = \pi_F \circ \alpha ( [p_0^E]) \\
			&= \pi_F ([p_0^F]) = [1_{C^*(F)}].
\end{align*}
This implies that $\beta$ is an isomorphism from $\mathrm{XK}\delta^{+,1} (C^*(E))$ to $\mathrm{XK}\delta^{+,1}(C^*(F))$.
\end{proof}

\section{Computing dimension data}

As we shall touch upon again in Section \ref{twosided}, computing and comparing dimension quadruples for regular graphs is an important and well-studied method in symbolic dynamics. In this section, we show how to compute the data in our dimension quadruples $\DQ(-)$ and $\DQp(-)$ by amending the classical results. We  follow the general approach in \cite[Section~4.4]{rh:dlpa} to detail results that originate in \cite{wk:dftmc}, but generalize to non-regular graphs along the way.   Note that we choose to work with right matrix actions.   


We first recall the eventual range and eventual kernel of an integral square matrix.  As detailed e.g. in  \cite[Remark~7.4.4 (2) and (3)]{dlbm:isdc}, the eventual range $\mathcal{R}_A = \bigcap_{k=1}^\infty \QQ^n A^k = \QQ^n A^n$, and the eventual kernel $\mathcal{K}_A = \bigcup_{ k = 1}^\infty \ker( A^k ) = \ker(A^n)$, of any $n\times n$ matrix satisfies
\[
\QQ^n = \mathcal{R}_A \overset{\mathbf{\cdot}}{+} \mathcal{K}_A
\]
in the sense that every element can be uniquely decomposed. Note also that both  $\mathcal{R}_A$ and  $\mathcal{K}_A$ are $A$-invariant subspaces of $\QQ^n$.   Moreover, $A$ is invertible on $\mathcal{R}_A$.  We write $\pK{\vv}$ for the first component and $\piK{\vv}$ for the second component of any $\vv\in\QQ^n$.  It is easy to check that $\pi_{\mathcal{R}}$ and $\pi_\mathcal{K}$ are homomorphisms.

\begin{defin}\label{kriegerpic}
Let $A$ be an $n \times n$ matrix with entries in $\NN_0$.  We define the triple $(\Delta_A, \Delta_A^+ , \delta_A)$ by\index{DeltaA@$\Delta_A$}\index{DeltaA+@$\Delta_A^+$}\index{deltaA@$\delta_A$}
\begin{align*}
\Delta_A &:= \left\{ \vv \in    \mathcal{R}_A  \ | \  \vv A^\ell \in \ZZ^n, \text{ for some } \ell \in \NN_0 \right\}\\
\Delta_A^+ &:= \left\{ \vv \in  \mathcal{R}_A    \ | \ \vv A^\ell \in \NN_0^n, \text{ for some } \ell \in \NN_0 \right\}.
\end{align*}
and $\delta_A \colon \Delta_A \to \Delta_A$ is the automorphism $\delta_A( \vv) = \vv A$.  
\end{defin}

Note that 
$$
(\Delta_A, \Delta_A^+ , \delta_A) \cong (\Delta_B, \Delta_B^+ , \delta_B) 
$$
if and only if 
$$
(\Delta_A, \Delta_A^+ , \delta_A^{-1}) \cong (\Delta_B, \Delta_B^+ , \delta_B^{-1}).
$$
So, as an invariant of shift equivalence for matrices, it makes no difference whether we use $\delta_A$ or $\delta_A^{-1}$ in the definition of the triple above.

\begin{lemma}\label{lem-regularDQ}
Let $A$ be an $n \times n$ matrix with entries in $\NN_0$.   
\begin{enumerate}[(i)]
\item For all $\vv \in \QQ^n$ and for all $m \in \NN_0$, $\pK{\vv A^m} = \pK{\vv}A^m$ and $\vv A^m = \pK{\vv}A^m$ when $m \geq n$.

\item For all $\vv \in \ZZ^n$, $\pK{\vv} \in \Delta_A$ and for all $\vv \in \NN_0^n$, $\pK{\vv} \in \Delta_A^+$.
\end{enumerate}
\end{lemma}

\begin{proof}
We first prove Item (i).  Let $\vv \in \QQ^n$ and let $m \in \NN_0$.  Note that
\[
\vv A^m = \pK{\vv}A^m + \piK{\vv}A^m,
\]
and $\pK{\vv}A^m  \in \mathcal{R}_A$ and $\piK{\vv}A^m \in \mathcal{K}_A$.  Uniqueness of the decomposition gives $\pK{\vv A^m} = \pK{\vv}A^m$.  Since $\piK{\vv} A^m = \mathbf{0}$ for $m \geq n$, $\vv A^m = \pK{\vv}A^m + \piK{\vv}A^m =\pK{\vv}A^m$ for $m \geq n$.

Item (ii) follows from $\pK{\vv}A^n = \vv A^n \in \ZZ^n$ if $\vv \in \ZZ^n$ and $\pK{\vv}A^n = \vv A^n \in \NN_0^n$ if $\vv \in \NN_0^n$.
\end{proof}

Let $E$ be a graph.  Using the partition $E^0 = E^0_\regX \sqcup E^0_\singX$, write 
$$
{\mathsf A}_E = \left[\begin{matrix} \adjRR_E & \adjRS_E \\ \adjSR_E &\adjSS_E  \end{matrix}\right].
$$  
We write $\nor=|E^0_\regX|$ and $\nos=|E^0_\singX|$.\index{n@$\nos$}\index{n@$\nor$}\index{AE@$\adjRR_E$}\index{AE@$\adjSR_E$}\index{AE@$\adjRS_E$}\index{AE@$\adjSS_E$}  We always think of elements of $\ZZ^{\nor}$ and $\ZZ^{\nos}$ as rows and mostly denote them $\vv$ and $\ww$, respectively. We write $\uu_\regX$ and $\uu_\singX$ for the vectors having $1$ in all entries.  
 
 We shall work inside the group of two-sided sequences of integer vectors $\tsX\ZZ^n$ with an emphasis on the subgroup
 \[
 \osX\ZZ^n=\{(\ww_k)\in\tsX\ZZ^n\mid \exists k\forall \ell<k:\ww_\ell=\oo\}
 \]
 of sequences that are eventually zero to the left. It is convenient to represent the elements of $\osX\ZZ^n$ by
 \[
 \www=\oXe{k}{ \ww_0, \ww_1, \dots }=\oXe{k}{ \ww_i : i=0,1,\dots}
 \]
indicating the two-sided sequence with $\www_\ell=\ww_{\ell-k}$ for $\ell\geq k$ and $\www_\ell=\oo$ otherwise. One may of course arrange that $\ww_0\not=\oo$ if $\www\not=\ooo$ (with $\ooo$ the zero of $\osX$), but we will not always do so, and hence it can not be assumed. We note that $\nsX\ZZ^n\subseteq \osX\ZZ^n$ and will represent the elements
 \[
 \oXe{k}{\vv_0,\dots,\vv_m}.
 \]
 The key to understanding the dimension quadruples in the presence of singular vertices is the following map.

 \begin{defin}\label{etaomegadef}
 Let $E$ be a graph and define $\eta:\Delta_{\adjRR_E}\to \osX\ZZ^{\nos}$\index{eta@$\eta$} by 
 \[
 \eta(\vv):=\oXe{\ell+1}{ \vv (\adjRR_E)^{\ell+i } \adjRS_E : i=0,1,\dots}
 \]
 with $\ell$ the smallest nonnegative integer so that $\vv (\adjRR_E)^\ell \in\ZZ^{\nor}$.   Let $\nDT_E$ be the subset of $\Delta_{\adjRR_E} \times \osX\ZZ^{\nos}$ given by
 \[
 \nDT_E = \{(\vv,\www)\in\Delta_{\adjRR_E}\times \osX\ZZ^{\nos}\mid \eta(\vv)-\www\in\nsX\ZZ^{\nos}\},
 \]
 with the convention that if $E$ is a regular graph, then $\nDT_E = \Delta_{\mathsf{A}_E}$.\index{omegaE@$\nDT_E$} 
 \end{defin}
 Note that $\eta$ is not a group homomorphism, but that it induces $\overline\eta:\Delta_{\adjRR_E}\to \osX\ZZ^\nos/\nsX\ZZ^\nos$ which is.

 \begin{lemma}\label{lem-dimgrp-inv}
Let $E$ be a graph and let $\vv \in \Delta_{\adjRR_E}$.  Suppose $\ell$ is the smallest nonnegative integer such that $\vv (\adjRR_E)^\ell \in \ZZ^\nor$ and suppose $\ell'$ is the smallest nonnegative integer such that $\vv(\adjRR_E)^{-1} (\adjRR_E)^{\ell'} \in \ZZ^\nor$.  Then
\[
\eta( \vv (\adjRR_E)^{-1}) = \begin{cases} \rt(\eta(\vv)) &\text{ if } \ell' \geq 1 \\ \oXe{1}{\vv(\adjRR_E)^{i-1} \adjRS_E : i=0,1,\dots } &\text{ if } \ell' = 0 \end{cases}.
\]
and 
\[
\eta( \vv \adjRR_E) = \begin{cases} \rt^{-1} (\eta(\vv)) &\text{ if } \ell \geq 1 \\ \oXe{1}{\vv(\adjRR_E)^{i+1} \adjRS_E : i=0,1,\dots } &\text{ if } \ell = 0 \end{cases}
\]

Moreover, $(\delta_{\adjRR_E}^{-1} \times \rt) ( \nDT_E) = \nDT_E$.
\end{lemma}

\begin{proof}
 Note $\vv(\adjRR_E)^{-1} (\adjRR_{E} )^{\ell+1} = \vv (\adjRR_E)^{\ell} \in \ZZ^\nor$.  Thus, $\ell' \leq \ell+1$.  Since 
$ \vv(\adjRR_{E} )^{\ell'-1} = \vv(\adjRR_E)^{-1} (\adjRR_{E} )^{\ell'} \in \ZZ^{\nor}$, we have $\ell'-1 \geq \ell$ if $\ell' \geq 1$ or $\ell' =0$. 

Suppose $\ell' \geq 1$.  Then $\ell' = \ell+1$.  Thus, 
\begin{align*}
\eta(\vv (\adjRR_E)^{-1} ) &= \oXe{\ell+2}{ \vv (\adjRR_E)^{-1} (\adjRR_E)^{\ell+1+i } \adjRS_E : i=0,1,\dots} \\
			&= \rt( \eta( \vv )).
\end{align*}
Suppose $\ell' = 0$.  Then $\vv(\adjRR_E)^{-1} \in \ZZ^{\nor}$ which implies $\vv \in \ZZ^\nor$.  Hence, $\ell = 0$.  Then 
\begin{align*}
\eta(\vv (\adjRR_E)^{-1} )  &= \oXe{1}{ \vv (\adjRR_E)^{-1} (\adjRR_E)^{i } \adjRS_E : i=0,1,\dots} \\
					&= \oXe{1}{\vv(\adjRR_E)^{i-1} \adjRS_E : i=0,1,\dots }. 
\end{align*}

Let $\ell''$ be the smallest positive integer such that $\vv \adjRR_E (\adjRR_E)^{\ell''} \in \ZZ^\nor$.  Suppose $\ell = 0$.  Then $\ell'' = 0$.  So, 
\begin{align*}
\eta(\vv \adjRR_E )  &= \oXe{1}{ \vv (\adjRR_E) (\adjRR_E)^{i } \adjRS_E : i=0,1,\dots} \\
					&= \oXe{1}{\vv(\adjRR_E)^{i+1} \adjRS_E : i=0,1,\dots }. 
\end{align*}
Suppose $\ell \geq 1$.  Note that $\vv \adjRR_E (\adjRR_E)^{\ell-1} = \vv (\adjRR_E)^\ell \in \ZZ^\nor$ which implies $\ell-1 \geq \ell''$.  Since $\vv (\adjRR_E)^{\ell''+1} = \vv \adjRR_E (\adjRR_E)^{\ell''}  \in \ZZ^\nor$, $\ell''+1 \geq \ell$.  So, $\ell'' = \ell-1$ which implies
\begin{align*}
\eta(\vv \adjRR_E) &= \oXe{\ell}{ \vv \adjRR_E (\adjRR_E)^{\ell-1+i}  \adjRS_E : i=0,1, \dots} \\
			&= \rt^{-1}( \eta(\vv)).
\end{align*}

We now show that $(\delta_{\adjRR_E}^{-1} \times \rt) ( \nDT_E) = \nDT_E$.  Let $(\vv, \www) \in \nDT_E$.  Suppose $\ell' \geq 1$.  Then 
$\eta( \vv (\adjRR_E)^{-1}) - \rt(\www) = \rt(\eta(\vv) - \www) \in \nsX\ZZ^\nos$.  Hence, $(\delta_{\adjRR_E}^{-1} \times \rt) (\vv, \www) \in \nDT_E$.  Suppose $\ell' = 0$.  For $k \geq 1$, $\eta(\vv (\adjRR_E)^{-1} )_k  = \vv(\adjRR_E)^{k-2}\adjRS_E$ and for $k \geq 2$, $\eta(\vv (\adjRR_E)^{-1} )_k  = \eta(\vv)_{k-1}$.  As, $\rt(\www)_k = \www_{k-1}$, we have that $\eta(\vv (\adjRR_E)^{-1} ) - \rt(\www)  \in \nsX\ZZ^{\nos}$.  So, $(\delta_{\adjRR_E}^{-1} \times \rt) ( \vv, \www) \in \nDT_E$.

For the other inclusion, we show $(\delta_{\adjRR_E}^{-1} \times \rt)^{-1} ( \vv, \www)  \in \nDT_E$.   Suppose $\ell \geq 1$.  Then 
$\eta( \vv \adjRR_E) - \rt^{-1}(\www) = \rt^{-1} (\eta(\vv) - \www ) \in \nsX\ZZ^\nos$.  Suppose $\ell =0$.  For $k \geq 1$, $\eta(\vv \adjRR_E )_k = \vv (\adjRR_E)^{k} \adjRS_E = \eta(\vv)_{k+1}$.  Since $\rt^{-1}(\www)_k = \www_{k+1}$, we have $\eta(\vv \adjRR_E )_k  - \rt^{-1}(\www) \in \nsX\ZZ^\nos$.  Hence, $(\delta_{\adjRR_E}^{-1} \times \rt)^{-1} ( \vv, \www) \in \nDT_E$.
  
Thus, concluding the desired result that $( \delta_{\adjRR_E} \times \rt^{-1} ) (\nDT_E) = \nDT_E$.
\end{proof}

 \begin{lemma}\label{lem-pull-back}
 Let $E$ be a graph, define $\rho \colon \ZZ^{\nor} \to \osX\ZZ^{\nos}$ by 
 \[
 \rho( \vv ) = \oXe{1}{ \vv \adjRS_E, \vv \adjRR_E \adjRS_E , \vv (\adjRR_E)^2 \adjRS_E , \ldots }.
 \]
 Then $\rho$ is a group homomorphism such that for all $\vv \in \ZZ^{\nor}$
 \[
 \rho(\vv)=  \oXe{1}{\vv\adjRS_E}  + \rt(\rho(\vv \adjRR_E ) )
 \]
 and 
 \[
 \rho(\vv) - \oXe{m+1}{  \pK{\vv} (\adjRR_E)^m \adjRS_E ,  \pK{\vv} (\adjRR_E)^{m+1} \adjRS_E , \dots}  \in \nsX\ZZ^{\nos}
 \]
for any nonnegative integer $m$ with $\pK{\vv}(\adjRR_E)^m \in \ZZ^\nor$.  In particular, for all $\vv \in \ZZ^{\nor}$, 
\[
\rho(\vv) - \eta( \pK{\vv}) \in \nsX\ZZ^{\nos}.
\]
 \end{lemma}
 
 \begin{proof}
 A computation shows $\rho(\vv)=  \oXe{1}{\vv\adjRS_E}  + \rt(\rho(\vv \adjRR_E ) )$.  For $k \geq m+1$, the $k$th coordinate of 
 \[
 \oXe{m+1}{  \pK{\vv} (\adjRR_E)^m \adjRS_E ,  \pK{\vv} (\adjRR_E)^{m+1} \adjRS_E , \dots}
 \]  
is $\pK{\vv} (\adjRR_E)^{k-1} \adjRS_E$.  Since $\rho(\vv)_k = \vv (\adjRR_E)^{k-1} \adjRS_E$, the lemma now follows from this observation and the fact that $\pK{\vv} (\adjRR_E)^k = \vv (\adjRR_E)^k$ for all $k \geq \nor$ (by Lemma~\ref{lem-regularDQ}).
 \end{proof}

For a graph $E$, $\ee_v$ denotes the element of $\ZZ^{E^0}$ with $1$ in the $v$th coordinate and $0$ in all other coordinates.  An important property of $K_0(C^*(E \times_1\ZZ))$ that we will use throughout this section is that for all $v \in E^0_\regX$, 
\begin{equation}\label{eq-ck-rel-kthy}
[ p_{(v,\ell) } ] = \sum_{ w \in E^0_\regX} \adjRR_E (v,w) [ p_{(w, \ell+1)} ] + \sum_{ w \in E^0_\singX} \adjRS_E(v,w) [ p_{(w, \ell+1)} ]
\end{equation}
by using the Cuntz-Krieger Relation (CK3).

 \begin{propo}\label{premaster}
 Let $E$ be a graph.  The map $\phi:K_0( C^*(E \times_1 \ZZ))\to \Delta_{\adjRR_E}\times \osX\ZZ^{\nos}$ defined by
  \[
 \phi ( [p_{(v, k) } ] ) = \begin{cases} (\delta_A^{-1} \times \rt)^k  (\oo, \oXe{0}{\ee_v} ) &\text{ if } v\in E^0_\singX \\ (\delta_A^{-1} \times \rt)^k \left( \pK{\ee_v} , \rho( \ee_v) \right) &\text{ if } v \in E^0_\regX\end{cases}.
 \]
 induces an isomorphism
 \[
\phi:K_0( C^*(E \times_1 \ZZ))\to \nDT_E
\]
with the property 
\[
\phi\circ \rt_*= ( \delta_{\adjRR_E}^{-1}  \times \mathrm{rt})  \circ \phi
\]
and
\[
(\vv, \oXe{m+1}{ \vv (\adjRR_E)^m \adjRS_E, \vv (\adjRR_E)^{m+1}\adjRS_E, \dots} ) = \sum_{w \in E^0_\regX} (\vv (\adjRR_E)^m)_w  \phi( [ p_{(w,m)} ])
\] 
where $\vv \in \Delta_{\adjRR_E}$ and $m$ is a nonnegative integer with $\vv (\adjRR_E)^m \in \ZZ^{\nor}$.  In particular, $(\vv, \eta(\vv))$ is in the image of $\phi$ for all $\vv \in \Delta_{\adjRR_E}$.
\end{propo}
 
 \begin{proof}
We first show that $\phi$ induces a homomorphism from $K_0(C^*(E \times_1 \ZZ))$ to $\Delta_A\times \osX\ZZ^{\nos}$.  By \cite[Theorem~3.1]{ddmt:ckegc}, it is enough to show for all $v \in E^0_\regX$ and for all $k$,
\begin{equation}\label{eq-ck-rel-kthy2}
x_{(v,k)} = \sum_{ s(e)=v } x_{ (r(e), k+1)},
\end{equation}
where
$$
x_{(v,k)}= \begin{cases} (\delta_A^{-1} \times \rt)^k  (\oo, \oXe{0}{\ee_v} ) &\text{ if } v\in E^0_\singX \\ (\delta_A^{-1} \times \rt)^k \left( \pK{\ee_v} , \rho( \ee_v) \right) &\text{ if } v \in E^0_\regX\end{cases}.
$$
Note that $x_{(v,k)} = (\delta_{A}^{-1} \times \rt)^{k} ( x_{(v,0)} )$ for all $(v,k) \in E^0 \times \ZZ$.  Hence, to prove \eqref{eq-ck-rel-kthy2}, it is enough to prove \eqref{eq-ck-rel-kthy2} for $k =0$.

Let $v \in E^0_\regX$.  Then 
\begin{align*}
&\sum_{ s(e)=v } x_{ (r(e), 1)} \\
&= \sum_{ w \in E^0} \Asf_E(v,w) x_{ w, 1} = \sum_{ w \in E^0_\regX } \adjRR_E (v,w) x_{ w, 1} + \sum_{ w \in E^0_\singX } \adjRS_E(v,w) x_{w, 1} \\ 
&=  \sum_{ w \in E^0_\regX } \adjRR_E(v,w) ( \delta_{\adjRR_E}^{-1} \otimes \rt) ( \pK{\mathbf{e}_w}, \rho(\ee_w) ) + \sum_{ w \in E^0_\singX } \adjRS_E(v,w) ( \delta_{\adjRR_E}^{-1} \otimes \rt) (\oo, \oXe{0}{\ee_w} )  \\
&=( \delta_{\adjRR_E}^{-1} \otimes \rt)  ( \pK{ \ee_v\adjRR_E} , \rho( \ee_v \adjRR_E) )+ (\oo, \oXe{1}{\ee_v \adjRS} )  \\
&=( \delta_{\adjRR_E}^{-1} \otimes \rt)  ( \pK{ \ee_v}\adjRR_E , \rho( \ee_v \adjRR_E) )+ (\oo, \oXe{1}{\ee_v \adjRS} ) \\
&= x_{(v,0)},
\end{align*}
where the last two equalities follow from Lemma~\ref{lem-regularDQ} and Lemma~\ref{lem-pull-back}.  Thus, $\phi:K_0( C^*(E \times_1 \ZZ))\to \Delta_{\adjRR_E}\times \osX\ZZ^{\nos}$ is a homomorphism.  

Note that $\phi \circ \rt_* ( [p_{(v,k)} ] ) = \phi( [ p_{(v,k+1)} ] ) = (\delta_{\adjRR_E}^{-1} \otimes \rt)^{k+1}  (x_{(v,0)}) = (\delta_A^{-1} \otimes \rt) x_{(v,k)} =  (\delta_{\adjRR_E}^{-1} \otimes \rt) \circ \phi( [ p_{(v,k)}] )$.  Since $K_0(C^*(E\times_1\ZZ))$ is generated by $[p_{(v,k)}]$ for all $(v,k)$, $\phi\circ \rt_*= ( \delta_{\adjRR_E}^{-1}  \times \mathrm{rt}) \circ \phi$.

We now show $\phi$ is an injection.  Let $a \in K_0(C^*(E \times_1\ZZ))$ with $\phi(a) = (\oo, \ooo)$.  Repeated use of \eqref{eq-ck-rel-kthy}, we may assume $a = \sum_{v \in E_\regX^0 } m_{v} [ p_{(v, N)}] +  \sum_{v \in E_\singX^0 } \sum_{k = -N}^N m_{(v,k)} [ p_{(v, k)}]$, where $N \geq \nor$.  Set $\mathbf{w} = \sum_{v \in E_{\regX}^0} m_v \mathbf{e}_v$ and $\mathbf{w}_k =  \sum_{v \in E^0_\singX} m_{(v,k)} \mathbf{e}_{v}$.  Then 
\[
(\oo, \ooo)=\phi(a) = (\pK{\ww}(\adjRR_E)^{-N} , \oXe{-N}{\ww_{-N}, \ldots, \ww_N, \ww \adjRS_E , \ww  \adjRR_E \adjRS_E, \dots}). 
\]
Consequently,$\ww_{k} = \oo$ for all $k$, and $\pK{\ww} = \oo$ and $\ww (\adjRR_E)^i \adjRS_E = \oo$ for all $i$.  By Lemma~\ref{lem-regularDQ}, $\ww (\adjRR_E)^m = \pK{\ww} (\adjRR_E)^m = \oo$ for all $m \geq \nor$.  Using \eqref{eq-ck-rel-kthy} in the third equality,
\begin{align*}
a &= \sum_{v \in E_\regX^0 } m_{v} [ p_{(v, N)}] +  \sum_{v \in E_\singX^0 } \sum_{k = -N}^N m_{(v,k)} [ p_{(v, k)}]=  \sum_{v \in E_\regX^0 } m_{v} [ p_{(v, N)}] \\
	&= \sum_{ w \in E_\regX^0 } (\ww (\adjRR_E)^{N+1})_w [ p_{(w,N+1)}] + \sum_{ w \in E_\singX^0 } (\ww (\adjRR_E)^{N+1} \adjRS_E)_w [ p_{(w,N +1)} ] \\
	&= 0.
\end{align*}
Hence, $\phi$ is an injection.

We now prove that 
\[
(\vv, \oXe{m+1}{ \vv \adjRS_E, \vv \adjRR_E \adjRS_E, \dots} ) = \sum_{w \in E^0_\regX} (\vv (\adjRR_E)^m)_w  \phi( [ p_{(w,m)} ])
\] 
where $m$ is a nonnegative integer with $\vv (\adjRR_E)^m \in \ZZ^{\nor}$.   Indeed,
\begin{align*}
&\sum_{w \in E^0_\regX}  (\vv (\adjRR_E)^m)_w ( \pK{\ee_w} (\adjRR_E)^{-m} , \oXe{m+1}{ \ee_w (\adjRR_E)^i  \adjRS_E : i =0,1,\dots}) \\
&=\left( \pK{\vv (\adjRR_E)^m} (\adjRR_E)^{-m} ,   \oXe{m+1}{ \ee_v (\adjRR_E)^{i+m}  \adjRS_E : i =0,1,\dots} \right) \\
	&= (\pK{\vv} , \oXe{m+1}{ \vv (\adjRR_E)^m \adjRS_E, \vv (\adjRR_E)^{m+1}\adjRS_E, \dots} ) \qquad \text{by Lemma~\ref{lem-regularDQ}}\\
	&= (\vv, \oXe{m+1}{ \vv (\adjRR_E)^m \adjRS_E, \vv (\adjRR_E)^{m+1}\adjRS_E, \dots}).
\end{align*}

Lastly, we show $\phi$ is an isomorphism onto $\nDT_E$.  By Lemma~\ref{lem-pull-back}, the image of $\phi$ is a subset of $\nDT_E$.  Let $(\vv, \www) \in \nDT_E$.  Recall that there exists $a \in K_0(C^*(E \times_1\ZZ))$ such that $\phi(a) = (\vv, \eta(\vv))$.  Since $\www - \eta(\vv) \in \nsX\ZZ^{\nos}$, $\www = \eta(\vv) + \www'$ with $\www' \in  \nsX\ZZ^{\nos}$.  Note that $b  = a + \sum_{(w,k) \in E^1_\singX \times \ZZ } \www'_{w,k} [ p_{(w,k)}] \in K_0(C^*(E \times_1\ZZ))$, where $\www'_{w,k}$ is the $w$th coordinate of $\www'_k$.  Moreover,
\[
\phi(b) = (\vv, \eta(\vv)) + ( \oo , \www') = (\vv, \www).
\]
Hence, showing that $\nDT_E$ is a subset of the image of $\phi$.  We now conclude that $\phi \colon K_0(C^*(E \times_1\ZZ)) \to \nDT_E$ is an isomorphism.
 \end{proof}
 
 \begin{theor}\label{regularDQ} Let $E$ be a regular graph.  Then 
 $$
\DT(E)\cong ( \Delta_{\Asf_E}, \Delta_{\Asf_E}^+, \delta_{\Asf_E}^{-1})
$$
and
\begin{eqnarray*}
\DQ(E) &\cong&  ( \Delta_{\Asf_E}, \Delta_{\Asf_E}^+, \delta_{\Asf_E}^{-1},\Delta_{\Asf_E}^+) \\
\DQp(E)&\cong&(\Delta_{\Asf_E}, \Delta_{\Asf_E}^+, \delta_{\Asf_E}^{-1},\pK{\uu}).
\end{eqnarray*}
\end{theor}

\begin{proof}
Since $E$ is a regular graph,  $\nDT_E=\Delta_{\mathsf{A}_E}$.  By Proposition~\ref{premaster}, the map $\phi \colon K_0(C^*(E \times_1 \ZZ)) \to \Delta_{\Asf_E}$ with 
\[
\phi( [ p_{(v,k)} ] ) = \pK{\ee_v}\Asf_E^{-k}
\]
 is an isomorphism such that $\phi \circ \lt_* = \delta_{\Asf_E}^{-1} \circ \phi$ and $\phi( [p_0^E]) = \pK{\uu}$.  By Proposition~\ref{premaster} the inverse of $\phi$ is given by 
 $\phi^{-1} ( \vv ) = \sum_{w \in E^0 } (\vv \Asf_E^m)_w [ p_{(w,m)}]$, where $m$ is a nonnegative integer with $\vv \Asf_E^{m} \in \ZZ^{E^0}$.   As $K_0(C^*(E \times_1 \ZZ))_+$ is generated by $[ p_{(v,k)}]$'s, we see that $\phi$ and $\phi^{-1}$ are positive maps.  Hence, $\phi$ is an order isomorphism.
 
 Since $\pK{\uu}$ is clearly an order unit, the order ideal in this case is all of $\Delta_{\mathsf{A}_E}$ and thus carries no information.
 \end{proof}

Away from the regular case, Proposition~\ref{premaster} still characterizes $K_0(C^*(E\times_1\ZZ))$ as a group, even equivariantly, using only the $\adjRR_E$ and $\adjRS_E$ submatrices and the number $\nos$ of singular vertices. Since we know from Proposition \ref{Kpartialinva} that the dimension group contains enough information to reconstruct $E$ up to \xyz{000}-invariance, it is clear that the submatrices $\adjSR_E$ and $\adjSS_E$ must play a role too, and indeed they enter through the order of this group.
Because $\adjSR_E$ and $\adjSS_E$ may contain infinite entries, we pass to matrices $C\leq \adjSR_E$ and  $D\leq \adjSS_E$ to describe this order, where the smaller matrices are dominated  entry-wise and only have finite entries.

To define the positive cone of $\nDT_E$ that makes $\phi \colon K_0(C^*(E \times_1\ZZ))\to \nDT_E$ an order isomorphism, we define $\eta^+ \colon \Delta_{\adjRR_E}^+ \to \osX \NN_0^{\nos}$ by\index{etaplus@$\eta^+$}
$$
\eta^+( \vv ) = \oXe{\ell+1}{ \vv \adjRS_E , \vv \adjRR_E \adjRS_E, \dots}.
$$
where $\ell$ is the smallest nonnegative integer such that $\vv (\adjRR_E)^\ell \in \NN_0^\nor$.  By Proposition~\ref{premaster}, $(\vv, \eta^+(\vv)) \in \nDT_E$ for all $\vv \in \Delta_{\adjRR_E}^+$.  Arguing as in Lemma~\ref{lem-dimgrp-inv} we get the following.  

\begin{lemma}\label{lem-dimgrp-inv-plus}
Let $E$ be a graph, let $\ell$ be the smallest nonnegative integer such that $\vv (\adjRR_E)^\ell \in \NN_0^\nor$, and let $\ell'$ be the smallest nonnegative integer such that and $\vv(\adjRR_E)^{-1} (\adjRR_E)^{\ell'} \in \NN_0^\nor$.  Then
\[
\eta^+( \vv (\adjRR_E)^{-1}) = \begin{cases} \rt(\eta^+(\vv)) &\text{ if } \ell' \geq 1 \\ \oXe{1}{\vv(\adjRR_E)^{i-1} \adjRS_E : i=0,1,\dots } &\text{ if } \ell' = 0 \end{cases}
\]
and 
\[
\eta^+( \vv \adjRR_E) = \begin{cases} \rt^{-1} (\eta^+(\vv)) &\text{ if } \ell \geq 1 \\ \oXe{1}{\vv(\adjRR_E)^{i+1} \adjRS_E : i=0,1,\dots } &\text{ if } \ell = 0 \end{cases}
\]
\end{lemma}

We now describe some elements of the image of the order ideal generated by $p_E^0$ under the isomorphism $\phi$.

\begin{lemma}\label{lem-scale}
Let $E$ be a graph.  All elements of the form
\[
\left(\vv,\eta^+(\vv)\right), 
\]
for $\vv\in \Delta_{\adjRR_E}^+$ and all elements of the form
\begin{equation}\label{singgen}
\left(\pK{-\ee_w C} (\adjRR_E)^{-1},\oXe{0}{\ee_w,-\ee_wD, - \mathbf{e}_w C\adjRS_E,  - \mathbf{e}_w C \adjRR_E \adjRS_E,\dots} \right)
\end{equation}
for $w\in E^0_\singX$,  $C\leq \adjSR_E$ and  $D\leq \adjSS_E$, lie in $\phi(I[p_E^0])$.
\end{lemma}

\begin{proof}
To prove that $\left(\vv,\eta^+(\vv)\right) \in \phi(I[p_E^0])$ for $\vv\in \Delta_{A}^+$, if there exists a path in $E$ of length $\ell$ from $v$ to $w$, then there exists a path $\mu$ in $E \times_1 \ZZ$ from $(v,0)$ and $(w, \ell)$ which implies $[ p_{(w,\ell)}  ] = [s_\mu^* s_\mu]  = [s_\mu s_\mu^*] \in   I[p_E^0]$.   Let $\vv \in \Delta_A^+$ and let $\ell$ be the smallest nonnegative integer with $\vv (\adjRR_E)^\ell \in \NN_0^\nor$.  By Proposition~\ref{premaster}, 
\[ \phi^{-1} ( \vv , \eta^+(\vv)) = \sum_{w \in E^0_\regX} (\vv (\adjRR_E)^\ell)_w [ p_{(w, \ell)} ] = \sum_{w \in E_\regX^0} \sum_{z \in E_\regX^0} \vv_z (\adjRR_E)^\ell (z,w) [ p_{(w, \ell)} ] .\]
Note that the only nonzero terms on the right-hand side of the above equation are $z , w \in E_\regX^0$ for which there exists a path of length $\ell$ from $z$ to $w$.  Thus, $( \vv , \eta^+(\vv)) \in \phi(I[p_E^0])$.

To prove the elements in \eqref{singgen} are elements of $\phi(I[p_E^0])$, observe that 
$$
0\leq a=[p_{ (w,0) }] - \sum_{ v \in E^0_\singX }D(w, v) [ p_{(v,1)}] - \sum_{ v \in E^0_\regX } C(w, v) [ p_{(v,1)}]
 \leq [p_{ (w, 0) }]$$
 as 
 $$
 [p_{(w,0)}] = a +  \sum_{ v \in E^0_\singX } D(w, v) [ p_{(v,1)}] + \sum_{ v \in E^0_\regX } C(w, v) [ p_{(v,1)}] 
 $$
 and 
 \[
 a = \left[ p_{(w,0)} - \sum_{e \in X } s_{ (e,1)} s_{(e,1)}^* \right]
 \]
 with $s(X) = \{w\}$, $| \{ e \in X :  r(e) =v \} | = C(w, v)$ for $v \in E^0_\regX$, and $| \{ e \in X :  r(e) =v \} | = D(w, v)$ for $v \in E^0_\singX$.  Thus, $\phi(a) \in \phi(I[p_E^0])$ and 
 \begin{align*}
 \phi(a) &= ( \oo, \oXe{0}{\ee_w} ) - \sum_{ v \in E^0_\singX } D(w, v) ( \oo, \oXe{1}{\ee_v})  - \sum_{v \in E^0_\regX} C(w, v)  (\delta_{\adjRR_E}^{-1} \times \rt) ( \pK{\ee_v} , \rho(\ee_v) )\\
 &= ( \oo, \oXe{0}{\ee_w} ) - ( \oo, \oXe{1}{\ee_w D} ) - (\delta_{\adjRR_E}^{-1} \times \rt)  ( \pK{\ee_w C } , \rho( \ee_w C))\\
		&=  \left(\pK{-\ee_w C} (\adjRR_E)^{-1},\oXe{0}{\ee_w,-\ee_w D, - \mathbf{e}_w C \adjRS_E,  - \mathbf{e}_w C \adjRR_E  \adjRS_E,\dots} \right).\qedhere
 \end{align*}
\end{proof}

We denote the set of elements in \eqref{singgen} by $\singgen$.\index{Gc@$\singgen$}
Note that when $v$ is a sink, $\ee_v \adjSR_E$ and $\ee_v \adjSS_E$ must both vanish, and the single contribution to $\singgen$ for $v$ is  $(\oo,\oXe{0}{\ee_v})$. When $v$ is an infinite emitter, at least one of the relevant rows of $\adjSR_E$ or $\adjSS_E$ has an infinite entry, and consequently there will be infinitely many contributions.

\begin{theor}\label{DQmaster}
Let $E$ be a graph.  Under the isomorphism 
$
\phi:K_0( C^*(E \times_1 \ZZ))\to\nDT_E$
in Proposition \ref{premaster}, 
\begin{enumerate}[(i)]
\item $\rt_*$ is taken to $\delta_{\adjRR_E}^{-1}\times \rt$,
\item $K_0( C^*(E \times_1 \ZZ))_+$ is taken to $\nDT_E^+$, the cone generated by 
\[
\{(\delta_{\adjRR_E}^{-1}\times \rt)^n \left(\vv,\eta^+(\vv)\right)\mid n \in \NN_0, \vv\in \Delta_{\adjRR_E}^+\}\cup \{(\delta_{\adjRR_E}^{-1}\times \rt)^n(\singgen)\mid n\in\ZZ\}, 
\]
\item $[p_E^0]$ is taken to 
\[
\left(\pK{\uu_\regX },\oXe{0}{\uu_\singX,\uu_\regX \adjRS_E, \uu_\regX \adjRR_E \adjRS_E, \uu_\regX (\adjRR_E)^2 \adjRS_E,\dots}\right)
\]
\end{enumerate}
\end{theor}

\begin{proof}
Item~(i) follows from Proposition~\ref{premaster}.  For Item~(ii), we first show $(\delta_{\adjRR_E}^{-1} \times \rt)(\nDT_E^+) = \nDT_E^+$.  It is clear that $(\delta_{\adjRR_E}^{-1} \times \rt) (\delta_{\adjRR_E}^{-1}\times \rt)^n(\singgen)$ and $(\delta_{\adjRR_E}^{-1} \times \rt)^{-1} (\delta_{\adjRR_E}^{-1}\times \rt)^n(\singgen)$ are subsets of $\nDT_E^+$ for all $n$.   From the definition of $\nDT_E^+$, $(\delta_{\adjRR_E}^{-1} \times \rt) (\delta_{\adjRR_E}^{-1}\times \rt)^n(\vv, \eta^+(\vv)) \in \nDT_E^+$ for all $\vv \in \Delta_{\adjRR_E}^+$ and for all $n \in \NN_0$, and $(\delta_{\adjRR_E}^{-1} \times \rt)^{-1} (\delta_{\adjRR_E}^{-1}\times \rt)^n(\vv, \eta^+(\vv)) \in \nDT_E^+$ for all $\vv \in \Delta_{\adjRR_E}^+$ and for all $n \in \NN$.  Thus, we are left to show that $(\delta_{\adjRR_E}^{-1} \times \rt)^{-1}( \vv, \eta^+(\vv)) \in \nDT_E^+$ for all $\vv \in \Delta_{\adjRR_E}^+$.  Let $\ell$ be the smallest nonnegative integer such that $\vv (\adjRR_E)^\ell \in \NN_0^{\nor}$By Lemma~\ref{lem-dimgrp-inv-plus},
\begin{align*}
( \vv \adjRR_E, \rt^{-1} (\eta^+( \vv))) &= \begin{cases} ( \vv \adjRR_E, \eta^+( \vv \adjRR_E)) , &\text{ if } \ell \geq 1 \\  ( \vv \adjRR_E, \eta^+( \vv \adjRR_E)) + (\oo, \oXe{0}{\vv \adjRS_E}  )&\text{ if } \ell =0\end{cases}
\end{align*}
which are elements of $\nDT_E^+$.  Hence, completing the proof of the claim $(\delta_{\adjRR_E}^{-1} \times \rt)(\nDT_E^+) = \nDT_E^+$.

For Item~(ii), let $(v, k) \in E^0\times \ZZ$.  Since $(\delta_{\adjRR_E}^{-1} \times \rt)(\nDT_E^+) = \nDT_E^+$, to show that $\phi( [ p_{(v,k)} ] ) \in \nDT_E^+$, it is enough to show the statement is true for $k =0$.  If $v \in E^0_\singX$, then $\phi( [ p_{(v,k)} ] ) = (\oo, \oXe{k}{\ee_v})  \in \nDT_E^+$.  Suppose $v \in E_\regX^0$.  By repeated use of \eqref{eq-ck-rel-kthy}, we have 
\begin{align*}
[ p_{(v,0)}] &= \sum_{ w \in E^0} (\adjRR_E)^\ell( v, w) [ p_{(w, \ell)} ] + \sum_{w \in E^0} \sum_{ k = 1}^\ell ((\adjRR_E)^{k-1} \adjRS_E)(v,w) [ p_{(w,k)} ] \\
	&= \sum_{ w \in E^0} (\ee_v (\adjRR_E)^\ell)_w [ p_{(w, \ell)} ] + \sum_{w \in E^0} \sum_{ k = 1}^\ell ((\adjRR_E)^{k-1} \adjRS_E)(v,w) [ p_{(w,k)} ].
\end{align*}
where $\ell$ is the smallest nonnegative integer for which $\pK{\ee_v} (\adjRR_E)^\ell \in \NN_0^\nor$.  By Proposition~\ref{premaster}, 
$$
\phi( [ p_{(v,0)} ] ) = ( \pK{\ee_v }, \eta^+( \ee_v ) ) +  \sum_{w \in E^0} \sum_{ k = 1}^\ell ((\adjRR_E)^{k-1} \adjRS_E(v,w)) \oXe{k}{\ee_w}.
$$
So, $\phi( [ p_{(v,0)} ] )  \in \nDT_E^+$ which completes the proof that $\phi ( K_0( C^*(E \times_1\ZZ))_+) \subseteq \nDT_E^+$.

We now show $\nDT_E^+ \subseteq \phi ( K_0( C^*(E \times_1\ZZ))_+)$.  By Lemma~\ref{lem-scale}, $\{(\delta_{\adjRR_E}^{-1}\times \rt)^n(\singgen)\mid n\in\ZZ\}$ is a subset of $\phi ( K_0( C^*(E \times_1\ZZ))_+)$.  Let $\vv \in \Delta_{\adjRR_E}^+$ and let $\ell$ be the smallest nonnegative integer such that $\vv (\adjRR_E)^{\ell} \in \NN_0^\nor$.  By Proposition~\ref{premaster},
$$
(\delta_{\adjRR_E}^{-1} \times \rt)^n (\vv, \eta^+(\vv)) = \phi\left( \sum_{w \in E^0_\regX} (\vv (\adjRR_E)^\ell)_w   [ p_{(w,\ell+n)} ] \right) \in \phi ( K_0( C^*(E \times_1\ZZ))_+)
$$
as $(\vv (\adjRR_E)^\ell)_w  \in \NN_0^\nor$.  

Lastly, we prove Item~(iii),
\begin{align*}
\phi( [p_E^0] ) &= \sum_{w \in E^0_\regX} \phi([ p_{(w,0)} ]) + \sum_{w \in E^0_\singX} \phi([ p_{(w,0)} ]) \\
			&= \sum_{w \in E^0_\regX}  ( \pK{\ee_w} , \rho(\ee_w)) + \sum_{w \in E^0_\singX} \oXe{0}{\ee_w} \\
				&=\left(\pK{\uu_\regX },\oXe{1}{\uu_\regX \adjRS_E, \uu_\regX \adjRR_E \adjRS_E, \uu_\regX (\adjRR_E)^2 \adjRS_E,\dots}\right) + (\oo, \oXe{0}{\uu_\singX})\\
			&=\left(\pK{\uu_\regX },\oXe{0}{\uu_\singX,\uu_\regX \adjRS_E, \uu_\regX \adjRR_E \adjRS_E, \uu_\regX (\adjRR_E)^2 \adjRS_E,\dots}\right). \qedhere
\end{align*}
\end{proof}

Note that we are only able to describe the positive cone by a set of generators, and in general, we know of no better description of $I([p^0_E])$ than just the elements in the positive cone dominated by a multiple of $[p^0_E]$. In all of the special cases we will study, much more satisfactory descriptions are possible.

We will give several examples in the ensuing sections of how to employ this result. In Propositions \ref{cyclecomplete},
\ref{cuntzcomplete} and in Lemmas\ref{RRkstable}, we give examples of how to compute the dimension quadruples in the regular case. In Proposition \ref{finitecomplete}, Lemma \ref{RSkstable} and Proposition \ref{RSisoexact}, we work with graphs with finitely many edges, but singular due to a sink. In Propositions \ref{descaforder}, \ref{descafunit}, Lemma  \ref{SRkstable}, and Proposition \ref{SRisoexact} the graphs are infinite and we use the full force of Theorem \ref{DQmaster}.

\part{The generation conjectures}\label{genpart}

\chapter{Foundational tools and examples}

In this chapter, we will analyze key classes of graphs, using moves to understand when they are mutually  $\xyzrel{xyz}$-isomorphic. This is straightforward in many cases, but some require deep insights from symbolic dynamics. We collect information along the way with the ultimate purpose of showing that all our isomorphism notions differ. We will prove this at the end of this part, when all the necessary tools and examples are in place.

\section{Three basic classes}


In this section we introduce three  classes of graphs and completely analyze their \xyz{xyz}-isomorphism relations using moves.

\begin{examp}[The graph $\GF(n_1, \dots, n_k)$]\label{graph-F}

For any $k\geq 0$ and any $n_1,\dots,n_k>0$, we consider the graph $\GF(n_1,\dots,n_k)$
\begin{equation}\label{GFgraph}\index{F@$\GF(\nn)$}
\vcenter{\xymatrix{\bullet\ar[d]^{n_1-1}&\bullet\ar[d]^{n_2-1}&&\bullet\ar[d]^{n_{k-1}-1}&\bullet\ar[d]^{n_k-1}\\
\circ&\bullet\ar[l]&\cdots\ar[l]&\bullet\ar[l]&\bullet\ar[l]&\bullet\ar[l]}}
\end{equation}
to be interpreted as whenever some $n_i=1$, the corresponding source is removed. The notation has been chosen so that $k$ is the length of the longest path in the graph, and the $n_i$ indicate the number of paths of length $i$ ending at the sink. The graph $\GF()$ in the case $k=0$ is simply a sink which is also a source.  
\end{examp}

\begin{examp}[The graph $\GC( n_1, \dots, n_k)$] \label{graph-E}
For any $k\geq 1$ and any $n_1,\dots,n_k\geq 0$, we consider $\GC(n_1,\dots,n_k)$ given as\index{E@$\GC(\nn)$}
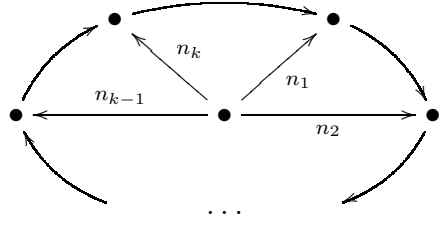
\[
\xymatrix{&\bullet\ar@/^/[rr]&&\bullet\ar@/^/[dr]&\\
\bullet\ar@/^/[ur]&&\bullet\ar[rr]_{n_2}\ar[ur]_{n_1}\ar[ll]_-{n_{k-1}}\ar[ul]_{n_k}&&\bullet\ar@/^/[dl]\\
&\ar@/^/[ul]&\dots&&}
\]
that is, $\GC(n_1, \ldots, n_k)$ is the graph consisting of a single cycle of length $k$ and a source for which the source emits finitely many edges to each vertex of the cycle.  Note that this time $n_i= 0$ is allowed, but if $\sum_{i=1}^kn_i=0$, then the source is removed.
\end{examp}

For notational convenience, for $k \geq 1$, we denote the ordered list $n_1, n_2, \ldots, n_k$ by $\nn$.  So, $\GF(n_1,\dots,n_k)$ and $\GC(n_1,\dots,n_k)$ becomes $\GF(\nn)$ and $\GC(\nn)$ respectively.

\begin{examp}[The graph $\GO(c,n)$] \label{graph-G}
For $c\geq 2$ and $n\geq 0$, $\GO(c,n)$ will be the graph\index{G@$\GO(\nn)$}
\[
\xymatrix{\bullet\ar@(ld,lu)[]^-{c}&\bullet\ar[l]^-{n}}
\]
(i.e. with adjacency matrix
$\left[\begin{smallmatrix}
0&n\\0&c\end{smallmatrix}\right]$), again with no source when $n=0$.  
\end{examp}

We first attend to the $C^*$-algebras given by the graphs $\GF(\nn)$.  These graphs give finite dimensional $C^*$-algebras.   In fact, $C^*(\GF(\nn))\cong M_d(\CC)$ with $d=1+\sum_{i=1}^k n_i$ by \cite[Corollary~2.3]{akdpir:ckadg}. 

\begin{propo}\label{finitecomplete} 
Within the class of graphs $\{\GF(\nn)\}$, $\xyzrel{xy0}=\xyzrel{xy1}$ for all $\xyz{x},\xyz{y}\in\{\xyz{0},\xyz{1}\}$, and 
\begin{eqnarray*}
\xyzrel{00z}&=&\langle\OOO,\SSS\rangle\\
\xyzrel{01z}&=&\langle\OOO,\IIIm\rangle\ =\ \langle\KKKm\rangle\\
\xyzrel{10z}&=&\langle\OOO,\RRRp\rangle\\
\xyzrel{11z}&=&\langle\rangle\ =\ \langle\KKKp\rangle.
\end{eqnarray*}
Complete invariants for the eight types of isomorphisms are given by the quantities
\begin{center}\begin{tabular}{|c|c|c|c|c|c|c|c|}\hline\small\xyz{000}&\small\xyz{001}&\small\xyz{010}&\small\xyz{011}&\small\xyz{100}&\small\xyz{101}&\small\xyz{110}&\small\xyz{111}\\\hline
\multicolumn{2}{|c}{$\emptyset$}&\multicolumn{2}{|c}{$k$}&\multicolumn{2}{|c}{$\displaystyle\sum_{i=1}^k n_i$}&\multicolumn{2}{|c|}{$\nn$}\\\hline
\end{tabular}
\end{center}
for each such graph  $\GF(\nn)=\GF(n_1,\dots,n_k)$.
\end{propo}

\begin{proof}
We first note that
\[
(\GF(\nn),\GF())\in\langle \OOO,\SSS\rangle\subseteq \xyzrel{001}
\]
by first outsplitting all sources fully and then removing them one at a time from furthest away from the sink with \SSS\ moves. This shows all claims for $\xyz{00z}$. For $\xyz{10z}$, we similarly note that
\[
\left(\GF(\nn),\GF\left(\sum_{i=1}^kn_i\right)\right)\in\langle \OOO,\RRRp\rangle\subseteq \xyzrel{101}
\]
by first performing an \RRRp\ move at all regular vertices that are not sources,  and then collecting the sources into one by an \OOO\ move. This reduces the study of \xyz{101}-equivalence to deciding  when $(\GF(m),\GF(n))\in\xyzrel{101}$, and since the class of the unit in $K_0(C^*(\GF(n)))$ under the identification $K_0(C^*(\GF(n)))\cong \ZZ$ is exactly $n+1$,  this requires that $m=n$ which is guaranteed already by \xyz{100}-equivalence. The case $\GF()$ is not covered by this analysis, but since the unit in  $K_0(C^*(\GF()))$ under the identification $K_0(C^*(\GF()))\cong \ZZ$ is $1$ we see that this fits as well.

For the remaining claims, we compute the dimension quadruples $\DQ(\GF(\nn))$ and $\DQp(\GF(\nn))$ using Theorem \ref{DQmaster}. 
Computing from the graph where all sources are combined into one by \OOO\ moves (the so-called antenna form, cf. the discussion preceding Corollary \ref{reducetoGF}), the adjacency matrix is subdivided 
\[
\left[\begin{array}{c|c}\adjRR_E&\adjRS_E\\\hline \adjSR_E&\adjSS_E\end{array}\right]=
\left[\begin{array}{ccccc|c}
0&n_k&n_{k-1}-1&\cdots&n_{2}-1&n_{1}-1\\
&&1&&&\\
&&&\ddots&&\\
&&&&1&\\
&&&&&1\\\hline
&&&&&0
\end{array}\right]
\]
Since $\adjRR_E$ is nilpotent, $\Delta_{\adjRR_E}=0$, so the dimension group is $\nsX\ZZ$, which is ordered canonically because $\singgen=\{\eee_0\}$ where $\eee_0=\oXe 01$. We see that
\[
[p_E^0]=\oXe{0}{1,n_1,\dots,n_k}
\]
because $\uu_\regX (\adjRR_E)^i$ vanishes for all $i\geq k$ and is $\begin{bmatrix}0&\cdots&0&n_k&\cdots& n_{i+1}\end{bmatrix}$ when $i<k$. Consequently
\[
I([p_E^0])=\{\www\mid w_i\not=0\Longrightarrow w_i>0\text{ and } 0\leq i\leq k\}
\]
 Note that the elements 
 \[
 \eee_i:=\oXe{i}{1}
 \]
 are exactly the   set of minimally positive elements, and thus any self-isomorphism $h$ on the dimension triple must have $h(\ee_0)=\ee_s$ for some $s$, which implies that $h=\lt^s$. To preserve either of the dimension quadruples, we must have $s=0$ and consequently $h$ must be the identity. This shows that $(\GF(\nn),\GF(\mm))\in \xyzrel{110}$ only when $\nn=\mm$, and all claims about \xyz{11z}-equivalence follow immediately. 

For \xyz{01z}-equivalence, we note from the computation above that $\GF(\nn)$ and $\GF(\mm)$ with $\mm=(m_1,\dots,m_\ell)$ can only be \xyz{010}-equivalent when $k=\ell$. In the other direction, it is easy to obtain any $\GF(\nn)$ from $$\GF(\overbrace{1,\dots,1}^k)$$ by first applying \IIIm\ moves with $n_i-1$ empty sets at  each of the $k$ regular vertices, and then collecting the resulting sources with reverse \OOO\ moves. When we define the ad hoc relation that $\GF(\nn)\equiv \GF(\mm)$ precisely when $k=\ell$, these  observations show
\[
\xyzrel{011}\subseteq \xyzrel{010}\subseteq\langle\KKKm\rangle\subseteq\  \equiv\ \subseteq \langle\OOO,\IIIm\rangle\subseteq \xyzrel{011}
\]
in combination with Proposition \ref{KKpremembers} and
 Theorems \ref{thm:outsplitting} and \ref{thm:insplitting2}.
 \end{proof}

The computation of the dimension quadruples above can be  easily visualized using the skew product graph, $\GF(\nn)\times_1\ZZ$, in this case. Illustrated with $k=3$ it becomes
\[
\xymatrix{
&\bullet\ar[dr]&\color{red}{\bullet}\ar@[red][dr]&\color{red}{\bullet}\ar@[red][dr]&\bullet\ar[dr]&\bullet\ar[dr]&\bullet\\
&\bullet\ar[dr]&\color{red}{\bullet}\ar@[red][dr]&\color{red}{\bullet}\ar@[red][dr]&\color{red}{\bullet}\ar@[red][dr]&\bullet\ar[dr]&\bullet\\
\cdots&\circ&\color{red}{\circ}&\color{red}{\circ}&\color{red}{\circ}&\color{red}{\circ}&\circ&\cdots&\\
&\bullet\ar_(0.35){n_1-1}[ur]\ar_(0.75){n_2-1}[uur]\ar^(0.9){n_3}[uuur]&\color{red}{\bullet}\ar@[red]_(0.35){n_1-1}[ur]\ar@[red]_(0.75){n_2-1}[uur]\ar@[red]^(0.9){n_3}[uuur]&\bullet\ar_(0.35){n_1-1}[ur]\ar_(0.75){n_2-1}[uur]\ar^(0.9){n_3}[uuur]&\bullet\ar_(0.35){n_1-1}[ur]\ar_(0.75){n_2-1}[uur]\ar^(0.9){n_3}[uuur]&\bullet\ar_(0.35){n_1-1}[ur]\ar_(0.75){n_2-1}[uur]\ar^(0.9){n_3}[uuur]&\bullet}
\]
With $w$ denoting the sink of $\GF(\nn)$, the elements $\ee_i$ with a single non-zero entry being $1$ at index $i$ are exactly given by $[p_{(w,i)}]$ in this picture. We note that computing  the $K_0$-classes of the vertex projections at index 0 (indicated in red above) comes out to counting the number of paths starting at level 0 and ending at sinks.

We emphasize that the reliance of our \KKKm\ move on the dimension quadruple $\DQ(-)$ is necessary to distinguish $\GF(2)$ from $\GF(1,1)$, as the dimension triples agree for these two graphs. We have also obtained:

\begin{propo}\label{OIIniIOO} $\xyzrel{011}\not\subseteq \xyzrel{100}$, $\langle\SSS\rangle \not\subseteq \xyzrel{100}$, and $\langle\IIIm\rangle \not\subseteq \xyzrel{100}$.
\end{propo}
\begin{proof} Proposition \ref{finitecomplete} shows that $(\GF(1),\GF(2))\in \xyzrel{011}\backslash \xyzrel{100}$, and it is easy to obtain $\GF(2)$ from $\GF(1)$ by either an \IIIm\ or an \SSS\ move.
\end{proof}

\begin{propo}\label{IOIniOIO} $\xyzrel{101}\not\subseteq \xyzrel{010}$, $\langle\SSS\rangle \not\subseteq \xyzrel{010}$ and $\langle \RRRp\rangle \not\subseteq \xyzrel{010}$.
\end{propo}
\begin{proof} Proposition \ref{finitecomplete} shows that $(\GF(1,1),\GF(2))\in \xyzrel{101}\backslash \xyzrel{010}$, and it is easy to obtain  $\GF(2)$ from $\GF(1,1)$ by an \RRRp\ move. We also have that $(\GF(),\GF(1))\in \langle \SSS\rangle\backslash \xyzrel{010}$.
\end{proof}

The third question of this type we must leave open in this paper; we do not know if $\xyzrel{110}$ is contained in $\xyzrel{001}$. We in fact conjecture it to be true in general, cf.~Conjecture \ref{tradein} below.

We now turn to the $\GC(\nn)$ graphs.  As we will see in the proof of Proposition \ref{cyclecomplete}, $\GC(\nn)$ defines a $C^*$-algebra that is isomorphic to $M_d(C(\TT))$, where $d = k + \sum_{ i = 1}^k n_i$.


\begin{propo}\label{cyclecomplete} 
Within the class of graphs $\{\GC(\nn)\}$, $\xyzrel{xy0}=\xyzrel{xy1}$ for all $\xyz{x},\xyz{y}\in\{\xyz{0},\xyz{1}\}$, and 
\begin{eqnarray*}
\xyzrel{00z}&=&\langle\OOO,\RRRp,\SSS\rangle\\
\xyzrel{01z}&=&\langle\OOO,\IIIm\rangle\ =\ \langle\KKKm\rangle\\
\xyzrel{10z}&=&\langle\OOO,\RRRp\rangle\\
\xyzrel{11z}&=&\langle\rangle\ =\ \langle\KKKp\rangle.
\end{eqnarray*}
Complete invariants for the eight types of isomorphisms are given by the quantities
\begin{center}\begin{tabular}{|c|c|c|c|c|c|c|c|}\hline\small\xyz{000}&\small\xyz{001}&\small\xyz{010}&\small\xyz{011}&\small\makebox[8mm][c]{\xyz{100}}&\small\xyz{101}&\small\makebox[18mm][c]{\xyz{110}}&\small\xyz{111}\\\hline
\multicolumn{2}{|c}{$\emptyset$}&\multicolumn{2}{|c}{$k$}&\multicolumn{2}{|c}{$\displaystyle k+\sum_{i=1}^k n_i$}&\multicolumn{2}{|c|}{$\{\sigma_k^j(\nn)\mid j\in\{1,\dots,k\}\} $}\\\hline
\end{tabular}
\end{center}
for each such graph  $\GC(\nn)=\GC(n_1,\dots,n_k)$, where $\sigma_k$ denotes the cyclic shift on $k$ elements to the right.\index{sk@$\sigma_k$}
\end{propo}

  \begin{proof}
  We first note that
  \[
  \left(\GC(\nn),\GC\left(\sum_{i=1}^kn_i+(k-1)\right)\right)\in \langle \OOO,\RRRp\rangle\subseteq \xyzrel{101}
  \]
  by  first using \RRRp\ $k-1$ times to reduce the cycle length to one, and collecting all sources by \OOO. The \RRRp\ moves  produce the $k-1$ additional sources indicated.  This reduces the study of \xyz{101}-equivalence to deciding  when $(\GC(m),\GC(n))\in\xyzrel{101}$.  We claim that $\GC(n) \cong M_{n+1}(C(\TT))$ which implies that $C^*(\GC(\nn)) \cong M_d(C(\TT))$, where $d = k + \sum_{i=1}^k n_i$.  Let $F_n$ be the graph
  \[
  \xymatrix{\bullet\ar@(ld,lu)[]&\bullet\ar[l]&\bullet\ar[l]&\cdots\ar[l]&\bullet\ar[l]
  }
  \] 
    obtained by attaching a strand of length $n$ to a single vertex supporting a loop.  Using the \RRRp\ move to reduce the length of the strand in $F_n$, we see that $C^*(\GC(n)) \cong C^*(F)$.  By \cite[Proposition~9.3]{gamt:imega}, $M_{n+1}(C(\TT))$ is isomorphic to $C^*(F_n)$.  Thus, $C^*(\GC(n)) \cong M_{m+1}(C(\TT))$.  We may now conclude that the class of the unit in $K_0(C^*(\GC(n)))$ under the identification $K_0(C^*(\GC(n))) \cong \ZZ$ is exactly $n+1$.  Therefore, $(\GC(m),\GC(n))\in\xyzrel{101}$ if and only if $m=n$ which is also guaranteed by $\xyz{100}$-equivalence. Continuing on with \SSS\ moves, we obtain
 \[
(\GC(\nn),\GC(0))\in \langle \OOO,\RRRp,\SSS\rangle\subseteq \xyzrel{001}
  \]
  which shows all claims in the \xyzrel{00z} cases as well.
  
The graph is regular with adjacency matrix
\[
A=\left[\begin{array}{ccccc}
0&n_k&n_{k-1}&\cdots&n_{1}\\
&&1&&\\
&&&\ddots&\\
&&&&1\\
&1&&&
\end{array}\right],
\]
so we can employ Theorem \ref{regularDQ}. The eventual kernel and range agrees with the kernel and range of $A$ itself, and we choose
\[
\begin{bmatrix}1&-n_k&\cdots&-n_1\end{bmatrix}
\]
as a basis for the kernel, and the standard basis vectors $\ee_2,\dots,\ee_{k+1}$ as a basis for the range.  $A^{-1}$ acts as a cyclic shift on the eventual range, and $\Delta_A=\ZZ^k$ ordered canonically. To compute $\pK{\uu}$ we write
\[
\begin{bmatrix}1&&\cdots&1\end{bmatrix}=\begin{bmatrix}0&n_k+1&\cdots&n_1+1\end{bmatrix}+\begin{bmatrix}1&-n_k&\cdots&-n_1\end{bmatrix}
\]

 Consequently, $(\GC(\nn),\GC(\mm))\in \langle\KKKp\rangle$ only when $\mm=\sigma_k^j(\nn)$ for some $j$ and $(\GC(\nn),\GC(\mm))\in \langle\KKKm\rangle$ precisely when the cycles given have the same length.  The argument is completed in a similar way as in the proof of Proposition \ref{finitecomplete}.
 \end{proof}

 We now move on to the graphs $\GO(c,n)$.  To answer the desired questions for these graphs, we start with a number of lemmas.
 
  \begin{lemma}\label{GOmoves} Let $c \geq 2$ and $n \geq 0$. Then
\begin{enumerate}[(1)]
\item $(\GO(c,n),\GO(c,0))\in\langle \OOO,\IIIm\rangle$
\item $(\GO(c,n),\GO(c,c(n+c-1))) \in\langle \OOO,\IIIp\rangle$
\end{enumerate}
\end{lemma}
  \begin{proof}
  We obtain the first claim by
  \[
\xymatrix@R=2mm@C=3mm{
&&&&&&&&&&&&&\bullet\ar@/_/[dddl]\ar@/^/[dddr]&\\
\bullet\ar[dd]_-{n}&&&&&\bullet\ar@/_/[ddl]_-{n}\ar@/^/[ddr]^-{n}&&&&&&&&\bullet\ar@/_/[ddl]\ar@/^/[ddr]\\
&\ar@{~>}^{\OOO}[rr]&&&&&&\ar@{~>}^{\OOO}[rr]&&&&&&\vdots&&&&&\\
\bullet\ar@(ld,rd)[]_-{c}&&&&\bullet\ar@(dl,ul)[]\ar@/^/[rr]&&\bullet\ar@(dr,ur)[]_-{c-1}\ar@/^/[ll]^-{c-1}&&&&&&\bullet\ar@(dl,ul)[]\ar@/^/[rr]&&\bullet\ar@(dr,ur)[]_-{c-1}\ar@/^/[ll]^-{c-1}&&&&&\\
&&&&&&&&&&&&&&\\
&&&&&&&&&&&&&&\\
&&&&&&&&&&&&&\ar@{~>}[uu]_-{\IIIm}\\\
&&&&&&&\bullet\ar@(ld,rd)[]_-{c}&\ar@{~>}[rr]_{\OOO}&&&&\bullet\ar@(dl,ul)[]\ar@/^/[rr]&&\bullet\ar@(dr,ur)[]_-{c-1}\ar@/^/[ll]^-{c-1}\
}
\]
with the \IIIm\ move applied at the vertex with only one loop, and $n$ empty sets in the partition.

For the next claim, we note first that the \OOO\ move splitting all loops into $c$ singletons gives a graph with $c$ instances of vertices receiving $n$ edges from a source, supporting a loop and receiving uniquely from all other vertices of this form. For $c=3$ this looks like 
 
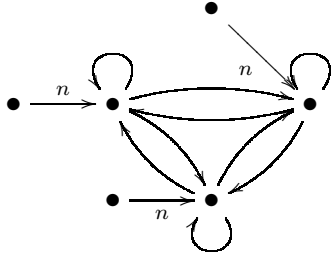
\[
 \xymatrix{&&\\\bullet\ar[r]^-{n}&\bullet\ar@(ur,ul)[]\ar@/^/[dr]\ar@/^/[rr]&&\bullet\ar@(ur,ul)[]\ar@/^/[ll]\ar@/^/[dl]&\bullet\ar_-{n}[l]\\
& \bullet\ar[r]_-{n}&\bullet\ar@(dr,dl)[]\ar@/^/[ur]\ar@/^/[ul]&\\&&&}
 \]
  where we have split the sources for legibility of the graph. Since all these $c$ vertices have exactly the same future, we may perform an \IIIp\ move to create $c-1$ sources each emitting to a select vertex among the $c$, and collect all the edges from sources there, as
 
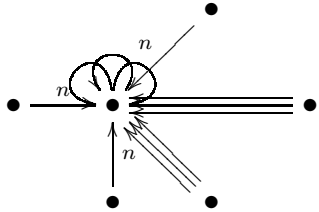
\[
 \xymatrix{&&\bullet\ar[dl]_-{n}&\\\bullet\ar[r]^-{n}&\bullet\ar@(ur,ul)[]\ar@(u,r)[]\ar@(u,l)[]&&\bullet\ar[ll]\ar@<-1mm>[ll]\ar@<1mm>[ll]&\\
& \bullet\ar[u]_-{n}&\bullet\ar[ul]\ar@<-1mm>[ul]\ar@<1mm>[ul]&\\&&&}
 \]
 for $c=3$. The total number of edges from sources is $c(c-1)$ from the other vertices arising from the outsplit, and $cn$ obtained as copies of the original $n$. This establishes (2).
 
 \end{proof}

We are now ready to analyze the $\GO(c,n)$ graphs.  As we will see in the proof of Proposition \ref{cuntzcomplete}, these graphs define $C^*$-algebras that are related to the Cuntz algebras  \cite{jc:scgi}.

 \begin{propo}\label{cuntzcomplete} 
Within the class of graphs $\{\GO(c,n)\}$, $\xyzrel{xy0}=\xyzrel{xy1}$ and $\xyzrel{00z}=\xyzrel{01z}$ for all $\xyz{x},\xyz{y},\xyz{z}\in\{\xyz{0},\xyz{1}\}$, and 
\begin{eqnarray*}
\xyzrel{0yz}&=&\langle\OOO,\SSS\rangle\ =\ \langle \KKKm\rangle\ = \ \langle\OOO,\IIIm\rangle\\
\xyzrel{10z}&=&\langle\OOO,\IIIp,\RRRp\rangle\\
\xyzrel{11z}&=&\langle\OOO,\IIIp\rangle\ =\ \langle\KKKp\rangle.
\end{eqnarray*}
Complete invariants for the eight types of isomorphisms are given by the quantities
\begin{center}\begin{tabular}{|c|c|c|c|c|c|c|c|}\hline\small\xyz{000}&\small\xyz{001}&\small\xyz{010}&\small\xyz{011}&\small\makebox[12mm][c]{\xyz{100}}&\small\xyz{101}&\small\makebox[18mm][c]{\xyz{110}}&\small\xyz{111}\\\hline
\multicolumn{4}{|c}{$c$}&\multicolumn{2}{|c}{$(c,(n+1)\ZZ_{c-1}^\times)$}&\multicolumn{2}{|c|}{$(c,\{(1+\tfrac nc)c^k\mid k\in\ZZ\})$}\\\hline
\end{tabular}
\end{center}
for each such graph  $\GO(c,n)$.
\end{propo}

We prove the proposition in two separate installments, dealing with the \xyz{0yz} and \xyz{11z} claims here, and the \xyz{10z} claims just after the proof of Theorem \ref{PSMMCEOR} below.\\

 \begin{proofoffirstpart}{Proposition \ref{cuntzcomplete}}
 We first show that $C^*(\GO(c,n))\cong M_{n+1}(\mathcal O_c)$. To see this, we use the $\OOO$ move to conclude that $C^*(\GO(c, n)) \cong C^*(E)$, where $E_{c,n}$ (depicted with $c=3$) 
 \[
 \xymatrix{
 \bullet\ar[dr]&\cdots&\bullet\ar[dl]\\
 &\bullet\ar@(l,d)[]\ar@(lu,ld)[]\ar@(d,r)[] \\&}
 \]
  is the graph obtained from $\GO(c,0)$ by attaching $n$ sources to the single vertex.  By \cite[Proposition~9.3]{gamt:imega}, $M_{n+1}(\mathcal O_c)$ is isomorphic to $C^*(F_{c,n})$, where $F_{c,n}$ 
  \[
  \xymatrix{
   \bullet\ar@(l,d)[]\ar@(lu,ld)[]\ar@(d,r)[] &\bullet\ar[l]&\cdots\ar[l]&\bullet\ar[l]\\&}
  \]
  is the graph obtained from $\GO(c,0)$ by attaching a head of length $n$ to the single vertex.  Using the \RRRp\ move to reduce the length of the strand of $F_{c,n}$, we see that $C^*(E_{c,n}) \cong C^*(F_{c,n})$, concluding the proof that $C^*(\GO(c, n)) \cong M_{n+1}(\mathcal O_c)$.

To establish the claims for \xyz{0yz}-equivalence, we consider
 \[
\langle\OOO,\IIIm\rangle\subseteq \xyzrel{011}\subseteq \langle\KKKm\rangle\subseteq\xyzrel{000}\subseteq\ \equiv\ \subseteq\langle\OOO,\IIIm\rangle
\]
with $\equiv$ defined by the ad hoc relation that $\GO(c,n)\equiv \GO(d,m)$ precisely when $c=d$. The three first inclusions hold in general by Theorems \ref{thm:outsplitting} and \ref{thm:insplitting2} and by Propositions  \ref{KKpremembers} and \ref{Kpartialinva}. Since $K_0(C^*(\GO(c,n))) \cong \ZZ_{c-1}$ we get that $\xyzrel{000}$ is contained in the ad hoc relation. The last inclusion follows from Lemma \ref{GOmoves}(1). We also get
\[
\equiv\ \subseteq\langle\OOO,\SSS\rangle\subseteq \xyzrel{000}
\]
by Theorems \ref{thm:outsplitting} and \ref{thm:source}, which establishes all the non-unital claims.

For the \xyzrel{11z} claims, we similarly consider
 \[
\langle\OOO,\IIIp\rangle\subseteq \xyzrel{111}\subseteq \xyzrel{110}\subseteq \langle\KKKp\rangle\subseteq\ \equiv\ \subseteq\langle\OOO,\IIIp\rangle
\]
where this time $\equiv$ is given by letting $\GO(c,n)\equiv \GO(d,m)$ precisely when $c=d$ and
\[
\{(1+\tfrac nc)c^k\mid k\in\ZZ\}=\{(1+\tfrac mc)c^k\mid k\in\ZZ\},
\]
which is equivalent to the statement
\[
\frac{c+n}{c+m}=c^k
\]
for some $k\in\ZZ$. Again the first three inclusions are general facts from Theorems \ref{thm:outsplitting}, \ref{thm:insplittingplus}, and \ref{Kpartialinva}.  To see that \KKKp\ is contained in the ad hoc relation, note that the graph is regular,
so we can employ Theorem \ref{regularDQ}. The eventual kernel and range agrees with the kernel and range of $A$ itself, and we choose
\[
\begin{bmatrix}c&-n\end{bmatrix}, \begin{bmatrix}0&1\end{bmatrix}
\]
as a basis for the kernel and the range, respectively.  $A$ acts as multiplication by $c$ on the eventual range, and $\Delta_A=\ZZ[\frac1c]$ ordered canonically and with $\delta_A=\times\frac1c$. To compute $\pK{\uu}$ we write
\[
\begin{bmatrix}1&1\end{bmatrix}=(1+\tfrac nc)\begin{bmatrix}0&1\end{bmatrix}+\tfrac 1c\begin{bmatrix}c&-n\end{bmatrix}
\]
As necessitated by our analysis of the \xyz{0yz}\ cases, $\DQ(\GO(c,n))$ does not depend on $n$, but $\DQp(\GO(c,n))$ contains the class $1+\tfrac nc$ which does. The only automorphisms of $\DT(\GO(c,n))$ are the actions $x\mapsto c^kx$, so $(\GO(c,n),\GO(d,m))\in\langle\KKKp\rangle$ shows that $(c,n)\equiv (d,m)$ as stipulated.

To show that $(\GO(c,n),\GO(d,m))\in\langle\OOO,\IIIp\rangle$ when $(c,n)\equiv (d,m)$, it is enough to show that for all $k \in \ZZ$, for all nonnegative integers $m, n$, if  $(c+m)/(c+n)=c^k$, then $(\GO(c,n),\GO(c,m))\in\langle\OOO,\IIIp\rangle$.  First note that  when $(c+m)/(c+n)=c$, we have that $m=c(c+n)-c=c(c+n-1)$, and the claim follows by Lemma \ref{GOmoves}(2).  Let $k \geq 1$.  Assume that for all nonnegative integers $m,n$, if $(c+m)/(c+n)=c^k$, then $(\GO(c,n),\GO(c,m))\in\langle\OOO,\IIIp\rangle$.  Assume $(c+m)/(c+n)=c^{k+1}$.  Then 
 \[
\frac{c+m}{ c + c(c+n-1) } = \frac{c+m}{c^2+cn} = c^k
 \]
 which implies $(\GO(c,m),\GO(c, c(c+n-1)))$ is an element of $\langle\OOO,\IIIp\rangle$.  By Lemma \ref{GOmoves}(2), $(\GO(c,n),\GO(c, c(c+n-1)))$ is an element of $\langle\OOO,\IIIp\rangle$.  Hence, $(\GO(c,n),\GO(c,m))$ is an element of $\langle\OOO,\IIIp\rangle$.  Consequently, for all $k \geq 1$ and for all nonnegative integers $m, n$, if $(c+m)/(c+n)=c^k$, then $(\GO(c,n),\GO(c,m))\in\langle\OOO,\IIIp\rangle$.  If $(c+m)/(c+n)=c^{-k}$, then $(c+n)/(c+m)=c^k$, then the previous case can be applied to get $(\GO(c,n),\GO(c,m))\in\langle\OOO,\IIIp\rangle$.
\end{proofoffirstpart}

 \section{$\SL$-equivalence and augmented standard forms}\label{GLandfriends}
 
 In preparation for the completion of our analysis of the graphs $\GO(c,n)$, and for many other applications to come, we now repeat results and notation from 
 \cite{segrerapws:ccuggs}  and \cite{seaseer:gciugc} which provide useful variations of the fundamental notions of standard form and $\GL$-equivalence as summarized already in Section \ref{classummary}. 
 
We specialize the notion of $\GL$-equivalence (resp. $\GLp$-equivalence)  to say that two $\GL$-equivalent  and thus necessarily matched graphs are $\SL$-equivalent  (resp. $\SLp$-equivalent) when the matrices $U,V$ in \eqref{GLdef}
may be chosen such that all diagonal blocks have determinant  $1$. And we generalize the notion of canonical form to allow a source by replacing (vii) by\index{SL-equivalence@$\SL$-equivalence}\index{SLp-equivalence@$\SLp$-equivalence}
 \begin{enumerate}[(i')]\addtocounter{enumi}{6}
 \item Every regular vertex except at most one supports a loop, and if one regular vertex supports no loop, it must be a source.
 \end{enumerate}
 We say that a graph is in \emph{augmented canonical form}\index{augmented!canonical form} when it satisfies (i)-(vi) from Section \ref{classummary} along with (vii'), and say that a matched pair of graphs in augmented canonical form is in \emph{augmented standard form}.\index{augmented!standard form}

 We first note that (augmented) standard forms may be obtained, in an algorithmic way, for all matched graphs using \xyz{x01}-invariant moves. Combined with Proposition \ref{getmatched}, this shows that we can always reduce to these cases when checking for \xyz{x0z}-invariance:  
 
 \begin{lemma}\mbox{}\label{getstd}
 \begin{enumerate}[(i)]
 \item For any matched pair of graphs $(E,F)$ there is an algorithm that produces a pair of graphs $(E',F')$ in standard form so that 
 \[
 (E,E'),(F,F')\in \langle \OOO,\IIIm,\RRRp,\SSS\rangle \subseteq \xyzrel{001}
 \]
 \item  For any matched pair of graphs $(E,F)$ there is an algorithm that produces a pair of graphs $(E',F')$ in augmented standard form so that 
 \[
 (E,E'),(F,F')\in \langle \OOO,\IIIp,\RRRp\rangle \subseteq \xyzrel{101}
 \] 
\end{enumerate}
 \end{lemma}
 \begin{proof} Statement (ii) is in \cite{seaseer:gciugc}, and removing the source with \OOO\ and \SSS\ moves to obtain (i) from this is straightforward.
 \end{proof}
 
 This result shows that we may reduce to the case of pairs in standard form at little expense. It is very easy to give examples of graphs, for instance $\GF(1)$, which cannot be placed in canonical form using $\xyz{x10}$-invariant moves, and it is similarly impossible to arrange for canonical forms using $\xyz{101}$-invariant moves, as the example $\GC((0,0))$ demonstrates. Before moving on, we state a reformulation of key results from  \cite{segrerapws:gcgcfg}, \cite{segrerapws:ccuggs}, and \cite{seaseer:gciugc} in this direction.
 
For compact reference, we use the notation $\standard$\index{std@$\standard$} for the set of pairs $(E,F)$ that are in standard form, and $\standardp$\index{std+@$\standardp$} for the set of pairs  in augmented standard form. Note that this is not a reflexive relation.

 \begin{theor}\label{masterGLSL}\mbox{}
\begin{enumerate}[(i)]
\item $\GL\subseteq  \langle \OOO,\IIIm,\RRRp,\SSS,\CCCp,\PPPp\rangle\subseteq \xyzrel{000}$, and 
\[
\standard\cap \xyzrel{000}=\standard\cap \langle \OOO,\IIIm,\RRRp,\SSS,\CCCp,\PPPp\rangle=\standard\cap\GL
\]
\item $\SL\subseteq \langle \OOO,\IIIm,\RRRp,\SSS\rangle\subseteq \xyzrel{001}$.
\item $\GLp\subseteq \langle \OOO,\IIIp,\RRRp,\CCCp,\PPPp\rangle\subseteq \xyzrel{100}$, and 
\[
\standardp\cap \xyzrel{100}=\standardp\cap  \langle \OOO,\IIIp,\RRRp,\CCCp,\PPPp\rangle=\standardp\cap\GLp
\]
\item $\SLp\subseteq  \langle \OOO,\IIIp,\RRRp\rangle\subseteq \xyzrel{101}$.
\end{enumerate}
Within the class of finite graphs, we further have
\begin{enumerate}[(i')]\addtocounter{enumi}{1}
\item
$
\standard\cap \langle \OOO,\IIIm,\RRRp,\SSS\rangle=\standard\cap\SL
$
\addtocounter{enumi}{1}
\item
$
\standardp\cap \langle \OOO,\IIIp,\RRRp\rangle=\standardp\cap\SLp
$
\end{enumerate}

\end{theor}

 \section{Classical symbolic dynamics}\label{twosided}
 
 In this section we summarize results rooted in  symbolic dynamics that are foundational for our work, and present them in a form which serves our purposes. The monograph \cite{dlbm:isdc} comes highly recommended as a resource for this topic, and we will provide precise references for the several notions we will, in the interest of brevity, not define explicitly here. 
 
 When a graph $E$ is finite, the two-sided shift space
 \[
 \XX E=\{(e_n)\in(E_1)^\ZZ\mid r(e_n)=s(e_{n+1})\}
 \]
 becomes a \emph{shift of finite type},\index{shift of finite type} thoroughly studied in symbolic dynamics as a key example of a compact space with a homeomorphism, which is given by the shift map\index{shift map}
 \[
 \sigma_E((e_n)_{n\in\ZZ})=(e_{n+1})_{n\in\ZZ}
 \]
as laid out in \cite[Definition 2.2.5 and Proposition 2.2.6]{dlbm:isdc}.  We note that any edge not contained in an infinite path does not contribute to $\XX E$, and indeed one often restricts attention in symbolic dynamics to graphs that are \emph{essential},\index{essential!graph} containing neither sinks nor sources. We use the term \emph{essential part}\index{essential!part} to denote the subgraph obtained by deleting all edges and vertices that are not on infinite paths. This can be obtained by deleting all sinks and sources successively, until no such vertices are left.
 
 We will need the notions of primitive and irreducible graphs:
 
 \begin{defin}[{\cite[Definitions 2.2.13 and 4.5.7, Theorem 4.5.8]{dlbm:isdc}}] \index{primitive graph}\index{irreducible graph} \label{dfn-irreducible}
 An essential graph $E$ is \emph{primitive} if there is an $N>0$ so that there is a path of length $N$ between any two (not necessarily distinct) vertices in $E^0$. It is  \emph{irreducible} when there is a path of positive length between any pair of (not necessarily distinct) vertices.
 \end{defin}
 
 Graphs of the form $\GC(\underline{0})$ are always irreducible, but only primitive when there is a single vertex. The difference may be described in terms of periods as in \cite[Definition 4.5.2]{dlbm:isdc}.

 We note that the \OOO\ move preserves essentiality, and indeed it was invented by Williams as a way to change the underlying essential graph $E$ into another graph $E_O$ in such a way that the shift spaces, $(\XX E,\sigma_E)$ and $(\XX {E_O},\sigma_{E_O})$, are related in a controlled manner. Our \IIIm,\IIIp, and \RRRp\ moves are not consistent with essentiality, so according to tradition in symbolic dynamics we will adjust them to moves \III\ and \RRR\ as follows. We also introduce \CCC,\PPP, and \KKK\ moves in line with our remaining moves; these are not standard in symbolic dynamics.
 
 \begin{defin}
 For a pair $(E,F)$ of regular and essential graphs, we say that 
 \begin{itemize}
 \item $F$ is obtained from $E$ by an \III\ move if  $F$ is obtained from $E$ by an \IIIm\ moves with no empty sets in the partition;
 \item $F$ is obtained from $E$ by an \RRR, \CCC\ or \PPP\ move, respectively, if  $F$ is obtained from $E$ by first performing an \RRRp, \CCCp, or \PPPp\ move, respectively, and then passing to the essential part;
  \item $F$ is obtained from $E$ by a \KKK\ move if  $\DT(E)\simeq\DT(F)$.
 \end{itemize}
  \end{defin}
 
 It would be possible, but somewhat circuitous, to define the \III\ move in parallel with the \RRR, \CCC\ and \PPP\ moves -- allowing empty sets in the partitions, and then deleting the resulting sources. Defining it as a special case of \IIIm\ has the advantage of leading directly to the first observation below.
 
 \begin{lemma}\label{essinv} For regular and essential graphs, we have
 \begin{eqnarray*}
 \langle\III\rangle&\subseteq&\langle\IIIm\rangle\ \subseteq\ \xyzrel{011}\\
 \langle\RRR\rangle&\subseteq&\langle\RRRp,\OOO,\SSS\rangle\ \subseteq\ \xyzrel{001}\\ 
 \langle\CCC\rangle&\subseteq&\langle\CCCp,\OOO,\SSS\rangle\ \subseteq\ \xyzrel{000}\\ 
 \langle\PPP\rangle&\subseteq&\langle\PPPp,\OOO,\SSS\rangle\ \subseteq\xyzrel{000}\\ 
 \langle\KKK\rangle&\subseteq&\xyzrel{000}
 \end{eqnarray*}
  \end{lemma}  
 
 \begin{proof}
 For the $\RRR,\CCC,$ and $\PPP$ moves we note that since there are no sinks, passing to the essential part can be obtained by $\OOO$ and $\SSS$ moves. They are both $\xyz{001}$-invariant, passing to the invariance as noted. The claim for \KKK\ follows from Proposition \ref{Kpartialinva} by noting that for regular graphs $E$ and $F$, $\DQ(E) \simeq \DQ(F)$ if and only if $\DT(E)\simeq\DT(F)$ as a consequence of Theorem \ref{regularDQ}.
 \end{proof}
 
 The following result is a superposition of a classical result by Williams (\cite{rfw:csft}), showing the equivalence of (i)--(iii), with the more recent characterization of \xyz{011}-equivalence by Carlsen and Rout (\cite{tmcjr:dgigc}) which complement the invariance observations that $\langle \OOO,\III\rangle\subseteq \xyzrel{011}$ by showing in fact that these relations agree. Conjugacy\index{conjugacy} between two shift spaces simply comes out to allowing a shift-commuting homeomorphism, and strong shift equivalence\index{strong shift equivalence} is a matrix relation introduced by Williams as fully explained in \cite[Section 7.2]{dlbm:isdc}.
 
 \begin{theor}[\cite{rfw:csft},\cite{tmcjr:dgigc}]\label{WCR}
  For a pair $(E,F)$ of regular and essential graphs, the following are equivalent
  \begin{enumerate}[(i)]
\item $(E,F)\in\langle \OOO,\III\rangle$;  
  \item $\XX E$ and $\XX F$ are conjugate;
  \item $\mathsf{A}_E$ and $\mathsf{A}_F$ are strong shift equivalent;
  \item $(E,F)\in\xyzrel{011}$.
    \end{enumerate}
  \end{theor}

  It can be very difficult to decide when two graphs/shifts of finite type/matrices satisfy these equivalent conditions, and no procedure is known for doing so. There are many invariants known, however, which may sometimes be used to show that two objects do not agree in this sense. We will need to work with two:
  
  \begin{defin}\label{someinvs}\index{spectrum away from zero}\index{Bowen-Franks invariant}\index{spx@$\spaz(E)$}\index{BF@$\BF(E)$}
  When $E$ is a regular and essential graph, we define the \emph{spectrum away from zero}, $\spaz(E)$, as the set of nonzero eigenvalues of ${\mathsf A}_E$, counted with multiplicity. We define the \emph{Bowen-Franks invariant}, $\BF(E)$, as the pair
 \[
 (\cok(\idmatrix-\Asf_{E}),\det(\idmatrix-\Asf_{E}))= (\cok \Bsf_{E},\det(-\Bsf_{E}))
  \]
  \end{defin}

The next key result also superposes a classical notion from Williams with a more recent $C^*$-algebraic result. \emph{Shift equivalence}\index{shift equivalence} is a relation on matrices defined by Williams as a candidate for a more computable invariant for conjugacy (see \cite[Definition 7.3.1]{dlbm:isdc}).  By work of Krieger, (i) and (ii) are known to be equivalent in general. The backward implication from (iii) to (i), which we have in fact proved as a special case of invariance (Proposition \ref{KKpremembers}), was also known to Krieger. 

The forward implication (i) to (iii) is open in general, but was proved by Bratteli and Kishimoto for the subclass of primitive shifts of finite types. Recent work by Szab\'o and the first named author allows the generalization to irreducible shifts by a reduction to the primitive case.

   \begin{theor}[\cite{obak:tsacsa},\cite{segs}]\label{KBKES}
  For a pair $(E,F)$ of regular and essential graphs, the following are equivalent
  \begin{enumerate}[(i)]
\item $(E,F)\in\langle \KKK\rangle$;
  \item $\mathsf{A}_E$ and $\mathsf{A}_F$ are shift equivalent;
  \end{enumerate}
  and are implied by 
  \begin{enumerate}[(i)]\addtocounter{enumi}{2}
  \item $(E,F)\in\xyzrel{010}$.
    \end{enumerate}
    When the graphs are irreducible, all statements are equivalent.
  \end{theor}
  
It is a deep result by Kim and Roush (\cite{khkfwr:dse}) that the equivalent conditions (i)--(ii) are decidable by terminating procedures. The invariants in Definition \ref{someinvs} remain invariant for this coarser relation.

The notion of \emph{flow equivalence}\index{flow equivalence} was introduced as a means of studying geodesic flows (in umbilical form already in \cite{mmgah:sd}). Using the so-called \emph{suspension flows}, it defines a coarse equivalence relation on all shift spaces. By 
a fundamental result by Parry and Sullivan, this relation is precisely $\langle\xyzrel{011},\RRR\rangle$, so we will not provide details of it here. Matsumoto and Matui  showed that in the irreducible case, this notion translates exactly to \xyz{001}-equivalence of the $C^*$-algebras. Work of the first named author with Carlsen, Ortega and Restorff generalized this to regular and essential graphs, so we have:

   \begin{theor}[\cite{bpds:tifos}, \cite{kmhm:coetmscka},\cite{ceor:coe}, \cite{mbdh:pbeim}]\label{PSMMCEOR}
  For a pair $(E,F)$ of regular and essential graphs, the following are equivalent
  \begin{enumerate}[(i)]
\item $(E,F)\in\langle \OOO,\III,\RRR\rangle$;  
  \item $\XX E$ and $\XX F$ are flow equivalent;
  \item $(E,F)\in\xyzrel{001}$.
    \end{enumerate}
and when $(E,F)$ are matched, they  are  implied by 
    \begin{enumerate}[(i)]\addtocounter{enumi}{3}
  \item $(E,F)\in\SL$.
  \end{enumerate}
When $(E,F)$ is in canonical form, all statements are equivalent.
  \end{theor}
  
  Franks proved that the Bowen-Franks invariant is in fact complete for flow equivalence for irreducible systems that have infinitely many elements. The general case is much more complicated, and was solved in \cite{mbdh:pbeim}. The notions of canonical form and \SL-equivalence are as given in our Sections \ref{classummary} and \ref{GLandfriends}, but may of course be simplified by the assumption of regularity (rendering (ii) and (iv) redundant) in this case. Our notions are generalizations of notions introduced by Boyle and Huang, and since the results of \cite{mbdh:pbeim} precede our notions by nearly two decades,  this order of presentation is an anachronism chosen again for brevity.
  
We record easy examples that show that the \CCC\ and \PPP\ moves have no nice properties in terms of symbolic dynamics, and hence should be thought of as belonging to the realm of operator algebras proper.
  
    \begin{propo}\label{psex}
  $\xyzrel{000}\not\subseteq \xyzrel{001}\cup\xyzrel{010}$,   $\langle\CCC\rangle\not\subseteq \xyzrel{001}\cup\xyzrel{010}$, and $\langle\PPP\rangle\not\subseteq \xyzrel{001}\cup\xyzrel{010}$.
  \end{propo}
  \begin{proof}
  We have
  \[\xymatrix{
E_1:&\bullet\ar@(l,d)[]\ar@(dl,dr)[]&\ar@{~>}[r]^-{\CCC}&&\bullet\ar@(l,d)[]\ar@(dl,dr)[]\ar@/^/[r]&\bullet \ar@/^/[r]\ar@(dl,dr)[]\ar@/^/[l]&\bullet\ar@/^/[l]\ar@(dl,dr)[]&:E_2\\&&&}
\]
with $E_1=\GO(2,0)$. It is seen directly that
\[
\BF(E_1)=(0,-1)\not=(0,1)=\BF(E_2)
\]
and we also have
\[
\spaz(E_1)=\{2\}\not=\spaz(E_2)
\]
where the set on the right hand side has three elements, namely the roots (all different) of $t^3 - 4t^2 + 3t + 1$. It follows that \CCC\ is neither \xyz{001}- nor \xyz{010}-invariant. By Lemma \ref{essinv}, we also have   $\xyzrel{000}\not\subseteq \xyzrel{001}\cup\xyzrel{010}$.

We argue similarly for \PPP. A minimal example is
  \[\xymatrix{&&&\\&\bullet\ar@(ur,ul)[]\ar[d]&&&\bullet\ar@(ur,ul)[]\ar[d]\ar@<-0.7mm>@/^/[drr]\ar@<0.7mm>@/^/[drr]\\
F_1:&\bullet\ar@(l,d)[]\ar@(dl,dr)[]&\ar@{~>}[r]^-{\PPP}&&\bullet\ar@(l,d)[]\ar@(dl,dr)[]\ar@/^/[r]&\bullet \ar@/^/[r]\ar@(dl,dr)[]\ar@/^/[l]&\bullet\ar@/^/[l]\ar@(dl,dr)[]&:F_2\\&&&}
\]
We have $\spaz(F_i)=\{1\}\cup \spaz(E_i)$ showing that $(F_1,F_2)\not\subseteq\xyzrel{010}$ as above, but the Bowen-Franks groups agree for $F_1$ and $F_2$. Using \SL-invariance instead easily establishes that $(F_1,F_2)\not\subseteq\xyzrel{001}$. 
  \end{proof}
  
  The last example of this section is anything but easy.
  
    \begin{examp}\label{kimroushex} (The Kim-Roush example) 
Kim and Roush provide in \cite{khkfwr:wcfis} the first example of a pair of shift equivalent primitive shifts of finite type which fail to be strong shift equivalent:
\[
\mathsf{A}:=\begin{bmatrix}0&0&1&1&3&0&0\\1&0&0&0&3&0&0\\0&1&0&0&3&0&0\\0&0&1&0&3&0&0\\0&0&0&0&0&0&1\\1&1&1&1&10&0&0\\1&1&1&1&0&1&0\end{bmatrix}\qquad
\mathsf{B}:=\begin{bmatrix}
0&0&1&1&3&0&0\\1&0&0&0&0&0&0\\0&1&0&0&0&0&0\\0&0&1&0&0&0&0\\0&0&0&0&0&0&1\\4&5&6&3&10&0&0\\4&5&6&3&0&1&0
\end{bmatrix}.
\]
\end{examp}
  
  \begin{propo}\label{kimroush}
  $\xyzrel{100}\cap\xyzrel{010}\cap\xyzrel{001}\not\subseteq \xyzrel{011}$
  \end{propo}
  \begin{proof}
  Paired with Theorems 
\ref{WCR} and \ref{KBKES}, Example   \ref{kimroushex} shows directly that $\xyzrel{010}\not\subseteq \xyzrel{011}$. Referring to the graphs with adjacency matrices $\mathsf{A}$ and $\mathsf{B}$  by $E$ and $F$ respectively, we compute
 \[
 (K_0(C^*(E)),[1]) \cong  (K_0(C^*(F)),[1]) \cong (\ZZ_{99},66)
 \]
 and
 \[
\BF(E)\cong \BF(F)\cong (\ZZ_{99},-1)
\]
 which shows that also $(E,F)\in \xyzrel{100}\cap\xyzrel{001}$. 
   \end{proof}
  
The result above shows that one may not always combine one gauge-preserving $*$-isomorphism and one diagonal-preserving $*$-isomorphism into a $*$-isomorphism preserving both. But in fact,
using the observation by Mike Boyle proved in the recent preprint \cite{mb:seifesft} that shift equivalence implies flow equivalence, we observe that a gauge-preserving $*$-isomorphism may always be replaced by a diagonal-preserving one.

\begin{propo}\label{pre-mike}
Among regular and essential graphs, $\xyzrel{010}\subseteq\langle\KKKm\rangle\subseteq \xyzrel{001}$.
\end{propo}
\begin{proof}
If $E,F$ are regular and  essential graphs with $(E,F)\in\xyzrel{010}$, then as discussed just before Theorem \ref{KBKES}, we get that the associated shifts of finite type are shift equivalent, and by \cite{mb:seifesft} we conclude that they are also flow equivalent, which shows $(E,F)\in\xyzrel{010}$ by Theorem \ref{PSMMCEOR}.
\end{proof}

\section{Non-invertible symbolic dynamics}\label{onesided}

 Understanding  \xyz{1yz}-equivalence in a dynamical setting requires work 
 with the one-sided version
  \[
 \XXp E=\{(e_n)\in(E_1)^{\NN_0}\mid r(e_n)=s(e_{n+1})\}
 \]
of shifts of finite type.  Here, edges can contribute precisely when they are contained in paths that are one-sided infinite, so the relevant condition to impose on $E$ is regularity, to avoid sinks.   The one-sided shift space $\XXp E$ is equipped with the shift map $\sigma_E \colon \XXp E \to \XXp E$
$$
\sigma_E (( e_n)_{n \in \NN_0 } ) = ( e_{n+1} )_{n\in \NN_0}.
$$
Such systems are  considerably less studied among dynamicists due to the fact that $\sigma_E$ is no longer invertible, but still support a substantial literature, which has been revisited and augmented due to operator algebraic questions.
We refer the reader to \cite{kab:insdc} for a full introduction, including the observation that the most fundamental notion of sameness amongst one-sided shifts of finite type, \emph{one-sided conjugacy}, is finer even than our \xyz{111}-equivalence and hence is not in the scope of this paper. Instead, we recall some definitions by Matsumoto (see \cite{km:oetmscka} and \cite{km:coefemscacka}).

\begin{defin}\index{continuous orbit equivalence}\index{eventual conjugacy}
Let $E$ and $F$ be regular graphs.  The one-sided shift spaces $\XXp E$ and $\XXp F$ are said to be \emph{one-sided orbit equivalent} when there exist a homeomorphism $h:\XXp E\to \XXp F$ and maps $k,\ell: \XXp E\to \NN_0$ and $k',\ell':\XXp F\to \NN_0$ so that
\begin{gather*}
\sigma_F^{k(x)}(h(x))=\sigma_F^{\ell(x)}(h(\sigma_E(x))\\ \sigma_E^{k'(y)}(h^{-1}(y))=\sigma_E^{\ell'(y)}(h^{-1}(\sigma_F(y))
\end{gather*}
for all $x\in \XXp E$ and $y\in \XXp F$. 

We say that $\XXp E$ and $\XXp F$ are \emph{continuous orbit equivalent} when all maps $k,\ell,k',\ell'$ are chosen to be continuous. We say that they are \emph{eventual conjugate} when 
furthermore we may take $k = \ell+1$ and $k' = \ell'+1$.
\end{defin}

Carlsen and Rout characterized eventual conjugacy as \xyz{111}-equivalence in \cite{tmcjr:dgigc}, and Brix proved that this notion is implemented by the \OOO\ and \IIIp\ moves.

   \begin{theor}[\cite{tmcjr:dgigc},\cite{kab:bssebiec}]\label{B}
  For a pair $(E,F)$ of regular graphs, the following are equivalent
  \begin{enumerate}[(i)]
\item $(E,F)\in\langle \OOO,\IIIp\rangle$;
   \item $\XXp E$ and $\XXp F$ are eventually conjugate;
    \item $(E,F)\in\xyzrel{111}$.
    \end{enumerate}
  \end{theor}

The following result will be proved in forthcoming work by Brix and the second author.

\begin{propo}(\cite{kaber:use})\label{pre-mike-unital}
Among regular graphs, $\xyzrel{110}\subseteq\langle\KKKp\rangle\subseteq \xyzrel{101}$.
\end{propo}

There is no reason known to us why the remaining results that we summarize in this section should not hold much more  generally, but at  this time they have only been shown as indicated. 

Completing a program initiated by Matsumoto (and executed in the case of an irreducible essential part), the authors characterized continuous orbit equivalence as \xyz{101}-equivalence with Arklint in \cite{seaseer:dcdpgc}. This was also obtained concurrently and independently by Carlsen and Winger in 
\cite{tmcmlw:oegigg}. It is a key question, which we must leave open here, whether or not \xyzrel{101} is generated by moves in general, but we may answer it positively in the irreducible case (see Theorem~\ref{MCEORhere}).  This result is new in the form given here. We will discuss potential generalizations further in Section \ref{regcase}.

   \begin{theor}\label{MCEORhere}
  For a pair $(E,F)$ of regular graphs with irreducible essential parts, the following are equivalent
  \begin{enumerate}[(i)]
\item $(E,F)\in\langle \OOO,\IIIp,\RRRp\rangle$;
   \item $\XXp E$ and $\XXp F$ are continuously orbit equivalent;
  \item $(E,F)\in\xyzrel{100}\cap \xyzrel{001}$;
  \item $(E,F)\in\xyzrel{101}$.
    \end{enumerate}
and  when $(E,F)$ are matched, they are implied by either of
    \begin{enumerate}[(i)]\addtocounter{enumi}{4}
  \item $(E,F)\in \SL\cap \GLp$; 
  \item $(E,F)\in \SLp$. 
    \end{enumerate}
    When $(E,F)$ is in augmented standard form, all statements are equivalent.
  \end{theor}\fxnote{Proof can now be shortened}
  \begin{proof}
By  \cite{seaseer:dcdpgc}, (ii)$\Longleftrightarrow $(iv).  Therefore, it is enough to prove the rest of the statements are equivalent to (iv).  We first claim that we may assume $C^*(E) \cong C^*(F)$.  By Theorems \ref{thm:outsplitting}, \ref{thm:insplittingplus}, and \ref{thm:reduction-plus}, $\langle \OOO,\IIIp,\RRRp\rangle \subseteq \xyzrel{101}$.  Therefore, (i) implies $C^*(E) \cong C^*(F)$.  By \cite[Corollary~3.6]{segrerapws:ccuggs}, (v) or (vi) implies that $C^*(E) \cong C^*(F)$.  It is clear that (iii) or (iv) implies $C^*(E) \cong C^*(F)$.  Thus, we may assume $C^*(E) \cong C^*(F)$ since each the statement of the theorem implies the graph $C^*$-algebras are isomorphic.

We now show that we may reduce to the case that $(E, F)$ is in augmented standard form.  Since $C^*(E) \cong C^*(F)$, by \cite[Lemma~5.2]{seaseer:gciugc}, there are graphs $E'$ and $F'$ such that $(E', F')$ are in augmented standard form for which $(E,E')$ and $(F,F')$ are elements of $\langle \OOO,\IIIp,\RRRp\rangle$.  Since $\langle \OOO,\IIIp,\RRRp\rangle \subseteq \xyzrel{101}$, we now have $(E,F)\in\langle \OOO,\IIIp,\RRRp\rangle$ if and only if $(E',F')\in\langle \OOO,\IIIp,\RRRp\rangle$, $(E,F)\in\xyzrel{100}\cap \xyzrel{001}$ if and only if $(E',F')\in\xyzrel{100}\cap \xyzrel{001}$, and  $(E,F)\in\xyzrel{101}$ if and only if  $(E',F')\in\xyzrel{101}$.  Thus, we may assume $(E, F)$ is in augmented standard form.

We now show that all statements are equivalent under the assumption that $(E, F)$ is in augmented standard form.  We have already noted that  (ii)$\Longleftrightarrow$ (iv) and since $\langle \OOO,\IIIp,\RRRp\rangle \subseteq \xyzrel{101}$, (i) implies (iv).  And it is clear (iv) implies (iii).  Now (iii) implies (v) by \cite[Corollary~6.3]{ceor:coe}, \cite[Theorem~3.1]{mb:fesftpf}, and \cite[Theorem~14.6]{segrerapws:ccuggs}.  By \cite[Theorem~5.4]{seaseer:gciugc}, (vi) implies (i).  Hence, we are left to show (v) implies (vi).

The asserted $\mathsf{GL}^+$-equivalence shows that \begin{equation}U\BB_EV=\BB_F\label{intertwine}\end{equation} with \begin{equation}\label{units}U\DD_E-\DD_F\in \operatorname{im}\BB_F\end{equation}  with $\det U,\det V\in\{\pm 1\}$.
We aim to find $U'$,$V'$ with  $\det U'=\det V'=1$ intertwining $\BB_E$ and $\BB_F$ so that \eqref{units} remains true, since this entails that the matrices are $\operatorname{SL}^+$-equivalent.  

Note that $\BB_E= \begin{bmatrix} 0 \end{bmatrix}$ if and only if $\BB_F=\begin{bmatrix} 0 \end{bmatrix}$, and if $\BB_E  = \begin{bmatrix} 0 \end{bmatrix}$, then $D_E = D_F$ since $\mathrm{im} \BB_F = 0$.  So, if $\BB_E= \begin{bmatrix} 0 \end{bmatrix}$, then $A_E = A_F$.  Assume $\BB_E$ is not equal to $\begin{bmatrix} 0 \end{bmatrix}$.  Then $\BB_E$ and $\BB_F$ have size at least $3$ since $(E, F)$ is in augmented standard form.  By the Smith normal form, we can find $U_0,V_0$ in $\mathrm{GL}(n, \ZZ)$ such that 
\[
U_0\BB_EV_0=\begin{pmatrix}d_1&&&&\\&d_2&&&\\&&\ddots&&\\&&&& d_n\end{pmatrix},
\]
where all $d_i$'s are nonnegative and arranged so that $d_i\mid d_{i+1}$.  It follows from the definition of augmented standard form that $d_1=d_2=1$, and we can use this to simultaneously change the sign on both of the matrices $U$ and $V$ at our discretion.  Indeed, if we let $U_1$ and $V_1$ be the appropriately sized matrices with the block $\left(\begin{smallmatrix}0&1\\1&0\end{smallmatrix}\right)$ in the top left diagonal block, and then ones down the diagonal, we have $\det U_1=\det V_1=-1$. We have 
\[
U_1U_0\BB_E V_0V_1=U_0\BB_E V_0
\]
so
\[
UU_0^{-1}U_1U_0\BB_E V_0V_1V_0^{-1}V=U\BB_E V=\BB_F
\]
and we retain that 
\[
UU_0^{-1}U_1U_0\DD_E-\DD_F\in  \operatorname{im}\BB_F
\]
because
\[
U_1U_0\DD_E-U_0\DD_E\in  \operatorname{im}U_0\BB_E=\operatorname{im}U_0U^{-1}\BB_F.
\]
Thus we may replace $U,V$ by $UU_0^{-1}U_1U_0,VV_0^{-1}V_1V_0$ to change both determinants without changing anything else.  We can also use this procedure to assume $\det(U)=1$ as if $\det(U)=-1$, we may replace $U$ and $V$ with $UU_0^{-1}U_1U_0,VV_0^{-1}V_1V_0$ and note $\det(UU_0^{-1}U_1U_0)=1$.

Now assume $d_n>0$. In this case the assumption of SL-equivalence entails that
\[\det \BB_E=\det \BB_F\not=0
\]
(this is the only time we use that assumption). 
 We conclude that $\det U=\det V$, and thus we can arrange as in the previous paragraph that $\det U=\det V=1$.  When  $d_n=0$ and $\det V=-1$ we replace $V$ by $V_0V_2V_0^{-1}V$ where
\[
V_2=\begin{pmatrix}1&&&\\&\ddots&&\\&&1&\\&&&-1\end{pmatrix}
\]
changes the sign of the determinant without changing anything else.  Thus, we have proved that $A_E$ and $A_F$ are $\mathrm{SL}^+$-equivalent.  So, (v) implies (vi) which allows us to conclude that all six statements are equivalent.  
  \end{proof}
  
  We are now ready to complete the analysis of the graphs $\GO(c,n)$.\\
  
  \begin{proofofsecondpart}{Proposition \ref{cuntzcomplete}}
  For the final claims,
we argue via
\[
\langle \OOO,\IIIp,\RRRp\rangle\subseteq \xyzrel{101}\subseteq \xyzrel{100}\subseteq\ \equiv\ \subseteq \langle \OOO,\IIIp,\RRRp\rangle
\]
where our final ad hoc relation is defined by saying that $\GO(c,n)\equiv \GO(d,m)$ precisely when $c=d$ and $n+1,m+1$ --- considered as elements of $\ZZ_{c-1}$ ---  differ by a multiple of an invertible element.
This follows from \xyz{100}-equivalence since $m+1$ and $n+1$, respectively, define the class of the units in the $K_0$-group.

To show the rightmost inclusion, we prepare to appeal to Theorem \ref{MCEORhere}. To simplify notation, let $q=c-1$. We first outsplit to  
\[
\xymatrix{&\bullet\ar[dl]_-{n}\ar[dr]^-{n}&\\
\bullet\ar@(d,l)[]\ar@/^/[rr]^q&&\bullet\ar@(d,r)[]_{q}\ar@/^/[ll]
}
\]
(and similar with $m$) to get the 
two graphs represented by the pairs
\[
\left(\begin{pmatrix}
m+1\\m+1\end{pmatrix},\begin{pmatrix}
0&1\\q&q-1\end{pmatrix}\right), 
\left(\begin{pmatrix}
n+1\\n+1\end{pmatrix},\begin{pmatrix}
0&1\\q&q-1\end{pmatrix}\right), 
\]
The matrix $B$ is of course \SL-equivalent to itself, so to establish (i) of Theorem \ref{MCEORhere} we just need to find a \GLp-equivalence. For this, we note that any invertible element of $\ZZ/q$ is given by multiplication by $p\in\NN$ with $(p,q)=1$. We fix $\alpha,\beta\in \ZZ$ with 
\[
\alpha p+\beta q=1
\]
and consider
\[
U=\begin{pmatrix}\alpha&-\beta\\q&p\end{pmatrix}\qquad 
V=\begin{pmatrix}p&1\\-\beta q&\alpha\end{pmatrix}
\]
\end{proofofsecondpart}

   \begin{theor}[\cite{obak:tsacsa},\cite{segs}]\label{BKEShere}
  For a pair $(E,F)$ of regular and irreducible  graphs, the following are equivalent
  \begin{enumerate}[(i)]
\item $(E,F)\in\langle \KKKp\rangle$;
  \item $(E,F)\in\xyzrel{110}$.
    \end{enumerate}
  \end{theor}

\chapter{Sources and transitional vertices}
\section{Simplified forms}\label{standardforms}

Already in \cite{jc:cctmc2} it was noted that vertices not supporting paths back to themselves could lead to complications, but were easily removable up to what we would call \xyz{000}-equivalence. Surprisingly, we will see that this to a rather large extent remains true if one works with finer notions of equivalence, and in fact all the way to \xyzrel{111}-equivalence in the classical case that was considered by Cuntz. This has important consequences for our work, as it will often allow the reduction to the case of graphs that have already been studied.

\begin{table}[h]
\begin{tabular}{|m{1em}m{1em}|c|c|c|c|}\hline
&& (1)&(2)&(3)&(4) \\\hline
\rotatebox{90}{Original}&\rotatebox{90}{graph} & $\xymatrix@C=3mm@R=5mm{&\textcolor{red}{\bullet}\ar[dr]\ar[dl]&&\textcolor{red}{\circ} \ar@{=>}[dl]\ar[dr]\\\bullet\ar@(dr,dl)[]&&\bullet\ar@(dr,dl)[]&&\bullet\ar@(dr,dl)[]}$&
$\xymatrix@C=3mm@R=5mm{&\bullet\ar[rd]\ar[dl]&\\\bullet\ar[dr]&&\bullet\ar[dl]\\&\textcolor{red}{\bullet}\ar[d]\\&\bullet\ar@(dr,dl)[]}$&
$\xymatrix@C=3mm@R=5mm{  & & \bullet \ar[d]\\ & & \textcolor{red}{\bullet} \ar[d]  \\ \bullet \ar@/_/[rr]  \ar@(dl,dr)[] &  & \bullet  \ar@/_/[ll] \ar@(dl,dr)[] }$&
$\xymatrix@C=3mm@R=5mm{&&\bullet\ar[d]\\\bullet\ar@(ul,ur)[]\ar[dr]&&\bullet\ar[dl]\\&\textcolor{red}{\bullet}\ar[d]\\&\circ}$ \\
&& &&&\\ \hline
\rotatebox{90}{\xyz{111} simplified}&\rotatebox{90}{form \ref{III-simplified}} & $\xymatrix@C=3mm@R=5mm{\bullet \ar[d]&\bullet \ar[dr]&&\circ \ar@{=>}[dl] & \bullet \ar[d]\\ \bullet\ar@(dr,dl)[]&&\bullet\ar@(dr,dl)[]&&\bullet\ar@(dr,dl)[]}$ &  $\xymatrix@C=3mm@R=5mm{ \bullet \ar@<-3mm>[d]  \ar@<-1mm>[d] \ar@<1mm>[d] \ar@<3mm>[d] \\ \bullet  \ar@(dr,dl)[] }$ & $\xymatrix@C=3mm@R=5mm{ & \bullet \ar[d] & & \bullet \ar[d] \\
\bullet \ar@(dr,dl)[]   \ar[r] & \bullet \ar@/_/[rr]  &  & \bullet  \ar@/_/[ll] \ar@(dl,dr)[]  \ar@/^1.5pc/[lll] }$
 & $\xymatrix@C=3mm@R=5mm{&&\bullet\ar[d]\\\bullet\ar@(ul,ur)[]\ar[rr]&&\bullet\ar[d]\\ & \bullet \ar[r] &\bullet \ar[d]\\ & &\circ}$ \\ 
&& & & & \\ \hline
\rotatebox{90}{\xyz{011} simplified}&\rotatebox{90}{form \ref{OII-simplified}}& $\xymatrix@C=3mm@R=5mm{& &\circ \ar@{=>}[d] & &\\ \bullet\ar@(dr,dl)[]&&\bullet\ar@(dr,dl)[]&&\bullet\ar@(dl,dr)[]}$ & $\xymatrix@C=3mm@R=5mm{\\ \bullet\ar@(dl,dr)[]}$  &$\xymatrix@C=3mm@R=5mm{&&\\
\bullet \ar@(dr,dl)[]   \ar[r] & \bullet \ar@/_/[rr]  &  & \bullet  \ar@/_/[ll] \ar@(dl,dr)[]  \ar@/^1.5pc/[lll] }$  & $\xymatrix@C=3mm@R=5mm{\\ \bullet\ar@(ul,ur)[]\ar[rr]&&\bullet\ar[d]\\ &  &\bullet \ar[d]\\ & &\circ}$  \\ 
&& & & & \\ \hline
\end{tabular}
\caption{$\xyzrel{111}$ and $\xyzrel{011}$ simplified forms}
\label{111-011-simplified-form}
\end{table}

\section{Controlling transitional vertices}

We say a vertex is \emph{transitional}\index{transitional vertex} if it is not a source, and supports no path back to itself. We distinguish between regular and singular transitional vertices, and note that any sink which is not a source is in the latter class.  

\begin{propo}\label{III-simplified}
\emph{(\xyz{111} simplified form):} Any graph may be altered to satisfy
\begin{enumerate}[(1)] 
\item Any regular transitional vertex emits exactly one edge, and any singular transitional vertex never emits with finite multiplicity to any vertex;
\item Any transitional vertex receives at most from one regular transitional vertex;
\item Any regular transitional vertex emits to another transitional vertex; 
\item If any transitional vertex receives from a regular vertex which is neither transitional nor a source, then it receives from no regular transitional vertex at all.
\end{enumerate}
by applying moves \OOO\ and \IIIp, and their inverses.
\end{propo}

Table~\ref{111-011-simplified-form} illustrates the various kinds of configurations that are removed in the proposition, with the offending vertex indicated in red.  It is important to note that a graph can have several different $\xyzrel{111}$ simplified forms.  For example, 
$$
\xymatrix@C=3mm@R=5mm{ & & \bullet \ar[d] & & \bullet \ar[d] \\
\bullet \ar@(dr,dl)[]   \ar[r]  & \bullet \ar[r] & \bullet \ar@/_/[rr]  &  & \bullet  \ar@/_/[ll]  \ar@(dl,dr)[]  \ar@/^1.5pc/[lll] \ar@/^2pc/[llll]\\& }
$$
 is another $\xyzrel{111}$ simplified form for Graph (3) in Table~\ref{111-011-simplified-form}.
\begin{proof}
The proof is constructive, and is organized in several steps. 

\step{1} Condition (1) is easily obtained by appropriately outsplitting any transitional vertex not satisfying the condition.  This is done in the proof of \cite[Lemma~3.17(iv)]{segrerapws:gcgcfg}.  A non-source transition state in \cite{segrerapws:gcgcfg} is precisely our definition of a regular transitional vertex.  The idea is to first outsplit any infinite emitter using the partition that separates the infinite parallel edges with the finite parallel edges.  We then outsplit any regular transitional vertices using the partition with exactly one edge in each set in the partition starting with those regular transitional vertices whose shortest path to a singular transitional vertex or a non-transitional vertex is one. 

\step{2} It follows from (1) that any regular transitional vertex lies on a \emph{strand}\index{strand} of such vertices, having a unique outgoing path which ends in a unique vertex which is not regular transitional. We call this vertex the \emph{anchor}\index{anchor} of the strand, and of all the regular transitional vertices on the strand. For any anchor $v$, we fix a regular strand of maximal length, and note that by applying \IIIp\ repeatedly, we may transform all other regular transitional vertices into sources emitting to the chosen strand. This accomplishes (2).

\step{3} We in fact arranged in the previous step that any anchor, transitional or not, receives from at most one strand of regular transitional vertices. Our next goal is to remove all strands anchored at a vertex with a path back to itself. We note that such vertices lie in components that we denote $C_1,\dots,C_m$ under the assumption that there is only a path from $C_i$ to $C_j$ when $i\leq j$. We will work from the top, and assume that $C_1,\dots,C_{\ell-1}$ has already been cleared of incoming strands when removing strands at $C_\ell$.

\substep{3}{a}
We assume first that there is a vertex $v_0$ in $C_\ell$ which does not receive from a strand. If some other vertex $w_0$ does receive from a strand, we choose any path
\[
\xymatrix{
v_0\ar[r]&\cdots\ar[r]&v\ar[r]^-{e}&w\ar[r]&\cdots\ar[r]&w_0}
\]
and single out the vertices $v$ and $w$ with the properties that $v$ is the last vertex in the path not receiving from a strand ($v=v_0$ would be  possible). Outsplitting $s^{-1}(v)$ with a partition into two sets, one of which is $\{e\}$, we obtain a vertex $v^1$ which emits uniquely to $w$, without increasing the number of vertices in strands anchored at $C_\ell$.  If $s^{-1}(v)= \{e\}$, then we do not perform any outsplitting.  An \IIIp\ move allows us to move the past of the last transitional vertex in the strand anchored at $w$ to $v^1$, after which this vertex becomes a source, and we have now reduced the number of regular transitional vertices anchored at $C_\ell$ by one. Note that $w$ now is a vertex not receiving from a strand in $C_\ell$.

It remains to consider the cases where all vertices in $C_\ell$ receive from a strand. 

\substep{3}{b}
Assume that there are at least two vertices in $C_\ell$, and they all receive from a strand. Fix any edge $e$ in $C_\ell$ which is not a loop, and outsplit with this edge constituting a set of the partition and the remainder of $s^{-1}(s(e))$ in the other (if $s^{-1}(s(e)) = \{e\}$, we do not perform any outsplitting).  Since $s(e)$ receives from a strand, the last vertex of that strand now emits to two vertices, and we must redo steps \mystep{1}-\mystep{2} to reestablish the previous conditions. This will increase the number of vertices in strands anchored at $C_\ell$. But now we can clear $r(e)$ of incoming strands using an \IIIp \ move via $e$, obtaining a component with exactly one vertex not receiving from a strand, and the argument is completed by \mysubstep{3}{a}.  

\substep{3}{c}
Assume that there is only one vertex in $C_\ell$, but two or more edges (possibly infinitely many). Outsplitting with two sets in the partition, we obtain two vertices in the component, but will increase the number of regular vertices in the strands ending there, as in the previous step. Still, we may now proceed as in that step.

\substep{3}{d} It remains to consider the case where $C_\ell$ consists of one vertex $v$ and one loop $e$. If nothing else emits from $v$, it is straightforward to use \IIIp\ to clear $v$ of any strand ending there, since the future of $v$ is the same as the future of the last transitional vertex in the strand. If $v$ emits anything else, it must necessarily emit outside $C_\ell$, and we outsplit with $\{e\}$ a set of the partition to obtain the situation
\[
\xymatrix@R=1mm{&&&v^1\ar@(ur,ul)[]_-{e^1}\ar[dd]_-{e^2}&\\\cdots\ar[r]&s_2\ar[r]&s_1\ar[ur]\ar[dr]&&\\&&&v^2\ar[ur]\ar[r]\ar[dr]&\\&&&&}
\]
The last transitional vertex $s_1$ of the strand ending at $C_\ell$ now has the same future as the new vertex constituting $C_\ell$, so by an \IIIp\ move we get
\[
\xymatrix@R=1mm{&&&\\&&&v^1\ar@(ur,ul)[]\ar[dd]&&\\\cdots\ar[r]&s_2\ar@/^/[urr]&s_1\ar[ur]\ar[dr]&&\\&&&v^2\ar[ur]\ar[r]\ar[dr]&\\&&&&}
\]
and reduce the number of transitional vertices anchored at $C_\ell$ by one by making $s_1$ a regular source. Note, however, that the  vertex $v^2$ resulting from the outsplitting is transitional by construction. Thus we need to perform steps \mystep{1}--\mystep{2} to reestablish the conditions, which will increase the number of transitional vertices anchored at components $C_{\ell+1},\dots,C_m$. But by our assumption that no component $C_1,\dots,C_{\ell-1}$ receives from $C_\ell$ we have not increased the number of transitional vertices ending at the these components. (This is in fact the only step where the chosen ordering of components is important).

\step{4} It remains to clear all vertices in a strand anchored at a singular transitional vertex, except for the very last one, from any finite set of incoming edges from a non-transitional regular vertex.
Our method for doing so is related to \mysubstep{3}{d}. 

Assume that the length of the given strand is $k$ and that the vertex $v$ at level $m$, $k>m$, receives (necessarily finitely) from some regular $w$ in a component $C_\ell$. It is possible that $v$ is the singular vertex, in which case we set $k=0$. With $n$ the number of edges from $w$ to $v$, we partition $s^{-1}(w)$ into $n+1$ sets with the aforementioned edges placed as singletons, and perform an outsplitting. This will result in $n$ new transitional vertices, each receiving from $C_\ell$, and one vertex $w'$ fully within $C_\ell$. But all the new transitional vertices have the same future as the vertex at level $k+1$ of the strand, so by an \IIIp\ move we obtain that $v$ receives nothing from $w'$. This may be repeated until all such vertices are cleared.
\end{proof}

\begin{remar}
When the graph is regular, there can be no singular transitional vertex to receive from another transitional vertex as in Graph~(3) in Table~\ref{111-011-simplified-form}, and hence the associated graph in simplified form is completely free of transitional vertices. \end{remar}

We emphasize that the result above does not attempt any regularization of the regular sources in a graph, even though it is easy to collect and redistribute regular sources by \OOO\ moves. We will follow \cite{seaseer:gciugc} in saying that a graph has regular sources in \emph{antenna form}\index{antenna form} when there is at most one regular source, and in \emph{shadow form}\index{shadow form} when any regular source emits to exactly one vertex.

\begin{corol}\label{reducetoGF}
The following are equivalent
\begin{enumerate}[(i)]
\item $E$ is \xyz{000}-equivalent to the graph $\circ$;
\item $C^*(E)\simeq M_n(\CC)$ for some $n$;
\item $C^*(E)$ is gauge simple and finite-dimensional;
\item $E$ is \xyz{111}-equivalent to a graph of the form $\GF(\underline n)$.
\end{enumerate}
\end{corol}
\begin{proof} 
Since $C^*(\circ)\simeq \CC$, (i)$\Longrightarrow$(ii) is clear, and obviously (ii)$\Longrightarrow$(iii). For (iii)$\Longrightarrow$(iv) we note that $E$ must be acyclic with exactly one sink, and apply Proposition \ref{III-simplified} to obtain the stipulated form after arranging all sources in shadow form. Proposition \ref{finitecomplete} completes the claim.
\end{proof}

\begin{corol}\label{reducetoGC}
The following are equivalent\\
\begin{enumerate}[(i)]
\item $E$ is \xyz{000}-equivalent to the graph $\xymatrix{\bullet\ar@(r,u)[]}$;
\item $C^*(E)\simeq M_m(C(\TT))$ for some $m$;
\item $C^*(E)$ is gauge simple and not simple;
\item $E$ is \xyz{111}-equivalent to a graph of the form $\GC(\underline{m})$
\end{enumerate}
\end{corol}
\begin{proof} 
(i)$\Longrightarrow$ (ii)$\Longrightarrow$ (iii) is clear as above. For (iii)$\Longrightarrow$ (iv) we note that $E$ must be regular with all non-transitional vertices lying on the same cycle, emitting uniquely. We can  apply Proposition \ref{III-simplified} to obtain the stipulated form after collecting  sources in antenna form. Proposition \ref{cyclecomplete} completes the claim.
\end{proof}

\section{Controlling sources}

Whereas the removal of transitional vertices usually requires many changes to the remaining part of the graph, it is often possible to remove sources up to \xyz{011}-equivalence without changing anything else. To establish this, we define an auxiliary move  $\SSSs$ by the addition of a strand of length $n$ to a vertex already receiving a path of length $n$. \index{Ss@$\SSSs$}

\begin{theor}\label{Sstar}
$\SSSs \subseteq \langle\OOO,\IIIm\rangle$.
\end{theor}
\begin{proof}
For the purposes of the proof, we define $\SSSsn$ as the class of moves which attach strands of length $n$ to a vertex which is already the range of a path of length $n$, and  \SSSsnm\ for the union of all moves  ${\mbox{\texttt{\textup{(S=m)}}}}$ with $m$ ranging in $\{0,\dots,n-1\}$. We will prove that
\[
\SSSsn \subseteq \langle \OOO,\IIIm,\SSSsnm \rangle
\]
To see this, assume that $v\in E^0$ is already the range of a path of length $n$, the last edge of which is $e\in E^1$. Suppose a strand of length $n$ going through the vertices $s_1,\dots, s_n$ is added to anchor at $v$. We aim to transform it into a strand of shorter length anchored at $s(e)=w$. 
This proves the claim because $\SSSs$ is the union of all \SSSsn, and because ${\mbox{\texttt{\textup{(S<1)}}}}={\mbox{\texttt{\textup{(S=0)}}}}$ is trivial.

Passing to  $\langle \OOO,\IIIm,\SSSsnm \rangle$ is straightforward if 
 $e$ is the only
edge emitting from $w$, as then $w$
and the last vertex $s_n$ in the added strand
have the same future, so that the past of
$s_n$ may be shifted to $w$, and the final edge
of the strand can be deleted, using one $\IIIm$
move. We note that $w$ already has a path
of length $n-1$ ending at it, thus allowing 
the  $\SSSsnm$ move as well.

If $w$ emits other edges, we distinguish
the cases when $w=v$ ($e$ is then a loop)
and where $w\not= v$.
In the former case, we outsplit $w$ with
$\{e\}$ one set of the partition, to obtain the
setting 
\[
\xymatrix@R=1mm{
&v^1\ar@/^/@(ur,ul)[]_-{e^1}\ar[dd]_-{e^2}&\\
&&s_n\ar[ul]\ar[dl]&s_{n-1}\ar[l]&\cdots\ar[l]&s_1\ar[l]\\
&v^2\ar[l]\ar[ul]\ar[dl]\\&}
\]
We note that $v^1$ and $s_n$ have the same future, so an \IIIm\ move  yields
\[
\xymatrix@R=1mm{
&v^1\ar@/^/@(ur,ul)[]_-{e^1}\ar[dd]_-{e^2}&\\
&&&s_{n-1}\ar[ull]&\cdots\ar[l]&s_1\ar[l]\\
&v^2\ar[l]\ar[ul]\ar[dl]\\&&}
\]
This is exactly the result of outsplitting with $\{e\}$ a set of the partition, and then applying  an \SSSsnm\ move. Note that $v^1$ is the source of a path of any length.

When $w\not=v$, we outsplit in the same way to get 
\[
\xymatrix@R=1mm{
&w^1\ar[rd]^-{e^1}&\\
&&v&s_n\ar[l]&s_{n-1}\ar[l]&\cdots\ar[l]&s_1\ar[l]\\
&w^2\ar[l]\ar[ul]\ar[dl]\\
&}
\]
Again $w^1$ and $s_n$ have the same future, and $w^1$ supports a path of length $n-1$.
\end{proof}

\begin{corol}\label{Sstarreg} Within the class of regular graphs,
\begin{enumerate}[(i)]
\item Any graph may be replaced by its essential part by moves in $\langle \OOO,\IIIm\rangle$;
\item $\SSS \subseteq \langle\OOO,\IIIm\rangle$.
\end{enumerate}
\end{corol}
\begin{proof}
For the first claim, note that any vertex outside the essential part must be a source or transitional. Splitting each transitional vertex fully puts the non-essential part of the graph into the form of a number of disjoint strands that anchor at vertices from the essential part of the graph, and consequently may be removed by \SSSs\ moves in reverse. The claim then follows  by Theorem \ref{Sstar}.

For the second claim, note that if an \SSS\ move in reverse is applied to a vertex which is not a source, then it is an \SSSs\ move and Theorem \ref{Sstar} applies again. If it is applied to a source, then only the non-essential part of the graph is altered, and (i) may be invoked.
\end{proof}

\begin{propo}\label{OII-simplified}
\emph{(\xyz{011} simplified form):} Any graph may be altered to satisfy (1)--(3) of Proposition \ref{III-simplified} as well as
\begin{enumerate}[(1')]\addtocounter{enumi}{3} 
\item Any regular source emits uniquely to a transitional vertex,
\item If any transitional vertex receives from a regular vertex which is not transitional, then it receives from no regular transitional vertex at all;
\item If any transitional vertex receives  from a regular source, then it receives from no other regular vertex;

\end{enumerate}
by applying moves \OOO\ and \IIIm, and their inverses.
\end{propo}
\begin{proof}
We place the graph in the \xyz{111} simplified form of Proposition \ref{III-simplified}, and completely outsplit all regular sources. Any source emitting to a vertex receiving from elsewhere may then be removed by an \SSSs\ move in reverse, and consequently all remaining sources emit to transitional vertices that receive nothing else. This shows (4') and (5'), and allows the strengthening of (6) from Proposition \ref{III-simplified} to (6') here.
\end{proof}

In Table~\ref{111-011-simplified-form}, we have examples of graphs and their $\xyzrel{011}$ simplified forms.  It is important to note that a graph can have several different $\xyzrel{011}$ simplified forms. For instance, the graph 
$$\xymatrix@C=3mm@R=5mm{  \\ \bullet \ar@/_/[rr]  \ar@(ur,ul)[] &  & \bullet  \ar@/_/[ll] \ar@(ul,ur)[] }$$
is a much simpler $\xyzrel{011}$ simplified form for Graph~(3) in Figure~\ref{111-011-simplified-form} than the one found by combining the two given algorithms.

\section{Exact counterexamples}

Because of Corollary \ref{Sstarreg}, we may now easily upgrade Lemma \ref{essinv} as follows.

 \begin{lemma}\label{essinviiRPC} For regular and essential graphs, we have
 \begin{eqnarray*}
 \langle\RRR\rangle&\subseteq&\langle\RRRp,\OOO,\IIIm\rangle\ \subseteq\ \langle \RRRp,\xyzrel{011}\rangle\\ 
 \langle\CCC\rangle&\subseteq&\langle\CCCp,\OOO,\IIIm\rangle\ \subseteq\ \langle\CCCp,\xyzrel{011}\rangle\\ 
 \langle\PPP\rangle&\subseteq&\langle\PPPp,\OOO,\IIIm\rangle\ \subseteq \langle\PPPp,\xyzrel{011}\rangle\\ 
 \end{eqnarray*}
  \end{lemma}  
  
  With this, we also get:
  
      \begin{propo}\label{psexplus}
  $\xyzrel{100}\not\subseteq \xyzrel{001}\cup\xyzrel{010}$,   $\langle\CCCp\rangle\not\subseteq \xyzrel{001}\cup\xyzrel{010}$, and $\langle\PPPp\rangle\not\subseteq \xyzrel{001}\cup\xyzrel{010}$.
  \end{propo}
  \begin{proof}
By Lemma \ref{essinviiRPC}, the pairs of examples given in Proposition \ref{psex} are in fact in  $\langle\CCCp,\xyzrel{011}\rangle$ and $\langle\PPP,\xyzrel{011}\rangle$, respectively. We then argue by contradiction.
  \end{proof}
  
In order to show a similar lack of invariance of the \KKKp\ move, our last order of business of this chapter is to establish $\KKK\subseteq \langle\KKKp,\OOO,\IIIm\rangle$ to some extent. As discussed below, we do not know how to do so in the general regular and essential case, so we work in the 
 irreducible case.
  
\begin{lemma}\label{lem:I-Kthy}
Let $E$ be a graph and let $w$ be a regular vertex.   Denote the graph obtained by insplitting $w$ using the partition 
$$r^{-1}(w)=r^{-1}(w) \sqcup \left( \bigsqcup_{j=1}^n \emptyset\right)$$
by $F$.  Then there exists an isomorphism $\alpha \colon \DT(E)\to\DT(F)$ with $\alpha ( [ p_{(v, k) }] )= [ p_{(v^1,  k)} ]$ for all $v\in E^0$.  Moreover, $\alpha ( [ p_0^E] + n [p_{(w,0)}] )= [ p_0^F]$.  
\end{lemma}

\begin{proof}
Note that $F^0 = \{ v^1 : v \in E^0 , v \neq w \} \sqcup \{ w^1, w^2, \ldots, w^{n+1} \}$ and $F^1 = \{ e^1 : e \in E^1, s_E(e) \neq w \} \sqcup \{ f^i : 1 \leq i \leq n+1, s_E(f) = w \}$ such that $s_F( e^i ) = s_E(e)^i$ and $r_F( e^i) = r_E(e)^1$.  Therefore, $F$ is $E$ with $n$ additional sources.  Thus, one can check by the Cuntz-Krieger relations that there exists a $*$-homomorphism 
$h \colon C^*(E \times_1 \ZZ ) \to C^*(F \times_1 \ZZ )$ such that $h( p_{(v,k)} ) = p_{(v^1, k)}$ and $h( s_{(e,k)} ) = s_{ (e^1, k) }$ and its image is $PC^*(F \times_1 \ZZ)P$, where $P = \sum_{(v,k) \in E^0 \times \ZZ } p_{( v^1, k)}$.

Set $\alpha = h_*$.  To show that $\alpha$ is an order isomorphism, it is enough to prove that $PC^*(F \times_1 \ZZ)P$ is full in $C^*(F \times \ZZ)$ by Brown (\cite{lgb:sihsc}).  Fix $2 \leq i \leq n+1$.  Set $V = \sum_{ f \in s^{-1}(w) } s_{(f^i, k-1)}s_{( f^1, k-1) }^*$ in the multiplier algebra of $C^*(F \times_1 \ZZ)$.  Then 
\begin{align*}
VV^* &= \sum_{ f, g \in s^{-1}(w) }  s_{(f^i, k-1)}s_{( f^1, k-1) }^*  s_{(g^1, k-1)}s_{( g^i, k-1) }^* \\
	&= \sum_{ f \in s^{-1} (w) }  s_{(f^i, k-1)} s_{(f^i,k-1)}^* = p_{( w^i, k) }
\end{align*}
and 
\begin{align*}
V^*V &= \sum_{ f, g \in s^{-1}(w) } s_{(f^1, k-1)}s_{( f^i, k-1) }^*  s_{(g^i, k-1)}s_{( g^1, k-1) }^* \\
	&= \sum_{ f \in s^{-1} (w) }  s_{(f^1, k-1)} s_{(f^1,k-1)}^* = p_{( w^1, k) }.
\end{align*}
Hence, $p_{( w^i, k) }$ is in the ideal generated by $PC^*(F \times_1 \ZZ)P$.  Thus, all vertex projections in $C^*(F\times_1 \ZZ)$ are in the ideal generated by $PC^*(F \times_1 \ZZ)P$.  So, $PC^*(F \times_1 \ZZ)P$ is full in $C^*(F \times_1 \ZZ )$.   It is clear from the definition of $\alpha$ that $\alpha \circ \lt_* = \lt_* \circ \alpha$ and that $\alpha ( [ p_{(v, k)} ] )= [ p_{(v^1,  k)} ]$ for all $v$. 

Note that we also proved that in $K_0( C^*(F \times_1 \ZZ ))$, $[p_{( w^i, k) }] = [ p_{( w^1, k) } ]$.  Consequently, 
\begin{align*}
\alpha ( [ p_0^E] + n [p_{(w,0)}] ) &= \sum_{ v \in E^0 } \alpha( [p_{(v,0)} ] ) + n [ p_{(w^1,0)} ] \\
	&=  \sum_{ v \in E^0 } [p_{(v^1,0)} ]  + \sum_{ i=2}^{n+1} [  p_{( w^i, 0) } ] \\
	&= [ p_0^F ].\qedhere
\end{align*}
\end{proof}

\begin{theor}\label{ImvsIp}
Let $E$ and $F$ be regular, essential and irreducible graphs. Then $\KKK\subseteq \langle\KKKp,\OOO,\IIIm\rangle$.
\end{theor}
\begin{proof}
Suppose $\DT(E)\simeq \DT(F)$. We fix an isomorphism $\alpha:\DT(E)\to\DT(F)$ with the aim of producing $F'$ with $(F,F')\in \langle\KKKp,\OOO,\IIIm\rangle$ and $\beta: \DT(E)\to\DT(F')$ so that $\beta([p^E_0])=[p^{F'}_0]$.

Set $A =\Asf_F^t$.  Since $F$ is an irreducible graph,   by \cite[Proposition~4.5.6]{dlbm:isdc}, $F^0 = \bigsqcup_{i=0}^{p-1} D_i$ such that $r(s^{-1}(D_i)) = D_{i+1}$ for $0 \leq i \leq p-1$, $r(s^{-1}(D_{p-1}) = D_0$, and if $\Asf_F$ is written as block matrices according to the decomposition $F^0 = \bigsqcup_{i=0}^{p-1} D_i$, then 

$$(A^t)^p=\Asf_F^p = \begin{bmatrix} A_0 & 0 & 0 & \cdots & 0 \\ 0 & A_1 & 0 & \cdots & 0 \\  0 & 0 & A_2 & \cdots & 0 \\ \vdots & \vdots & \vdots & \ddots & \vdots  \\ 0 & 0 & 0 & \vdots & A_{p-1} \end{bmatrix} $$
such that each $A_i $ is a primitive matrix.  By \cite[Theorem 4.5.8]{dlbm:isdc}, there exists an integer $N$ such that $A_i^N > 0$ for all $i$.  Consequently, $F^0 \times_1 \mathbb{Z} = \bigsqcup_{i=0}^{p-1}  H_i$ where 

$$H_i = \bigcup_{ j = 0}^{p-1} D_{ [j+i] } \times (p\ZZ+j)$$
with $[j+i] \in \ZZ_{p} = \{ 0 , 1, 2, \ldots, p-1\}$.  A computation shows that each $H_i$ is a hereditary and saturated subset of $F^0 \times \ZZ$.

Since the vertex projections of $C^*(F\times_1\ZZ)$ generate the positive cone of $K_0 ( C^*(F\times_1\ZZ))$ and since $[p_{(w, k)}^F ] = \sum_{ v \in F^0 } A( w, v) [ p_{(v, k+1)}^F ]$, there are $m_v \geq 0$, $v \in F^0$ and $k \geq 0$ such that $\alpha( [p_0^E]) = \sum_{ v \in F^0 } m_v [ p_{(v, k)} ]$.  By replacing $\alpha$ with $\lt_*^{k} \circ \alpha$, we may assume that $\alpha( [p_0^E]) = \sum_{ v \in F^0 } m_v [ p_{(v, 0)} ]$.  Note that the order ideal generated by $[p_0^E]$ is $K_0(C^*(E\times_1\ZZ))_+$  since $p_0^E = \sum_{ v \in E } p_{(v,0)}$ is full in $C^*(E\times_1\ZZ)$.  Using the fact that $\alpha$ is an order isomorphism, the order ideal generated by $\alpha( [p_0^E])$ is the positive cone of $K_0 ( C^*(F\times_1\ZZ))$.  Thus the above paragraph implies $\{  v \in F^0 : m_v > 0 \} \cap D_i \neq \emptyset$ for all $i$.  By applying the Cuntz-Krieger relation (CK3), i.e.,  $[p_{(w, k)}^F ] = \sum_{ v \in F^0 } A( w, v) [ p_{(v, k+1)}^F ]$ and using the fact that $A_i^N > 0$ for all $i$, we may write $\alpha( [p_0^E]) = \sum_{ v \in F^0 } n_v [ p_{(v, \ell)} ]$ for some $\ell \geq 0$ and $n_v > 0$.  By replacing $\alpha$ with $\lt_*^{\ell } \circ \alpha$, we may assume that $\alpha ( [ p_0^E ] ) =  \sum_{ v \in F^0 } n_v [ p_{(v,0)}^F ]$ with $n_v \geq 1$.  Therefore,
$$\alpha( [p_0^E ] ) = \sum_{ v \in F^0 } n_v [ p_{(v,0)}^F ] = \sum_{ v \in F^0 } [ p_{(v,0)}^F ] +  \sum_{ v \in F^0 } (n_v - 1) [ p_{(v,0)}^F ].$$

Insplitting each vertex $v$ in $F^0$ with $n_v \geq 2$ using the partition $r^{-1}(v) \sqcup \left( \bigsqcup_{ i = 1}^{n_v-1} \emptyset \right)$ and using Lemma~\ref{lem:I-Kthy}, we get a graph $F'$ such that there exists an order isomorphism $\gamma \colon K_0( C^*(F'\times_1 \ZZ)) \to K_0( C^*(F\times_1\ZZ))$ that commutes with $\lt_*$ and $\gamma( [ p_0^{F'} ] ) = \sum_{ v \in F^0 } [ p_{(v,0)}^F ] +  \sum_{ v \in F^0 } (n_v - 1) [ p_{(v,0)}^F ] = \sum_{ v \in F^0 } n_v [ p_{(v,0)}^F ]$.  

Set $\beta=  \gamma^{-1} \circ \alpha$.  Then  $\beta \colon K_0( C^*(E\times_1\ZZ)) \to K_0( C^*(F'\times_1 \ZZ))$ is an order isomorphism commuting with $\lt_*$ and $\beta( [p_0^{E}]) = [ p_0^{F'} ]$.
\end{proof}

\begin{corol} \label{kimroushplus}$\xyzrel{110}\not\subseteq \xyzrel{011}$, and $\langle\KKKp\rangle\not\subseteq \xyzrel{011}$
\end{corol}
\begin{proof}
By Lemma \ref{ImvsIp}, the pairs of examples given in Proposition \ref{kimroush} are in fact in  $\langle\KKKp,\xyzrel{011}\rangle$. We then argue by contradiction.
\end{proof}

We  exhibit a concrete pair of graphs in $\xyzrel{110}\backslash \xyzrel{011}$:

\begin{examp}
We recall the example in Example \ref{kimroushex} and  consult \cite[Section~7]{khkfwr:wcfis} for details of the fact that the matrices $\mathsf{A}$ and $\mathsf{B}$ given there are shift equivalent. 
For this, we set
\begin{align*}
\mathsf{R}' &=\begin{bmatrix} 2 & 2 & 2 & 1 & 3 & 0 & 0 \\ 1& 2 & 2 & 1 & 3 & 0 & 0 \\ 1& 1 & 2 & 1 & 3 & 0 & 0 \\ 1& 1 & 1 & 1 & 3 & 0 & 0 \\ 0& 0 & 0 & 0 & 0 & 0 & 1 \\ 4 & 5 & 6 & 3 & 10 & 0 & 0 \\  4 & 5 & 6 & 3 & 0 & 1 & 0   \end{bmatrix} &
\mathsf{S}' &=\begin{bmatrix} -1 & 0 & 1 & 1 & 0 & 0 & 0 \\ 1& -1 & 0 & 0 & 0 & 0 & 0 \\ 0& 1 & -1 & 0 & 0 & 0 & 0 \\ 0& 0 & 1 & -1 & 0 & 0 & 0 \\ 0& 0 & 0 & 0 & 1 & 0 & 0 \\ 0& 0 & 0 & 0 & 0 & 1 & 0 \\ 0& 0 & 0 & 0 & 0 & 0 & 1   \end{bmatrix}.
\end{align*}
The matrices $\mathsf{A}$ and $\mathsf{B}$ are primitive matrices and a computation shows that 
\begin{align*}
\mathsf{A} = \mathsf{R}'\mathsf{S}' \quad \text{and} \quad \mathsf{S}'\mathsf{R}' = \mathsf{B}, 
\end{align*}
so in particular, $\mathsf{A}\mathsf{R}'=\mathsf{R}'\mathsf{B}$ and $\mathsf{B}\mathsf{S}'=\mathsf{S}'\mathsf{A}$.  Therefore, 
\begin{align*}
(\mathsf{R}'\mathsf{B}^6) (\mathsf{S}'\mathsf{A}^6) &= \mathsf{A}^{13} & (\mathsf{S}'\mathsf{A}^6)(\mathsf{R}'\mathsf{B}^6) &= \mathsf{B}^{13} \\
\mathsf{A} (\mathsf{R}'\mathsf{B}^6) & = (\mathsf{R}'\mathsf{B}^6) \mathsf{B} & \mathsf{B} (\mathsf{S}'\mathsf{A}^6) &= (\mathsf{S}' \mathsf{A}^6) \mathsf{A}.  
\end{align*}
Moreover, the entries of $\mathsf{R}'\mathsf{B}^6$ and $\mathsf{S}'\mathsf{A}^6$ are non-negative integers.  So, $(\mathsf{R},\mathsf{S}):= ( \mathsf{R}'\mathsf{B}^6, \mathsf{S}'\mathsf{A}^6)$ is a shift equivalence with lag $13$ between $\mathsf{A}$ and $\mathsf{B}$.  

Note that $\det(\mathsf{A})=\det(\mathsf{B})=-1$ and since $\mathsf{A}$ and $\mathsf{B}$ are primitive matrices, $\QQ^{7} = \mathcal{R}_{\mathsf{A}}=\Delta_\mathsf{A}$ and $\QQ^{7} = \mathcal{R}_\mathsf{B} = \Delta_\mathsf{B}$, and $\pi_\mathcal{R} \colon \QQ^{13} \to \mathcal{R}_\mathsf{A}$ and $\pi_\mathcal{R} \colon \QQ^{13} \to \mathcal{R}_\mathsf{B}$ are just the identity maps.  Moreover, $\det(\mathsf{R})=1$.  

 We now set $\mathbf{v} =  \mathbf{1} \mathsf{R}\mathsf{B}- \mathbf{1}\mathsf{B}$ and compute
 \[
\mathbf{v} =\begin{bmatrix} 40911& 50209& 60197& 31748& 61044& 3999&17694\end{bmatrix}
\]
We set $\mathsf{B}' = \left[\begin{smallmatrix} 0 & \mathbf{v} \\ \mathbf{0} & \mathsf{B} \end{smallmatrix}\right]$. Set $\mathsf{R}_1 = \begin{bmatrix} \mathbf{0} & \mathsf{R} \end{bmatrix}$ and $\mathsf{S}_1 = \left[\begin{smallmatrix} \mathbf{v} \mathsf{B}^{12} \mathsf{R}^{-1}\\ \mathsf{S} \end{smallmatrix}\right]$.  We now claim that $(\mathsf{R}_1, \mathsf{S}_1)$ is a shift equivalence between $A$ and $\mathsf{B}'$ such that the induce order preserving isomorphism $\widehat{\mathsf{R}_1} \colon \Delta_\mathsf{A} \to \Delta_{\mathsf{B}'}$ sends $\mathbf{1}$ to $\pi_\mathcal{R}(\mathbf{1})$.  Since the entries of $\mathsf{R}$ are nonnegative integers, the entries of $\mathsf{R}_1$ are nonnegative integers.  A computation shows that the entries of $\mathsf{S}_1$ are nonnegative integers.  The matrix relations for a shift equivalence are satisfied.  Indeed,
\begin{align*}
\mathsf{R}_1 \mathsf{S}_1 &= \begin{bmatrix} \mathbf{0} & \mathsf{R} \end{bmatrix} \begin{bmatrix} \mathbf{v} \mathsf{B}^{12} \mathsf{R}^{-1}\\ \mathsf{S} \end{bmatrix} = \mathsf{R}\mathsf{S} = \mathsf{A}^{13} \\
\mathsf{S}_1 \mathsf{R}_1 &= \begin{bmatrix} \mathbf{v} \mathsf{B}^{12} \mathsf{R}^{-1}\\ \mathsf{S} \end{bmatrix}\begin{bmatrix} \mathbf{0} & \mathsf{R} \end{bmatrix}  = \begin{bmatrix} 0 & \mathbf{v} \mathsf{B}^{12} \\ 0 & \mathsf{S}\mathsf{R}\end{bmatrix} = (\mathsf{B}')^{13} \\
\mathsf{A} \mathsf{R}_1 &= \mathsf{A} \begin{bmatrix} \mathbf{0} & \mathsf{R} \end{bmatrix}  = \begin{bmatrix} \mathbf{0} & \mathsf{A}\mathsf{R} \end{bmatrix}  = \begin{bmatrix} \mathbf{0} & \mathsf{R} \mathsf{B} \end{bmatrix}  = \mathsf{R}_1 \mathsf{B}' \\
\mathsf{B}' \mathsf{S}_1 &= \begin{bmatrix} 0 & \mathbf{v} \\ \mathbf{0} & \mathsf{B} \end{bmatrix}  \begin{bmatrix} \mathbf{v} \mathsf{B}^{12} \mathsf{R}^{-1}\\ \mathsf{S} \end{bmatrix} = \begin{bmatrix} \mathbf{v} \mathsf{S} \\ \mathsf{B} \mathsf{S} \end{bmatrix} = \begin{bmatrix} \mathbf{v} \mathsf{B}^{13} \mathsf{R}^{-1} \\ \mathsf{B} \mathsf{S} \end{bmatrix} \\
&=\begin{bmatrix} \mathbf{v} \mathsf{B}^{12} \mathsf{R}^{-1} \mathsf{A}\\ \mathsf{S} \mathsf{A} \end{bmatrix}=\begin{bmatrix} \mathbf{v} \mathsf{B}^{12} \mathsf{R}^{-1}\\ \mathsf{S} \end{bmatrix}\mathsf{A}=  \mathsf{S}_1 \mathsf{A}
\end{align*}
Thus, showing $(\mathsf{R}_1, \mathsf{S}_1)$ is a shift equivalence between $\mathsf{A}$ and $\mathsf{B}'$.  

Finally, we show that $\mathbf{1}\mathsf{R}_1 = \pi_\mathcal{R} (\mathbf{1})$.  Note that
$$
\begin{bmatrix} 0 & \mathbf{v}\mathsf{B}^{-1} + \mathbf{1}\end{bmatrix}  = \begin{bmatrix} 0 & (\mathbf{1} \mathsf{R}\mathsf{B}- \mathbf{1}\mathsf{B})\mathsf{B}^{-1} + \mathbf{1} \end{bmatrix} = \begin{bmatrix} 0 & \mathbf{1} \mathsf{R} \end{bmatrix} = \mathbf{1}\mathsf{R}_1.
$$
Since 
\begin{align*}
\mathbf{1} (\mathsf{B}')^8 &= \begin{bmatrix} 1 & \mathbf{1} \end{bmatrix} \begin{bmatrix} 0 & \mathbf{v} \mathsf{B}^{7} \\ \mathbf{0} & \mathsf{B}^8 \end{bmatrix} = \begin{bmatrix} 0 & \mathbf{v}\mathsf{B}^7 + \mathbf{1} \mathsf{B}^8 \end{bmatrix} = \begin{bmatrix} 0 & (\mathbf{v} + \mathbf{1} \mathsf{B})\mathsf{B}^7\end{bmatrix} \\
&= \begin{bmatrix} 0 & (\mathbf{v} + \mathbf{1} \mathsf{B}) \mathsf{B}^{-1} \end{bmatrix} \begin{bmatrix} 0 & \mathbf{v} \mathsf{B}^{7} \\ \mathbf{0} & \mathsf{B}^8 \end{bmatrix}  =\begin{bmatrix} 0 & (\mathbf{v} + \mathbf{1} \mathsf{B}) \mathsf{B}^{-1} \end{bmatrix}   (\mathsf{B}')^8 \\
&= \mathbf{1} \mathsf{R}_1(\mathsf{B}')^8
\end{align*}  
and since $\mathsf{1} \mathsf{R}_1 \in \Delta_{\mathsf{B}'} \subseteq \mathcal{R}_{\mathsf{B}'}$,  uniqueness of the decomposition gives $\mathbf{1}\mathsf{R}_1 = \pi_\mathcal{R} (\mathbf{1})$.

We associate graphs $E,F$ to $\mathsf A$ and $\mathsf B$ as in Propostion \ref{kimroush}, and similarly associate $F'$ to  $\mathsf{B}'$.  By \cite{khkfwr:wcfis}, $(E, F) \notin \xyzrel{011}$.  By \cite[Corollary~4.1]{obak:tsacsa}, we have $(E, F') \in \xyzrel{110}$.   Since $F'$ is built from $F$, by just adding sources, by Theorem~\ref{Sstar}, $(F, F')\in \xyzrel{011}$.  Therefore, $(E, F') \in \xyzrel{110} \setminus \xyzrel{011}$.
\end{examp}

\begin{openq}\label{kplusq}
Is 
\[
\langle \KKKp,\IIIm,\OOO\rangle =\langle \KKKm\rangle
\]
true in general?
\end{openq}

We know of no counterexamples, but proving it even for graphs defining simple $C^*$-algebras is currently outside our reach.


\chapter{General conclusion}\label{gencon}

  \section{Distinguishing \xyz{xyz}-equivalences}

We are now ready to prove that all eight \xyz{xyz}-relations differ.  Since $\xyzrel{1yz} \subseteq \xyzrel{0yz}$, $\xyzrel{x1z} \subseteq \xyzrel{x0z}$, and $\xyzrel{xy1} \subseteq \xyzrel{xy0}$, we have $\xyzrel{xyz} \subseteq  \xyzrel{x'y'z'}$ when $\xyz{xyz}\geq \xyz{x'y'z;}$ with ordering given by $\xyz{x_1y_1z_1}\leq \xyz{x_2y_2z_2}$ if and only if $\xyz{x_1} \leq \xyz{x_2}$, $\xyz{y_1} \leq \xyz{y_2}$, and $\xyz{z_1} \leq \xyz{z_2}$.  In the previous sections, we have collected a number of results showing that whenever \xyz{xyz} and \xyz{x'y'z'} are not related, neither are the relations: 

\begin{propo}\label{notincl}
If $\xyzrel{x_1y_1z_1}\subseteq \xyzrel{x_2y_2z_2}$, then one of
\begin{enumerate}
\item $\xyz{x_1y_1z_1}\geq \xyz{x_2y_2z_2}$
\item $\xyz{y_1}=\xyz{1}$, $\xyz{z_1}=\xyz{0}$, $\xyz{x_2y_2z_2}=\xyz{001}$
\item $\xyz{x_1y_1z_1}=\xyz{110}$, $\xyz{y_2}=\xyz{0}$, $\xyz{z_2}=\xyz{1}$
\end{enumerate}
hold.
\end{propo}
\begin{proof}
Consider the entries in row $\xyz{x_1y_1z_1}$ and column $\xyz{x_2y_2z_2}$ of Table \ref{xyzrelations}.  A blank entry means $\xyzrel{x_1y_1z_1}\subseteq \xyzrel{x_2y_2z_2}$.  A reference to a result in the entry means that the result gives counter examples to the inclusion $\xyzrel{x_1y_1z_1}\subseteq \xyzrel{x_2y_2z_2}$.  Lastly, a question mark means that it is unknown whether or not $\xyzrel{x_1y_1z_1}\subseteq \xyzrel{x_2y_2z_2}$.   Thus, the proposition follows from the information given in the table.
\end{proof}

\begin{corol}\label{xyzdiffer}
$\xyzrel{x_1y_1z_1}=\xyzrel{x_2y_2z_2}$ only occurs when $\xyz{x_1y_1z_1}=\xyz{x_2y_2z_2}$
\end{corol}
\begin{proof}
The only options for $\xyzrel{x_1y_1z_1}\subseteq \xyzrel{x_2y_2z_2}$ and $\xyz{x_1y_1z_1}\not=\xyz{x_2y_2z_2}$ left open by Proposition \ref{notincl} will have $\xyz{y_1}=\xyz{1}$ and $\xyz{y_2}=\xyz{0}$. Thus   $\xyzrel{x_2y_2z_2}\subseteq \xyzrel{x_1y_1z_1}$ is impossible.
\end{proof}

\begin{table}
 \begin{center}
\begin{NiceTabular}{|c||c|c|c|c|c|c|c|c|}\hline
&\xyz{000}&\xyz{001}&\xyz{010}&\xyz{011}&\xyz{100}&\xyz{101}&\xyz{110}&\xyz{111}\\\hline\hline
\xyz{000}&\nsp&\ref{psex}&\multicolumn{2}{c|}{\multirow{2}{*}{\ref{IOIniOIO}}}&\multicolumn{4}{c|}{\multirow{4}{*}{\ref{OIIniIOO}}}\\\cline{1-3}
\xyz{001}&\nsp&\nsp&\multicolumn{2}{c|}{}&\multicolumn{4}{c|}{}\\\cline{1-5}
\xyz{010}&\nsp&\open&\nsp&\ref{kimroush}&\multicolumn{4}{c|}{}\\\cline{1-5}
\xyz{011}&\nsp&\nsp&\nsp&\nsp&\multicolumn{4}{c|}{}\\\hline
\xyz{100}&\nsp&\ref{psexplus}&\multicolumn{2}{c|}{\multirow{2}{*}{\ref{IOIniOIO}}}&\nsp&\ref{psexplus}&\multicolumn{2}{c|}{\multirow{2}{*}{\ref{IOIniOIO}}}\\\cline{1-3}\cline{6-7}
\xyz{101}&\nsp&\nsp&\multicolumn{2}{c|}{}&\nsp&\nsp&\multicolumn{2}{c|}{}\\\hline
\xyz{110}&\nsp&\open&\nsp&\ref{kimroushplus}&\nsp&\open&\nsp&\ref{kimroushplus}\\\hline
\xyz{111}&\nsp&\nsp&\nsp&\nsp&\nsp&\nsp&\nsp&\nsp\\\hline
\end{NiceTabular}
\end{center}\caption{Relations among \xyz{xyz}-invariance}\label{xyzrelations}
\end{table}

\begin{conje}\label{tradein}
$\xyzrel{010}\subseteq \xyzrel{001}$ and $\xyzrel{110}\subseteq \xyzrel{101}$ for general graphs with finitely many vertices.
\end{conje}
\begin{discu}
We know this to be true in the regular case by Propositions \ref{pre-mike} and \ref{pre-mike-unital}, and in many other cases as we will detail in the next section. We can prove $\xyzrel{010}\subseteq \xyzrel{001}$ for any graph defining a gauge simple $C^*$-algebra (Theorem \ref{OIOtoOOI} below), but we have not been able to establish  $\xyzrel{110}\subseteq \xyzrel{101}$ in that case.

It is worth noting that we know of no direct way of trading in  an \xyz{x10}-equivalence
for  an \xyz{x01}-equivalence; all the evidence for the conjecture goes via classification and/or moves. On might ask the same question for general (non-unital) graph $C^*$-algebras; here we presently have no insights at all.
 \end{discu}
 
\section{The generation conjectures}


In previous sections, we have provided extensive (but admittedly somewhat sporadic) evidence that  
each \xyzrel{xyz} is generated as an equivalence relation by those moves that leave the graphs invariant in the relevant sense, using the unproved convention that \KKKp\ is a \xyz{110}-move and 
\KKKm\ is a \xyz{010}-move. We now take the leap and conjecture as follows:

\begin{conje}\label{master}
\addtocounter{equation}{-1}
\begin{eqnarray}
\xyzrel{000}&=&\langle \OOO,\IIIm,\RRRp,\SSS,\CCCp,\PPPp,\KKKp\rangle\\
\xyzrel{001}&=&\langle \OOO,\IIIm,\RRRp,\SSS\rangle\\
\xyzrel{010}&=&\langle \OOO,\IIIm,\KKKp\rangle\\
\xyzrel{011}&=&\langle \OOO,\IIIm\rangle\\
\xyzrel{100}&=&\langle \OOO,\IIIp,\RRRp,\CCCp,\PPPp,\KKKp\rangle\\
\xyzrel{101}&=&\langle \OOO,\IIIp,\RRRp\rangle\\
\xyzrel{110}&=&\langle \OOO,\IIIp,\KKKp\rangle\\
\xyzrel{111}&=&\langle \OOO,\IIIp\rangle
\end{eqnarray}
\end{conje}

We refer to the statements (\ref{gencon}.0)--(\ref{gencon}.7) as a whole as the \emph{generation conjecture}.\index{generation conjecture}\ Two of these statements are theorems. 

\begin{theor}[Cf. \cite{segrerapws:ccuggs}]\label{OOO-done}\fxnote{Must be adjusted}
\begin{eqnarray*}
\xyzrel{000}&=&\langle \OOO,\IIIm,\RRRp,\SSS,\CCCp,\PPPp\rangle\\
&=&\langle \OOO,\IIIm,\RRRp,\SSS,\CCCp,\PPPp,\KKKp\rangle
\end{eqnarray*}
and define a decidable relation.
\end{theor}
\begin{proof}
By \cite[Theorem~3.1]{segrerapws:ccuggs}, $\xyzrel{000}$ is generated by the moves on the list
\[
\OOO,\III,\RRR,\SSS,\CCC,\PPP,
\]
with the moves \III,\RRR,\CCC,\PPP\ defined as in Section \ref{twosided}, but for general graphs (removing sources added by the ``{\mbox{\texttt{\textup{+}}}}'' variations of the moves). Arguing as in Lemma \ref{essinv}, we see  that the moves not on both lists may be generated by the lists considered in the present paper.

We note that the \KKKp\ move is redundant in this case, but by Theorem \ref{Kpartialinva} we do know it is invariant. Decidability follows from the \GL-picture, cf.~Theorem \ref{masterGLSL}.
\end{proof}


In what may be considered a companion paper to the present one, joint with Arklint, we establish the corresponding claim for exact $*$-isomorphism:

\begin{theor}[Cf. \cite{seaseer:gciugc}]\label{IOO-done}
\begin{eqnarray*}
\xyzrel{100}&=&\langle \OOO,\IIIp,\RRRp,\CCCp,\PPPp\rangle\\
&=&\langle \OOO,\IIIp,\RRRp,\CCCp,\PPPp,\KKKp\rangle
\end{eqnarray*}
and define a decidable relation.
\end{theor}
\begin{proof}
We appeal to \cite[Corollary 5.6]{seaseer:gciugc}. This shows that the \KKKp\ move is redundant in this case, but by Theorem \ref{Kpartialinva} we do know it is invariant. Decidability follows from the \GLp-picture, cf.~Theorem \ref{masterGLSL}.
\end{proof}

\section{Challenges}

In Part \ref{casepart} of this work we will provide casewise evidence for the generation conjectures which we find compelling. However, each such casewise result presented also indicates that we do not know how to prove the corresponding component of the generation conjecture without added hypotheses. In this final section of Part \ref{genpart}, we summarize the main obstacles we see for obtaining complete results. Having solved the \xyz{x00} case completely, we discuss the cases \xyz{x01}, \xyz{x10} and \xyz{x11} cases separately.

\subsection{Generating  \xyz{x01}}

The generation conjecture's subcases (\ref{gencon}.1) and (\ref{gencon}.5) is best understood by contrasting (i) and (iii) of Theorem \ref{masterGLSL} to (ii) and (iv) of that same result. Indeed, (\ref{gencon}.1) and (\ref{gencon}.5)  would follow from a positive answer to

\begin{openq}\label{SLconj}
Do
\begin{gather*}
\xyzrel{001}\cap \standard \subseteq \SL\\
\xyzrel{101}\cap \standardp \subseteq \SLp
\end{gather*}
hold?
\end{openq}

We will show in Part \ref{casepart} that $\xyzrel{001}\cap \standard\subseteq \SLrel$ when the graphs are regular or define gauge simple $C^*$-algebras, but the general case remains open. The methods allowing a resolution in the regular case go back to \cite{kmhm:coetmscka} and requires an understanding of groupoid cohomology which eludes us in the presence of infinite emitters. The methods allowing a resolution in the gauge simple case go back to \cite{apws:gcsga} and seem to fail outside this case. We present a minimal example of a pair $(E,F)\in \xyzrel{000}\backslash \langle\OOO,\IIIm,\RRRp,\SSS\rangle$ where we are not able to determine if $(E,F)\in \xyzrel{001}$. Our conjecture requires this to fail.

If we knew that Question \ref{SLconj} had a positive answer, it would follow immediately that
\[
\standard\cap \langle \OOO,\IIIm,\RRRp,\SSS\rangle\subseteq \standard\cap\SL
\]
and
\[
\standardp\cap \langle \OOO,\IIIp,\RRRp\rangle\subseteq \standard\cap\SLp
\]
in parallel with (i) and (iii) of Theorem \ref{masterGLSL}. This would also entail decidability of the two relations as above.
But we do not, so in order to establish pairs of examples that can not be related by sequences of moves, we cannot work with matrices in standard form, but must consider matrices that are of arbitrary size and show that there is no \SL- or \SLp-relation there. To do so, we let  $\idmatrix_k$ be the $k\times k$ identity matrix, and write $\zmatrix$ for zero matrices of arbitrary size.

\begin{lemma}\label{gunnarprob}
 Consider the $(m+1+1+n+1)\times (m+1+n)$ block matrix
\[
B(x) =
\left( \begin{array}{cc|c}
\idmatrix_{m} & \zmatrix  & \zmatrix \\
\zmatrix & 3 &  \zmatrix \\
\zmatrix & 0  &  \zmatrix\\ \hline
\zmatrix & \zmatrix & \idmatrix_n   \\
\zmatrix & x & \zmatrix
\end{array}
\right).
\]
For $m ,n \geq 2$, $B(x_1)$ and $B(x_2)$ are $\SL$-equivalent via
\[
U=\left(\begin{array}{cc}
U_{1} & \zmatrix \\ 
U_{2} & U_{3}
\end{array}\right) \quad \text{and} \quad 
V=\left(\begin{array}{cc}
V_{1} & \zmatrix \\ 
V_{2} & V_{3}
\end{array}\right)
\]
 if and only if $x_1=\pm x_2$ modulo $3$. When $x_1$ and $x_2$ are units modulo $3$, the $(m+1,m+1)$ entry of $V_1$ is equivalent to $x_1/x_2 \; \bmod 3$.
\end{lemma}

\begin{proof}
Assume there are \SL-matrices for which $U B(x_1) = B(x_2) V$.  So, 
\begin{align*}
UB(x_1) &= \left(\begin{array}{c|c}
U_1 \left( \begin{smallmatrix} \idmatrix_n & \zmatrix \\ \zmatrix & 3 \\ \zmatrix & 0 \end{smallmatrix} \right) & \\ \hline
U_2 \left( \begin{smallmatrix} \idmatrix_n & \zmatrix \\ \zmatrix& 3 \\ \zmatrix & 0 \end{smallmatrix} \right)  + U_3 \left( \begin{smallmatrix} \zmatrix & \zmatrix \\ \zmatrix  & x_1 \end{smallmatrix} \right)  & U_3 \left( \begin{smallmatrix} \idmatrix_n  \\ \zmatrix  \end{smallmatrix} \right)  \end{array} \right) \quad \text{and} \\
B(x_2) V &= \left(\begin{array}{c|c}
\left( \begin{smallmatrix} \idmatrix_n & \zmatrix \\ \zmatrix & 3 \\ \zmatrix & 0 \end{smallmatrix} \right) V_1 & \\ \hline
\left( \begin{smallmatrix} \zmatrix & \zmatrix \\ \zmatrix  & x_2 \end{smallmatrix} \right) V_1 
+ \left( \begin{smallmatrix} \idmatrix_n  \\ \zmatrix \end{smallmatrix} \right) V_2  & \left( \begin{smallmatrix} \idmatrix_n  \\ \zmatrix \end{smallmatrix} \right) V_3
\end{array}
\right).
\end{align*}
Since the last row of $\left(\begin{smallmatrix} \idmatrix_n \\ \zmatrix \end{smallmatrix}\right) V_3$ is the zero row and since the last row of $U_3 \left(\begin{smallmatrix} \idmatrix_n \\ \zmatrix \end{smallmatrix}\right)$ is the last row of $U_3$ except for the last column, we see that $U_3 = \left(\begin{smallmatrix} U_3' & \mathbf{v} \\ \zmatrix & a \end{smallmatrix}\right)$.  Since $1=\det(U_3) = a$, we have $U_3 = \left(\begin{smallmatrix} U_3' & \mathbf{v} \\   \zmatrix & 1 \end{smallmatrix}\right)$.  Since $U_1 \left( \begin{smallmatrix} \idmatrix_n & \zmatrix \\ \zmatrix & 3 \\ \zmatrix & 0 \end{smallmatrix} \right) = \left( \begin{smallmatrix} \idmatrix_n & \zmatrix \\ \zmatrix & 3 \\ \zmatrix & 0 \end{smallmatrix} \right) V_1$, entries of the last column of $V_1$ must be divisible by three except possibly the $(m+1, m+1)$th entry.  Since $\det(V_1)=1$, it must be the case that the $(m+1,m+1)$th entry of $V_1$ is not divisible by 3, and hence invertible modulo 3.  Comparing the $(n+1, m+1)$th entries of $U_2 \left( \begin{smallmatrix} \idmatrix_n & \zmatrix \\ \zmatrix & 3 \\ \zmatrix & 0 \end{smallmatrix} \right)  + U_3 \left( \begin{smallmatrix} \zmatrix & \zmatrix \\ \zmatrix  & x_1 \end{smallmatrix} \right)$ and $\left( \begin{smallmatrix} \zmatrix & \zmatrix \\ \zmatrix  & x_2 \end{smallmatrix} \right) V_1 
+ \left( \begin{smallmatrix} \idmatrix_n  \\ \zmatrix \end{smallmatrix} \right) V_2$, we see that 
\begin{equation}\label{keyeqhere}
3u + x_1 = x_2 v,
\end{equation}
where $v$ is the $(m+1,m+1)$th entry of $V_1$.  Since $v$ is invertible module 3, \eqref{keyeqhere} implies that $x_1 = \pm x_2$ modulo 3.  Moreover, if $x_1, x_2$ are units modulo 3, \eqref{keyeqhere} implies $v \equiv x_1/x_2 \; \bmod 3$.

Conversely, assume that $x_1=\pm x_2$ modulo $3$.  Thus, there exists $u \in \ZZ$ such that $3u + x_1 = x_2 v$, where $v = \pm 1$.  Set 
$$
U=\left( \begin{array}{c|c}
\left( \begin{smallmatrix}
\idmatrix_{m-1} & \zmatrix & \zmatrix & \zmatrix \\
\zmatrix & v & 0  & 0 \\
\zmatrix & 0 & v & 0 \\
\zmatrix & 0 & 0 & 1
\end{smallmatrix}\right)
 \\ \hline
\left(\begin{smallmatrix}
\zmatrix & \zmatrix & \zmatrix & \zmatrix \\
\zmatrix & 0 & 0  & 0 \\
\zmatrix & 0 & 0 & 0 \\
\zmatrix & 0 & u & 0
\end{smallmatrix}\right)
&
\idmatrix_{n+1}
\end{array} 
\right) \quad \text{and} \quad 
V= 
\left( 
\begin{array}{c|c}
\left(
\begin{smallmatrix}
\idmatrix_{m-1} & \zmatrix & \zmatrix  \\
\zmatrix & v & 0  \\
\zmatrix & 0 & v  \\
\end{smallmatrix}\right)
& \\ \hline
	& \idmatrix_{n+1}
\end{array}
\right).
$$
Since $v = \pm 1$, $U$ and $V$ are \SL\ matrices of the desired form.  A computation shows that 
$$
U B(x_1) = B(x_2) V,
$$
concluding the proof.
\end{proof}

\begin{theor}\label{gp}
With
\[
E:\qquad \xymatrix{\bullet\ar[d]\ar[dr]\ar@(u,l)[]\ar@(ul,dl)[]\ar@(l,d)[]\ar@(ur,ul)[]\ar@/^/@<-1mm>[r]\ar@/^/@<1mm>[r]\ar@/^/[r]&\circ\ar@{=>}@/^/[l]\ar@{=>}@(dr,ur)[]\\
\circ\ar@{=>}@(dr,ur)[]&\circ\ar@{=>}@(dr,ur)[]}\qquad\qquad F:\qquad
\xymatrix{\bullet\ar@<0mm>[d]\ar@<1mm>[d]\ar[dr]\ar@(u,l)[]\ar@(ul,dl)[]\ar@(l,d)[]\ar@(ur,ul)[]\ar@/^/@<-1mm>[r]\ar@/^/@<1mm>[r]\ar@/^/[r]&\circ\ar@{=>}@/^/[l]\ar@{=>}@(dr,ur)[]\\
\circ\ar@{=>}@(dr,ur)[]&\circ\ar@{=>}@(dr,ur)[]}
\]
we have that $(E,F)\in \xyzrel{000}\backslash \langle \OOO,\IIIm,\RRRp,\SSS\rangle$. 
\end{theor}
\begin{proof}
For $k, m, n \geq 2$, set 
$$
\widetilde{B}_{k,m, n} (x, y)  =
\left( \begin{array}{cc|c|c}
\idmatrix_{k} & \zmatrix  & & \\
\zmatrix & 3 & &  \\
\zmatrix & 0  & & \\ \hline
\zmatrix & \zmatrix & \idmatrix_m &  \\
\zmatrix & x & \zmatrix &  \\ \hline
\zmatrix & \zmatrix & & \idmatrix_n   \\
\zmatrix & y & & \zmatrix
\end{array}
\right).
$$
It is straightforward to expand the two graphs with \OOO\ moves so that there are four vertices in each component, and then to use row and column additions to obtain two graphs $E'$ and $F'$ so that the pair $(E',F')$ is in standard form with \SL-equivalences from $\mathsf{B}^\bullet _{E'}$ and $\mathsf{B}^\bullet _{F'}$ to $\widetilde{B}_{2,3,3}(1,1)$ and $\widetilde{B}_{2,3,3}(1,-1)$ respectively.  A diagonal choice shows that $\widetilde{B}_{2,3,3}(1,1)$ and $\widetilde{B}_{2,3,3}(1,-1)$ are \GL-equivalent.  And hence $(E', F') \in \xyzrel{000}$ by \cite[Theorem~9.11]{segrerapws:ccuggs}.  Thus, $(E, F) \in \xyzrel{000}$.

We now claim that $(E, F) \notin \mathsf{SL}$.  Assume to the contrary that $(E, F) \in \mathsf{SL}$.  Thus, there are graphs $E'$ and $F'$ such that $(E, E')$ and $(F,F')$ are elements of $\langle \OOO,\IIIm,\RRRp,\SSS\rangle$, $(E', F')$ is in standard form, and there is an \SL-equivalence from $\mathsf{B}^\bullet _{E'}$ to $\mathsf{B}^\bullet _{F'}$.  Since there are \SL-equivalences from $\mathsf{B}^\bullet _{E'}$ to $\mathsf{B}^\bullet _{F'}$ to $\widetilde{B}_{k,m,n}(1,1)$ and $\widetilde{B}_{k,m,n}(1,-1)$ for some $k, m ,n \geq 2$, it must be the case that $\widetilde{B}_{k,m,n}(1,1)$ and $\widetilde{B}_{k,m,n}(1,-1)$ are \SL-equivalent via matrices 
$$
U = 
\left( 
\begin{array}{c|c|c}
U_{11} & &  \\ \hline 
U_{21} & U_{22} &  \\ \hline
U_{31} &  & U_{33} 
\end{array}
\right)
\quad \text{and} \quad 
V = 
\left( 
\begin{array}{c|c|c}
V_{11} & &  \\ \hline 
V_{21} & V_{22} &  \\ \hline
V_{31} &  & V_{33}
\end{array}
\right)
$$
We can now pass to $2\times 2$ block submatrices and apply Lemma \ref{gunnarprob} twice to show that the $(k+1,k+1)$th entry of $V_{11}$ must be equivalent to $1$ and $-1$ modulo $3$ which is a contradiction.  Therefore, $(E, F) \notin \mathsf{SL}$ which implies $(E, F)$ is not an element of $\langle \OOO,\IIIm,\RRRp,\SSS\rangle$.
\end{proof}

\begin{remar}
The example above grew out of stimulating conversations with Gunnar Restorff with the aim of producing a similar example showing that $\GL\not=\SL$ for a $C^*$-algebra with exactly one non-trivial ideal. It remains unknown, and is an interesting problem, if such examples exist.
\end{remar} 

Theorem \ref{gp} shows that the generation conjecture fails if indeed $(E,F)\in \xyzrel{001}$. The key point of the example is that since all components contain infinite emitters, there are no determinants available to distinguish the two graphs. The problem may be translated by \cite[Corollary 3.6]{tmcerasmt:rgcds} to stable isomorphism or (weak) Kakutani equivalence of the associated graph groupoids.

To analyze the \xyz{101} setting, we have found in many cases (such as Theorem \ref{MCEORhere}) that it equals $\xyzrel{100}\cap\xyzrel{001}$, and show generation this way. However, examples that we provide below show that in fact
\[
\xyzrel{101}=\xyzrel{100}\cap\xyzrel{001}
\]
cannot be true if the generation conjecture holds, both in the general simple case and in the general regular case. It would seem more promising to try to establish generation by finding a way to use \RRRp\ moves to pass from continuous orbit equivalence to eventual conjugacy. It is possible to extend these notions beyond the regular case using Webster's infinite path space \cite{sbgw:psdg}, but one must then also require that the orbit equivalence preserves periodic points, cf.~\cite{seaseer:dcdpgc} or \cite{tmcmlw:oegigg}.

\subsection{Generating  \xyz{x10}}

As discussed in Section \ref{hazratdisc}, our generation conjectures (\ref{gencon}.2) and (\ref{gencon}.6) in the \xyz{x10} case are essentially the same as the Hazrat conjectures, subject already to a serious investigation extending over more than a decade. 

\begin{openq} Do
\begin{eqnarray*}
\xyzrel{010}&=&\langle \OOO,\IIIm,\KKKp\rangle\\
\xyzrel{110}&\supseteq&\langle \KKKp\rangle\\
\end{eqnarray*}
hold?
\end{openq}

Note that here the generation question is solved in the \xyz{110} case, and the invariance is open. The open question in the generation direction for \xyz{010} is of our own making, since we have elected to not list \KKKm\ on our proposed generation collection of moves. We expect that $\langle \KKKm\rangle=\langle \OOO,\IIIm,\KKKp\rangle$ holds true in general, so that the \KKKm\ move is redundant, but only know this in some cases.

Our understanding of the other inclusions are hampered by a lack of examples in $\xyzrel{x10}\backslash \xyzrel{x11}$, where we essentially only have the one provided by Kim and Roush in \cite{khkfwr:wcfis}. The earlier non-irreducible example 
found in \cite{khkfwr:wcfrs} to lie in $\langle\KKKm\rangle \backslash \xyzrel{011}$ is an interesting test case that we do not know how to handle:

\begin{openq}
Let $E$ and $F$ be the graphs given by
\[
\begin{bmatrix}
0&0&1&1&0&0&0&0\\
1&0&0&0&0&0&0&0\\
0&1&0&0&0&0&0&0\\
0&0&1&0&0&0&0&0\\
1&0&0&0&0&0&1&1\\
0&1&0&0&1&0&0&0\\
0&0&1&0&0&1&0&0\\
0&0&0&1&0&0&1&0\\\end{bmatrix}\qquad
\begin{bmatrix}
0&0&1&1&0&0&0&0\\
1&0&0&0&0&0&0&0\\
0&1&0&0&0&0&0&0\\
0&0&1&0&0&0&0&0\\
5&1&5&5&0&0&1&1\\
5&5&1&0&1&0&0&0\\
0&5&5&1&0&1&0&0\\
1&0&5&4&0&0&1&0\\\end{bmatrix},
\]
respectively. Is $(E,F)\in \xyzrel{010}$?
\end{openq}

The authors' substantial, but not yet successful, efforts to solve this problem is detailed in \cite{tmcadse:selcka}, \cite{bbader:ehcgc}, \cite{bbader:serlc}.

We do not yet have much insight into what happens in the presence of infinite emitters, even for graphs defining simple $C^*$-algebras. It is also essential to know the following:

\begin{openq}
Are $\langle \KKKp\rangle$ and $\langle \KKKm\rangle$ decidable relations?
\end{openq}

This is known classically for the \KKKm\ case (see \cite{khkfwr:dse}), but this result is hard and not obviously portable.

\subsection{Generating  \xyz{x11}}

For the  (\ref{gencon}.3) and (\ref{gencon}.7) cases, we again have full control over invariance and only ask:

\begin{openq} Do
\addtocounter{equation}{-1}
\begin{eqnarray*}
\xyzrel{011}&\subseteq&\langle \OOO,\IIIm\rangle\\
\xyzrel{111}&\subseteq&\langle \OOO,\IIIp\rangle
\end{eqnarray*}
hold in general?\end{openq}

Although we can make no firm claims in this direction, we are generally optimistic that the approach that lead to the full resolution of these questions in the regular case could generalize, arguing via graded groupoids as in \cite{tmcjr:dgigc}. Again, there are general reinterpretations known of two-sided conjugacy and of eventual conjugacy using infinite path spaces, encoding exactly these equivalence relations.

Allowing sinks seems less daunting to us than allowing infinite emitters.  This was already seen for the $\xyzrel{000}$ relation:  In \cite{segrerapws:gcgcfg}, we solved the generation problem for finite graphs by reducing the problem to the regular graphs.  The general case was completed in \cite{segrerapws:ccuggs} and its proof required us to generalize results from symbolic dynamics to accommodate for infinite emitters.

\part{Casewise studies}\label{casepart}
\newcommand{\elsewhere}{\cellcolor{lgray}{}}
\newcommand{\elsewhered}{\cellcolor{lgray}{+}}
\newcommand{\here}{\cellcolor{orange}{}}
\newcommand{\hered}{\cellcolor{orange}{+}}
\newcommand{\sopen}{\cellcolor{black}{}}
\newcommand{\sopenbut}{\cellcolor{black}{\textcolor{white}{!}}}
\newcommand{\statusbar}[8]{\begin{center}\begin{tabular}{|c|c|c|c|c|c|c|c|}\hline\small\xyz{000}&\small\xyz{001}&\small\xyz{010}&\small\xyz{011}&\small\xyz{100}&\small\xyz{101}&\small\xyz{110}&\small\xyz{111}\\\hline#1&#2&#3&#4&#5&#6&#7&#8\\\hline\end{tabular}\end{center}}

In this part, we discuss results and open questions that address only subclasses of graphs. Depending on the nature of our results, we will present some in decreasing generality and others in increasing. For a quick overview of the status quo in each subcase, we start each subsection with a status bar, where the color coding is to be understood as 
\begin{center}
\begin{tabular}{|c|p{10cm}|}\hline
\sopen&Conjecture \ref{master} remains open for \xyzrel{xyz} even when restricted to this subcase.\\\hline
\sopenbut&Conjecture \ref{master} remains open for \xyzrel{xyz} even when restricted to this subcase, but we present partial results.\\\hline\here&We prove Conjecture \ref{master} for \xyzrel{xyz} in this section.\\\hline
\hered&We prove Conjecture \ref{master} for \xyzrel{xyz} in this section, and also show that \xyzrel{xyz} is decidable.\\\hline
\elsewhere&We have already proved Conjecture \ref{master} for \xyzrel{xyz} in higher generality.\\\hline
\elsewhered&We have already proved Conjecture \ref{master} for \xyzrel{xyz} in higher generality, showing (either here or there) that \xyzrel{xyz} is decidable\\\hline
\end{tabular}
\end{center}
The status for all graphs going in to this part will be

\statusbar{\elsewhered}{\sopenbut}{\sopenbut}{\sopenbut}{\elsewhered}{\sopenbut}{\sopenbut}{\sopenbut}

\chapter{Graphs defining gauge simple $C^*$-algebras}

In this chapter, we exclusively study gauge simple $C^*$-algebras. We note that since the \PPPp\ move is not applicable in this case, we have
\begin{eqnarray*}
\xyzrel{000}&=&\langle \OOO,\IIIp,\RRRp,\SSS,\CCCp\rangle\\
\xyzrel{100}&=&\langle \OOO,\IIIp,\RRRp,\CCCp\rangle
\end{eqnarray*}

\section{Otherwise general graphs}

\statusbar{\elsewhere}{\hered}{\sopen}{\sopen}{\elsewhere}{\sopenbut}{\sopen}{\sopen}

It was an early surprise, contained in the work of S\o{}rensen \cite{apws:gcsga}, that  in the gauge simple case, any classical Cuntz splice can be undone by more basic moves in the presence of infinite emitters. We will start by reproving and generalizing this result.

\begin{lemma}\label{GLisSL}
In the class of graphs defining gauge simple $C^*$-algebras with infinite $K_0$-groups,
\[
\standard\cap \GL=\standard\cap \SL.
\]
\end{lemma}
\begin{proof}
We fix $(E,F)$ a pair in standard form, assume that they define gauge simple $C^*$-algebras, and that 
\begin{equation}U\BB^\bullet_EV=\BB^\bullet_F\label{intertwine}\end{equation}  with $\det U,\det V\in\{\pm 1\}$.
We aim to find $U'$,$V'$ with  $\det U'=\det V'=1$ intertwining $\BB_E^\bullet$ and $\BB_F^\bullet$. 

By a straightforward variation of Smith normal form, we can find $U_0,V_0$ with $\det U_0=\det V_0=1$ so that 
\begin{equation}\label{smithbetter}
U_0\BB^\bullet _EV_0=\begin{pmatrix}d_1&&&&\\&d_2&&&\\&&\ddots&&\\&&&&\pm d_n\\\\\\\end{pmatrix}
\end{equation}
with $t\in\NN_0$ zero rows at the bottom. All $d_i$ are nonnegative and arranged to that $d_i\mid d_{i+1}$, but a sign may be necessary as indicated. 
Our assumption on the $K_0$-groups implies that either $t>0$ or $d_n=0$, so in any case the last row of $U_0\BB^\bullet _EV_0$ is identically zero.

It follows from the definition of standard form that $d_1=d_2=1$, and just as in the proof of Theorem \ref{MCEORhere} we can use this to simultaneously change the sign on both of the determinants $U$ and $V$ at our discretion. Thus we may assume that $\det V=1$. When $\det U=-1$ we replace $U$ by $UU_0U_2U_0^{-1}$ where
\[
U_2=\begin{pmatrix}1&&&\\&\ddots&&\\&&1&\\&&&-1\end{pmatrix}
\]
changes the sign of the determinant without changing anything else. 
\end{proof}

\begin{propo}\label{simpleandinfK}
In the class of graphs defining gauge simple $C^*$-algebras with infinite $K_0$-groups,
\[
\langle \CCCp\rangle\subseteq \langle \OOO,\IIIm,\RRRp,\SSS\rangle
\]
and consequently
\[
\xyzrel{000}=\xyzrel{001}=\langle \OOO,\IIIm,\RRRp,\SSS\rangle
\]
in this class.
\end{propo}
\begin{proof}
Assume that $F$ is obtained from $E$ by a \CCCp\ move. We know that at least one vertex in $E$ emits more than one edge, so we can outsplit $E$ twice to obtain $\widetilde{E}$ so that $(\widetilde{E},F)$ is a matched pair. Applying to Lemma \ref{getstd}(i) we obtain a pair $(E,F)$ in standard form so that $(E,E'),(F'F')\in\langle \OOO,\IIIp\rangle$, and we conclude that $(E',F')\in\GL$.  By Lemma \ref{GLisSL} combined with Theorem \ref{masterGLSL}(ii) we get the desired conclusion.
\end{proof}

\begin{corol}\label{OOI-simple}
Within the class of  graphs defining gauge simple $C^*$-algebras, 
$$\xyzrel{001}=\langle \OOO,\IIIm,\RRRp,\SSS\rangle$$
\end{corol}
\begin{proof}
When the graph is regular, we may use Corollary \ref{Sstarreg} to reduce to the essential subgraph, and then prove the claim by Theorem \ref{PSMMCEOR} (in fact, the original result from \cite{kmhm:coetmscka} applies).


When the graph has a singular vertex, we note that $K_0$ must be infinite, and appeal to Proposition \ref{simpleandinfK}.
\end{proof}


\begin{propo}\label{OIOtoOOI}
Within the class of  graphs defining gauge simple $C^*$-algebras, 
$$\xyzrel{010}\subseteq \xyzrel{001}$$
\end{propo}
\begin{proof}
When the graphs are not regular, we have that $K_0$ is infinite and infer
\[
\xyzrel{010}\subseteq \xyzrel{000}=\xyzrel{001}
\]
from Proposition \ref{simpleandinfK}. When they are regular, we apply Proposition \ref{pre-mike}.
\end{proof}

We end this section by discussing the \xyz{101} case.

\begin{propo}\label{notbadcase}
Within the class of  graphs $E$ defining gauge simple $C^*$-algebras for which\footnote{In an earlier version of this paper, we falsely claimed to have solved the full case}
\[
\operatorname{rank} K_0(E)\not =\operatorname{rank} K_1(E)+1
\]
we have
\[
\standardp\cap \GLp\cap \SL= \standardp\cap \SLp
\]
and consequently
\[
\xyzrel{101}=\xyzrel{100}\cap\xyzrel{001}.
\]
\end{propo}
\begin{proof}
We have already proved the result in the essential case in Theorem \ref{MCEORhere}, and by Corollary \ref{Sstarreg},  it follows in the general regular case, corresponding to the case 
\[
\operatorname{rank} K_0(E) =\operatorname{rank} K_1(E)
\]
Thus, the case 
\[
\operatorname{rank} K_0(E) >\operatorname{rank} K_1(E)+1
\]
remains. To analyze this, we pass to the diagonal matrix \eqref{smithbetter} as in the proof of Lemma \ref{GLisSL}, noting that we have $t>1$. Again we may assume that $\det V=1$.

 We note that there are at least $2$ zero rows in $U_0\BB_E V_0$ and consider the two last entries $\alpha=d'_{n+t-1}$ and $\beta=d'_{n+t}$ in $U_0\DD_E$. If one of these is zero, we may change the sign of $\det U$ by multiplying from the left by a diagonal matrix with one $-1$ entry and the rest $1$ just as in the previous paragraph. Because of the zero entries, nothing else is changed. If neither is zero, let $d=\gcd(|\alpha|,|\beta|)$ and find $x,y\in \ZZ$ so that $\alpha x-\beta y=d$. The matrix
\[
Z=\frac1d\begin{pmatrix}\alpha x+\beta y&-2\alpha y\\2\beta x&-\alpha x-\beta y\end{pmatrix}
\]
is integral, has determinant $-1$, and satisfies $Z\left(\begin{smallmatrix}\alpha \\\beta\end{smallmatrix}\right)=\left(\begin{smallmatrix}\alpha \\\beta\end{smallmatrix}\right)$. Hence we may change the sign of $U$ by placing $Z$ in the lower right diagonal block of a matrix with ones in the remaining diagonal.
\end{proof}

We do not know what happens in the remaining case, but can pinpoint the difficulties more precisely as follows. 

\begin{lemma}\label{simpleversion}
Consider the $(m+1)\times m$ and $(m+1)\times 1$ matrices
\[
B=\left(\begin{array}{cc}\idmatrix_m\\\zmatrix\end{array}\right)
\qquad
 D(x)=\left(\begin{array}{c}\zmatrix\\x\end{array}\right).
\]
$(D(x_1),B)$ and $(D(x_2),B)$ are \GLp-equivalent if and only if $|x_1|=|x_2|$. They  are only \SLp-equivalent when $x_1=x_2$.
\end{lemma}
\begin{proof}
When  implementing matrices $U,V$ are given, we write $U$ as an $(m+1)\times (m+1)$ block matrix
\[
U=\begin{pmatrix}
U_1&U_2\\
U_3&u
\end{pmatrix}\qquad
\]
with $u\in\ZZ$, and note (like in the proof of Lemma \ref{gunnarprob}) that
\[
\begin{pmatrix}
U_1\\U_3\end{pmatrix}=UB=BV=\begin{pmatrix}V\\\zmatrix\end{pmatrix}
\]
so that $U_1=V$ and $U_3=\zmatrix$. Then we conclude
 that 
\[
\det U=\det\left(\begin{array}{cc}
V&*\\
0&u\end{array}\right)=u\det V,
\]
showing that $u=\pm 1$ for any \GL-equivalence and that $u=1$ for any \SL-equivalence.

Because of the form of $B$, $UD(x_1)-D(x_2)$ vanishes in $\cok B$ precisely when the last entries of the two vectors agree, and this they do when $ux_1=x_2$, showing the claims.
\end{proof}

\begin{theor}\label{exsimple}
With
\[
E:\qquad \xymatrix@R=5mm{&\\\bullet\ar@(ul,dl)[]\ar@(u,l)[]\ar@/^/@<-0.5mm>[r]\ar@/^/@<0.5mm>[r]&\circ\ar@{=>}@/^/[l]\ar@{=>}@(dr,ur)[]}\qquad\qquad F:\qquad
\xymatrix@R=5mm{&\bullet\ar[d]\\\bullet\ar@(ul,dl)[]\ar@(u,l)[]\ar@/^/@<-0.5mm>[r]\ar@/^/@<0.5mm>[r]&\circ\ar@{=>}@/^/[l]\ar@{=>}@(dr,ur)[]\\&\bullet\ar[u]}
\]
we have that $(E,F)\in (\xyzrel{100}\cap \langle \SSS\rangle )\backslash \langle \OOO,\IIIp,\RRRp\rangle$.
\end{theor}
\begin{proof}
We argue from Lemma \ref{simpleversion} just like in Theorem \ref{gp}, but appeal here to \cite{seaseer:gciugc} to get that the pairs are in \xyzrel{100}, but not in $ \langle \OOO,\IIIp,\RRRp\rangle$. Obviously, their differences may be obtained by \SSS\ moves.
\end{proof}

\begin{openq} 
Is $(E,F)\in \xyzrel{101}$?
\end{openq}
\begin{discu}
Our example shows that 
\begin{equation}\label{GLSL}
\GLp\cap \SL\subsetneq \SLp
\end{equation}
and also that the generation conjecture fails if indeed $(E,F)\in \xyzrel{101}$. We  know that equality holds in \eqref{GLSL} for graphs defining simple $C^*$-algebras except when there is exactly one infinite emitter, and the pair $(E,F)$ demonstrates the necessity of avoiding that case.

We find it striking that we cannot decide whether or not $\xyzrel{100}\cap \langle \SSS\rangle\subseteq \xyzrel{101}$. We know -- and it is easy to see -- that any \SSS\ move preserves the diagonal after stabilization, but it will of course not usually preserve the graph algebra exactly. The question is if the diagonal must follow along in those instances where an \SSS\ move does not change the $C^*$-algebra.
\end{discu}

\section{Further specialized graphs}

We study more cases where gauge simplicity is  relevant at the end of ensuing sections.

\chapter{Finite  graphs}

Since finite graphs of course have no infinite emitters, they will be regular precisely when they have no sinks. We return to the regular case first.

\section{Regular graphs}\label{regcase}

\statusbar{\elsewhered}{\hered}{\sopen}{\here}{\elsewhered}{\sopen}{\sopen}{\here}

\ifthenelse{\boolean{casewise}}{
We have already presented several generation results in the regular case in Section \ref{twosided} and \ref{onesided} above. The optimal form concerning \xyz{001}-, \xyz{011}- and \xyz{111}-equivalence is

\begin{theor}\label{OOI-cr}
In  the class of regular graphs,
\begin{eqnarray*}
\xyzrel{001}&=&\langle \OOO,\IIIm,\RRRp,\SSS\rangle\\
\xyzrel{011}&=&\langle \OOO,\IIIm\rangle\\
\xyzrel{111}&=&\langle \OOO,\IIIp\rangle\\
\xyzrel{010}&\subseteq&\langle\KKKm\rangle\ \subseteq\ \xyzrel{001}\\
\xyzrel{110}&\subseteq&\langle\KKKp\rangle\ \subseteq\ \xyzrel{101}
\end{eqnarray*}
\end{theor}
\begin{proof}
The last statement is contained in Theorem \ref{B}. We noted the three first statements  in Theorems \ref{pre-mike}, \ref{PSMMCEOR} and \ref{WCR}, but only in the essential case. Using Proposition \ref{Sstarreg}, they now follow in general.
\end{proof}

Since the ensuing discussion takes place in the classical SFT setting, we can perform the analysis in matrices of fixed size (cf. the discussion after Open Question \ref{SLconj}).

\begin{lemma}\label{ckversion}
Let
\[
B=\left(\begin{array}{ccc|ccc}
1&0&0&&&\\
0&1&0&&&\\
0&0&0&&&\\\hline
0&0&0&1&0&0\\
0&0&0&0&1&0\\
0&0&3&0&0&0\\
\end{array}\right) D(x)=\left(\begin{array}{c}0\\0\\3\\\hline 0\\0\\x\end{array}\right)
\]
$(D(x_1),B)$ and $(D(x_2),B)$ are \GLp-equivalent if and only if $x_1\equiv \pm x_2 \mod 3$. They  are  \SLp-equivalent if and only if $x_1\equiv x_2 \mod 3$.
\end{lemma}
\begin{proof}
When 
\[
U=\left(\begin{array}{ccc|ccc}\scriptsize
u_{11}&u_{12}&u_{13}&&&\\
u_{21}&u_{22}&u_{23}&&&\\
u_{31}&u_{32}&u_{33}&&&\\\hline
u_{41}&u_{42}&u_{43}&u_{44}&u_{45}&u_{46}\\
u_{51}&u_{52}&u_{53}&u_{54}&u_{55}&u_{56}\\
u_{61}&u_{62}&u_{63}&u_{64}&u_{65}&u_{66}\\
\end{array}\right)
\qquad
V=\left(\begin{array}{ccc|ccc}\scriptsize
v_{11}&v_{12}&v_{13}&&&\\
v_{21}&v_{22}&v_{23}&&&\\
v_{31}&v_{32}&v_{33}&&&\\\hline
v_{41}&v_{42}&v_{43}&v_{44}&v_{45}&v_{46}\\
v_{51}&v_{52}&v_{53}&v_{54}&v_{55}&v_{56}\\
v_{61}&v_{62}&v_{63}&v_{64}&v_{65}&v_{66}\\
\end{array}\right)
\]
implement the equivalence, we get
\[
\left(\begin{array}{ccc|ccc}\scriptsize
u_{11}&u_{12}&0&&&\\
u_{21}&u_{22}&0&&&\\
u_{31}&u_{32}&0&&&\\\hline
u_{41}&u_{42}&3u_{46}&u_{44}&u_{45}&0\\
u_{51}&u_{52}&3u_{56}&u_{54}&u_{55}&0\\
u_{61}&u_{62}&3u_{66}&u_{64}&u_{65}&0\\
\end{array}\right)
=\left(\begin{array}{ccc|ccc}\scriptsize
v_{11}&v_{12}&v_{13}&&&\\
v_{21}&v_{22}&v_{23}&&&\\
0&0&0&&&\\\hline
v_{41}&v_{42}&v_{43}&v_{44}&v_{45}&v_{46}\\
v_{51}&v_{52}&v_{53}&v_{54}&v_{55}&v_{56}\\
3v_{31}&3v_{32}&3v_{33}&0&0&0
\end{array}\right),
\]
providing a plethora of identities between the entries of $U$ and $V$. Using the notation $w_{ij}$ at the indices where
\[
w_{ij}=u_{ij}=v_{ij},
\]
we collect all pertinent information in
\[
U=\left(\begin{array}{ccc|ccc}\scriptsize
w_{11}&w_{12}&u_{13}&&&\\
w_{21}&w_{22}&u_{23}&&&\\
0&0&u_{33}&&&\\\hline
u_{41}&u_{42}&u_{43}&w_{44}&w_{45}&u_{46}\\
u_{51}&u_{52}&u_{53}&w_{54}&w_{55}&u_{56}\\
3v_{31}&3v_{32}&u_{63}&0&0&v_{33}\\
\end{array}\right)
\]
which gives
\[
UD(x)=\left(\begin{array}{c}*\\ {*}\\3u_{33}\\\hline *\\ {*}\\3 u_{63}+x v_{33}\end{array}\right)
\]
and consequently that $UD(x)-D(y)$ vanishes in $\cok B$ precisely when
\[
u_{33}=1 \wedge (x v_{33}\equiv y\mod 3)
\]
As usual, we see that $v_{33}=\pm 1$, and the claim concerning \GLp-equivalence is proved as soon as one realizes that both choices can be implemented with diagonal matrices $U,V$.

To show that \SLp-equivalence is more restrictive, we just note
\[
1=\frac{\det U_{1,1}}{\det V_{1,1}}=\frac{u_{33}\det\left(\begin{smallmatrix}w_{11}&w_{12}\\
w_{21}&w_{22}\end{smallmatrix}\right)}{v_{33}\det\left(\begin{smallmatrix}w_{11}&w_{12}\\
w_{21}&w_{22}\end{smallmatrix}\right)}=\frac{u_{33}}{v_{33}}.
\]
\end{proof}

\begin{theor}\label{exck}
With
\[
E:\qquad\qquad \xymatrix{&\bullet\ar[d]\ar@<-1mm>[d]\ar@<1mm>[d]\\\bullet\ar@(ul,dl)[]\ar@(u,l)[]\ar@/^/[r]&\bullet\ar@(ur,dr)[]\ar@(u,r)[]\ar@/^/[l]\ar[d]\ar@<-1mm>[d]\ar@<1mm>[d]\\
\bullet\ar@(ul,dl)[]\ar@(u,l)[]\ar@/^/[r]&\bullet\ar@(ur,dr)[]\ar@(u,r)[]\ar@/^/[l]\\&\bullet\ar[u]}\qquad
F:\qquad\qquad \xymatrix{&\bullet\ar[d]\ar@<-1mm>[d]\ar@<1mm>[d]\\\bullet\ar@(ul,dl)[]\ar@(u,l)[]\ar@/^/[r]&\bullet\ar@(ur,dr)[]\ar@(u,r)[]\ar@/^/[l]\ar[d]\ar@<-1mm>[d]\ar@<1mm>[d]\\
\bullet\ar@(ul,dl)[]\ar@(u,l)[]\ar@/^/[r]&\bullet\ar@(ur,dr)[]\ar@(u,r)[]\ar@/^/[l]\\\bullet\ar[ur]&\bullet\ar[u]}
\]
we have that $(E,F)\in (\xyzrel{100}\cap \langle \SSS\rangle )\backslash \langle \OOO,\IIIp,\RRRp\rangle$.
\end{theor}
\begin{proof}
We argue from Lemma  \ref{ckversion} just like in Theorem \ref{gp}, but appeal here to \cite{seaseer:gciugc} to get that the pairs are in \xyzrel{100}, but not in $ \langle \OOO,\IIIp,\RRRp\rangle$. Obviously, $E$ may be obtained from $F$ by an \SSS\ move.
\end{proof}}

\section{Finite graphs defining gauge simple $C^*$-algebras}\label{finsimcase}

\statusbar{\elsewhered}{\elsewhered}{\hered}{\elsewhere}{\elsewhered}{\hered}{\sopenbut}{\elsewhere}

If the graph is not regular, it is of the form studied in Proposition \ref{finitecomplete}, cf. Corollary  \ref{reducetoGF}. Thus we only need to look at the regular case.

\begin{theor}\label{IOI-cr}
In  the class of finite graphs defining gauge simple $C^*$-algebras,
\begin{eqnarray*}
\xyzrel{010}&=&\langle \KKKp,\OOO,\IIIm\rangle\\
\xyzrel{101}&=&\langle \OOO,\IIIm,\RRRp\rangle\\
\end{eqnarray*}
\end{theor}
\begin{proof}
The last statement was noted in Theorem  \ref{KBKES}. We noted the first statement  in Theorems \ref{WCR} and \ref{ImvsIp}, but only in the essential case. Using Proposition \ref{Sstarreg}, it now follows in general.
\end{proof}

\chapter{Singular graphs}
\ifthenelse{\boolean{casewise}}{This class is not well studied and has the unfortunate property
that when  any of our moves is applied to a graph within the class, the resulting graph is outside it. Still, it
provides interesting examples, as we shall see.}{}

\section{Amplified graphs}

\statusbar{\elsewhered}{\hered}{\hered}{\hered}{\elsewhered}{\hered}{\hered}{\hered}

\ifthenelse{\boolean{casewise}}{We say that a graph is \emph{amplified}\index{amplified graphs} when the number of edges between two vertices is always either $0$ or $\infty$.
To study this case, the authors introduced with  S\o{}rensen in \cite{seerapws:agc}
an ad hoc move  \TTT\ adding infinitely many edges from $v$ to $w$ whenever there is a path from $v$ to $w$ starting with an edge which is parallel to infinitely many edges.\index{TTT@$\TTT$} Visually, a \TTT\ move goes from
\[
\xymatrix{
v\ar@{=>}[r]&u_1\ar[r]&\cdots\ar[r]&u_n\ar[r]&w}
\]
to
\[
\xymatrix{&&\\
v\ar@{=>}@/^2em/[rrrr]\ar@{=>}[r]&u_1\ar[r]&\cdots\ar[r]&u_n\ar[r]&w}
\]
We do not require the vertices $v,w,u_k$ to be mutually different.

\begin{lemma}\label{Tasmove}
 $\TTT\subseteq \langle\OOO,\RRRp\rangle$.
 \end{lemma}
\begin{proof}
We may assume that $n=1$ since we can use that case repeatedly to obtain both
\[
\xymatrix{&&&\\
v\ar@{=>}@/^3em/[rrrrr]\ar@{=>}@/^2em/[rrrr]\ar@{=>}@/^1em/[rr]\ar@{=>}[r]&u_1\ar[r]&u_2\ar[r]&\cdots\ar[r]&u_n\ar[r]&w}
\]
and
\[
\xymatrix{&&&\\
v\ar@{=>}[r]\ar@{=>}@/^2em/[rrrr]\ar@{=>}@/^1em/[rr]&u_1\ar[r]&u_2\ar[r]&\cdots\ar[r]&u_n\ar[r]&w}
\]

If $u_1$ emits uniquely, it is regular, and an \RRRp\ move takes us to
\[
\xymatrix{&&\\
v\ar@{=>}@/^2em/[rr]&u_1\ar[r]&w}
\]
But we can also go to this graph with an \RRRp\ move on
\[
\xymatrix{&&\\
v\ar@{=>}@/^2em/[rr]\ar@{=>}@/^1em/[r]&u_1\ar[r]&w}
\]

If $u_1$ does not emit uniquely, we outsplit to create a vertex which does, perform the moves above, and out-amalgamate back.
\end{proof}

\begin{theor}\label{classifyamplified}
In the class of amplified graphs,
\[
\xyzrel{000}=\xyzrel{001}=\xyzrel{100}=\xyzrel{101}=\langle\OOO,\RRRp\rangle
\]
and
\[
\xyzrel{010}=\xyzrel{011}=\xyzrel{110}=\xyzrel{011}=\langle\rangle
\]
with decidable relations.
\end{theor}
\begin{proof}
We proved in \cite{seerapws:agc}
\[
\xyzrel{000}=\xyzrel{100}=\langle\TTT\rangle
\]
so we get from Lemma \ref{Tasmove} that 
\[
\xyzrel{000}\subseteq \langle\TTT\rangle\subseteq \langle\OOO,\RRRp\rangle\subseteq \xyzrel{101}\subseteq\xyzrel{000}.
\]
One may test of equivalence within $\langle\TTT\rangle$ by performing \TTT\ moves until no further such moves are possible, and then check for graph isomorphism using a graph where all infinite collections of edges are replaced by a single one.

The authors proved with Sims in  \cite{seeras:agciir} that $\xyzrel{010}=\langle\rangle$, showing the second claim. Decidability is proven as above.
\end{proof}

The identity $\xyzrel{000}=\xyzrel{100}$ goes back to \cite{seerapws:agc}, and  in fact $\xyzrel{000}=\xyzrel{101}$ was observed already in \cite{nbtmcmfw:gaoe}.}{}

\section{Singular graphs defining simple $C^*$-algebras}

\statusbar{\elsewhered}{\hered}{\sopen}{\sopen}{\elsewhered}{\hered}{\sopen}{\sopen}

\ifthenelse{\boolean{casewise}}{For more general singular graphs, the relations are no longer overlapping so dramatically. For instance, it is easy to see that with
\[
E: \xymatrix{\circ\ar@{=>}[r]&\circ\ar@{=>}[r]&\circ}\qquad\qquad F:\xymatrix{\circ\ar@{=>}[r]\ar@/^1em/[rr]&\circ\ar@{=>}[r]&\circ}
\]
we have $(E,F)\in \xyzrel{011}\backslash \xyzrel{100}$ so that $\xyzrel{001}\not=\xyzrel{100}$. 
But in the simple case, we retain the identity amongst the \xyz{x0z} relations:

\begin{theor}\label{singxOz}
In the class of singular graphs defining gauge simple $C^*$-algebras,
\[
\xyzrel{000}=\xyzrel{001}=\xyzrel{100}=\xyzrel{101}=\langle\OOO,\RRRp\rangle.
\]
The number
\[
\begin{cases}0&E^1=\emptyset\\
|E^0|&E^1\not =\emptyset\end{cases}
\]
is an (obviously decidable) invariant for all of these relations.
\end{theor}
\begin{proof}
We have seen in Lemma \ref{Tasmove} that $\TTT\subseteq \langle\OOO,\RRRp\rangle$, so we may apply this move with no loss of generality for our purposes. 

Assume that $(E,F)\in\xyzrel{000}$ with both graphs singular. If there is a sink in $E$, $E$ has no edges and just one vertex, and the same is true for $F$. Thus all vertices are infinite emitters in both graphs, and we can apply \TTT\ moves until the graphs are complete amplified graphs, and then apply Theorem \ref{classifyamplified}.
\end{proof}

The various relations \xyz{x1z} do not agree in general in this case. We pursue this question in forthcoming work with Sims.

}{}


\chapter{Type I $C^*$-algebras}

\section{Monocyclic graphs}

\statusbar{\elsewhered}{\hered}{\sopen}{\sopen}{\elsewhered}{\hered}{\sopen}{\sopen}

We say that a graph $E$ is \emph{monocyclic} if any vertex $v$ supports at most one path back to itself which avoids $v$ except at the endpoints. In the unital case, this corresponds exactly to $C^*(E)$ being a type I $C^*$-algebra. We easily get:

\begin{propo}\label{monoxOz} Within the class of  monocyclic graphs
\[
\xyzrel{000}=\xyzrel{001}=\langle \OOO,\IIIm,\RRRp,\SSS\rangle
\]
and
\[
\xyzrel{100}=\xyzrel{101}=\langle \OOO,\IIIp,\RRRp\rangle
\]
with all relations being decidable.
\end{propo}
\begin{proof}
This follows directly from Theorems \ref{OOO-done} and \ref{IOO-done} since in the monocyclic case, 
the moves  \CCCp\ and \PPPp\ are never applicable.
\end{proof}

It is worth noting that we know of no direct way of promoting an \xyz{x00}-equivalence to an  \xyz{x01}-equivalence in this case. Similarly, our evidence for the following conjecture is purely casewise as we will detail below.

\begin{conje}\label{typeIconj}
Within the class of  monocyclic graphs
\[
\xyzrel{x10}=\xyzrel{x11}
\] 
\end{conje}

We will prove this for all acyclic graphs, and whenever there is only one non-trivial gauge-invariant ideal in the $C^*$-algebras. In all cases, we do so by showing that  $\langle\KKKm\rangle\subseteq \langle\OOO,\IIIm\rangle$ and $\langle\KKKp\rangle\subseteq \langle\OOO,\IIIp\rangle$, establishing also the generation conjectures.

\section{Acyclic graphs}

\statusbar{\elsewhered}{\elsewhered}{\hered}{\hered}{\elsewhered}{\elsewhered}{\hered}{\hered}

We study first the class of acyclic graphs, exactly corresponding to the AF graph $C^*$-algebras. Our approach involves the following simplified models for such graphs.

\begin{defin}\label{dfn:w-decorated-graphs}
A \emph{\WDG} is a sextuple\index{weighted, decorated graph}
\[
( \Gamma^0,\Gamma^1,r,s,D,\underline{k})
\]
where
\begin{enumerate}[(i)]
\item
 $(\Gamma^0,\Gamma^1,r,s)$ is a finite acyclic graph which is simple, i.e. satisfies for all $v, w \in \Gamma^0$ that
\[ |s^{-1}(v) \cap r^{-1}(w)| \in \{ 0 , 1 \} ;\]
\item $D:\Gamma^1\to {\mathcal P}(\NN_0)\backslash\{\emptyset\}$ is a map which to each edge associates a finite, nonempty \emph{delay set}  of nonnegative integers;
\item $\underline{k}:\Gamma^0\to\bigcup_{\ell=0}^\infty \NN^\ell$ is a map which to each vertex associates a \emph{weight tuple} (possibly empty) of positive integers;
\end{enumerate}
so that at every  $e \in \Gamma^1$, 
\[
\max D(e) \leq |\underline{k}(r(e))|.
\]
where $|(k_1,\dots,k_\ell)|=\ell$.
\end{defin}

Suppressing notation selectively, we usually denote a \WDG\ just by $\Gamma[\underline{k}]$.

\begin{defin}\label{dfn:decorated-graphs}
A \emph{\DG}\index{decorated graph} is a sextuple, 
\[
( \Gamma^0,\Gamma^1,r,s,D,n)
\]
where (i),(ii) holds as above, and
\begin{enumerate}[(i')]\addtocounter{enumi}{2}
\item $n:\Gamma^0\to\NN_0$ is a map which to each vertex associates a \emph{strand length}\index{strand length}
\end{enumerate}
so that at every  $e \in \Gamma^1$, 
\[
\max D(e) \leq n(r(e)). 
\]
\end{defin}

We usually denote a \DG\ just by $\Gamma$. When a \WDG\ $\Gamma[\underline{k}]$ is given, we always associate to it the \DG\ $\Gamma$ given with
\begin{equation}\label{kchoice}
n(v)=|\underline{k}(v)|.
\end{equation}
When a \DG\ $\Gamma$ is given, we are free to equip it as a \WDG\ by choosing any $\kk$ with $\kk(v)(i)>0$ throughout, satisfying \eqref{kchoice}. A special role is played by   the \WDG\ where
\[
\underline{k}(v)=(\overbrace{1,\dots,1}^{n(v)})=\underline{1}.
\]
for all $v$. Suppressing notation again,  we denote this \WDG\ by $\Gamma[\underline{1}]$, and note that $\Gamma$ is the \DG\ associated to $\Gamma[\underline{1}]$ in the sense just defined.

\begin{defin}
Given a \DG\ $\Gamma$, we define a directed graph $G_{\Gamma}$ as follows:  the set of vertices is 
\[
G_{\Gamma}^0 := \{ v^i : 0 \leq i \leq n(v) \}
\]
and the set of edges is 
\begin{align*}
G_{\Gamma}^0 &:= \{ e^i_m : e \in \Gamma^0, i \in D(e), m \in \mathbb{N} \} \\
	&\qquad \sqcup \{ f^j_v : v \in \Gamma^0 ,  1 \leq j \leq n(v) \} \end{align*}
with source and range maps given by 
\begin{align*}
s_{ \Gamma } (e^i_m)&= s_\Gamma(e)^0 & s_{ \Gamma } (f^j_v) &= v^j \\
r_{ \Gamma } (e^i_m)&= r_\Gamma(e)^i & r_{ \Gamma } (f^j_v) &= v^{j-1} 
\end{align*}

\end{defin}

\begin{defin}
Given a \WDG\ $\Gamma[\underline k]$, we define a directed graph $G_{\Gamma[\underline{k}]}$ as follows. When $\underline{k}= \underline 1$,  $G_{\Gamma[\underline{k}]}=G_\Gamma$. In all other cases, the set of vertices is 
\[
G_{\Gamma[\underline{k}]}^0 := \{ v^i : 0 \leq i \leq |\underline{k}(v)| \} \sqcup \{ s\}
\]
and the set of edges is 
\begin{align*}
G_{\Gamma[\underline{k}]}^0 &:= \{ e^i_m : e \in \Gamma^0, i \in D(e), m \in \mathbb{N} \} \\
	&\qquad \sqcup \{ f^j_v : v \in \Gamma^0 , 1 \leq j \leq |\underline{k}(v)| \} \\
	&\qquad \sqcup \{ g^{j,\ell}_v : v \in \Gamma^0 , j\leq |\kk(v)|, 1\leq \ell\leq \kk(v,j)-1 \}
\end{align*}
with source and range maps given by 
\begin{align*}
s_{ \Gamma[\underline{k}] } (e^i_m)&= s_{ \Gamma[\underline{k}] }(e)^0 & s_{ \Gamma[\underline{k}] } (f^j_v) &= v^i & s_{ \Gamma[\underline{k}] } (g^{k,\ell}_v) &= s \\
r_{ \Gamma[\underline{k}] } (e^i_m)&= r_{ \Gamma[\underline{k}] }(e)^i & r_{ \Gamma[\underline{k}] } (f^j_v) &= v^{i-1} & r_{ \Gamma[\underline{k}] } (g^{k,\ell}_v) &= v_{j-1}. 
\end{align*}
\end{defin}

The difference between these two types of graphs for $\underline k\not=\underline 1$ is just a regular source $s$ defined according to the entries of $\kk(v)$. Note, however, that there will also be a regular  source at the end of any strand where $|\kk(r(e))|>0$, so the graph is only on  antenna form when $\Gamma$ is the graph with one vertex $w$ and no edges,  and in this case, $G_{\Gamma[\kk]}=\GF(\kk(w))$ from Example \ref{graph-F}.  We will often pass to shadow form in illustrations for visual legibility.

\begin{theor}\mbox{}
\begin{enumerate}[(i)]
\item  If $\Gamma[\underline{k}]\simeq \Lambda [\underline{\ell}]$ then $G_{\Gamma[\underline{k}]}\simeq G_{\Lambda[\underline{\ell}]}$.
\item For any acyclic graph $E$, there exists a \WDG\ $\Gamma[\underline k]$ so that $(E,G_{\Gamma[\underline k]})\in\langle \OOO,\IIIp\rangle$.
\item For any \WDG\ $\Gamma[\kk]$, $(G_{\Gamma[\kk]},G_\Gamma)\in \langle \OOO,\IIIm\rangle$.
\end{enumerate}
\end{theor}
\begin{proof}
The first claim is obviously true.

For (ii), fix such a graph $E$.
By Proposition \ref{III-simplified} we may pass to the  simplified form $F$, where each singular vertex $v$  comes anchored with a strand of length $n_v\geq 0$, and where each singular vertex emits with infinite multiplicity only. There will also be a number of regular sources that emit into the strands; we assign one at the end of each nontrivial strand. If there are any left, we collect them to emit from a source $s$.

We construct a graph $\Gamma$ by having a vertex for each singular vertex in $F$ and an edge from the vertex corresponding to the singular vertex $v$ to the edge corresponding to the singular vertex $w$ precisely when $v$ emits into the strand at $w$. We decorate this edge $e$ by $D(e)$, the set of indices of vertices in the strand that receive directly from $v$. Finally, we let $\underline{k}(v)$ be a vector with $n_v$ entries defined as one more than the number of edges received from $s$. Clearly $G_{\Gamma[\kk]}=F$, proving (ii).

For (iii), we note that each vertex emitted from $s$ lands at a vertex which receives from another vertex in the strand, namely the one furthest away from the anchor point. Hence, Theorem \ref{Sstar} applies.
\end{proof}

We now study to what extent \DGs\ can be manipulated by \OOO\ and \IIIp\ moves,  and start with an example. Leaving the numbers $x_0,y_0,y_1,y_2\in\NN_0$ unspecified, we consider the decorated graphs
\[
\xymatrix{*+[o][F-]{0}\ar[d]_{\{1\}}\\*+[o][F-]{1}\ar[d]_{\{0,2\}}\\*+[o][F-]{3}}
\qquad
\xymatrix{u\ar[d]_-e\\v\ar[d]_f\\w}
\qquad
\xymatrix{*+[F]{\scriptstyle ()}\ar[d]_{\{1\}}\\
*+[F]{\scriptstyle(1+x_0)}\ar[d]_{\{0,2\}}\\
*+[F]{\scriptstyle(1+y_2,1+y_1,1+y_0)}
}
\]
where the middle graph is included to introduce names for the vertices and edges, the leftmost is a \DG\ and the rightmost a \WDG. 
Up to \OOO\ moves, $G_{\Gamma[\kk]}$ is 
\[
\xymatrix@R=10mm{&&&\circ\ar@{=>}[dl]\\
&&\bullet\ar[r]&\circ\ar@{=>}[d]\ar@{=>}[dll]&\bullet\ar[l]_-{x_0}\\
\bullet\ar[r]&\bullet\ar[r]&\bullet\ar[r]&\circ&\\
&\bullet \ar[u]_-{y_2}&\bullet \ar[u]_-{y_1}&\bullet \ar[u]_-{y_0}&}
\]
with the convention that when any $x_i$ or $y_j$ vanishes, there is also no vertex emitting the edges.

Pick an edge amongst the infinitely many corresponding to $2\in D(f)$ and outsplit with this edge in a set of the partition, and all others in the other. We get
\[
\xymatrix@R=10mm{&&&\circ\ar@{=>}[dl]\\
\bullet\ar[r]^-{x_0}&\bullet\ar[d]&\star\ar[r]\ar[l]&\circ\ar@{=>}[d]\ar@{=>}[dll]&\bullet\ar[l]_-{x_0}\\
\bullet\ar[r]&\bullet\ar[r]&\bullet\ar[r]&\circ\\
&\bullet \ar[u]_-{y_2}&\bullet \ar[u]_-{y_1}&\bullet \ar[u]_-{y_0}
}
\]
and outsplit at $\star$  to get
\[
\xymatrix@R=10mm{&\bullet\ar[d]&&\circ\ar@{=>}[dl]\ar@{=>}[ll]\\
\bullet\ar[r]^-{x_0}&\bullet\ar[d]&\bullet\ar[r]&\circ\ar@{=>}[d]\ar@{=>}[dll]&\bullet\ar[l]_-{x_0}\\
\bullet\ar[r]&\bullet\ar[r]&\bullet\ar[r]&\circ\\
&\bullet \ar[u]_-{y_2}&\bullet \ar[u]_-{y_1}&\bullet \ar[u]_-{y_0}}
\]
Using \IIIp\ moves as in the proof of  Proposition \ref{III-simplified}, we ``zip up'' to get
\[
\xymatrix@R=10mm{&&&&\circ\ar@{=>}[dl]\ar@{=>}@/_2em/[ddllll]\\
&&&\bullet\ar[r]&\circ\ar@{=>}[dll]\ar@{=>}[d]&\bullet\ar[l]_-{x_0}\\
\bullet\ar[r]&\bullet\ar[r]&\bullet\ar[r]&\bullet\ar[r]&\circ\\
&\bullet\ar[u]_-{x_0}&\bullet \ar[u]_-{y_2+1}&\bullet \ar[u]_-{y_1}&\bullet \ar[u]_-{y_0}}
\]
which gives
\[
\xymatrix{*+[o][F-]{0}\ar[d]^-{\{1\}}\ar@/_5em/[dd]_-{\{4\}}\\*+[o][F-]{1}\ar[d]^-{\{0,2\}}\\*+[o][F-]{4}}\qquad
\xymatrix{*+[F]{\scriptstyle ()}\ar[d]^-{\{1\}}\ar@/_5em/[dd]_-{\{4\}}\\
*+[F]{\scriptstyle(1+x_0)}\ar[d]^-{\{0,2\}}\\
*+[F]{\scriptstyle(1+x_0,2+y_2,1+y_1,1+y_0)}
}
\]
Doing the same construction at $0\in D(f)$ would take us to
\[
\xymatrix{*+[o][F-]{0}\ar[d]^-{\{1\}}\ar@/_5em/[dd]_-{\{2\}}\\*+[o][F-]{1}\ar[d]^-{\{0,2\}}\\*+[o][F-]{3}}\qquad
\xymatrix{*+[F]{\scriptstyle ()}\ar[d]^{\{1\}}\ar@/_5em/[dd]^-{\{2\}}\\
*+[F]{\scriptstyle(1+x_0)}\ar[d]_{\{0,2\}}\\
*+[F]{\scriptstyle(1+y_2,2+x_0+y_1,1+y_0)}
}
\]
and at $1\in D(e)$ we would get
\[
\xymatrix{*+[o][F-]{0}\ar[d]^-{\{1\}}\\*+[o][F-]{2}\ar[d]^-{\{0,2\}}\\*+[o][F-]{3}}\qquad
\xymatrix{*+[F]{\scriptstyle ()}\ar[d]^{\{1\}}\\
*+[F]{\scriptstyle(1,1+x_0)}\ar[d]^{\{0,2\}}\\
*+[F]{\scriptstyle(1+y_2,1+y_1,1+y_0)}
}
\]
First outsplitting and then zipping up like this, we get:

\begin{lemma}\label{increaseWDG}
Let a \WDG\ $\Gamma[\kk]$ be given, and fix $i\in D(f)$ for some edge $f$ with $s(f)=v$ and $r(f)=w$. Then
\[
(G_{\Gamma[\kk]},G_{\Gamma'[\kk']})\in \langle \OOO,\IIIp\rangle
\]
where the \WDG\ $\Gamma'$ is obtained from $\Gamma$ by
\begin{enumerate}[(i)]
\item For any $e$ with $s(e)=u$ and $r(e)=v$, if there is no edge from $u$ to $w$ in $\Gamma$, a new such edge $g$ is added;
\item $D'(g)=1+i+D(e)$ if $g$ was newly added in (i), and
\[
D'(g)=D(g)\cup (1+i+D(e))
\]
if not;
\item 
\begin{equation}\label{taueq}
\kk'(w)=\kk(w)+\tau^i(1)+\tau^{i+1}\kk(v).
\end{equation}
\end{enumerate}
and copying the remaining parts of $\Gamma'$, $D'$ and $\kk'$ from $\Gamma$, $D$  and $\kk$.
\end{lemma}

In (iii), $\tau:\bigcup_{k=0}^\infty\NN^k\to\bigcup_{k=0}^\infty\NN^k$ is the map which fixes $()\in\NN^0$ and extends any other vector by a zero to the right. On the right hand side of \eqref{taueq}, we extend the shorter vectors with zeros to the left before performing the sum.

Translated to decorated graphs, the previous result shows that up to \OOO\ and \IIIm, we can add edges to $\Gamma$ and increase the $n(-)$ and $D(-)$ entries in a specific way. We now associate a decorated graph directly from the original graph which we will prove bounds the entries that can be obtained this way.

\newcommand{\DGmax}[1][E]{\Gamma^{\max}_{#1}}

\begin{defin}
For any acyclic graph $E$, the \DG\ $\DGmax$ is defined by letting $\Gamma^0=E^0_\singX$ and defining $\Gamma^1$ by placing an edge $(w,v)$ exactly when there is a path in $E$ from $w$ to $v$. We decorate $\Gamma$ by
\begin{eqnarray*}
n(w)&=&\max \{i\mid \text{There is a path in }E\text{ of length }i\text{ ending at } w\}\\
D(w,v)&=& \{i\mid \text{There is a path in }E\text{ of length }i\text{ from } w\text{ to }v\}
\end{eqnarray*}
\end{defin}

Note that $n(w)$ and $D(e)$ are always finite because of the acyclicity of $E$.
We will see that this \DG\ is maximal amongst all representing $E$ using the ordering that $\Gamma\leq \Lambda$ when there is a bijection $\pi:\Gamma^0\to\Lambda ^0$ so that
\begin{enumerate}[(i)]
\item If $(v,w)\in \Gamma^1$, then  $(\pi(v),\pi(w))\in \Lambda^1$, and $D(v,w)\subseteq D(\pi(v),\pi(w))$
\item For all $v\in \Gamma^0$, $n(v)\leq n(\pi(v))$
\end{enumerate}

\begin{propo}\label{maxinvariant}
Let $E$ and $F$ be acyclic graphs.
\begin{enumerate}[(i)]
\item When $(E,G_{\Gamma})\in\langle\OOO,\IIIm\rangle$, then $\Gamma\leq \DGmax$.
\item $(E,G_{\DGmax})\in\langle \OOO,\IIIm\rangle$
\end{enumerate}
\end{propo}
\begin{proof}
The first claim is clear because \OOO\ or \IIIm\ moves cannot change the data used to define $\DGmax$. More precisely, the number of singular vertices match up, and if there is a path of length $i$ between two singular vertices before performing such a move, there is also a path of length $i$ between them after the move. The same is true for paths starting in an arbitrary vertex and ending in a singular one.

For (ii), we first recall that we have established already that there  is a \DG\ $\Gamma$ with $(E,G_{\Gamma})\in \langle \OOO,\IIIm\rangle$. By (i), $\Gamma\leq \DGmax$. By Lemma \ref{increaseWDG} we can keep increasing $\Gamma$ by the moves described there until the \DGs\ agree.
\end{proof}

The result above shows that $\DGmax$ is a complete invariant for the relation $\langle\OOO,\IIIm\rangle$. By now showing that in fact $\DGmax$ is invariant with respect to $\DQ(-)$ we will reach our end goal in the stable case.

\begin{propo}\label{descaforder} When $E$ is an acyclic graph,
\[
\DQ(E)=(\nsX \ZZ^\nos,\mathcal P,\rt,\mathcal I)
\]
with the positive cone $\mathcal P$ given by
\begin{gather*}
\oXe{\ell}{\textstyle \sum \beta^0_w\ee_w,\dots,\sum \beta^k_w\ee_w}\in\mathcal P
\end{gather*}
if and only if 
\begin{equation}\label{pathcrit}
\beta^i_w<0\Longrightarrow\left[\begin{array}{c}\exists j<i\exists v\in E_\singX: \beta^j_v>0\text{ and there is a}\\\text{path in }E\text{ of length }i-j\text{ from } v \text{ to } w\end{array}\right]
\end{equation}
and the order ideal $\mathcal I$ given by
\begin{gather*}
\oXe{0}{\textstyle \sum \beta^0_w\ee_w,\dots,\sum \beta^k_w\ee_w}\in\mathcal I
\end{gather*}
if and only if \eqref{pathcrit} holds along with
\[
\beta^i_w\not= 0\Longrightarrow\text{There is is a path in }E\text{  of length }i\text{ ending at }w 
\]
\end{propo}
\begin{proof}
We pass to a graph  in simplified form according to Proposition \ref{OII-simplified}, using antenna form of sources, and note when the adjacency matrix is subdivided 
\[
\left[\begin{array}{c|c}
\adjRR_E&\adjRS_E\\\hline
\adjSR_E&\adjSS_E
\end{array}\right]
\]
then all entries in $\adjSR_E$ and $\adjSS_E$ are in $\{0,\infty\}$, and except in the row corresponding to the source, all entries in $\adjRR_E$ and $\adjRS_E$ are in $\{0,1\}$. 
Because $\adjRR_E$ is nilpotent, $\Delta_{\adjRR_E}=0$ and $\Omega_E$ is simply $\osX \ZZ^\nos$ equipped with the standard shift.
We note that $\eee_w:=\oXe{0}{\ee_w}$ generates $(\nsX \ZZ^\nos,\rt)$. We say that a generator of this form is \emph{simple}.

When a given element 
\[
\oXe{\ell}{\textstyle \sum \beta^0_w\ee_w,\dots,\sum \beta^k_w\ee_w}
\]
is positive, we can write it
\[
\sum_{\kappa\in K} \rt^{n_\kappa}(g_\kappa)
\]
with $n_\kappa\in\ZZ$ and $g_\kappa\in \singgen$.  It may well happen that $n_\kappa=n_{\kappa'}$ and/or $g_\kappa=g_{\kappa'}$ even though $\kappa\not=\kappa'$.  If some $\beta^i_w$ is negative, there will be at least one $\kappa$ with
$
g_{\kappa}
$
given by $u,{C},{D}$ with $C\leq \adjSR_E$, $D\leq \adjSS_E$,
so that
\[
 \rt^{n_{\kappa}}(g_\kappa)=\oXe{n_\kappa}{\ee_{u},- \ee_u D,-\ee_uC\adjRS_E ,-\ee_u C\adjRR_E\adjRS_E,\cdots}
 \]
 contributes with a negative coefficient of $\ee_w$ at index $i$. Translating this to the graph,  we see that there is a path in exactly $i-n_{\kappa}$ steps from $u$ to $w$, so if $\beta^{n_{\kappa}}_u>0$, we are done. If not, we see that there must be another summand $ \rt^{n_{\kappa'}}(g_{\kappa'})$ with $g_{\kappa'}$ given by $u',\underline{C}',\underline{D}'$  with a negative coefficient of $\ee_u$ at index $n_{\kappa}$, so that there is a path of the required length from $u'$ to $w$ via $u$. This process terminates because the sequence $n_{\kappa},n_{\kappa'},n_{\kappa''},\dots$ is strictly decreasing, so that no $\kappa\in K$ is encountered twice.
 
 In the other direction, we assume the existence of paths and prove the existence of generators by induction. We posit 
\[
\oXe{0}{\textstyle \sum \beta^0_w\ee_w,\dots,\sum \beta^{m-1}_w\ee_w}=\sum_{\kappa\in K}\rt^{n_\kappa}(g_\kappa)
\]
and will outline how to write
\[
\oXe{0}{\textstyle \sum \beta^0_w\ee_w,\dots,\sum \beta^{m-1}_w\ee_w,\sum \beta^{m}_w\ee_w}=\sum_{\kappa\in K\backslash K'}\rt^{n_\kappa}(g_\kappa)+\sum_{\kappa\in K'}\rt^{n_\kappa}(g'_\kappa)+\sum_{\kappa\in K''}\rt^{n_\kappa}(g_\kappa)
\]
where the generators in $K$ are kept unchanged, the generators in $K'$ are adjusted, and the generators in $K''$ are newly added. 

It is straightforward to adjust for $\beta^m_w>0$ by simple generators indexed by $K''$. When $\beta^m_w<0$, we know by assumption that there is a $u$ supporting a path of length $i$ to $w$ and having $\beta^{m-i}_u>0$. There must hence be a $\kappa\in K$ with $g_\kappa$ given by $u,\underline{C},\underline{D}$ and $n_{\kappa}=m-i$, and if there is no other singular vertex on the path from $u$ to $w$, me may transfer $\kappa$ to $K'$ and adjust one entry of $\underline{C}$ or $\underline{D}$ upwards to get the coefficient of $\ee_w$ to be $\beta^m_w$ at index $m$. In the former case, this may also alter other entries downwards; this is undone easily by adding more simple generators to $K''$. 

If there is exactly one singular vertex $u'$ on the given path from $u$ to $w$, encountered exactly $i'$ steps before the $w$, we adjust as above on a generator ensuring positivity of $\beta^{m-i}_w$ so that the entry $\beta^{m-i'}_{u'}$ is decreased by one (we do not know its sign). This leaves room to add a new generator $\kappa\in K''$ with $n_{\kappa}=m-i'$ so that an entry of $\beta^m_w$ is obtained as the coefficient of $\ee_w$ at index $m$. This argument can be generalized to any number of singular vertices on the given path.

To compute $\mathcal I$, we assume that $\beta^i_w$ defines a positive element so that $\beta^i_w\not=0$ only when $i\leq n(w)$, with $n(w)$ the length of the longest path ending at $w$.  For the purposes of the proof, we also let  $m(w)$ denote the length of the longest path in $E$ ending at $w$, but traversing no other singular vertices. It follows from Theorem \ref{DQmaster} that
\[
[p^0_E]=\oXe{0}{\uu_\singX,\textstyle \sum \epsilon^0_w\ee_w,\dots,\sum \epsilon^{m}_w\ee_w}
\]
where $m=\max\{ m_w\mid w\in E^0_\singX\}$ and  $\epsilon^i_w>0$ exactly when $i\leq m(w)$. Letting 
\[
n=\max \beta^k_w
\]
we now quite easily get that 
\[
0\leq\oXe{0}{\textstyle \sum \beta^0_w\ee_w,\dots,\sum \beta^k_w\ee_w}\leq n[p^0_E].
\]
Indeed, every coefficient $\gamma^i_w$ in
\[
\oXe{0}{\textstyle \sum \gamma^0_w\ee_w,\dots,\sum \gamma^k_w\ee_w}= n[p^0_E]-\oXe{0}{\textstyle \sum \beta^0_w\ee_w,\dots,\sum \beta^k_w\ee_w}
\]
will be positive when $i\leq m(w)$, so when  $\gamma^i_w<0$ we have $\gamma^i_w=-\beta^i_w$ and consequently $m(w)<i\leq n(w)$. This means that there is a path in $E$ of length $i$ ending at $w$, but only one that visits another singular vertex $u$, say $i'$ steps before arriving at $w$, so that no other singular vertex is encountered before $u$. We get that $i-i'\leq m(u)$ and consequently $\gamma^{i-i'}_u>0$ as needed. The argument in the other direction is similar.
\end{proof}

Before we start the next proof, we note that any isomorphism $\zeta:\DQ(E)\to \DQ(F)$ is given by a shift-commuting automorphism of $\nsX \ZZ^{\nos}$ that we also denote $\zeta$, and that $\zeta$ is completely determined by the images of $\eee_w$. Consequently, any such automorphism can be described by a finite number of $\nos\times \nos$-matrices
\[
Z_{-m},\dots,Z_0,\dots Z_k,
\]
taken to be acting on the right. By our description of the order ideal in $\DQ(E)$ and $\DQ(F)$, we further see that all matrices with negative indices must vanish, and we will write $\zeta\sim [Z_0,\dots Z_k]$ to describe this situation. When another such automorphism 
\[
\psi\sim [\id,\Psi_1,\dots, \Psi_{\nos}]
\]
is given, we get
\begin{equation}\label{mulZPsi}
\zeta\circ\psi\sim[\id,\Psi_1+Z_1,\Psi_2+\Psi_1Z_1+Z_2,,\Psi_3+\Psi_2Z_1+\Psi_1Z_2+Z_3,\dots]
\end{equation}
(recall that we act from the right on dimension groups).
It is also straightforward to show that such a $\zeta$ will be invertible if and only if $Z_0$  is invertible, and the sequence of matrices
\[
Z_1,Z_1^2+Z_2,Z_1^3+Z_1Z_2+Z_2Z_1+Z_3,\dots,\widetilde{Z}_m,\dots
\]
 with
\[
\widetilde{Z}_m:=\sum_{i_1+\dots+ i_r=m}Z_{i_1}\cdots Z_{i_r}
\]
eventually vanishes. In this case,
\begin{equation}\label{invertZ}
\zeta^{-1}\sim[Z_0^{-1},Z_0^{-1}Z_1,-Z_0^{-1}(Z_1^2+Z_2),\dots,-Z_0^{-1}\widetilde{Z}_m,\dots]
\end{equation}

\begin{theor}\label{AFstable}
If $\zeta:\DQ(E)\to \DQ(F)$ is an isomorphism, then (after permutation of the vertices if necessary) $\Lambda^{\max}_E= \Lambda^{\max}_F=:\Lambda^{\max}$, and $\zeta\sim [\id, H_1,\dots, H_k]$ where the matrices $Z_1,\dots, Z_k$ have the property that whenever $\zeta^i_{wv}\not=0$, then there is an edge $e$ from $v$ to $w$ in $\Lambda^{\max}$ for which $j\in D(e)$.
\end{theor}
\begin{proof}
Since the elements $\zeta(\ee_w)$ must generate the order ideal, the first entry must be of the form $\ee_{\pi(w)}$ with $\pi:E_\singX^0\to E_\singX^0$ some permutation. We assume $\pi$ is the identity. Note  that then $\zeta^{-1}$ is induced by the matrices
\begin{equation}\label{inversezeta}
\id,-Z_1^2-Z_2,-Z_1^3-Z_1Z_2-Z_2Z_1-Z_3,\dots,-\widetilde{Z}_m,\dots
\end{equation}
It follows from the assumption that $\zeta$ is invertible that the sequence is eventually zero, but \emph{a priori} we do not know at which step this happens.
 
For the purposes of the proof, we also talk about $D_E(w,v)$ and $D_F(w,v)$ as the sets of possible lengths of paths from $w$ to $v$ in the two graphs even when there is no edge connecting $w$ and $v$ in the decorated graphs. Since there will be an edge from $w$ to $v$ in the maximal decorated graphs precisely when these sets are non-empty, this collection of sets determines the underlying graphs, as well as the decorations of the edges. We return to the decorations of the vertices at the end of the proof.

Our first goal is now to prove
\begin{equation}\label{ih}
\xymatrix{
j\in D_E(w,v)\ar@{<=>}[rr]&&j\in D_F(w,v)\\
&\zeta^j_{vw}\not=0\ar@{=>}[ul]\ar@{=>}[ur]
}
\end{equation}
where here and below, $\zeta^j_{vw}$ is the $vw$ entry of $Z_j$.
Formally, this must be proven in tandem by induction over $j$, but since the inductive component of the argument is straightforward and notationally demanding, we will just do the first steps.

Let first $j=1$ and assume that $j\in D_E(w,v)$. By Lemma \ref{descaforder} we have that $\oXe{0}{\ee_w,-n\ee_v}\geq 0$ for all $n\in \NN$, so we conclude that
\[
\zeta( \oXe{0}{\ee_w,-n\ee_v})=\oXe{0}{\textstyle \ee_w,\sum\zeta^1_{uw}\ee_u-n\ee_v,\cdots}\geq 0
\]
The coefficient $\zeta^1_{vw}-n$ of $\ee_v$ is eventually negative, so we conclude  by Lemma \ref{descaforder} that there is a path from $w$ to $v$ in one step in $F$. Thus, $j=1\in D_F(w,v)$. We argue similarly in the other direction.

Now assume that $\zeta^1_{wv}\not=0$. If $\zeta^1_{wv}<0$, we conclude from the positivity of
\[
\zeta( \oXe{0}{\ee_w})=\oXe{0}{\textstyle \ee_w,\sum\zeta^1_{uw}\ee_u,\cdots}
\]
that $1\in D_F(w,v)$, which we now know also implies  $1\in D_E(w,v)$. If $\zeta^j_{wv}>0$, we instead invoke the positivity of 
\[
\zeta^{-1}( \oXe{0}{\ee_w})=\oXe{0}{\ee_w,-\sum\zeta^1_{uw}\ee_u,\cdots}.
\]

Now consider $j=2$. We assume $2\in D_E(w,v)$ and note as above that 
\[
\zeta( \oXe{0}{\ee_w,\oo,-n\ee_v})=\oXe{0}{\textstyle \ee_w,\sum\zeta^1_{uw}\ee_u,\sum\zeta^2_{uw}\ee_u-n\ee_v,\cdots}\geq 0
\]
so the coefficient of $\ee_v$ at index 2 is eventually negative, and we conclude that in $F$ there is a path to $v$ which is either in two steps starting at $w$, or in one step starting at $u$ with $\zeta^1_{uw}>0$. In the first case, we get  $2\in D_F(w,v)$ directly, in the other, we note that by the case $j=1$ we know that there is a path in one step from $w$ to $u$, which we can concatenate with the given path from $u$ to $v$. In the other direction, the sign changes to show $\zeta^1_{uw}<0$, which allows the same conclusion.

Now assume that $\zeta^2_{vw}\not=0$. If $\zeta^2_{vw}<0$, we conclude from the positivity of
\[
\zeta( \oXe{0}{\ee_w})=\oXe{0}{\textstyle \ee_w,\sum\zeta^1_{uw}\ee_u,\sum\zeta^2_{uw}\ee_u,\cdots}
\]
that there is either a path from $w$ to $v$ in two steps, or a path in one step from $u$ to $v$, where $\zeta^1_{uw}>0$. By the case $j=1$, the latter is only possible when there is a path in one step from $w$ to $u$, and in either case, we conclude 
that $2\in D_F(w,v)$. If $\zeta^2_{wv}>0$, we instead invoke the positivity of 
\[
\zeta^{-1}( \oXe{0}{\ee_w})=\oXe{0}{\textstyle \ee_w,-\sum\zeta^1_{uw}\ee_u,-\sum\zeta^2_{uw}\ee_u-\sum\sum \zeta^1_{uu'}\zeta^1_{u'w}\ee_u,\cdots}
\]
If there is a path in two steps from $w$ to $v$ in $E$ via the singular vertex $u'$, we are done. If not, the double sum does not contribute to the coefficient of $\ee_v$ at index 2, and we conclude that it is negative. As above, we conclude that $2\in D_E(w,v)= D_F(w,v)$. 

We can continue the proof of \eqref{ih} this way, and this proves that $Z_k=\widetilde{Z}_k=0$ when there are no paths of length $k$ between singular vertices, which happens no later than at $k=\nos$.
\end{proof}

\begin{theor}
The following are equivalent for acyclic graphs $E,F$:
\begin{enumerate}[(i)]
\item $(E,F)\in\langle \OOO,\IIIm\rangle$
\item $(E,F)\in\langle \KKKm\rangle$
\item $\DGmax\simeq \DGmax[F]$
\end{enumerate}
\end{theor}

\begin{proof}
That (i) implies (ii) is a general fact following from
\[
\langle \OOO,\IIIm\rangle\subseteq \xyzrel{011}\subseteq\xyzrel{010}\subseteq\langle \KKKm\rangle,
\]
and we proved the equivalence of (i) and (iii) in Proposition \ref{maxinvariant}. The last implication is provided by Theorem \ref{AFstable}.
\end{proof}

In preparation for our results on exact isomorphism, we extract several key observations from our work this far.

First, we note that when $E,F$ are acyclic and $(E,F)\in\langle \KKKp\rangle$ or $(E,F)\in\langle \OOO,\IIIp\rangle$, then their associated \DGs\ have the same maximal graph, and hence we may assume within both equivalence relations that in fact $E=\Gamma[\kk]$ and $F=\Gamma[\myll]$, with $\Gamma$ maximal, in order to study these relations. In particular, this reduces the study of $\langle \KKKp\rangle$ to the study of the orbits of $[p_{G_{\Gamma[\kk]}}^0]$ under automorphisms of $\DQ(G_{\Gamma})$. For this purpose, we compute:

\begin{lemma}\label{descafunit}
$\DQp(G_{\Gamma[\kk]}^0)$ is given by $\DT(G_\Gamma)$ as in Lemma \ref{descaforder} along with the element
\[
[p_{G_{\Gamma[\kk]}}^0]=\oXe{0}{\uu_{\singX}}+\sum_{w\in E^0_\singX}\sum_{i=1}^{|\kk(w)|} \oXe{i+1}{k(w)(i)\ee_w}
\]
\end{lemma}

For brevity, we denote this element $\uuu+\rt(\kkk)$ with $\uuu$ and $\kkk$ defined in the obvious way.

We now further recall that we have already found a complete description of $\operatorname{Aut}(\DQ(G_\Gamma))$ as being given as 
\[
\zeta\sim [\id,Z_1,\dots, Z_{\nos}]
\]
satisfying the geometric property that an entry $\zeta^i_{vw}$ can only be nonzero when there is a path of length $i$ from $w$ to $v$ in the underlying acyclic graph, which, by maximality, is the same as requiring that $i\in D(e)$ with $e$ the edge from $w$ to $v$. Importantly, any composite (cf. \eqref{mulZPsi}) or inverse (cf. \eqref{invertZ}) of maps satisfying the geometric condition again satisfies the geometric condition, so the composition is guaranteed to end at index $\nos$, and the inverse of such a map always exists.

Finally, we note that we already know how to implement some of these  automorphisms by moves in $\langle \OOO,\IIIp\rangle$. Indeed, Lemma \ref{increaseWDG} with $i\in D(w,v)$ applied to a maximal \DG\ will not affect the graph, nor the $D(-)$ decorations, but will change the $\kk$ data consistent with the automorphism 
\[
[\id,0,\dots,0,E^i_{wv},0,\dots,0]
\]
where $E^i_{wv}$ has a one in the $wv$ entry, and zero everywhere else. Recall that knowing that $i\in D(v,w)$ exactly corresponds to our geometric condition on this tuple.

\begin{theor}\label{AFexact}
The following are equivalent for two weights $\kk,\myll$  of a maximal \DG\ $\Gamma$:
\begin{enumerate}[(i)]
\item $(G_{\Gamma[\kk]},G_{\Gamma[\ell]})\in\langle \OOO,\IIIp\rangle$
\item $(G_{\Gamma[\kk]},G_{\Gamma[\ell]})\in\langle \KKKp\rangle$
\item There exists $\zeta\in \operatorname{Aut}(\DQ(G_\Gamma))$ so that $\zeta(\uuu+\rt(\kkk))=\uuu+\rt(\mylll)$,
\end{enumerate}
\end{theor}
\begin{proof}
The forward implications follow as above, so we assume (iii) and aim for (i). We have
\[
\xymatrix{\DQp(\Gamma[\kk])\ar[rrr]^{[\id,Z_1,\dots,Z_n]}&&&\DQp(\Gamma[\myll])}
\]
and start by noting that we can find matrices $X_2,\dots X_{n}$ and $\kk'$ so that 
\[
\xymatrix{\DQp(\Gamma[\kk])\ar[rrr]^{[\id,Z_1^+,X_2,\dots,X_n]}&&&\DQp(\Gamma[\kk'])}
\]
and so that $(G_{\Gamma[\kk]},G_{\Gamma[\kk']})\in\langle \OOO,\IIIp\rangle$, where $Z=Z^+-Z^-$ is the canonical decomposition of $Z$ as a difference of matrices with nonnegative entries only.

This follows by the observation above that we can implement any automorphism $[\id,E^1_{wv}]$ satisfying the geometric condition by applying  Lemma \ref{increaseWDG}. Composing maps of this form will create the $X_i$, which will always satisfy the geometric condition. One can in fact arrange $X_i=0$ by ordering the moves carefully, but this is not useful given the rest of the argument; we just write $[\id,Z_1^+,\dots]$ in a situation where the remaining matrices do not need to be specified. For the same reason, 
\[
\xymatrix{\DQp(\Gamma[\myll])\ar[rrr]^{[\id,Z_1^-,\dots]}&&&\DQp(\Gamma[\myll'])}
\]
with $(G_{\Gamma[\myll]},G_{\Gamma[\myll']})\in\langle \OOO,\IIIp\rangle$. We have
\[
[\id,Z_1^-,\dots]\circ[\id,Z_1,\dots,Z_n]\circ [\id,Z_1^+,\dots]^{-1}=[\id,Z_1+Z_1^--Z_1^+,\dots]
\]
so the square
\[
\xymatrix{\DQp(\Gamma[\kk])\ar[rrr]^{[\id,Z_1,Z_2,\dots,Z_n]}\ar[d]_-{[\id,Z_1^+,\dots]}&&&\DQp(\Gamma[\myll])\ar[d]^-{[\id,Z_1^-,\dots]}\\
\DQp(\Gamma[\kk'])\ar[rrr]_{[\id,0,Z_2',\dots,Z_{n}']}&&&\DQp(\Gamma[\myll'])
}
\]
commutes, with the vertical arrows implemented by \OOO\ and \IIIp\ moves. Arguing similarly, we obtain
\[
\xymatrix{\DQp(\Gamma[\kk'])\ar[rrr]^{[\id,0,Z_2',\dots,Z_n']}\ar[d]_-{[\id,0,(Z_2')^+,\dots]}&&&\DQp(\Gamma[\myll'])\ar[d]^-{[\id,0,(Z_2')^-,\dots]}\\
\DQp(\Gamma[\kk''])\ar[rrr]_{[\id,0,0,Z_3'',\dots,Z_n'']}&&&\DQp(\Gamma[\myll''])
}
\]
etc., and eventually  we arrive at
\[
\xymatrix{\DQp(\Gamma[\kk^{(n)}])\ar[rrr]^{[\id]}&&&\DQp(\Gamma[\myll^{(n)}])
}
\]
which implies $\kk^{(n)}=\myll^{(n)}$ via $\uuu+\rt(\kkk^{(n)})=\uuu+\rt(\mylll^{(n)})$.
\end{proof}

\begin{corol}\label{acyclicxIz}
Within the class of acyclic graphs
\[
\xyzrel{011}=\xyzrel{010}=\langle \KKKm\rangle=\langle \OOO,\IIIm\rangle
\]
and
\[
 \xyzrel{111}=\xyzrel{110}=
\langle \KKKp\rangle =\langle \OOO,\IIIp\rangle
\]
and all  these relations are decidable.
\end{corol}
The identity $\xyzrel{110}=\langle \KKKp\rangle$ follows already by \cite{rhlv:kcguai}, shown by a very different method, namely employing continuity of the invariants. The work in  \cite{rhlv:kcguai} does not address decidability, but we suspect that  the methods presented there could be used to establish several, if not all, of the identities listed above.

\begin{proof}
We just need to prove that the relations are decidable, and this is clear in the first case because isomorphism between maximal \DGs\ obviously is. In the second case, we note that whether or not there exist $Z_1,\dots,Z_n$ satisfying the geometric condition so that 
\[
[\id,Z_1,\dots,Z_n](\uuu+\rt(\kkk))=\uuu+\rt(\mylll)
\]
translates to a linear problem over $\ZZ$ which is hence decidable. Such a map is automatically invertible, so this establishes the needed fact.
\end{proof}

\section{One non-trivial gauge invariant ideal}

\statusbar{\elsewhered}{\elsewhered}{\hered}{\hered}{\elsewhered}{\elsewhered}{\hered}{\hered}

In this final section of our work we establish Conjecture \ref{typeIconj}, and hence all generation conjectures, for all type I graph $C^*$-algebras with exactly one non-trivial gauge-invariant ideal. Up to \xyz{000}-equivalence, there are exactly four graphs defining such $C^*$-algebras, namely
\begin{center}
\begin{tabular}{|c|c|c|c|}\hline
\RRc&\RSc& \SRc& \SSc\\\hline
&&&\\
$\xymatrix{\bullet\ar@(ul,ur)[]\ar[d]\\\bullet\ar@(dl,dr)[]}$&
$\xymatrix{\bullet\ar@(ul,ur)[]\ar[d]\\\circ}$&
$\xymatrix{\circ\ar@{=>}[d]\\\bullet\ar@(dl,dr)[]}$&
$\xymatrix{\circ\ar@{=>}[d]\\\circ}$\\
&&&\\\hline
\end{tabular}
\end{center}

Using the notation \RRc, \SRc, \RSc, \SSc\ to refer to these cases, we note that the \RRc\ case is exactly the class of meteor graphs studied in \cite{lgcegdgrh:wcmg}, and that the \RSc\ case  contains the standard Toeplitz algebra. The case \SRc\ is less well  studied, and since the class \SSc\ is contained in the class classified in the previous section, we already know that Conjecture \ref{typeIconj} holds for it. We will state the key results in the \SSc\ case also, but refer to the previous section for proofs.

We note from the outset that the stable part of the conjecture for \RRc\ has already been established in \cite{lgcegdgrh:wcmg}, by similar methods. We will reprove their result using our machinery in preparation for our analysis of the exact version. 

We also point out that the strict bounds on the ideal lattice in the results presented here are a consequence of the absence of a standard form such as the maximal \DGs\ that we used in the previous section. We are convinced that  no such standard form can exist in general, and that  new approaches must be developed to solve the general case.

\newcommand{\Rc}{``$\bullet$''}

We work with the following collection of model graphs.

\begin{gather}
\xymatrix{
\bullet\ar[r]\ar[d]^-{x_0}\ar[dr]^-{x_1}\ar[drrr]^-{x_{d-1}}&\bullet\ar[r]&\cdots\ar[r]&\bullet \ar `ur^l[lll]`^dr[lll]^k [lll]\\
\bullet\ar[r]&\bullet\ar[r]&\cdots\ar[r]&\bullet\ar[r]&\cdots\ar[r]&\bullet \ar@{-<} `dr_l[lllll]`_ur[lllll]_{\ell} [lllll]
}\label{RRstab}\\
\xymatrix{
\bullet\ar[r]\ar[d]_-{x_0}&\bullet\ar[r]\ar[dl]_-{x_1}&\cdots\ar[r]&\bullet\ar[dlll]^-{x_{k-1}}\ar `ur^l[lll]`^dr[lll]^k [lll]\\
\circ&&
}\label{RSstab}
\end{gather}
\begin{gather}
\xymatrix{
\bullet\ar[r]&\bullet\ar[r]^-k&\bullet\ar[r]&\cdots\ar[r]&\circ\ar@{=>}@/_/[dlll]_(0.8){x_0}\ar@{=>}[dll]_(0.6){x_1}\ar@{=>}[d]^-{x_{\ell-1}}\\
&\bullet\ar[r]&\bullet\ar[r]&\cdots\ar[r]&\bullet \ar@{-<} `dr_l[lll]`_ur[lll]_{\ell} [lll]
}\label{SRstab}\\
\xymatrix{
&\bullet\ar[r]^-k&\bullet\ar[r]&\cdots\ar[r]&\circ\ar@{=>}@/_/[dllll]_(0.8){x_{\ell-1}}\ar@{=>}[dlll]^(0.6){x_{\ell-2}}\ar@{=>}[ld]^-{x_{1}}\ar@{=>}[d]^-{x_{0}}\\
\bullet\ar[r]&\bullet\ar[r]&\cdots\ar[r]&\bullet\ar[r]&\circ
}\label{SSstab}
\end{gather}

We have here $k,\ell\geq 0$ throughout, and $k,\ell>0$ in all the \Rc\ cases, i.e.~when the numbers describe the length of a cycle. It is also required that
not every $x_i$ is zero; and in the \SRc\ and \SSc\ cases, we require that $x_i\in \{0,\infty\}$ as indicated by double arrows.  In case \RRc, $d:=\gcd(k,\ell)$.

\begin{lemma}\label{typeIoiisstd}
When $C^*(G)$ is type I and has exactly one non-trivial gauge-invariant ideal, then $G$ can be transformed into one of the graphs \eqref{RRstab}--\eqref{SSstab}
with \OOO\ and \IIIm\ moves, and their inverses.
\end{lemma}
\begin{proof}
Cases \SRc\ and \SSc\ follow directly from Theorem \ref{OII-simplified}.  In the \RSc\ case we  obtain a graph on the form
\[
\xymatrix{
\bullet\ar[r]\ar[d]_-{x_0}&\bullet\ar[r]\ar[dl]_-{x_1}&\cdots\ar[r]&\bullet\ar[dlll]^-{x_{k-1}}\ar `ur^l[lll]`^dr[lll]^k [lll]\\
\bullet\ar[r]&\cdots\ar[r]&\circ
}
\]
 by Proposition \ref{OII-simplified}. But starting from \eqref{RSstab}, we can make a complete outsplit to
 \[
\xymatrix@C=2.5mm{
\bullet\ar[rr]\ar[d]\ar[dr]&&\bullet\ar[rr]\ar[d]\ar[dr]&&\cdots\ar[r]&\bullet\ar `ur^l[lllll]`^dr[lllll]^k [lllll]\ar[d]\ar[dr]\\
\bullet\ar[d]\ar@{..}[r]^-{x_1}&\bullet\ar[dl]&\bullet\ar@{..}[r]^-{x_2}\ar[dll]&\bullet\ar[dlll]&\cdots&\bullet\ar[dlllll]\ar@{..}[r]^-{x_0}&\bullet\ar[dllllll]\\
\circ&&&
}
\]
 with $\sum x_i$ vertices in the middle layer so that $x_i$ of them receive from vertex number $i-1$ in the cycle. An \IIIm\ move takes this to 
 \[
\xymatrix{
\bullet\ar[r]\ar[d]_-{x_1}&\bullet\ar[r]\ar[dl]_-{x_2}&\cdots\ar[r]&\bullet\ar[dlll]^-{x_{0}}\ar `ur^l[lll]`^dr[lll]^k [lll]\\
\bullet\ar[r]&\circ,
}
\]
which is graph isomorphic to 
 \[
\xymatrix{
\bullet\ar[r]\ar[d]_-{x_0}&\bullet\ar[r]\ar[dl]_-{x_1}&\cdots\ar[r]&\bullet\ar[dlll]^-{x_{k-1}}\ar `ur^l[lll]`^dr[lll]^k [lll]\\
\bullet\ar[r]&\circ.
}
\]
Repeating the process in reverse shows how to arrive at \eqref{RSstab}.

 Applying  Proposition \ref{OII-simplified} in the \RRc\ case only reduces to the case where there are $k+\ell$ vertices organized in cycles as in \eqref{RRstab}, and with no transitional vertices as desired, but with no control on where the edges from the cycle above to the cycle below begin and end. We easily see, though, that when $x_{11}>0$ in the situation
\[
\xymatrix{
\cdots\ar[r]&\bullet\ar[rr]\ar[drr]_(0.3){x_{01}}\ar[d]_{x_{00}}&&\bullet\ar[r]\ar[dll]^(0.3){x_{10}}\ar[d]^{x_{11}}&\cdots\\
\cdots\ar[r]&\bullet\ar[rr]&&\bullet\ar[r]&\cdots
}
\]
(disregarding the multiplicities of edges not shown), we can  out-split to 
\[
\xymatrix{
&&\bullet\ar[ddr]&&\\
\cdots\ar[r]&\bullet\ar[ur]\ar[rr]\ar[drr]_(0.3){x_{01}}\ar[d]_{x_{00}}&&\bullet\ar[r]\ar[dll]^(0.3){x_{10}}\ar[d]^{x_{11}-1}&\cdots\\
\cdots\ar[r]&\bullet\ar[rr]&&\bullet\ar[r]&\cdots
}
\]
and use an \IIIm\ move to obtain
\[
\xymatrix{
\cdots\ar[r]&\bullet\ar[rr]\ar[drr]_(0.3){x_{01}}\ar[d]_{x_{00}+1}&&\bullet\ar[r]\ar[dll]^(0.3){x_{10}}\ar[d]^{x_{11}-1}&\cdots\\
\cdots\ar[r]&\bullet\ar[rr]&&\bullet\ar[r]&\cdots,
}
\]
essentially moving the edge one step leftwards in both cycles, while leaving everything else the same. For easy reference, we call this move \AAArr;\index{Arr@$\AAArr$} we just saw  that $\AAArr\subseteq \langle\OOO,\IIIm\rangle$.

Using \AAArr\ moves, it is straightforward to move all edges so that they emit from the same upstairs vertex, and finding $a,b$ so that $ak+b\ell=d$ we see that moving an edge around the upstairs cycle $ak$ times moves its downstairs position by $d$. Consequently, the form in \eqref{RRstab} is obtained.
\end{proof}

We denote by $\DQ^\RRi(k,\ell,\xx)$, $\DQ^\RSi(k,\xx)$, $\DQ^\SRi(k,\ell,\xx)$, and  $\DQ^\SSi(k,\ell,\xx)$ the dimension quadruples of the four types of models that we now know are generic. We will use $\sigma_j$ with $j\in\{j,k,\ell\}$ to denote the forward cyclic shift of vectors of the appropriate length. 

\begin{lemma}\label{RRkstable}
The  group of $\DQ^{\RRi}(k,\ell,\xx)$  is isomorphic to $\ZZ^{k+\ell}$ ordered so that $(\alal,\bebe)\geq 0$ when all $\alpha_i\geq 0$ and
\[
\beta_j<0\Longrightarrow\exists i:\begin{cases}\alpha_i>0\\x_{j-i+1}>0\end{cases}
\] 
The shift map becomes 
\[
\begin{bmatrix}\sigma_k&\xi\\0&\sigma_\ell
\end{bmatrix}
\]
with 
\[
\xi(\alal)=(x_0\alpha_0,\dots,x_{d-1}\alpha_{d-1},0,\dots,0),
\]
the class of the unit is $(\uu,\uu)$, and the order ideal coincides with the positive cone.
\end{lemma}
\begin{proof}
We have
\[
A=\left[\begin{array}{cccc:cccc}
&1&&&x_0&\cdots&x_{d-1}&\\
&&\ddots&&&\\
&&&1&&&&\\
1&&&&&&\\\hdashline
&&&&&1&&\\
&&&&&&\ddots&\\
&&&&&&&1\\
&&&&1&&
\end{array}\right]
\]
with $A$ invertible over $\ZZ$, so $\mathcal{R}_A=\QQ^{k+\ell}$ and $\Delta _A=\ZZ^{k+\ell}$ acted upon by $A^{-1}$ on the right. Since
\begin{equation}\label{lcmeq}
A^{\operatorname{lcm}(k,\ell)}=\begin{bmatrix}\id&\widetilde{X}\\0&\id\end{bmatrix}
\end{equation}
where each entry $\widetilde{x}_{ij}$ of  $\widetilde{X}$ is just $x_{j-i+1}$ (index reduced modulo $d$), we see that whenever $\alal\geq \oo$ and $\alpha_i>0$ with $x_{j-i+1}>0$, negative $\beta_j$ are freely allowed in positive pairs $(\alal,\bebe)$.
\end{proof}

\begin{propo}\label{RRisostable}
We have $\DQ^\RRi(k,\ell,\xx)\simeq \DQ^\RRi(k',\ell',\xx')$ precisely when $k=k'$, $\ell=\ell'$, and  $\xx'=\sigma_d^m(\xx)$ for some $m$.
\end{propo}
\begin{proof}
Passing to ideal and quotient, we see that $k'=k$ and $\ell'=\ell$. When $\phi:\DQ^\RRi(k,\ell,\xx)\to\DQ^\RRi(k,\ell,\xx')$ is an isomorphism, it must be implemented by a matrix 
\begin{equation}\label{RRformi}
\begin{bmatrix}
S_k^i&Y\\
0&S_\ell^j
\end{bmatrix}
\end{equation}
with $S_m$ the matrix inducing the forward cyclic shift of order $m$. Indeed, the matrices in the diagonal must be permutation matrices to preserve the generators of the positive cone, and hence must be shifts to commute with the given shift map implemented by $A$ above. The lower off-diagonal block is zero by positivity of $\phi$ and $\phi^{-1}$, and since the matrix intertwines the $A$ and $A'$ matrices from $\xx$ and $\xx'$, and hence by \eqref{lcmeq} we have that
\[
\begin{bmatrix}
S_k^i&Y\\
0&S_\ell^j
\end{bmatrix}
\begin{bmatrix}
\id&\widetilde{X}\\
0&\id
\end{bmatrix}
=
\begin{bmatrix}
\id&\widetilde{X'}\\
0&\id
\end{bmatrix}
\begin{bmatrix}
S_k^i&Y\\
0&S_\ell^j
\end{bmatrix}
\]
and hence
\[
S_k^i\widetilde X=\widetilde{X'}S_\ell^i
\]
which implies $\sigma_d^{j-i}(\xx)=\xx'$. 

In the other direction, when $\sigma_d^{j-i}(\xx)=\xx'$ we get that all the $d$ rays of the right hand side of
\begin{equation}\label{RRformii}
S_kY-YS_\ell=S_k^iX-X'S_\ell^j
\end{equation}
have the property that either all entries are zero, or all but two entries are zero, and one is $x_i$ and another $-x_i$. Therefore it is easy to see that there is a solution to \eqref{RRformii} for any given choice of values $y_{00},\dots,y_{0,d-1}$ which is constant along  all the rays that vanish, and attains three different values on the ones that do not. For the map induced by \eqref{RRformi} to be  an order isomorphism we need to know 
\[
x_{j}=0\Longrightarrow y_{0j}=0
\]
because the single positive entry in the left half of the first row of the adjacency matrix is located in the $01$ position. And since the corresponding ray is constant, this is also a sufficient condition for order isomorphism.
\end{proof}

\begin{lemma}\label{RSkstable}
With $\eta:\ZZ^k\to \osX\ZZ$ defined by 
\[
\eta(\begin{bmatrix}\alpha_0&\cdots&\alpha_{k-1}\end{bmatrix})=\oXe{1}{\textstyle \sum \alpha_ix_i,\sum \alpha_{i-1}x_i,\dots,\sum \alpha_{i-k+1}x_i,\dots}
\]
where the first $k$ entries are repeated periodically \emph{ad infinitum}, 
the  group of $\DQ_{\RSi}(k,\xx)$  is isomorphic to
\[
\{(\alal,\bebebe)\in\ZZ^k\times \osX\ZZ\mid \bebebe-\eta(\alal)\in\nsX\ZZ\}
\]
equipped with the canonical order and the shift map $\sigma_k\times\rt$. Under this isomorphism, the class of the unit is
\[
(\uu,\oXe{0}{\textstyle \sum x_i, \sum x_i,\dots})
\]
and the order ideal is the intersection of the  positive cone with $\ZZ^k\times \sum_0^\infty \ZZ$.
\end{lemma}
\begin{proof}
The adjacency matrix is subdivided into 
\[
\left[\begin{array}{c|c}\adjRR_E&\adjRS_E\\\hline\adjSR_E&\adjSS_E\end{array}\right]=\left[\begin{array}{cccc|c}
&1&&&x_0\\
&&\ddots&&\vdots\\
&&&1&x_{k-2}\\
1&&&&x_{k-1}\\\hline&&&&0
\end{array}\right]
\]
with $\adjRR_E$ invertible over $\ZZ$, so $\Delta_{\adjRR_E}=\ZZ^k$ and $\eta$ is as described.

Theorem \ref{DQmaster} shows all claims except the descriptions of the positive cone and the order ideal. We obviously have $\Delta_{\adjRR_E}^+=\NN_0^k$ because $\adjRR_E$ is a permutation matrix, so all generators of the positive cone of the form $(\alal,\eta^+(\alal))=(\alal,\eta(\alal))$ have exclusively nonnegative entries. Since this is also true of $\singgen=\{\oXe{0}{1}\}$, we conclude that any positive element $(\alal,\bebebe)$ must have all $\alpha_i,\beta_j\geq 0$. In the other direction, we fix such an element $(\alal,\bebebe)$ and take $m\geq 0$ so that $\beta_j=\sum \alpha_{i-j+1}x_i$ for all $j\geq km$. We then have that $(\alal,\bebebe)$ agrees with
\[
(\sigma_k\times\rt)^{km}(\alal,\eta(\alal))=\left(\alal,\oXe{km}{\textstyle \sum \alpha_ix_i,\sum \alpha_{i-1}x_i,\dots,\sum \alpha_{i-k+1}x_i,\dots}\right)
\]
from entry $km+1$ and onwards, and we can now adjust on the  finitely many remaining entries with generators of the form $(\sigma_k\times\rt)^{\ell}(\oo,\oXe{ 0}{1})$ to obtain the desired sum.
The description of the order ideal follows.
\end{proof}

\begin{propo}\label{RSisostable}
We have $\DQ^\RSi(k,\xx)\simeq \DQ^\RSi(k',\xx')$ precisely when $k=k'$, and $\xx'=\sigma_k^m(\xx)$ for some $m$.
\end{propo}
\begin{proof}
Passing to the quotient, we get $k=k'$. Now assume $\phi$ induces an isomorphism from $\DQ^\RSi(k,\xx)$ to $\DQ^\RSi(k,\xx')$. We note that
\[
g=(\ee_0,\eta(\ee_0))\qquad h=(\oo,\oXe{0}{1})
\]
generate $\DQ^\RSi(k,\xx)$ as a group with shift, and generate the positive cone as well. Here
\[
\eta(\ee_0)=\oXe{1}{x_0,x_1,\dots,x_{k-1},\dots}
\]
We set $\widetilde g=\phi(g)$ and $\widetilde h=\phi(h)$, and note right away that $\widetilde h=\oXe{0}{1}$ since it must be a minimally positive element, and must generate the order ideal. We can describe the first entry of $\widetilde g$ similarly, but here only know that
\[
\widetilde g=(\ee_{-m},\bebebe)
\]
for some $m$ and $\bebebe$.

By the periodicity of $\eta(\ee_0)$, we get
\[
g-(\sigma_k\times\rt)^{k}(g)=(\oo,\oXe{1}{x_0,\dots x_{k-1}})=\sum_{i=1}^{k-1} x_i(\sigma_k\times\rt)^{i+1}(h),
\]
so it follows that also 
\[
\widetilde g-(\sigma_k\times\rt)^{k}(\widetilde g)=(\oo,\oXe{1}{x_0,\dots x_{k-1}}),
\]
which implies that 
\[
\bebebe=\oXe{1}{x_0,\dots,x_{k-1},\dots}
\]
This must agree eventually with $\eta(\ee_{-m})=\oXe{1}{x'_m,\dots,x'_{m-1},\dots}$, proving that $x'_j=x_{j+m}$ for all $j$. It is easy to see that $\sigma_k^m\times\id$ implements an isomorphism in this case.
\end{proof}

\begin{lemma}\label{SRkstable}
The  group of $\DQ^\SRi(k,\ell,\xx)$  is isomorphic to $\ZZ^\ell \times\nsX\ZZ$ ordered so that $(\alal,\bebebe)\geq 0$ when all $\beta_j\geq 0$ and 
\[
\alpha_i<0\Longrightarrow \exists j:\begin{cases}\beta_j>0\\x_{1+i-j}=\infty\end{cases}
\]
The shift map becomes $\sigma_\ell\times \rt$, the class of the unit is
\[
\left(\uu,\oXe{0}{1,\overbrace{1,\cdots,1}^k}\right)
\]
and the order ideal is the intersection of the  positive cone with $\ZZ^\ell\times \sum_0^k \ZZ$.
\end{lemma}

\begin{proof}
The adjacency matrix is divided into $\left[\begin{smallmatrix}A&B\\C&D\end{smallmatrix}\right]$ as
\[
\left[\begin{matrix}\adjRR_E&\adjRS_E\\\adjSR_E&\adjSS_E\end{matrix}\right]=\left[\begin{array}{cccc:cccc|c}
&1&&&&&&&\\
&&\ddots&&&&&&\\
&&&1&&&&&\\
&&&&&&&&1\\\hdashline
&&&&&1&&&\\
&&&&&&\ddots&&\\
&&&&&&&1&\\
&&&&1&&&&\\\hline
&&&&x_0&x_1&\cdots&x_{\ell-1}&0
\end{array}\right]
\]
where we subdivide $\adjRR_E$ further as indicated by the dashed lines with the first $k\times k$ subblock nilpotent and the second $\ell\times \ell$ subblock invertible. As before, we get that $\Delta_{\adjRR_E}=\ZZ^\ell$ ordered canonically, and $\singgen$ consists of all elements of the form
\[
\left(-\pi_{\mathcal R}({C})(\adjRR_E)^{-1},\oXe 01\right)=\left((-c_1,\dots,-c_{k-1},-c_0),\oXe 01 \right)
\]
where $c_i=0$ whenever $x_i=0$ but otherwise can be any nonnegative number. This shows that $\singgen$ is exactly the elements on the form
\[
(\alal,\oXe{0}{1})
\]
for which all $\alpha_i\leq 0$ and $\alpha_i=0$ whenever $x_{i+1}=0$, and we can determine the order as indicated, arguing as before.
\end{proof}

\begin{propo}\label{SRisostable}
We have $\DQ^\SRi(k,\ell,\xx)\simeq \DQ^\SRi(k',\ell',\xx')$ precisely when $k=k'$, $\ell=\ell'$, and $\xx'=\sigma_\ell^m(\xx)$ for some $m$.
\end{propo}
\begin{proof}
We get that $k=k'$ and $\ell=\ell'$ from passing to the ideal and quotient.  Now assume $\phi$ induces an isomorphism from $\DQ^\SRi(k,\ell,\xx)$ to $\DQ^\SRi(k,\ell,\xx')$. We note that
\[
g=(\ee_0,0)\qquad h=(0,\oXe 01)
\]
generate the group with shift, and are elements of the order ideal  (they do not generate the positive cone). We note that
\[
(\sigma_\ell\times \rt)^\ell(g)=(\id\times\rt^\ell)(g)=g
\]
and consequently the same is true for $\widetilde g:=\phi(g)$, showing that $\widetilde g=(\ee_{-m},0)$ for some $m$. Because  $\widetilde h:=\phi(h)$ is then the only possibility for generating the second component of the order ideal, we get that $\widetilde h=(\alal,\oXe 01)$ for some $\alal$. We argue 
\begin{eqnarray*}
x_i=\infty&\Longrightarrow &\forall n: (n\ee_{i-1},\oXe 01)\geq 0\\
&\Longrightarrow &\forall n: h-n(\sigma_\ell\times \rt)^{i-1}(g)\geq 0\\
&\Longrightarrow &\forall n: \widetilde{h}-n(\sigma_\ell\times \rt)^{i-1}(\widetilde{g})\geq 0\\
&\Longrightarrow &\forall n: (\alal-n\ee_{i-m-1},\oXe 01)\geq 0\\
&\Longrightarrow &x_{i-m}'=\infty
\end{eqnarray*}
showing the necessity of $\xx'=\sigma_\ell^m(\xx)$.

In the other direction, we show that there is a map with images  $\widetilde g=(\ee_{-m},0)$  and $\widetilde h=(\alal,\oXe 01)$ on the form
\[
\begin{bmatrix}
\sigma_\ell^m&0\\
\upsilon_{\alal}&\id
\end{bmatrix}
\]
for any  $\alal$. Indeed, with $\upsilon: \nsX\ZZ\to\ZZ^\ell$ defined by
\begin{equation}\label{SRformi}
\upsilon(\bebebe)=(\textstyle \sum \beta_{i\ell},\dots,\sum\beta_{i\ell+\ell-1})
\end{equation}
we get that $\sigma_\ell\circ \upsilon=\upsilon\circ\rt$, and that $\upsilon(\oXe 01)=\ee_0$, so that $\upsilon$ induces the choice $\alal=\ee_0$. In general,
\begin{equation}\label{SRformii}
\upsilon_{\alal}=\sum_{i=0}^{\ell-1} \alpha_i\sigma_\ell^i\circ \upsilon
\end{equation}
implements the desired group isomorphism. It will be an order isomorphism precisely when $\alal$ has the property
\[
x_i=0\Longrightarrow \alpha_{i+1}=0.
\]
\end{proof}

\begin{theor}\label{oneidealstable}
We have
\[
\xyzrel{011}=\xyzrel{010}=\langle\OOO,\IIIm\rangle=\langle\KKKm\rangle
\]
among all graphs defining Type I $C^*$-algebras with exactly one non-trivial gauge-invariant ideal, and all these relations are decidable. 
\end{theor}

\begin{proof}
We have always
\[
\langle\OOO,\IIIm\rangle\subseteq \xyzrel{011}\subseteq \xyzrel{010}\subseteq \langle\KKKm\rangle
\]
so our task is to show that when $\DQ(G)\simeq \DQ(G')$, it follows that $(G,G')\in\langle\OOO,\IIIm\rangle$, where we may assume that both $G,G'$ are in the model form \eqref{RRstab}--\eqref{SSstab}. This is now obvious from Propositions \ref{RSisostable} and \ref{SRisostable} in the \RSc\ and \SRc\ cases; we have in fact shown the stronger conclusion $(G,G')\in\langle\rangle$. In the \RRc\ case, Proposition \ref{RRisostable} does not establish graph isomorphism, but we note that we can use \AAArr\ moves to shift the multiplicities as required. The conditions in Propositions \ref{RRisostable}, \ref{RSisostable}, and \ref{SRisostable} are obviously decidable.
\end{proof}

We now analyze the exact problems, using model graphs given below.

\begin{gather}
\xymatrix{
\bullet\ar[d]|(0.4)\hole_(0.2){y_0}&\bullet\ar[d]|(0.4)\hole_(0.2){y_1}&&\bullet\ar[d]|(0.4)\hole_(0.2){y_{k-1}}\\
\bullet\ar[r]\ar[d]^-{x_0}\ar[dr]^-{x_1}\ar[drrr]^-{x_{d-1}}&\bullet\ar[r]&\cdots\ar[r]&\bullet \ar `ur^l[lll]`^dr[lll] [lll]\\
\bullet\ar[r]&\bullet\ar[r]&\cdots\ar[r]&\bullet\ar[r]&\cdots\ar[r]&\bullet  \ar@{-<} `dr_l[lllll]`_ur[lllll] [lllll]\\
\bullet\ar[u]|(0.4)\hole_(0.2){z_0}&\bullet\ar[u]|(0.4)\hole_(0.2){z_1}&&\bullet\ar[u]|(0.4)\hole_(0.2){z_{d-1}}&&\bullet\ar[u]|(0.4)\hole_(0.2){z_{\ell-1}}\\
&&&&
}\label{RR}\\
\xymatrix{
\bullet\ar[d]|(0.4)\hole_(0.2){y_0}&\bullet\ar[d]|(0.4)\hole_(0.2){y_1}&&\bullet\ar[d]|(0.4)\hole_(0.2){y_{k-1}}\\
\bullet\ar[r]\ar[d]_-{x_0}&\bullet\ar[r]\ar[dl]_-{x_1}&\cdots\ar[r]&\bullet\ar[dlll]^-{x_{k-1}}\ar `ur^l[lll]`^dr[lll] [lll]\\
\bullet\ar[r]&\bullet\ar[r]&\cdots\ar[r]&\bullet\ar[r]&\circ\\
\bullet\ar[u]_-{z_{\ell-1}}&\bullet\ar[u]_-{z_{\ell-2}}&&\bullet\ar[u]_-{z_{1}}&\bullet\ar[u]_-{z_0}
}\label{RS}\end{gather}
\begin{gather}
\xymatrix{&\bullet\ar[d]_-{y_{k-1}}&\bullet\ar[d]_-{y_{k-2}}&&\bullet\ar[d]_-{y_0}\\
\bullet\ar[r]&\bullet\ar[r]&\bullet\ar[r]&\cdots\ar[r]&\circ\ar@{=>}@/_/[dlll]_(0.8){x_0}\ar@{=>}[dll]_(0.6){x_1}\ar@{=>}[d]^-{x_{\ell-1}}\\
&\bullet\ar[r]&\bullet\ar[r]&\cdots\ar[r]&\bullet \ar@{-<} `dr_l[lll]`_ur[lll] [lll]\\
&\bullet\ar[u]|(0.4)\hole_(0.2){z_0}&\bullet\ar[u]|(0.4)\hole_(0.2){z_1}&&\bullet\ar[u]|(0.4)\hole_(0.2){z_\ell}
}\label{SR}\\
\xymatrix{&\bullet\ar[d]_-{y_{k-1}}&\bullet\ar[d]_-{y_{k-2}}&&\bullet\ar[d]_-{y_0}\\
&\bullet\ar[r]&\bullet\ar[r]&\cdots\ar[r]&\circ\ar@{=>}@/_/[dllll]_(0.8){x_{\ell-1}}\ar@{=>}[dlll]^(0.6){x_{\ell-2}}\ar@{=>}[ld]^-{x_{1}}\ar@{=>}[d]^-{x_{0}}\\
\bullet\ar[r]&\bullet\ar[r]&\cdots\ar[r]&\bullet\ar[r]&\circ\\
\bullet\ar[u]_-{z_{\ell-1}}&\bullet\ar[u]_-{z_{\ell-2}}&&\bullet\ar[u]_-{z_{1}}&\bullet\ar[u]_-{z_0}
}\label{SS}
\end{gather}

As above, we denote the dimension quadruples for these model graphs $\DQp^{**}(k,\ell,\xx,\yy,\zz)$.

\begin{lemma}
When $C^*(G)$ is type I and has exactly one non-trivial gauge-invariant ideal, then $G$ can be transformed into one of the graphs \eqref{RR}--\eqref{SS}
with \OOO\ and \IIIp\ moves, and their inverses.
\end{lemma}
\begin{proof}
The standard forms in \eqref{RS}--\eqref{SS} follow directly by Theorem \ref{III-simplified}, applying shadow form for sources. For \eqref{RR}, we argue as in Lemma \ref{typeIoiisstd}, using that when one 
has $x_{11}>0$ in the situation
\[
\xymatrix{
&\bullet\ar[d]_-{y_0}&&\bullet\ar[d]_-{y_1}\\
\cdots\ar[r]&\bullet\ar[rr]\ar[drr]_(0.3){x_{01}}\ar[d]_{x_{00}}&&\bullet\ar[r]\ar[dll]^(0.3){x_{10}}\ar[d]^{x_{11}}&\cdots\\
\cdots\ar[r]&\bullet\ar[rr]&&\bullet\ar[r]&\cdots\\
&\bullet\ar[u]_-{z_0}&&\bullet\ar[u]_-{z_1}\\
}
\]
an \OOO\ and an \IIIp\ move leads to 
\[
\xymatrix{
&\bullet\ar[d]_-{y_0}&&\bullet\ar[d]_-{y_1}\\
\cdots\ar[r]&\bullet\ar[rr]\ar[drr]_(0.3){x_{01}}\ar[d]_{x_{00}+1}&&\bullet\ar[r]\ar[dll]^(0.3){x_{10}}\ar[d]^{x_{11}-1}&\cdots\\
\cdots\ar[r]&\bullet\ar[rr]&&\bullet\ar[r]&\cdots\\
&\bullet\ar[u]_-{z_0+y_1}&&\bullet\ar[u]_-{z_1+1}.\\
}
\]
Calling this move \AAArrp, we see that it is in $\langle\OOO,\IIIp\rangle$, and we can obtain the model form by moving edges backwards as in the proof of Lemma \ref{typeIoiisstd}.
\end{proof}

Note the more complicated form in \eqref{RS} compared to \eqref{RSstab}. We will see that this is necessary. 

Appealing to our analysis of the stable case, we see that we can reduce the exact isomorphism case to working with automorphisms of $\DQ^{**}(-)$. We write $\operatorname{per}(\xx)$ for the shortest period of $\xx$; the smallest number $p$ so that $p\mid d,k,\ell$ as the case may be, and so that $\xx=\sigma^p(\xx)$ for the relevant cyclic shift.

\begin{propo}\label{RRisoexact}
We have $\DQp^\RRi(k,\ell,\xx,\yy,\zz)\simeq \DQp^\RRi(k,\ell,\xx,\yy',\zz')$ precisely when $\yy'=\sigma^i(\yy)$ for some $i$, and
\[\zz'+\uu=(\yy+\uu)Y+(\zz+\uu)S^{i+a\operatorname{per}(\xx)}
\]
for some $a$, where $Y$ is a $k\times \ell$-matrix satisfying
\[
S_kY-YS_\ell=S_k^iX-XS_\ell^{i+a\operatorname{per}(\xx)}
\]
and
\[
x_{i}=0\Longrightarrow y_{0i}=0
\]
\end{propo}
\begin{proof}
The class of the unit in $\DQp^\RRi(k,\ell,\xx,\yy,\zz)$ is $(\yy+\uu,\zz+\uu)$, and we saw in the proof of Proposition \ref{RRisostable} that all  isomorphisms are of the form $$\begin{bmatrix}S_k^i&Y\\0&S_\ell^j\end{bmatrix}$$ with $Y$ vanishing on ray number $i$ when $x_{i}=0$, and with $S_kY-YS_\ell$ equal to a difference of shifts of $X$ matrices. We proved while establishing   Proposition \ref{RRisostable} that $\sigma_d^{i-j}(\xx)=\xx$, so we get that $\operatorname{per}(\xx)|j-i$, showing the claim.
\end{proof}

\begin{propo}\label{RSisoexact}
We have $\DQp^\RSi(k,\ell,\xx,\yy,\zz)\simeq \DQp^\RSi(k,\ell',\xx,\yy',\zz')$ with $\ell\geq \ell'$ precisely when $\yy'=\sigma_k^{ap+\ell'-\ell}(\yy)$ and $\zz'$ extends $\zz$ by
\[
\textstyle (\sum x_i (y_i+1),\dots,\sum x_{i+\ell-\ell'} (y_i+1))
\] 
\end{propo}
\begin{proof}
The class of the unit in $\DQp^\RSi(k,\ell,\xx,\yy,\zz)$ is 
\[
\left(\yy+\uu,\oXe{0}{\textstyle 1,z_0+1,\dots,z_{\ell-2}+1,z_{\ell-1}+\sum x_i,\sum x_i(y_i+1),\dots,\sum x_{i+k-1}(y_i+1),\dots}\right)
\]
and we saw in the proof of Proposition \ref{RSisostable} that all  isomorphisms are of the form $\sigma_k^m\times\id$, where in the automorphism case we must have  $\operatorname{per}(\xx)|m$. This shows the claim.
\end{proof}

Because $\DQp^\RSi(1,1,(1),(0),(1,0))$ it not isomorphic to any  $\DQp^\RSi(1,0,(1),\yy,\zz)$,  the graph
\[
\xymatrix{
&\bullet\ar[d]&\\
\bullet\ar@(ld,lu)[]\ar[r]&\bullet\ar[r]&\circ}
\]
is not \KKKp\ to any graph without transitional vertices, even though, by Lemma \ref{typeIoiisstd}, it is  \KKKm\ to one.

\begin{propo}\label{SRisoexact}
We have $\DQp^\SRi(k,\ell,\xx,\yy,\zz)\simeq \DQp^\SRi(k,\ell,\xx,\yy',\zz')$ precisely when $\yy'=\yy$ and 
\[
\zz'-\sigma_\ell^{a\operatorname{per}(\xx)}(\zz)\in\operatorname{span}\{\sigma_\ell^i(\widetilde \yy)\mid x_{i-1}=\infty\}
\]
with 
\[
\widetilde{\yy}=\left(\textstyle 1+\sum (y_{i\ell-1}+1),\sum (y_{i\ell}+1),\dots, \sum (y_{i\ell+\ell-2}+1)\right)
\]
\end{propo}
\begin{proof}
The class of the unit in $\DQp^\SRi(k,\ell,\xx,\yy,\zz)$ is 
\[
\left(\zz+\uu,\oXe{0}{1,y_0+1,\dots, y_{k-1}+1}\right)\
\]
and we saw in the proof of Proposition \ref{SRisostable} that all  isomorphisms are of the form
\[
\begin{bmatrix}
\sigma_\ell^m&0\\
\upsilon_{\alal}&\id
\end{bmatrix},
\] 
where $\alpha_i\not=0$ only when $x_{i-1}=\infty$, and where in the automorphism case we must have  $\operatorname{per}(\xx)|m$. Since 
\[
\widetilde{\yy}=\upsilon\left(\oXe{0}{1,y_0+1,\dots, y_{k-1}+1}\right)
\]
(cf. \eqref{SRformi}) the claim follows from reference to \eqref{SRformii}.
\end{proof}

\begin{theor}\label{oneidealexact}
We have
\[
\xyzrel{111}=\xyzrel{110}=\langle\OOO,\IIIp\rangle=\langle\KKKp\rangle
\]
among all graphs defining type I $C^*$-algebras with exactly one non-trivial gauge-invariant ideal, and all these relations are decidable.
\end{theor}
\begin{proof}
The three previous propositions describe to what extent $\ell,\ell'$ and the $\yy,\yy'$ and $\zz,\zz'$ vectors can vary and still give $\DQp^{**}(k,\ell,\xx,\yy,\zz)\simeq \DQp^{**}(k,\ell',\xx,\yy',\zz')$, with $\ell'=\ell$ except in case \RSc. We need to show how to obtain these variations with \OOO\ and \IIIp\ moves.

In case \RRc, we first show how to reduce to the case $i=a=0$. Using \AAArrp\ moves, we go from \eqref{RR} to
\[
\xymatrix{
&\bullet\ar[d]|(0.4)\hole_(0.2){y_0}&\bullet\ar[d]|(0.4)\hole_(0.2){y_1}&&\bullet\ar[d]|(0.4)\hole_(0.2){y_{k-1}}\\
&\bullet\ar[r]&\bullet\ar[r]&\cdots\ar[r]&\bullet \ar `ur^l[lll]`^dr[lll] [lll]\ar[drr]_{x_0}\ar[dllll]^-{x_1}\ar[dl]^-{x_{d-1}}\\
\bullet\ar[r]&\bullet\ar[r]&\cdots\ar[r]&\bullet\ar[r]&\bullet\ar[r]&\cdots\ar[r]&\bullet  \ar@{-<} `dr_l[llllll]`_ur[llllll] [llllll]\\
\bullet\ar[u]|(0.4)\hole_(0.2){z_0'}&\bullet\ar[u]|(0.4)\hole_(0.2){z_1'}&&\bullet\ar[u]|(0.4)\hole_(0.2){z_{d-2}'}&\bullet\ar[u]|(0.4)\hole_(0.2){z_{d-1}'}&&\bullet\ar[u]|(0.4)\hole_(0.2){z_{\ell-1}'},\\
&&&&
}
\]
changing the downstairs multiplicities from sources in a way we do not need to compute. A graph isomorphism takes this back to 
\[
\xymatrix{
\bullet\ar[d]|(0.4)\hole_(0.2){y_{k-1}}&\bullet\ar[d]|(0.4)\hole_(0.2){y_0}&&\bullet\ar[d]|(0.4)\hole_(0.2){y_{k-2}}\\
\bullet\ar[r]\ar[d]^-{x_0}\ar[dr]^-{x_1}\ar[drrr]^-{x_{d-1}}&\bullet\ar[r]&\cdots\ar[r]&\bullet \ar `ur^l[lll]`^dr[lll] [lll]\\
\bullet\ar[r]&\bullet\ar[r]&\cdots\ar[r]&\bullet\ar[r]&\cdots\ar[r]&\bullet  \ar@{-<} `dr_l[lllll]`_ur[lllll] [lllll]\\
\bullet\ar[u]|(0.4)\hole_(0.2){z_{\ell-1}'}&\bullet\ar[u]|(0.4)\hole_(0.2){z_0'}&&\bullet\ar[u]|(0.4)\hole_(0.2){z_{d-2}'}&&\bullet\ar[u]|(0.4)\hole_(0.2){z_{\ell-2}'},\\
&&&&
}
\]
implementing an automorphism of the form
\[
\begin{bmatrix}\sigma_k&\upsilon\\0&\sigma_\ell\end{bmatrix}.
\]
When $\operatorname{per}(\xx)<d$ we can similarly shift the first $\operatorname{per}(\xx)$ downwards arrows $\ell-d$ times back, and employ a graph isomorphism to implement
an automorphism of the form
\[
\begin{bmatrix}\id&\upsilon'\\0&\sigma_\ell^{\operatorname{per}(\xx)}\end{bmatrix}.
\]

After composing with such automorphisms, we obtain without loss of generality that $i=a=0$. When this is the case, the defining property for $Y$ implementing $\upsilon$ is just
\[
SY-YS=0
\]
and we see that $Y$ must be constant on the $d$ rays. With $E_i$ the $k\times\ell$-matrix with entries 1 in ray number $i$ and entries 0 elsewhere, we have consequently have
\[
Y=\sum y_{0i}E_i
\]
with $y_{0i}$ zero whenever $x_{i}=0$. When  $x_{i}>0$, we can implement the effect of $\left[\begin{smallmatrix}\id&E_i\\0&\id\end{smallmatrix}\right]$ on $(\yy+\uu,\zz+\uu)$ by taking one edge among the ones represented by $x_{i+1}$ and sending it back to its original position by $\operatorname{lcm}(k,\ell)$ \AAArrp\ moves, so in this way we see that $\left[\begin{smallmatrix}\id&Y^+\\0&\id\end{smallmatrix}\right]$ and  $\left[\begin{smallmatrix}\id&Y^-\\0&\id\end{smallmatrix}\right]$ can be implemented by \OOO\ and \IIIp\ moves, proving that this is then the case for any automorphism.

In case \RSc, we note that whenever $\ell'>\ell$, we can use \OOO\ and \IIIp\ moves to extend the shorter strand, so that we may assume without loss of generality that $\ell'=\ell$. Then the graphs must be isomorphic.

In case \SRc, we may assume that $a=0$ after implementing a graph isomorphism. When $x_{0}=\infty$, we note that after an outsplit
\begin{gather*}
\xymatrix{&\bullet\ar[d]_-{y_{k-1}}&\bullet\ar[d]_-{y_{k-2}}&&\bullet\ar[d]_-{y_{1}}&\bullet\ar[d]_-{y_0}\ar@/_1.5em/[ddlll]_(0.15){y_0}\\
\bullet\ar[r]&\bullet\ar[r]&\bullet\ar[r]&\cdots\ar[r]&\bullet\ar[r]\ar[dll]&\circ\ar@{=>}@/_/[ddlll]_(0.8){x_0}\ar@{=>}[ddll]_(0.6){x_1}\ar@{=>}[dd]^-{x_{\ell-1}}\\
&&\bullet\ar[d]&&&&\\
&&\bullet\ar[r]&\bullet\ar[r]&\cdots\ar[r]&\bullet \ar@{-<} `dr_l[lll]`_ur[lll] [lll]\\
&&\bullet\ar[u]|(0.4)\hole_(0.2){z_0}&\bullet\ar[u]|(0.4)\hole_(0.2){z_1}&&\bullet\ar[u]|(0.4)\hole_(0.2){z_\ell}
}
\end{gather*}
and zipping up, we go from $\zz$ to $\zz+\widetilde{\yy}$; and more generally, when $x_i=\infty$, we implement the action of $\left[\begin{smallmatrix}\id&0\\\upsilon_{\ee_{i}}&\id\end{smallmatrix}\right]$. This shows the claim by implementing ${\alal^+}$ on one side and  ${\alal^-}$ on the other as above.
\end{proof}

\appendix

\renewcommand{\thetable}{A.\arabic{table}}   

\chapter{Leavitt path algebras}\label{LPA}
\section{Algebraic \xyz{xyz}-equivalence}

In this section, we provide a succinct overview of some immediate consequences of our work to the problem of classifying Leavitt path algebras. We address it to readers already familiar with this topic.

\begin{defin}\label{xyzLPAdef}
With $\xyz{x}, \xyz{y}, \xyz{z} \in \{ \xyz{0}, \xyz{1}\}$ we 
say that two graphs $E$ and $F$ with finitely many vertices are  \emph{algebraically \xyz{xyz}-equivalent} when there exists a ring isomorphism $\phi \colon L_{\CC}(E) \otimes M_\infty({\CC})\to L_{\CC}(F)\otimes M_\infty({\CC})$ which additionally satisfies 
\begin{itemize}
\item $\phi ( 1_{ L_{\CC}(E)} \otimes e_{11} ) = 1_{ L_{\CC}(F)} \otimes e_{11}$ when $\xyz{x} = \xyz{1}$

\item $\phi(   L_{\CC}(F)_n \otimes  M_\infty({\CC})) = L_{\CC}(E)_n \otimes  M_\infty({\CC})$ for all $n\in\ZZ$ when $\xyz{y} = \xyz{1}$

\item $\phi ( {D}_E \otimes c_0 ) = {D}_F \otimes c_0$ when $\xyz{z} = \xyz{1}$.
\end{itemize}
\end{defin}

We use the notation $\xyzLPArel{xyz}$ to refer to the equivalence relation among graphs with finitely many vertices defined this way. We now note:

\begin{propo}\mbox{}\\
\begin{enumerate}[(i)]
\item $\xyzrel{xy1}=\xyzLPArel{xy1}$
\item $ \xyzLPArel{x00}\subseteq \xyzrel{x00}$
\item $\xyzLPArel{010}\subseteq\langle\KKKm\rangle$
\item $\xyzLPArel{110}\subseteq\langle\KKKp\rangle$
\end{enumerate}
 \end{propo}
 \begin{proof}
 The four claims in (i) are contained in Corollaries 4.2, 4.5, 4.7, 4.8 of \cite{tmcjr:dpgisa}. The two claims in (ii) follow from the classification result in \cite{segrerapws:ccuggs} as explained in \cite{segrerapws:fkga}. The claim in (iv) is contained in Proposition 5.7 of \cite{parhhlas:gsar}, and (iii) follows in turn.
 \end{proof}

 Applying (i)-(iv) only, we may now extract equalities between algebraic \xyz{xyz} classes and equivalence classes generated by moves from our previously noted results. We present these in Tables \ref{LPAgen}, \ref{LPAone}, \ref{LPAnone}, organized by the restrictions put (here and in the original $C^*$-algebraic versions) on the ideal lattices. When necessary, further restrictions are listed in the rightmost column.
 
 \begin{remar}
Carlsen and Rout study relations of the form \xyzLPArel{xy1} for any choice $R$ of a commutative integral domain in place of $\CC$, and show that this relation does not depend on the choice of $R$. Our observations --  summarized in Corollary \ref{xyzdiffer} and Table \ref{xyzrelations} -- that all $C^*$-algebraic relations are different show that the four algebraic relations \xyzLPArel{xy1} differ, but there  are no examples known establishing that $\xyzLPArel{xy0}\not=\xyzLPArel{xy1}$. It is similarly unknown if the relations 
 $\xyzLPArel{xy0}$ defined above with $\CC$ coincide with their counterparts over all other rings.
 
  It is further noted in \cite{tmcjr:dpgisa} that our $C^*$-algebraic relation \xyzrel{xy1} coincides not only with \xyzLPArel{xy1} defined for general $R$, but also with notions of isomorphism where the Leavitt path algebras are considered as  $*$-algebras. Taking $R=\ZZ$ is particularly interesting in our context, as it is known (\cite{tmc:lpaz}, \cite{rjapws:csplpaz}) that all notions with \xyz{xy0} coincides with those for \xyz{xy1}, so that our moves \CCCp, \PPPp\ fail to be \xyz{00z}-invariant and  \KKKp\ fails  to be  \xyz{01z}-invariant. But it also follows that our generation conjectures (\ref{gencon}.1), (\ref{gencon}.3), (\ref{gencon}.5), (\ref{gencon}.7) based on the remaining moves are equivalent to those for direct or graded $*$-isomorphism amongst $L_\ZZ(E)$, and consequently our partial results transfer verbatim.

  
  
  
  \end{remar}
 
 \setlength{\extrarowheight}{1.7mm}
 \begin{table}
 \begin{center}
 \begin{tabular}{|c|c|c|}\hline
 \ref{thm:outsplitting}&$\OOO\subseteq \xyzLPArel{111}$&\\\hline
 \ref{thm:insplittingplus}&$\IIIp\subseteq \xyzLPArel{111}$&\\\hline 
 \ref{thm:insplitting2}&$\IIIm\subseteq \xyzLPArel{011}$&\\\hline 
 \ref{thm:reduction-plus}&$\RRRp\subseteq \xyzLPArel{101}$&\\\hline 
 \ref{thm:source}&$\SSS\subseteq \xyzLPArel{001}$&\\\hline 
 \ref{finitecomplete}, \ref{cyclecomplete},  \ref{cuntzcomplete} &$\xyzrel{xyz}=\xyzLPArel{xyz}$&$\{\GF(\nn),\GC(\nn),\GO(c,n)\}$\\\hline
   \ref{OOI-cr}&$\xyzLPArel{001}=\langle\OOO,\IIIm,\RRRp,\SSS\rangle$&regular\\\hline
  \ref{OOI-cr}&$\xyzLPArel{011}=\langle\OOO,\IIIm\rangle$&regular\\\hline  
  \ref{OOI-cr}&$\xyzLPArel{111}=\langle\OOO,\IIIp\rangle$&regular\\\hline  
  \ref{OOI-cr}&$\xyzLPArel{010}\subseteq\langle\KKKm\rangle\subseteq\xyzLPArel{001}$&regular\\\hline
  \ref{OOI-cr}&$\xyzLPArel{110}\subseteq\langle\KKKp\rangle\subseteq\xyzLPArel{101}$&regular\\\hline
    \ref{classifyamplified}&$\xyzLPArel{000}=\xyzLPArel{001}=\xyzLPArel{100}=\xyzLPArel{101}=\langle\OOO,\RRRp\rangle$&amplified\\\hline  
  \ref{classifyamplified}&$\xyzLPArel{010}=\xyzLPArel{011}=\xyzLPArel{110}=\xyzLPArel{111}=\langle\rangle$&amplified\\\hline  
    \ref{monoxOz}&$\xyzLPArel{000}=\xyzLPArel{001}=\langle\OOO,\IIIm,\RRRp,\SSS\rangle$& monocyclic\\\hline  
  \ref{monoxOz}&$\xyzLPArel{100}=\xyzLPArel{101}=\langle\OOO,\IIIp,\RRRp\rangle$& monocyclic\\\hline 
  \ref{acyclicxIz}&$\xyzLPArel{010}=\xyzLPArel{011}=\langle\OOO,\IIIm\rangle=\langle\KKKm\rangle$&acyclic\\\hline 
  \ref{acyclicxIz}&$\xyzLPArel{110}=\xyzLPArel{111}=\langle\OOO,\IIIp\rangle=\langle\KKKp\rangle$&acyclic\\\hline 
                             \end{tabular}
 \end{center}
 \caption{Identities for graphs defining $C^*$-algebras with arbitrary ideal lattices}\label{LPAgen}
 \end{table}

 \begin{table}
  \begin{center}
 \begin{tabular}{|c|c|c|}\hline
  \ref{oneidealstable}&$\xyzLPArel{010}=\xyzLPArel{011}=\langle\OOO,\IIIm\rangle=\langle\KKKm\rangle$&monocyclic\\\hline 
  \ref{oneidealexact}&$\xyzLPArel{110}=\xyzLPArel{111}=\langle\OOO,\IIIp\rangle=\langle\KKKp\rangle$&monocyclic\\\hline   
                            \end{tabular}
 \end{center}
   \caption{Identities for graphs defining $C^*$-algebras  with exactly one non-trivial ideal}\label{LPAone}
 \end{table}
 
 \begin{table}
  \begin{center}
 \begin{tabular}{|c|c|c|}\hline
  \ref{OOI-simple}&$\xyzLPArel{001}=\langle \OOO,\IIIm,\RRRp,\SSS\rangle$&\\\hline
  \ref{OIOtoOOI}&$\xyzLPArel{010}\subseteq\langle\KKKm\rangle\subseteq\xyzLPArel{001}$&\\\hline

  \ref{IOI-cr}&$\xyzLPArel{101}=\langle\OOO,\IIIp,\RRRp\rangle$&regular\\\hline  

  \ref{singxOz}&$\xyzLPArel{000}=\xyzLPArel{001}=\xyzLPArel{100}=\xyzLPArel{101}=\langle\OOO,\RRRp\rangle$&singular\\\hline  
                             \end{tabular}
 \end{center}
  \caption{Identities for graphs defining simple $C^*$-algebras}\label{LPAnone}
 \end{table}
 
\newcommand{\etalchar}[1]{$^{#1}$}

\begin{theindex}

  \item admissible, 11
  \item $\adjRR_E$, 51
  \item $\adjRS_E$, 51
  \item $\adjSR_E$, 51
  \item $\adjSS_E$, 51
  \item $\mathsf{A}^\bullet_E$, 12
  \item $\mathsf{A}^\circ_E$, 12
  \item amplified graphs, 107
  \item anchor, 80
  \item antenna form, 82
  \item $\AAArr$, 125
  \item augmented
    \subitem canonical form, 68
    \subitem standard form, 68

  \indexspace

  \item $\mathsf{B}^\bullet_E$, 12
  \item $\BF(E)$, 71
  \item Bowen-Franks invariant, 71

  \indexspace

  \item $c_0$, 10
  \item canonical form, 13
  \item (CK1)--(CK3), 10
  \item conjugacy, 70
  \item continuous orbit equivalence, 74
  \item \CCCp, 29
  \item $C^*$-algebra over $X$, 13
  \item Cuntz-Krieger families, 10
  \item Cuntz-Krieger relations, 10

  \indexspace

  \item decorated graph, 112
  \item $\Delta_A$, 51
  \item $\delta_A$, 51
  \item $\Delta_A^+$, 51
  \item diagonal, 10
  \item dimension quadruple, 36
  \item $\DQ(-)$, 36
  \item $\DQp(-)$, 36
  \item $\DT(-)$, 36

  \indexspace

  \item $\GC(\nn)$, 61
  \item essential
    \subitem graph, 69
    \subitem part, 69
  \item $\eta$, 52
  \item $\eta^+$, 56
  \item eventual conjugacy, 74

  \indexspace

  \item $\GF(\nn)$, 61
  \item $\FKX(-)$, 13
  \item $\FKXp(-)$, 13
  \item $\FKXpu(-)$, 14
  \item flow equivalence, 71

  \indexspace

  \item $\GO(\nn)$, 61
  \item $\singgen$, 57
  \item generalized components, 12
  \item generation conjecture, 90
  \item \GL-equivalence, 12
  \item $\GLp$-equivalence, 14

  \indexspace

  \item hereditary vertex sets, 11

  \indexspace

  \item \IIIm, 20
  \item \IIIp, 21
  \item irreducible graph, 70

  \indexspace

  \item $\K$, 10
  \item $\KKKm$, 36
  \item $\KKKp$, 36

  \indexspace

  \item matched graphs, 12

  \indexspace

  \item $\nor$, 51
  \item $\nos$, 51

  \indexspace

  \item \OOO\ move, 17
  \item obstruction class, 38
  \item $\nDT_E$, 52
  \item order ideal, 36

  \indexspace

  \item $p_0^E$, 33
  \item \PPPp, 31
  \item primitive graph, 70

  \indexspace

  \item regular
    \subitem graph, 10
    \subitem vertex, 10
  \item \RRRp, 27
  \item $\lt$, 33

  \indexspace

  \item \SSS, 29
  \item saturated vertex sets, 11
  \item shadow form, 82
  \item shift equivalence, 71
  \item shift map, 69
  \item shift of finite type, 69
  \item singular
    \subitem graph, 10
    \subitem vertex, 10
  \item $\sigma_k$, 64
  \item skew product, 33
  \item $\SL$-equivalence, 68
  \item $\SLp$-equivalence, 68
  \item spectrum away from zero, 71
  \item $\spaz(E)$, 71
  \item $\SSSs$, 83
  \item standard form, 13
  \item $\standard$, 69
  \item $\standardp$, 69
  \item strand, 80
  \item strand length, 112
  \item strong shift equivalence, 70

  \indexspace

  \item transitional vertex, 79
  \item $\mbox  {\texttt  {\textup  {(T)}}}$, 107

  \indexspace

  \item weighted, decorated graph, 111

  \indexspace

  \item $\mathrm{XK}\delta(-)$, 38
  \item $\mathrm{XK}\delta^+(-)$, 39
  \item $\mathrm{XK}\delta^{+,1}(-)$, 39
  \item $\mathrm{XK}_*(-)$, 38
  \item $\xyzrel{xyz}$, 10
  \item $\xyz{xyz}$-equivalence, 10

\end{theindex}

\end{document}